\documentclass[11pt]{article}
\usepackage{amsmath}                     
\usepackage{mathrsfs}                    
\usepackage{enumerate}                   
\usepackage{amsthm}                      
\usepackage{amssymb}                     
\usepackage[matrix,arrow,curve]{xy}      
\usepackage{dsfont}                      
\usepackage{bm}                          
\usepackage{color}                       
\usepackage{tikz}                        
\usepackage{graphicx}                    

\theoremstyle{plain}                 
\newtheorem{thm}{Theorem}[subsection]
\newtheorem{lem}[thm]{Lemma}
\newtheorem{prop}[thm]{Proposition}
\newtheorem{cor}[thm]{Corollary}

\theoremstyle{definition}           
\newtheorem{defi}[thm]{Definition}
\newtheorem{ex}[thm]{Example}
\newtheorem{rem}[thm]{Remark}
\numberwithin{equation}{section}

\begin{document}
\thispagestyle{empty}
{
\begin{center}
\includegraphics[width=60pt]{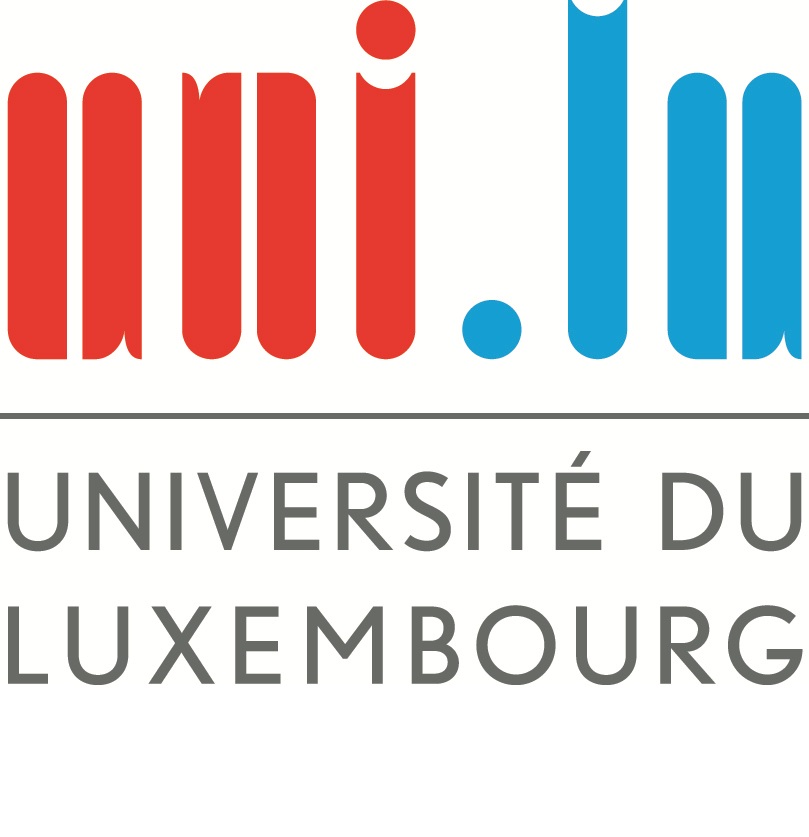}\\
PhD-FSTC-2013-22\\
The Faculty of Sciences, Technology and Communication\\
\vspace{60 pt}
\huge
DISSERTATION\\
\vspace{10 pt}
\normalsize
Presented on 11/09/2013 in Luxembourg\\
\vspace{5 pt}
to obtain the degree of \\
\vspace{20pt}
\huge
DOCTEUR~DE~L'UNIVERSIT\'E~DU~LUXEMBOURG\\
EN MATH\'EMATIQUES\\
\normalsize
\vspace{10 pt}
by\\
\vspace{10 pt}
\LARGE
\baselineskip=1pt David Khudaverdyan\\ \vspace{3 pt}
{\tiny Born on 12 October 1986 in Yerevan (Armenia)}\\\vspace{5 pt}
\huge
\textbf{Higher Lie and Leibniz algebras}\\

\end{center}
\vspace{3.5cm}
{\large
Doctoral committee}\\\\\\\\\\\\\\\\\\\\\\\\
{\normalsize Prof. Dr. Martin Schlichenmaier, Chairman}\\\baselineskip=1pt {\tiny\emph{University of Luxembourg, Luxembourg}}\\\\\\\\\\\\\\
Prof. Dr. Dimitry Leites, Deputy chairman\\
{\tiny\emph{University of Stockholm, Sweden} } \\\\\\\\\\\\\\
Prof. Dr. Norbert Poncin, Supervisor\\
{\tiny\emph{University of Luxembourg, Luxembourg}}\\\\\\\\\\\\\\
Prof. Dr. Valentin Ovsienko, Referee\\
{\tiny\emph{CNRS, Institute Camille Jordan, University Claude Bernard Lyon 1, France}}\\\\\\\\\\\
Dr. Vladimir Dotsenko, Referee \\
{\tiny\emph{Trinity College Dublin, Ireland}}

}
\newpage
\
\thispagestyle{empty}
\newpage
\thispagestyle{empty}
\section*{Acknowledgements}
First and foremost I would like to thank my Ph.D. advisor Norbert Poncin for giving generously of his time and energy to help me learn Mathematics. Special thanks also to Vladimir Dotsenko for reading the survey part of my manuscript and making valuable suggestions. I would like to thank Professors Martin Schlichenmaier, Dimitry Leites, and Valentin Ovsienko for accepting to be members of my Dissertation Committee. Thanks to the University of Luxembourg for welcoming me in the warmest way, as well as for financial support.

\newpage
\
\thispagestyle{empty}
\newpage
\tableofcontents
\newpage
\section{Introduction}
Higher structures -- infinity algebras and other objects up to homotopy, categorified algebras, `oidified' concepts, operads, higher categories, higher Lie theory, higher gauge theory $...$ -- are currently intensively investigated in Mathematics and Physics. \medskip

The present thesis deals with abstractions and flexibilizations of algebraic structures, and more precisely, with algebraic operads, homotopification and categorification.

\subsection{Algebraic operads}

Operads can be traced back to the fifties and sixties -- let us mention the names of Boardman, MacLane, Stasheff, Vogt $...$ They were formally introduced by J. P. May in \cite{May72}, who also suggested the denomination `operad' (contraction of `operation' and `monad'). Operads are to algebras, what algebras are to matrices, or, better, to representations. More precisely, an operad encodes a type of algebra. It heaves the algebraic operations of the considered type, their symmetries, their compositions, as well as the specific relations they verify, on a more abstract and universal level, which is best thought of by viewing a universal abstract operation as a kind of tree with a finite number of leaves (or inputs) and one root (or output). To `any' type of algebra one can associate an operad; the representations of this operad form a category, which is equivalent to the category of algebras of the considered type.\medskip

Operads were initially studied as a tool in homotopy theory, but found some thirty years later interest in a number of other domains like homological algebra, category theory, algebraic geometry, mathematical physics $...$ Among various powerful aspects of operads, let us mention that the operadic language simplifies not only the formulation of the mathematical results but also their proofs, that it allows to gain a more conceptual and deeper insight into classical theorems and to extend them to other types of algebra $...$ ; e.g. if some construction is possible `mutatis mutandis' for several types of algebra, it is a very enriching challenge to prove that it goes through for operads.

\subsection{Homotopy algebras}

Homotopy, sh, or infinity algebras \cite{Sta63} are homotopy invariant extensions of differential graded algebras. They appear for instance when examining whether a compatible algebraic structure on some chain complex can be transferred to a homotopy equivalent complex $(V,d_V)$. For differential graded associative or Lie algebras, the naturally constructed product on $V$ is no longer associative or Lie, but verifies the associativity or Jacobi condition up to homotopy. More precisely, we obtain a sequence $m_n:V^{\times n}\to V$, $n\in \mathbb{N}^*$, of multilinear maps on the graded vector space $V$ that (have degree $n-2$ and possibly some symmetry properties) verify a whole sequence of defining relations. We refer to these data as a (an associative or Lie) homotopy algebra structure on $V$.\medskip

Their homotopy invariance, explains the origin of infinity algebras in {\small BRST} formalism of closed string field theory. One of the prominent applications of Lie infinity algebras ($L_{\infty}$-algebras) \cite{LS93} is their appearance in Deformation Quantization of Poisson manifolds. The deformation map can be extended from differential graded Lie algebras ({\small DGLA}s) to $L_{\infty}$-algebras and more precisely to a functor from the category {\tt L}$_{\infty}$ to the category {\tt Set}. This functor transforms a weak equivalence into a bijection. When applied to the {\small DGLA}s of polyvector fields and polydifferential operators, the latter result, combined with the formality theorem, provides the 1-to-1 correspondence between Poisson tensors and star products.\medskip

Homotopy Lie and Leibniz algebras are central concepts of the present text.

\subsection{Categorified algebras}

- Categorification \cite{CF94}, \cite{Cra95} is characterized by the replacement of sets (resp., maps, equations) by categories (resp., functors, natural transformations). Rather than considering two maps as equal, one details a way of identifying them. Categorified Lie algebras were introduced by J. Baez and A. Crans \cite{BC04} under the name of Lie 2-algebras. Such an algebra is made up by a category $L$, a functor $[-,-]:L\times L\to L$, and a natural transformation $J_{x,y,z}:[x,[y,z]]\to [[x,y],z]+[y,[x,z]]$ that verifies some coherence law. The generalization of Lie 2-algebras, weak Lie 2-algebras, was studied by D. Roytenberg \cite{Roy07}. Other generalizations, Leibniz 2-algebras and Lie 3-algebras, are part of the present work.\medskip

Categorification is thus a sharpened viewpoint that leads to astonishing results in {\small TFT}, bosonic string theory $...$ Considerable effort with convincing output has been made by J. Baez and his school to show that categorification leads to a deeper understanding of Physics.\medskip

- The preceding `vertical' categorification has a `horizontal' counterpart: if a notion can be interpreted as some category with a single object, it might be generalized, horizontally categorified, or oidified, by considering many object categories of the same type. Well-known examples of this process are e.g. Lie groupoids and Lie algebroids. The concept of groupoid was introduced in 1926 by H. Brandt; topological and differential groupoids go back to C. Ehresmann and are transitive Lie groupoids in the sense of A. Kumpera and A. Weinstein. Lie algebroids were first considered by J. Pradines \cite{Pr}, following works by C. Ehresmann and P. Libermann.\medskip

Lie groupoids and algebroids are of importance in particular in view of their relations with Poisson Geometry and the theory of connections. A generalization of Lie algebroids to the Leibniz case can be found in this work.

\subsection{Structure and main results}

This dissertation is divided into five parts: a survey work on the operadic approach to homotopy algebras, three research papers, as well as an Appendix containing additional information.

\subsubsection*{Operadic approach to homotopy algebras}

After the introduction of homotopy algebras, it was understood quite quickly that the maps $m_n$ of a homotopy associative or Lie algebra on a space $V$ can be viewed as the corestrictions of a coderivation on the free graded associative or commutative coalgebra over the suspended space $sV$ and that, more surprisingly, the above-mentioned sequence of defining relations can be encrypted in the unique requirement that this coderivation be a codifferential: a (an associative or Lie) homotopy algebra can be interpreted as a codifferential of an appropriate coalgebra.\medskip

In their famous paper on `Koszul duality for operads' \cite{GK94}, V. Ginzburg and M. Kapranov gave a conceptual approach to a broad family of homotopy algebras and extended the preceding interpretation to any type of homotopy algebra whose corresponding algebra type can be encoded in a so-called Koszul operad. This operadic approach to homotopy algebras is the main aspect of the survey part of this thesis. An extensive treatment of algebraic operads is provided by a recent textbook \cite{LodayVallette}; a large part of background section follows selected parts of that textbook.

\subsubsection*{Higher categorified algebras versus bounded homotopy algebras}

The work `Higher categorified algebras versus bounded homotopy algebras' (vertical categorification and homotopification of Lie algebras) was transmitted by J. Stasheff to the journal 'Theory and Applications of Categories' -- Theo. Appl. Cat., 25(10) (2011), 251-275.\medskip

A categorical definition of Lie 3-algebras is given and their 1-to-1 correspondence with 3-term Lie infinity algebras whose bilinear and trilinear maps vanish in degree $(1,1)$ and in total degree 1, respectively, is proven. Further, an answer to a question of Roytenberg pertaining to the use of the nerve and normalization functors in the study of the relationship between categorified algebras and truncated sh algebras, is provided.\medskip

We refer the reader to the introduction of the paper for more details -- see page~\pageref{HigherCatAlgVSBoundHomAlg}.

\subsubsection*{The Supergeometry of Loday Algebroids}

The paper `The Supergeometry of Loday Algebroids' (horizontal categorification of Leibniz (Loday) algebras) has recently been published in the `Journal of Geometric Mechanics' -- J. Geo. Mech., 5(2) (2013), 185-213 (American Institute of Mathematical Sciences).\medskip

A new concept of Loday algebroid (and its pure algebraic version -- Loday pseudoalgebra) is proposed and discussed in comparison with other similar structures present in the literature. The structure of a Loday pseudoalgebra and its natural reduction to a Lie pseudoalgebra are studied. Further, Loday algebroids are interpreted as homological vector fields on a `supercommutative manifold' associated with a shuffle product, and the corresponding Cartan calculus is introduced. Several examples, including Courant algebroids, Grassmann-Dorfman and twisted Courant-Dorfman brackets, as well as algebroids induced by Nambu-Poisson structures, are given.\medskip

For further details, we refer the reader to the introduction of this paper -- see page~\pageref{SupergeomLodAlebroids}.

\subsubsection*{On the infinity category of homotopy Leibniz algebras}

The third article `On the infinity category of homotopy Leibniz algebras' (homotopification, vertical categorification, higher category theory) is being published in the ArXiv preprint database and submitted for publication in a peer-reviewed international journal.\medskip

Various concepts of $\infty$-homotopies, as well as the relations between them (focussing on the Leibniz type) are discussed. In particular $\infty$-$n$-homotopies appear as the $n$-simplices of the nerve of a complete Lie ${\infty}$-algebra. In the nilpotent case, this nerve is known to be a Kan complex \cite{Get09}. The authors argue that there is a quasi-category of $\infty$-algebras and show that for truncated $\infty$-algebras, i.e. categorified algebras, this $\infty$-categorical structure projects to a strict 2-categorical one. The paper contains a shortcut to $(\infty,1)$-categories, as well as a review of Getzler's proof of the Kan property. The latter is made concrete by applying it to the 2-term $\infty$-algebra case, thus recovering the concept of homotopy of \cite{BC04}, as well as the corresponding composition rule \cite{SS07Structure}. An answer to a question of \cite{BS07} about composition of $\infty$-homotopies of $\infty$-algebras is provided.\medskip

Again the reader is referred to the introduction of the paper for additional information -- see page~\pageref{InfCatHomLeibAlg}.

\subsubsection*{Appendix}

In the appendix, we give proofs that were considered too technical to be part of a research paper.

\newpage
\section{Operadic approach to homotopy algebras}

A good set-up for working with algebraic structures involves operads. This means that one considers all $n$-ary operations that can be obtained from some elementary `building blocks', with all possible structures this collection naturally has. In the setting of homotopical algebra, identities between those operations only hold up to a hierarchy of coherent homotopies. Determining that hierarchy is an important problem, and there are several ways to approach it. One method that works in a large class of examples is described in the paper on Koszul duality for operads, by Ginzburg and Kapranov, which puts many disjoint ideas of homotopy theory in a uniform context.\\

We start our survey recalling the concepts of algebra and coalgebra. Then we introduce the bar and cobar constructions. We build the bar-cobar resolution of an algebra. Then we restrict ourselves to quadratic algebras, introduce the notions of Koszul (co)algebra and Koszul dual (co)algebra, and construct a `smaller' resolution -- the minimal model.\\

We give the classical and functorial definitions of an operad. The functorial one is similar to that of associative algebras. We adapt Koszul duality theory to this operadic framework. For each Koszul operad $P$, we define the $P_\infty$ operad as the cobar construction of the Koszul dual cooperad of $P$. Finally we explain the Ginzburg-Kapranov theorem, which gives a compact characterization of  $P_\infty$-algebras by codifferentials on certain coalgebras.\\


\subsection{Associative algebras and coalgebras}
\subsubsection{Notations and sign convention. Graded vector spaces.}
Here we fix notations and give some basic definitions. We work over a field $K$ of characteristic zero.
\begin{defi}
A vector space $V$ is {\it $\mathbb{Z}-$graded (or just graded)} if there exists a collection of vector spaces $\{V_i\},i\in\mathbb{Z}$, such that
\begin{equation*}
V=\bigoplus\limits_{i\in\mathbb{Z}} V_i.
\end{equation*}
 Any element $a\in V_i$ is called a {\it homogeneous element of degree $i$}. The degree of this element is denoted by $\overline{a}$ or $|a|$.
\end{defi}
The tensor product of two graded vector spaces has a natural grading. The degree of the tensor product of two homogeneous elements is equal to the sum of their degrees:
$$
\overline{v\otimes w}=\overline{v}+\overline{w},\quad v\in V_i, w\in W_j.
$$
We dualize the graded vector space degree by degree:
$$
V^*:=\bigoplus\limits_{i\in\mathbb{Z}} V^*_i.
$$
To the elements of $V^*_i$ we assign the degree $-i$. This kind of the dualization is not the `honest' dualization $\mathrm{Hom}(V,K)$, but in the case if $V$ is finite dimensional $V^*\simeq \mathrm{Hom}(V,K)$.
\begin{defi}
The {\it tensor module} over $V$ is, by definition, a direct sum
$$
T(V):=K\oplus V\oplus V^{\otimes 2}\oplus V^{\otimes 3}\oplus...
$$
and the {\it reduced tensor module} over $V$ is defined as follows:
$$
\overline{T}(V):=V\oplus V^{\otimes 2}\oplus V^{\otimes 3}\oplus...
$$
\end{defi}
The vectors of the (reduced) tensor module which lie in the $n$-th tensor power of $V$: $v_1\otimes...\otimes v_n\in V^{\otimes n}$ are said to be of {\it weight} $n$. The vectors $v_1\otimes...\otimes v_n$ we usually denote by $v_1...v_n$, if it doesn't cause an ambiguity.
\begin{defi}
A linear map $f$ between graded vector spaces $V$ and $W$ is called a {\it homogeneous map} of a degree  $\overline{f}=k$, if it adds $k$ to the degree of any homogeneous element: $fV_i\in W_{i+k}$.
\end{defi}
For any homogeneous graded maps $f:V\to W$ and $f':V'\to W'$ their tensor product $f\otimes f':V\otimes V'\to W\otimes W'$ is defined in the following way:
\begin{equation}\label{KoszulSign1}
(f\otimes f')(v\otimes v')=(-1)^{\overline{f'}\cdot\overline{v}}fv\otimes f'v'.
\end{equation}
The transposed map $f^*: W^*{\longrightarrow}V^*$ defined as follows:
\begin{equation}\label{KoszulSign2}
\begin{array}{l}
(f^*\alpha)v=(-1)^{\overline{f}\cdot\overline{\alpha}}\alpha (fv) \qquad  v\in V,\alpha \in W^*.\\
\end{array}
\end{equation}
Note that the dual spaces $V^*$ and $W^*$ are dualized, by definition, degree by degree, that is in general $V^*$ and $W^*$ are not isomorphic to $\mathrm{Hom}(V,K)$ and, respectively, $\mathrm{Hom}(W,K)$, nevertheless the map $f^*$ is well defined. It follows from the fact that the map $f$ has a certain degree $k$, and the map  $f^*$ sends all vectors from the vector space $(V_i)^*$ to the  $(V_{i+k})^*$.\\
The way how the signs chosen in formulas (\ref{KoszulSign1}) and  (\ref{KoszulSign2}) is called {\it the Koszul sign convention}.
From this sign convention immediately follows that the transposition of composed maps
\begin{equation}\label{rule_for_transposition_of_comp}
(fg)^*=(-1)^{\overline{f}\cdot\overline{g}}g^*f^*
\end{equation}
and no sign appears in the transposition of tensor product:
\begin{equation}\label{rule_for_transposition_of_tens}
(f\otimes g)^*=f^*\otimes g^*.
\end{equation}
Note that in the last two formulas the argument's degree doesn't appear in the sign. That is one of the reasons that the Koszul sign convention helps to decrease the amount of signs in computations.
\paragraph*{Chain and cochain complexes}
\begin{defi}
A {\it chain complex} $(V,d)$ is a graded vector space $V$ together with a linear map $d$ of degree $-1$ and satisfying $d^2=0$, called {\it a differential of the chain complex}.
$$
\xymatrix@C=1.3pc{
...&V_{-1}\ar@{->}_{d}[l]&V_0\ar@{->}_{d}[l]&V_1\ar@{->}_{d}[l]&...\ar@{->}_{d}[l]
}
$$
\end{defi}
\begin{defi}
A {\it  cochain complex} $(V,d)$ is a graded vector space $V$ together with a linear map $d$ of degree $1$ and satisfying $d ^2=0$, called {\it a differential of the cochain complex}.
$$
\xymatrix@C=1.3pc{
...&V^{-1}\ar@{<-}_{d}[l]&V^0\ar@{<-}_{d}[l]&V^1\ar@{<-}_{d}[l]&...\ar@{<-}_{d}[l]
}
$$
\end{defi}
\begin{defi}
A {\it morphism of (co)chain complexes} $(V,d_V)$ and $(W,d_W)$ is a linear map $f:V\to W$ of degree 0, which commutes with differentials, that is $d_W\circ f=f\circ d_V$.
\end{defi}
Considering a chain complex $(V,d)$, it is sometimes useful to denote the differential $d$, in a more explicit way, by $d_n:V_n\to V_{n-1}$. Note that  $d^2=0$ explicitly reads as $d_n\circ d_{n+1}$, and thus $\mathrm{im}d_{n+1}\subset \mathrm{ker}d_n$. Elements of $\mathrm{ker}d_n$ are called {\it cycles} and elements of $\mathrm{im}d_{n+1}$ --- {\it boundaries}.
The {\it $n$-th  homology group} is by definition
$$
H_n:=\mathrm{ker}d_n\slash\mathrm{im}d_{n+1}.
$$
We denote $H_{\bullet}(V,d):=\bigoplus\limits_{n\in\mathbb{Z}} H_n(V,D).$\medskip

Similarly for the cochain complex $(V,d)$, where we denote the differential, in more explicit way, by $d^{n}:V^n\to V^{n+1}$. The {\it $n$-th  cohomology group} is by definition
$$
H^n:=\mathrm{ker}d^n\slash\mathrm{im}d^{n-1}.
$$
We denote $H^{\bullet}(V,d):=\bigoplus\limits_{n\in\mathbb{Z}} H^n(V,D).$\medskip

A (co)chain complex is called {\it acyclic} if its (co)homology is 0 everywhere. Note that a chain map $f:V\to W$ induces a linear map $f_\bullet$ in homology. If the map $f_\bullet$ is an isomorphism, we say that $f:V\xrightarrow{\sim}W$ is a {\it quasi-isomorphism.} Similarly one can define the notion of quasi-isomorphism for cochain maps.
\medskip
\begin{defi}
A {\it (co)chain homotopy} between two (co)chain maps $f,g: (V,d_V)\to (W,d_W)$ is a map $\eta: V\to W$ of degree 1 (respectively -1), such that $\eta d+d\eta=g-f$,
$$
\mbox{chain complex}\qquad\qquad\qquad\qquad\qquad\qquad\qquad\mbox{cochain complex}$$ $$
\xymatrix@C=1.3pc@R=2pc{
...\ar@{->}_{\eta}[dr]&V_{-1}\ar@{->}_{d}[l]\ar@{->}_{\eta}[dr]&V_0\ar@{->}_{d}[l]\ar@{->}_{\eta}[dr]&V_1\ar@{->}_{d}[l]\ar@{->}_{\eta}[dr]&...\ar@{->}_{d}[l]\\
...&W_{-1}\ar@{->}_{d}[l]&W_0\ar@{->}_{d}[l]&W_1\ar@{->}_{d}[l]&...\ar@{->}_{d}[l]\\
}\qquad\qquad
\xymatrix@C=1.3pc@R=2pc{
...&V^{-1}\ar@{<-}_{d}[l]\ar@{->}_{\eta}[dl]&V^0\ar@{<-}_{d}[l]\ar@{->}_{\eta}[dl]&V^1\ar@{->}_{d}[l]\ar@{->}_{\eta}[dl]&...\ar@{->}_{d}[l]\ar@{->}_{\eta}[dl]\\
...&W^{-1}\ar@{->}_{d}[l]&W^0\ar@{<-}_{d}[l]&W^1\ar@{<-}_{d}[l]&...\ar@{<-}_{d}[l]\\
}
$$
\end{defi}

If two (co)chain maps are homotopic, the induced maps in (co)homology coincide.
\subsubsection{Associative algebras}
In this section we give  definitions of associative algebras and coalgebras. The definitions are somehow dual to each other. It means that any associative algebra on a
finite dimension vector space $V$ gives rise to the coassociative coalgebra on $V^*$ and vice versa. We will work in the graded framework.
\begin{defi}
{\it A graded associative algebra $A$ over a field $K$} is a graded vector space $A$ endowed with a $K$-linear 0-degree map
\begin{equation*}
\mu: A\otimes A\rightarrow A,
\end{equation*}
which called a product (depending on a context the product $\mu$  denoted by a dot ($\cdot$), star ($*$), wedge ($\wedge$),...), such that the associativity condition satisfies:
\begin{equation*}
(a\cdot b)\cdot c=a\cdot(b\cdot c)\quad\Leftrightarrow\quad \mu(\mu\otimes \mathrm{id})=\mu(\mathrm{id}\otimes \mu),\qquad  a,b,c\in A.
\end{equation*}
\end{defi}
\begin{defi}
A graded associative algebra $A$ is called \\
\begin{itemize}
\item {\it unital}, if there exists an element $1_A\in A$, such that for any $a\in A$:
\begin{equation*}
1_A\cdot a=a\cdot 1_A=a,
\end{equation*}
\item {\it commutative}, if for any homogeneous elements $a,b\in A$:
 $$a\cdot b= (-1)^{\overline{a}\cdot \overline{b}} b\cdot a,$$
\item {\it anticommutative}, if for any homogeneous elements $a,b\in A$:
 $$a\cdot b= -(-1)^{\overline{a}\cdot \overline{b}} b\cdot a.$$
\end{itemize}
\end{defi}
One can show that if there exists the unit then it is unique. Indeed, let us assume that there exists another unit then $1'_A$, then
\begin{equation*}
1'_A=1'_A\cdot 1_A=1_A.
\end{equation*}

Sometimes it is useful to express the associativity and unital properties in terms of commutative diagrams. Now we reformulate the axioms of associative algebras in this language. One of the advantages of the diagram approach
that it doesn't refer to objects of an algebra.
\\

{\it associativity axiom}:
\begin{equation}\label{AlgebraAssociativityDiagramm}
\begin{array}{l}
\xymatrix@C=4pc@R=4pc{
A\otimes A \otimes A \ar@{->}^{\mathrm{id}\otimes\mu}[r]\ar@{->}^{\mu\otimes \mathrm{id}}[d]&A\otimes A\ar@{->}^\mu[d]\\
A\otimes A \ar@{->}^{\mu}[r]&A
}
\end{array}
{
\quad\Leftrightarrow\quad
\mu(\mu\otimes \mathrm{id})=\mu(\mathrm{id}\otimes \mu).
}
\end{equation}\medskip

The existence of unit in the algebra $A$ is equivalent to the existence
of $K$-linear map $u:K\rightarrow A$, such that for any
$k\in K$ and for any $a\in A$
\begin{equation}\label{algebraUnitalIdentity}
ka=u(k)\cdot a=a\cdot u(k),
\end{equation}
which can be expressed in the following commutative diagram:\medskip

{\it unital axiom}:
\begin{equation}\label{algebraUnitdiagramm}
\begin{array}{l}
\xymatrix@R=4pc{
K\otimes A\ar@{->}^{u\otimes \mathrm{id}}[r]\ar@{->}_{\backsimeq}[rd]&A\otimes A\ar@{->}^{\mu}[d]&A\otimes K\ar@{->}_{\mathrm{id}\otimes  u}[l]\ar@{->}^{\backsimeq}[ld]\\
&A&
}
\end{array}\quad\Leftrightarrow\quad \mu (u\otimes \mathrm{id})=\mathrm{id}=\mu(\mathrm{id}\otimes u).
\end{equation}
\begin{defi}
{\it An algebra homomorphism} between two graded associative algebras $A$ and $A'$ is a 0-degree linear map $F:A\rightarrow A'$ that respects the algebra structures:
\begin{equation*}
F(a\cdot b)=F(a)\cdot F(b).
\end{equation*}
If algebras are unital then the map $F$ is required to send the unit to unit:
\begin{equation*}
F(1_A)=1_{A'}.
\end{equation*}
\end{defi}
\begin{defi}
Algebras $A$ and $A'$ are called {\it isomorphic} if there exists the homomorphism $F:A\to A'$ that is bijective. The map $F$ is called an {\it algebra isomorphism}.
\end{defi}\bigskip
If $F:A\to A'$ is an algebra isomorphism then the map $F^{-1}$ is also an algebra isomorphism. Indeed,
\begin{equation*}
F^{-1}(a\cdot b)=F^{-1}(FF^{-1}a\cdot FF^{-1}b)=F^{-1}F(F^{-1}a\cdot F^{-1}b)=F^{-1}a\cdot F^{-1}b.
\end{equation*}
\begin{defi}
A unital algebra $A$ is called {\it augmented} if there exists the subalgebra $\overline{A}\subset A$, such that the algebra $A$ splits into two components
$$
A=1_AK\oplus \overline{A}.
$$
\end{defi}
The definition of the augmented algebra can be reformulated in the functorial spirit. One can show that the unital algebra is augmented if and only if there exists an algebra homomorphism $\varepsilon: A\to K$. This map is called the {\it augmentation map}.
\begin{ex}
The {\it algebra of polynomials} $K [x_1,...,x_n]$ is a commutative associative unital algebra.
\end{ex}
\begin{ex}\label{tensorAlgebraExample}
The {\it tensor algebra} over the tensor module
\begin{equation*}
T(V)=K\oplus V\oplus V^{\otimes 2}\oplus V^{\otimes 3}\oplus...\\
\end{equation*}
which is given by the concatenation multiplication:
\begin{equation*}
(v_1\otimes...\otimes v_p)\cdot(v_{p+1}\otimes...\otimes v_{p+q})=v_1\otimes...\otimes v_{p+q}.
\end{equation*}
The tensor algebra $(T(V),\cdot)$ is an associative unital algebra.
\end{ex}
\begin{ex}
The {\it reduced tensor algebra} over the reduced tensor module
\begin{equation*}
\overline{T}(V)=V\oplus V^{\otimes 2}\oplus V^{\otimes 3}\oplus...
\end{equation*}
with the concatenation product is an associative algebra.
\end{ex}
\begin{ex}\label{symmetricAlgebraExample}
The {\it symmetric algebra} is by definition
\begin{equation*}
S(V):=T(V)/I,
\end{equation*}
where the two-sided ideal $I$ is generated by all differences of homogeneous elements' products:\break $v\otimes w-(-1)^{\overline{v}\cdot\overline{w}}w\otimes v$ and the product descends from the tensor algebra's product. The reduced symmetric algebra is defined similarly:
$$
\overline{S}(V):=\overline{T}(V)/I.
$$
The symmetric algebra $S(V)$ is a graded commutative associative unital algebra. The  reduced symmetric algebra $\overline{S}(V)$ is a graded commutative associative  algebra. The (reduced) symmetric algebra can be decomposed in a direct sum of vector spaces:
\begin{equation*}
S(V)=\bigoplus\limits_{n=0}^\infty S^n(V),\quad \overline{S}(V)=\bigoplus\limits_{n=1}^\infty S^n(V),
\end{equation*}
 where $S^n(V)_{n\geqslant 0}$ is spanned by all elements of $S(V)$ of the weight $n$.
If the vector space $V$ is finite dimensional and has a dimension $k$ then the symmetric algebra $S(V)$ is isomorphic to the
algebra of polynomials $K [x_1,...,x_k].$
 \end{ex}
\begin{defi}
For any collection of homogeneous vectors $v_1,...,v_n\in S(V)$ and for any permutation $\sigma\in S(n)$ the {\it Koszul sign} $\varepsilon(\sigma,v_1,...,v_n)$ is defined by the commutativity relation
\begin{equation*}
\varepsilon(\sigma,v_1,...,v_n)v_{\sigma(1)}\cdot v_{\sigma(2)}\cdot...\cdot v_{\sigma(n)}=v_1\cdot v_2\cdot...\cdot v_n.
\end{equation*}
If it doesn't cause the confusions we will denote it just by $\varepsilon(\sigma)$.
\end{defi}
 \begin{ex}
The exterior algebra is by definition
\begin{equation*}
\Lambda(V):=T(V)/I,
\end{equation*}
where the two-sided ideal $I$ is generated by all sums of homogeneous elements' products $v\otimes w+\break+(-1)^{\overline{v}\cdot\overline{w}}w\otimes v$ and the product descends from the tensor algebra's product.\\
The reduced exterior algebra is defined similarly:
$$
\overline{\Lambda}(V):=\overline{T}(V)/I.
$$
The exterior algebra is anticommutative associative unital algebra. The reduced exterior algebra is anticommutative associative algebra. These algebras can be decomposed in a direct sum of vector spaces:
\begin{equation*}
\Lambda(V)=\bigoplus\limits_{n=0}^\infty \Lambda^n(V),\quad \overline{\Lambda}(V)=\bigoplus\limits_{n=1}^\infty \Lambda^n(V),
\end{equation*}
where $\Lambda^n(V)$ is spanned by all elements of $S(V)$ of the weight $n$.
The anticommutativity implies that for any homogeneous vectors $v_1,...,v_n \in \Lambda(V)$ and for any
permutation $\sigma\in S_n$ holds
\begin{equation*}
v_1\wedge v_2\wedge...\wedge v_n = \mathrm{sign}(\sigma)\cdot\varepsilon(\sigma,v_1,...,v_n)v_{\sigma(1)}\wedge v_{\sigma(2)}\wedge...\wedge v_{\sigma(n)}.
\end{equation*}
In the case if the vector $V$ is finite dimensional and possesses only 0-degree vectors then the exterior algebra is isomorphic to the algebra of linear forms over $V^*$ with standard exterior product.
 \end{ex}

\subsubsection{Coassociative coalgebras}
The general idea to give definitions in the `coalgebraic' world is to take the `algebraic' definitions and revert all arrows in diagrams.
\begin{defi}
{\it A graded coassociative coalgebra $C$ over a field $K$} is a graded vector space $C$ endowed with a $K$-linear 0-degree map
\begin{equation*}
\Delta: C\to C\otimes C,
\end{equation*}
such that the it satisfies the coassociativity relation:
\begin{equation}\label{CoalgebraCossociativityDiagramm}
\begin{array}{l}
\xymatrix@C=4pc@R=4pc{
C\otimes C \otimes C \ar@{<-}^{\mathrm{id}\otimes\Delta}[r]\ar@{<-}^{\Delta\otimes \mathrm{id}}[d]&C\otimes C\ar@{<-}^\Delta[d]\\
C\otimes C \ar@{<-}^{\Delta}[r]&C
}
\end{array}
\quad\Leftrightarrow\quad (\Delta\otimes \mathrm{id})\Delta=(\mathrm{id}\otimes \Delta)\Delta.
\end{equation}
The coalgebra $C$ is called
\begin{itemize}
\item  {\it counital} if there exists a map $\varepsilon: C\to K$ such that
\begin{equation}\label{CoalgebraCounitDiagramm}
\begin{array}{l}
\xymatrix@R=4pc{
K\otimes C\ar@{<-}^{\varepsilon\otimes \mathrm{id}}[r]\ar@{<-}_{\backsimeq}[rd]&C\otimes C\ar@{<-}^{\Delta}[d]&C\otimes K\ar@{<-}_{\mathrm{id}\otimes  \varepsilon}[l]\ar@{<-}^{\backsimeq}[ld]\\
&C&
}
\end{array}\quad\Leftrightarrow\quad  (\varepsilon\otimes \mathrm{id})\Delta=\mathrm{id}=(\mathrm{id}\otimes \varepsilon)\Delta.
\end{equation}
\end{itemize}
\end{defi}
\begin{rem}
Consider the finite dimensional associative unital algebra $(A,\mu,u)$. The multiplication map $\mu$ and the unity map $u:K\to A$ satisfies respectively associativity and unitality conditions (\ref{AlgebraAssociativityDiagramm}) and (\ref{algebraUnitdiagramm}). Consider the transposed maps $\mu^*:A^*\to (A\otimes A)^*$ and $u^*:A^*\to K$. Since the vector space $A$ is finite dimensional  $(A\otimes A)^*\simeq A^*\otimes A^*$; the transposed associativity and unitality conditions will looks like:
$$
\begin{array}{l}
(\mu^*\otimes \mathrm{id})\mu^*=(\mathrm{id}\otimes \mu^*)\mu^*,\\[5mm]
(u^*\otimes \mathrm{id})\mu^*=\mathrm{id}=(\mathrm{id}\otimes u^*)\mu^*.
\end{array}
$$
We see that if the associative unital algebra $(A,\mu,u)$ is finite dimensional then the datum
$(A^*,\mu^*,u^*)$ form a coassociative counital coalgebra. Similarly any finite dimensional coalgebra $(C,\Delta,\varepsilon)$ gives rise to the finite dimensional algebra $(C^*,\Delta^*,\varepsilon^*)$.
\end{rem}
\begin{rem}\label{RemInfDimCoalgAlg}
In infinite-dimensional case the situation is more tricky. The spaces  $(V\otimes W)^*$ and $V^*\otimes W^*$ are not isomorphic, but there exists an obvious embedding $V^*\otimes W^*\hookrightarrow (V\otimes W)^*$ defined by the formula $(\alpha\otimes \beta)(x\otimes y)\stackrel{def}{=}(-1)^{|x|\cdot|\beta|} \alpha(x)\beta(y)$.
So for infinite-dimensional algebra $(A,\mu)$ the transposed map $\mu^*$ defines the coalgebra structure on $A^*$ if and only if $\mathrm{Im}(\mu^*)\subset \mathrm{Im}(i)$:
\begin{equation}
\xymatrix@R=0.5pc{
&A \otimes A \ar@{->}^-{\mu}[r]& A\\
A^*\otimes A^*\ar@{^{(}->}^{i}[r]&(A\otimes A)^*\ar@{<-}^-{\mu^*}[r]&A^*.
}
\end{equation}
The situation for coalgebras is  better. The transposed map $\Delta^*$ always defines a product on $C^*\otimes C^*$. Indeed,
\begin{equation}
\xymatrix@R=0.5pc{
C \ar@{->}^-{\Delta}[r]& C \otimes C&\\
C^*\ar@{<-}^{\Delta^*}[r]&(C\otimes C)^*&C^*\otimes C.^*\ar@{_{(}->}_-{i}[l]
}
\end{equation}

\end{rem}
\begin{defi}
{\it A coalgebra homomorphism} between two graded associative coalgebras $C$ and $C'$ is a 0-degree linear map $F:C\rightarrow C'$ that respects the coalgebra structures:
\begin{equation}
\Delta F=(F\otimes F)\Delta.
\end{equation}
If the coalgebras are counital then the map $F$ is required to be compatible with counit maps:
\begin{equation*}
\varepsilon F=\varepsilon'.
\end{equation*}
\end{defi}
\begin{ex}
If $G$ is a finite group and $C(G)$ is a linear space of functions: $G\to K$. There is a natural coalgebra structure on the vector space $C(G)$. The coproduct
$$
\Delta: C(G)\to C(G)\otimes C(G)\simeq C(G\times G)
$$
 is given by the following formula:
$$
(\Delta f)(g_1,g_2)=f(g_1\cdot g_2),\ g_1,g_2\in G,
$$
the counit map is given by
$$
\varepsilon f=f(1_G).
$$
The datum $(C(G),\Delta,\varepsilon)$ form a coassociative counital coalgebra
\end{ex}
\begin{ex}
{\it The tensor coalgebra $T^c(V)$} as a vector space coincides with the tensor algebra $T(V)$. For any element $v_1...v_p\in V^{\otimes p}\subset T^c(V)$ the coproduct $\Delta$ acts in the following way:
$$
\Delta(v_1...v_p)=1_K\otimes v_1...v_p+\sum\limits_{i=1}^{p-1}v_1...v_i\otimes v_{i+1}...v_{p}+
v_1...v_p\otimes 1_K.
$$
The coproduct $\Delta$ is called a {\it deconcatenation}. The counit map $\varepsilon$ is the trivial projection $T^c(V)\to K$.
the datum $(T^c(V),\Delta,\varepsilon)$ form a coassociative counital coalgebra. If the vector space $V$ is finite dimensional then the tensor coalgebra $T^c(V)$ is dual to the tensor algebra, in the sense that  the datum $(T(V^*),\Delta^*,\varepsilon^*)$ form a tensor algebra (with the concatenation product).\medskip

Similarly {\it the reduced tensor coalgebra} $\overline{T}^c(V)$ as a vector space coincides with the reduced tensor coalgebra $\overline{T}^c(V)$ and the coproduct $\overline{\Delta}$ is defined as follows:
$$
\overline{\Delta}(v_1...v_p)=\sum\limits_{i=1}^{p-1}v_1...v_i\otimes v_{i+1}...v_{p}.
$$
\end{ex}
\begin{defi}
A coalgebra $C$ is called {\it coaugmented} if there exists the subcoalgebra $\overline{C}\subset C$, such that the coalgebra $C$ splits into two components
$$
C=1_CK\oplus \overline{C}.
$$
\end{defi}
Similarly to the algebraic case this definition can be encrypted in the existence of the coalgebra homomorphism $u: K\to C$, which called the {\it coaugmentation map}.

\begin{defi}\label{unshuffledPermutationDef}
To any finite set of natural numbers $i_1,...,i_p$ one can correspond the set of {\it $(i_1,...,i_p)$-unshuffled permutations $Sh(i_1,...,i_p)$}, which is a set of all $(i_1+...+i_p)$-permutations, such that for any
permutation $\sigma\in Sh(i_1,...,i_p)$ the following  conditions are required to satisfy:
$$
\begin{array}{l}
\sigma(1)<...<\sigma(i_1),\\
\sigma(i_1+1)<...<\sigma(i_1+i_2),\\
 ...\\[2mm]
\sigma(i_1+...+i_{p-1}+1)<...<\sigma(i_1+...+i_p).
\end{array}
$$
If, in addition to this, the last elements in each group are ordered:
$$
\sigma(i_1)<\sigma(i_1+i_2)<...<\sigma(i_1+...+i_p),
$$
then we say that the permutation $\sigma$ is an {\it $(i_1,...,i_p)$-half-unshuffled permutation}. The set of such elements we denote by $Hsh(i_1,...,i_p)$.
\end{defi}
\begin{ex}
{\it The symmetric coalgebra $S^c(V)$} as a vector space coincides  with the symmetric algebra $S(V)$. For any element $v_1...v_p\in S^c(V)$ the coproduct acts in the following way:
\begin{align*}
\Delta(v_1...v_p)&=1_K\otimes(v_1...v_p)+\sum\limits_{i=1}^{p-1}\sum\limits_{\sigma\in Sh(i,p-i)}\varepsilon(\sigma,v_1,...,v_p)(v_{\sigma(1)}...v_{\sigma(i)})\otimes(v_{\sigma(i+1)}...v_{\sigma(p)})\\
&+(v_1...v_p)\otimes 1_K,
\end{align*}
where $\varepsilon(\sigma,v_1,...,v_p)$ is the Koszul sign. The counit map $\varepsilon$ (don't confuse with the Koszul sign that is denoted by the same symbol) is the trivial projection $S^c(V)\to K$.
The datum $(S^c(V),\Delta,\varepsilon)$ form a coassociative counital coalgebra. If the vector space $V$ is finite dimensional then the symmetric coalgebra $S^c(V)$ is dual to the symmetric algebra, in the sense that the datum $(S^c(V^*),\Delta^*,\varepsilon^*)$ form the symmetric algebra.
\end{ex}\bigskip
\subsubsection{Differential graded associative algebras and coalgebras}
\begin{defi}
{\it A differential graded associative (DGA) algebra}  $(A,\mu, d)$ is an associative algebra $(A,\mu)$ equipped with $-1$-degree map $d$, called {\it the differential}, which satisfies the following identities:
$$
\begin{array}{l}
\bullet\ d(u\cdot v)=du\cdot v+(-1)^{\overline{u}}\cdot dv \quad\Leftrightarrow\quad d\mu=\mu(d\otimes \mathrm{id}+\mathrm{id}\otimes d)\quad \mbox{-- }\ d \mbox{ is a derivation,}\\[2ex]
\bullet\ d^2=0.
\end{array}
$$
\end{defi}

\begin{defi}
{\it A differential graded associative (DGA) coalgebra} $(C,\Delta, D)$ is an coassociative coalgebra $(C,\Delta)$ equipped with $-1$-degree map $D$, called {\it the codifferential}, that satisfies the following identities:
$$
\begin{array}{l}
\bullet\ \Delta D=(D\otimes \mathrm{id}+\mathrm{id}\otimes D)\Delta \quad \mbox{-- }\ D \mbox{ is a coderivation,}\\[2ex]
\bullet\ D^2=0.
\end{array}
$$
\end{defi}
\begin{defi}
{\it A DGA algebra homomorphism} between DGA algebras $A$ and $A'$ is a graded algebra homomorphism $F:A\to A'$ that additionally respects the differential structure of algebras:
$$
Fd-dF=0.
$$
\end{defi}
\begin{defi}
{\it A DGA coalgebra homomorphism} between DGA coalgebras $C$ and $C'$ is a graded coalgebra homomorphism $F:C\to C'$ that additionally respects the codifferential structure of coalgebras:
$$
FD-DF=0.
$$
\end{defi}

Often (DGA) (co)algebras are imposed by additional non-negative grading, called {\it the weight}:
$$
A:=A^{(0)}\oplus A^{(1)}\oplus ...
$$
If the weight is preserved by the (co)multiplication of the (co)algebra then we say that we deal with a {\it WGDA (weight-graded differential associative) (co)algebra}. For example, the tensor module $T(V)$ besides the homological grading that comes from the grading of $V$, carries the weight grading. The weight of each element in $V^{\otimes n}\subset T(V)$ is equal to $n$.
\begin{defi}
The WGDA (co)algebra $A$ is called {\it connected} in the weight decomposition $A^{(0)}=K$.
\end{defi}

\subsubsection*{Symmetric product}
Consider maps $f,g:S(V)\to S(W)$. The vector spaces $S(V)$ and $S(W)$
may be endowed with symmetric product $\mu$ (symmetric algebra) as well as with coproduct $\Delta$ (symmetric coalgebra). The symmetric product $f\odot g: S(V)\to S(W)$ is defined as follows:
$$
f\odot g:=\mu\circ(f\otimes g)\circ\Delta.
$$
For example, if the vector $v_1...v_p\in S^P(V)$ then
$$
(f\odot g)(v_1...v_{p})=\sum\limits_{i+j=p}\ \sum\limits_{\sigma \in Sh(i,j)}\varepsilon(\sigma)\cdot (-1)^{\overline g\cdot(\overline{v}_{\sigma(1)}+...+\overline{v}_{\sigma(i)})}f(v_{\sigma(1)}...v_{\sigma(i)})\cdot g(v_{\sigma(i+1)}...v_{\sigma(i+j)}).
$$
The symmetric product is graded commutative:
$$
f\odot g=(-1)^{\overline{f}\cdot\overline{g}}g\odot f.
$$
The transposition
$$
(f\odot g)^*=(\mu\circ(f\otimes g)\circ\Delta)^*=\Delta^*\circ(f^*\otimes g^*)\circ\mu^*=\mu\circ(f^*\otimes g^*)\circ\Delta=f^*\odot g^*.
$$
Similarly, the symmetric product of $n$ maps $f_1,f_2...,f_n$ is defined by
$$
f_1\odot...\odot f_n:=\mu^{n-1}\circ(f_1\otimes...\otimes f_n)\circ\Delta^{n-1},
$$
where the $n$-th powers of multiplication and comultiplication are well defined because of their
associativity. This also implies that the symmetric product is associative.

\subsubsection*{Suspension maps}
\begin{defi}
For any graded vector space $V$  {\it the shifted by n vector space $s^nV$ (or V[n])}
  is vector space, such that $ V_{i}\simeq (s^nV)_{i+n}$.
\end{defi}\medskip
In other words, there exists an isomorphism $s^n: V\rightarrow s^nV$. The linear map $s^n$ does nothing but increases degrees of vectors by n, so it has a degree n. The map $s$ is called a {\it suspension map} and the $s^{-1}$ is a {\it desuspension map}.\\
The dual of shifted vector space
$(s^nV)^*\simeq s^{-n}V^*.$\\
For the $n$-th tensor power of (de)suspension maps:
$
\xymatrix{
V^{\otimes n}\ar@<1ex>^{s^{\otimes n}}[r]&(sV)^{\otimes n}\ar@<1ex>^{(s^{-1})^{\otimes n}}[l]
}
$,
according to the Koszul sign rule, their composition satisfy the following relations:
\begin{equation}\label{identity_formula}
 s^{\otimes n}\circ {(s^{-1})}^{\otimes n}={(s^{-1})}^{\otimes n}\circ s^{\otimes n}=(-1)^{\frac{n(n-1)}{2}}\mathrm{id}^{\otimes n}.
\end{equation}
For any map $f: (sV)^{\otimes n}\rightarrow (sW)^{\otimes k}$ one can define the suspended map $f^{\mathrm{susp}}: V^{\otimes n}\rightarrow W^{\otimes k}$ via
the following diagram:
$$
\begin{array}{l}
\xymatrix{
V\otimes...\otimes V \ar@{->}^-{f^{\mathrm{susp}}}[r]\ar@{->}_-{s^{\otimes n}}[d]&W\otimes...\otimes W\\
sV\otimes...\otimes sV\ar@{->}^-{f}[r]&sW\otimes...\otimes sW\ar@{->}_-{(s^{-1})^{\otimes k}}[u]
}
\end{array}\quad\Leftrightarrow\quad
   f^{\mathrm{susp}}:={(s^{-1})^{\otimes k}\circ f\circ \left(s\right)}^{\otimes n}.
$$
We see that $\overline{f^{\mathrm{susp}}}=\overline{f}+n-k.$\medskip

One can show that for the permutation map $\sigma_n: V^{\otimes n}\to V^{\otimes n}$ the following identity holds:
\begin{equation}\label{suspension_of_permutation}
\sigma_n^{\mathrm{susp}}=(-1)^{\frac{n(n-1)}{2}}\mathrm{sign}(\sigma_n)\cdot\sigma_n.
\end{equation}
\begin{prop}
For any graded vector space $V$ the exterior algebra $\Lambda(V)$ is isomorphic to $\bigoplus\limits_{n=0}^{\infty}(s^{-1})^{\otimes n}S^n(sV)$.
\end{prop}
\begin{proof}
We should prove that the map $(s^{-1})^{\otimes n}: S^n(sV)\to \Lambda(V)$ is well defined on the $S^n(sV)=(sV)^{\otimes n}/I$, that is it preserves the equivalence classes. We check this for the case when $n=2$.\\
Take the representative element $sv_1\otimes sv_2$ in the equivalence class $sv_1\cdot sv_2\in S^2(sV)$:
$$
[(s^{-1}\otimes s^{-1})(sv_1\otimes sv_2)]=[(-1)^{v_1+1}v_1\otimes v_2]= (-1)^{v_1+1}v_1\wedge v_2.
$$
Take the other element $(-1)^{(v_1+1)(v_2+1)}(sv_2\otimes sv_1)$ which belongs to the same class $sv_1\cdot sv_2$. The map $(s^{-1}\otimes  s^{-1})$ sends this element to the same equivalence class $(-1)^{v_1+1}[v_1\wedge v_2]$. Indeed,
$$
[(s^{-1}\otimes s^{-1})(-1)^{(v_1+1)(v_2+1)}(sv_2\otimes sv_1)]=[(-1)^{v_1v_2+1}\cdot (-1)^{v_1+1}v_2\otimes v_1]= (-1)^{v_1+1}v_1\wedge v_2.
$$
Similarly one can show that the map $(s^{-1})^{\otimes n}$ is well defined for any $n$. Thus we prove the space isomorphism.
\end{proof}\medskip
For any map $f: S^n(sV)\rightarrow S^k(sW)$ one can define the suspended map $f^{\mathrm{susp}}: \Lambda^n(V)\rightarrow \Lambda^k(W)$ via
the following diagram:
$$
\begin{array}{l}
\xymatrix{
\Lambda^n(V) \ar@{->}^-{f^{\mathrm{susp}}}[r]\ar@{->}_-{s^{\otimes n}}[d]&\Lambda^k(W)\\
S^n(sV)\ar@{->}^-{f}[r]&S^k(sW)\ar@{->}_-{(s^{-1})^{\otimes k}}[u]
}
\end{array}\quad\Leftrightarrow\quad
   f^{\mathrm{susp}}\stackrel{def}{=}{(s^{-1})^{\otimes k}\circ f\circ \left(s\right)}^{\otimes n}.
$$
We see that $\overline{f^{\mathrm{susp}}}=\overline{f}+n-k.$\medskip
 \subsection{Koszul theory on associative algebras and coalgebras}
This section deals with twisting and Koszul morphisms for associative algebras and coalgebras. Moreover, we take a special interest in the bar and cobar construction, which will finally provide a model (given by the bar-cobar resolution) of the considered differential graded associative algebra. Then we restrict ourselves to specific type of (co)algebras, namely `quadratic Koszul (co)algebras' and construct for them a `smaller' model.
\subsubsection{Twisting morphisms and twisted tensor complexes}
In this section $A$ is a augmented unital DGA algebra and $C$ is a counital coaugmented  DGA coalgebra.

The {\it convolution algebra} in the space of linear maps  $\mathrm{Hom}(C,A)$ is  given by the product
$$
f\star g=\mu\circ (f\otimes g)\circ\Delta.
$$
The unit of the product $\star$ is given by $u\circ \varepsilon\in \mathrm{Hom}(C,A)$. The associativity of the product  follows from the associativity of the product and coproduct on $A$ and $C$.
\begin{prop}
The map
$$\partial f=d_A\circ f-(-1)^{\overline{f}}f\circ d_C$$
is a differential on the convolution algebra $(\mathrm{Hom}(C,A),\star)$.
\end{prop}
\begin{proof}
Take any homogeneous elements $f,g\in \mathrm{Hom}(C,A)$
\begin{align*}
\partial(f\star g)&=d_A\circ(f\star g)-(-1)^{\overline{f}+\overline{g}}(f\star g)\circ d_C\\
&=d_A\circ\mu\circ(f\otimes g)\circ\Delta-(-1)^{\overline{f}+\overline{g}}\mu\circ(f\otimes g)\circ\Delta\circ d_C\\
&=\mu\circ(d_Af-(-1)^{\overline{f}}fd_C\otimes g)\circ\Delta+(-1)^{\overline{f}}\mu\circ(f\otimes d_Ag-(-1)^{\overline{g}}gd_C)\circ\Delta\circ d_C\\
&=\partial f\star g+(-1)^{\overline{f}}f\star\partial g.
\end{align*}
This shows that the operation $\partial$ is a derivation of degree -1 on the convolution algebra. The derivation $\partial$ also satisfies the condition $\partial^2=0$, indeed
\begin{align*}
\partial(\partial f)&=\partial(d_Af-(-1)^{\overline{f}}fd_C)\\
&=d^2_Af-{(-1)^{\overline{f}+1}d_Afd_C}-{(-1)^{\overline{f}}d_Afd_C}-fd^2_C=0.
\end{align*}
So the derivation $\partial$ is a differential.
\end{proof}
\begin{defi}
The {\it twisting morphism} $\alpha\in\mathrm{Tw}(C,A)$ is a morphism $\alpha\in \mathrm{Hom}_K(C,A)$ of degree $-1$, that satisfies the {\it Maurer-Cartan equation}
\begin{equation}\label{MCequation}
\partial\alpha+\alpha\star\alpha=0,
\end{equation}
and which is null when composed with the augmentation of $A$ and also when composed with the coaugmentation of $C$.
\end{defi}
The last condition is of technical purpose and can be formulated as
$$
\alpha\circ u=0\Leftrightarrow \alpha(1_C)=0 \quad\mbox{ and }\quad \varepsilon\circ\alpha=0\Leftrightarrow \alpha(\overline{C})\subset\overline{A}.
$$
where $u:K\to C$ denotes the coaugmentation map and $\varepsilon: A\to K$ the augmentation map.\\

Consider the tensor complex $(C\otimes A,d_{C\otimes A})$, where $d_{C\otimes A}=d_C\otimes \mathrm{id}+\mathrm{id}\otimes d_A$. Moreover, consider a morphism $\alpha\in\mathrm{Hom}_K(C,A)$ and  $\overline{d}_\alpha:C\otimes A\to C\otimes A$ defined by the following diagram:
$$
\xymatrix@C=4pc{
C\otimes A\ar@{->}^-{\Delta\otimes \mathrm{id}}[r]&C\otimes C\otimes A\ar@{->}^{\mathrm{id}\otimes\alpha\otimes \mathrm{id}}[r]&C\otimes A\otimes A\ar@{->}^-{\mathrm{id}\otimes\mu}[r]&C\otimes A.
}
$$
\begin{prop}
For any map $\alpha:C\to A$ of degree -1 the condition $d^2_\alpha=0$ (where $d_\alpha=d_{C\otimes A}+\overline{d}_\alpha$) holds true if and only if the map $\alpha$ satisfies the Maurer-Cartan equation (\ref{MCequation}).
\end{prop}
\begin{proof}
$$d^2_\alpha=(d_{C\otimes A}+\overline{d}_\alpha)^2=d_{C\otimes A}^2+d^2_\alpha+d_{C\otimes A}\circ\overline{d}_\alpha+\overline{d}_\alpha\circ d_{C\otimes A},$$
where $d_{C\otimes A}^2=0$, because $d_{C\otimes A}$ is a differential of the complex $C\otimes A$.\\
The next term $\overline{d}^2_\alpha$ is equal to $\overline{d}_{\alpha\star\alpha}$, indeed
\begin{align*}
\overline{d}^2_\alpha&=(\mathrm{id}\otimes\mu)\circ(\mathrm{id}\otimes\alpha\otimes \mathrm{id})\circ(\Delta\otimes \mathrm{id})\circ(\mathrm{id}\otimes\mu)\circ(\mathrm{id}\otimes\alpha\otimes \mathrm{id})\circ(\Delta\otimes \mathrm{id})\\
&=(\mathrm{id}\otimes\mu)\circ(\mathrm{id}\otimes\alpha\otimes \mathrm{id})\circ(\mathrm{id}^{\otimes 2}\otimes\mu)\circ(\Delta\otimes \mathrm{id}^{\otimes 2})\circ(\mathrm{id}\otimes\alpha\otimes \mathrm{id})\circ(\Delta\otimes \mathrm{id})\\
&=(\mathrm{id}\otimes\mu)\circ(\mathrm{id}\otimes\mu\otimes \mathrm{id})\circ(\mathrm{id}\otimes\alpha\otimes\alpha\otimes \mathrm{id})\circ(\mathrm{id}\otimes\Delta\otimes \mathrm{id})\circ(\Delta\otimes \mathrm{id})\\
&=(\mathrm{id}\otimes\mu)\circ(\mathrm{id}\otimes\alpha\star\alpha\otimes \mathrm{id})\circ(\Delta\otimes \mathrm{id})=\overline{d}_{\alpha\star\alpha}.
\end{align*}
 Finally for the last two terms $d_{C\otimes A}\circ\overline{d}_\alpha+\overline{d}_\alpha\circ d_{C\otimes A}=\overline{d}_{\partial(\alpha)}$, indeed
\begin{align*}
&d_{C\otimes A}\circ\overline{d}_\alpha+\overline{d}_\alpha\circ d_{C\otimes A}=\\
&=(d_C\otimes \mathrm{id}+\mathrm{id}\otimes d_A)\circ(\mathrm{id}\otimes\mu)\circ(\mathrm{id}\otimes\alpha\otimes \mathrm{id})\circ(\Delta\otimes \mathrm{id})+\\
&+(\mathrm{id}\otimes\mu)\circ(\mathrm{id}\otimes\alpha\otimes \mathrm{id})\circ(\Delta\otimes \mathrm{id})\circ(d_C\otimes \mathrm{id}+\mathrm{id}\otimes d_A)\\
&=(\mathrm{id}\otimes\mu)\circ(d_C\otimes \mathrm{id}^{\otimes 2}+\mathrm{id}\otimes d_A\otimes \mathrm{id}+ \mathrm{id}^{\otimes 2}\otimes d_A)\circ(\mathrm{id}\otimes\alpha\otimes \mathrm{id})\circ(\Delta\otimes \mathrm{id})+\\
&+(\mathrm{id}\otimes\mu)\circ(\mathrm{id}\otimes\alpha\otimes \mathrm{id})\circ(d_C\otimes \mathrm{id}^{\otimes 2}+\mathrm{id}\otimes d_C\otimes \mathrm{id}+ \mathrm{id}^{\otimes 2}\otimes d_A)\circ(\Delta\otimes \mathrm{id})\\
&=(\mathrm{id}\otimes\mu)\circ(\mathrm{id}\otimes d_A\alpha+\alpha d_C\otimes \mathrm{id})\circ(\Delta\otimes \mathrm{id})=\overline{d}_{\partial(\alpha)}.\\
\end{align*}
We  see that
$$
d^2_\alpha=\overline{d}_{\alpha\star\alpha+\partial(\alpha)}.
$$
If the map $\alpha$ satisfies the Maurer-Cartan equation (\ref{MCequation}) then $\overline{d}_{\alpha\star\alpha+\partial(\alpha)}=0$.\\
The converse statement follows from the fact that for any map $f:C\to A$ the restriction of the map $\overline{d}_f$ on $C\otimes K\to K\otimes A$ is equal to $f$, indeed
\begin{align*}
(\varepsilon\otimes \mathrm{id})\circ\overline{d}_f\circ(\mathrm{id}\otimes u)&=
(\varepsilon\otimes \mathrm{id})\circ(\mathrm{id}\otimes\mu)\circ(\mathrm{id}\otimes f\otimes \mathrm{id})\circ(\Delta\otimes \mathrm{id})\circ(\mathrm{id}\otimes u)\\
&=\underbrace{(\mathrm{id}\otimes\mu)\circ(\mathrm{id}^{\otimes 2}\otimes u)}_{\mathrm{id}\otimes\simeq}\circ(\mathrm{id}\otimes f\otimes \mathrm{id})\circ\underbrace{(\varepsilon\otimes \mathrm{id}^{\otimes 2})\circ(\Delta\otimes \mathrm{id})}_{\simeq\otimes \mathrm{id}}\\
&=f.
\end{align*}
\end{proof}
Defined in the preceding proposition differential $d_\alpha$ is called the {\it perturbation of the differential $d_{C\otimes A}$}, and the complex $C\otimes_\alpha A:=(C\otimes A,d_\alpha)$ is called the {\it twisted tensor complex}.
\begin{lem}[Comparison lemma for twisted tensor complexes]
Let $g:A\to A'$ be a morphism of weight-DGA connected algebras and $f:C\to C'$ be a morphism of weight-DGA connected coalgebras. Let $\alpha: C\to A$ and $\alpha'\to C'\to A'$ be two twisting morphisms, such that $f$ and $g$ are compatible with $\alpha$ and $\alpha'$, that is the following diagram commutes:
$$
\xymatrix{
C\ar@{->}^{\alpha}[r]\ar@{->}_-{f}[d]&A\ar@{->}^-{g}[d]\\
C'\ar@{->}^{\alpha'}[r]&A'\\
}
$$
If two morphisms among $f,g$ and $f\otimes g:C\otimes_\alpha A\to C'\otimes_{\alpha'} A'$ are quasi-isomorphisms, then so is the third one.
\end{lem}
\subsubsection{Bar and cobar complexes}
Consider an arbitrary DGA coalgebra $C$ with codifferential $D:C\to C$ and comultiplication $\Delta:C\to C\otimes C$. The idea behind the cobar construction is to encode the information about the codifferential $D$ and the comultiplication $\Delta$ in a derivation of the tensor algebra $T(C)$ (up to suspensions). A derivation on $T(C)$ is characterized by its restriction on $C$. Up to suspension we define it to be $\Delta+D:C\to T(C)$. The coassociativity of $\Delta$ and the fact that $D$ is a codifferential imply that the constructed derivation is a differential.\\
Let us be more precise, consider an augmented conilpotent DGA coalgebra $(C=1_CK\oplus \overline{C},\Delta,d_C)$ and the tensor algebra $T(s^{-1}\overline{C})$. Any derivation $d$ on $T(s^{-1}\overline{C})$ is given by its restriction on $s^{-1}\overline{C}$.
$$
d|_{s^{-1}\overline{C}}=d_1+d_2+...,\mbox{ where } d_k: s^{-1}\overline{C}\in \left(s^{-1}\overline{C}\right)^{\otimes k}.
$$
For any element $v_1...v_n\in T(s^{-1}\overline{C})$ its derivation
$$
d(v_1...v_n)=\sum_{i=1}^n(\mathrm{id}^{(i-1)}\otimes d\otimes \mathrm{id}^{(n-i)})(v_1...v_n)=\sum_{i=1}^n\sum^{\infty}_{k=1}(\mathrm{id}^{(i-1)}\otimes d_{k}\otimes \mathrm{id}^{(n-i)})(v_1...v_n).
$$
\begin{defi}
The {\it cobar complex} $\Omega C=(T(s^{-1}\overline{C}),d_{\Omega C})$ is defined by the -1 degree derivation $d_{\Omega C}$, which is given by restrictions
\begin{align*}
d_1&=-s^{-1} d_C s\quad (s^{-1}\overline{C}\to s^{-1}\overline{C}),\\
d_2&=-(s^{-1}\otimes s^{-1})\Delta s\quad (s^{-1}\overline{C}\to \left(s^{-1}\overline{C}\right)^{\otimes 2}),\\
d_k&=0,\mbox{ for other }k.
\end{align*}
\end{defi}
\begin{prop}
Defined above derivation $d_{\Omega C}$ is a differential, that is $d^2_{\Omega C}=0$.
\end{prop}
\begin{proof}
Since $d^2_{\Omega C}$ is a derivation, it is determined by its action on generators:
$$
d^2_{\Omega C}(v)=d_{\Omega C}\circ(d_1+d_2)(v)=(d_1^2+(\mathrm{id}\otimes d_1)d_2+(d_1\otimes \mathrm{id})d_2+d_2d_1+(d_2\otimes \mathrm{id})d_2+(\mathrm{id}\otimes d_2)d_2)(v).
$$
The right-hand side of the last equation splits into three parts:\\
\begin{enumerate}
\item $d_1^2=(s^{-1}d_Cs)^2=s^{-1}d^2_Cs=0$ ($d_C$ - is a differential),
\item $(\mathrm{id}\otimes d_1)d_2+(d_1\otimes \mathrm{id})d_2+d_2d_1=(s^{-1}\otimes s^{-1})(-(\mathrm{id}\otimes d_C+d_C\otimes \mathrm{id})\Delta+\Delta d_C)s=0$ ($d_C$ is a derivation for the coproduct $\Delta$),
\item $(d_2\otimes \mathrm{id})d_2+(\mathrm{id}\otimes d_2)d_2=(s^{-1}\otimes s^{-1}\otimes s^{-1})((\Delta\otimes \mathrm{id})\Delta-(\mathrm{id}\otimes \Delta)\Delta)s=0$ (coassociativity of the coproduct $\Delta$).
\end{enumerate}
Finally  we see that the $d^2_{\Omega C}=0$, so the $d_{\Omega C}$ is a differential.
\end{proof}
Similarly, in the dual case consider an augmented DGA algebra $(A=1_AK\oplus\overline{A},\mu,d_A)$ and the tensor coalgebra $T^c(s\overline A)$. Any coderivation D on $T^c(s\overline A)$ is given by its corestriction's $D_n:(s\overline A)^{\otimes n}\to s\overline{A}$:
$$
D(v_1...v_n)=\sum_{i=1}^n\sum^{n-i+1}_{k=1}(\mathrm{id}^{(i-1)}\otimes D_{k}\otimes \mathrm{id}^{(n-i-k+1)})(v_1...v_n).
$$
\begin{defi}
The {\it bar complex} $\mathrm{B}A=(T^c(s\overline A),D_{\mathrm{B}A})$ is defined by the -1 degree coderivation $D_{\mathrm{B}A}$, which is given by corestrictions
\begin{align*}
D_1&=-sd_As^{-1}\quad (s\overline{A}\to s\overline{A}),\\
D_2&=-s\mu (s^{-1}\otimes s^{-1})\quad ((s\overline{A})^{\otimes 2}\to s\overline{A}),\\
D_k&=0,\mbox{ for other }k.
\end{align*}
\end{defi}
\begin{prop}
Defined above coderivation $d_{\mathrm{B}A}$ is a codifferential, that is $d^2_{\mathrm{B}A}$.
\begin{proof}
The proof is similar to the one for cobar complex.
\end{proof}
\end{prop}
\begin{prop}\label{twistingMorphismProposition}
There is an isomorphism of spaces:
$$
\mathrm{Hom}_{DGAA}(\Omega C,A)\simeq\mathrm Tw(C,A)\simeq\mathrm{Hom}_{DGAC}(C,BA).
$$
\end{prop}
\begin{proof}
We will prove only the first isomorphism, the  proof for the second one will be similar.  Any DGA algebra homomorphism $F: \Omega(C)\simeq T(s^{-1}\overline{C})\to A$ is given by its restriction
$f:s^{-1}\overline{C}\to A$. The condition $Fd-dF=0$ on the generators reads as
\begin{align*}
&F(d_1+d_2)-d_Af=0,\\
&-fs^{-1}d_Cs-F(s^{-1}\otimes s^{-1})d_cs-d_Afs^{-1}s=0,\\
&fs^{-1}d_C+\mu(fs^{-1}\otimes fs^{-1})d_c+df_As^{-1}=0,\\
&\partial\alpha+\alpha\star\alpha=0,
\end{align*}
where the restriction of the map $\alpha:C\to A$ on the $\overline{C}$ is $\alpha|_{\overline{C}}:=fs^{-1}$ and $\alpha(1_C):=0$. Note that $\alpha$ satisfies the Maurer-Cartan equation. Since the $F$ is the homomorphism of augmented algebras, that is it respects the augmentation structure, $\alpha(\overline{C})\subset\overline{A}$. We see that the map $\alpha\in\mathrm{Tw}(C,A)$. \\
\end{proof}
\begin{rem}
Note that the spaces $\mathrm{Hom}_{DGAA}(\Omega C,A)$ and $\mathrm{Hom}_{DGAC}(C,BA)$ are the spaces of homomorphisms of the (co)augmented (co)algebras, which means, by definition, that they respect the (co)augmentation structures. This corresponds to the fact that twisting morphisms have to  vanish on units and counits.
\end{rem}
\subsubsection{Koszul morphisms and bar-cobar resolution}
\begin{defi}
A twisting morphism $\alpha\in\mathrm{Tw}(C,A)$ is called a {\it Koszul morphism}, if the twisted tensor complex $C\otimes_{\alpha}A$ is acyclic. The set of Koszul morphisms from $C$ to $A$ is denoted by $\mathrm{Kos}(C,A)$
\end{defi}
\begin{thm}[Fundamental theorem of twisting morphisms]\label{fundamentalTheoremTwistingMorphisms}
Let $A$ be a connected WGDA algebra and $C$ be a connected WGDA conilpotent coalgebra. For any twisting morphism $\alpha\in\mathrm{Tw}(C,A)$ and constructed from it homomorphisms $f_\alpha\in\mathrm{Hom}_{DGAA}(\Omega C, A)$ and $g_\alpha\in\mathrm{Hom}_{DGAC}(C,BA)$ the following propositions are equivalent:
\begin{enumerate}
\item $\alpha\in\mathrm{Kos}(C,A)$, that is the twisted complex $C\otimes_\alpha A$ is acyclic,
\item $f_\alpha$ is a quasi-isomorphism,
\item $g_\alpha$ is a quasi-isomorphism.
\end{enumerate}
\begin{equation}\label{QisoKoszulCorrespondence}
\mathrm{QIso}_{DGAA}(\Omega C,A)\simeq\mathrm{Kos}(C,A)\simeq\mathrm{QIso}_{DGAC}(C,BA).
\end{equation}
\end{thm}
\begin{proof}
We will prove the equivalence of the first and the second propositions. To the identity map $\mathrm{id}\in Hom_{DGAC}(\Omega C,\Omega C)$ one can correspond the twisting morphism $i\in\mathrm{Tw}(C,\Omega C)$ due to the Proposition \ref{twistingMorphismProposition}. Since the twisted complex $C\otimes_i\Omega C$ is acyclic, the map $\mathrm{id}\otimes f_\alpha$ is a quasi-isomorphism. Obviously the morphism $\mathrm{id}:C\to C$ is a quasi-isomorphism as well. We apply the comparison lemma for the commutative diagram:
$$
\xymatrix{
C\ar@{->}^{i}[r]\ar@{->}_-{\mathrm{id}}[d]&\Omega C\ar@{->}^-{f_\alpha}[d]\\
C\ar@{->}^{\alpha}[r]& A
}
$$
and get that $f_\alpha$ is a quasi-isomorphism.
\end{proof}
Take $BA$ for the coalgebra $C$, then one can correspond to the identity map $\mathrm{id}\in \mathrm{QIso}_{DGAC}(BA,BA)$ the map $\varepsilon\in\mathrm{QIso}_{DGAC}(\Omega BA,A)$ (due to the isomorphisms (\ref{QisoKoszulCorrespondence})). We say that $\varepsilon$ is a {\it bar-cobar resolution} and that $\Omega BA$ is a {\it model} of $A$. The model $\Omega BA$ generally is too `big'. In the next section we will study when it is possible to replace it by a smaller one.
\subsubsection{Quadratic algebras and coalgebras. Koszul dual}
\begin{defi}
{\it Quadratic data} $(V,R)$ consists of a graded vector space $V$ and a graded vector subspace $R\subset V^{\otimes 2}$
\end{defi}
\begin{defi}
The {\it quadratic algebra} $A(V,R)$ associated to the quadratic data $(V,R)$ is the quotient algebra $T(V)/(R)$, where $(R)$ denotes the two-sided ideal generated by $R$
\end{defi}\bigskip
Note that
$$
(R)=\bigoplus_{n\geqslant 2}\ \sum_{i+j=n}V^{\otimes i}\otimes R\otimes V^{\otimes j}
$$
and that
$$
A(V,R)=K\oplus V\oplus V^{\otimes 2}/R\oplus...\oplus {\left(V^{\otimes n}\Big/\sum_{i+j=n-2}V^{\otimes i}\otimes R\otimes V^{\otimes j}\right)\oplus...}=:\bigoplus_{n\geqslant0}A^{(n)}(V,R).
$$
\begin{defi}
The {\it quadratic coalgebra} $C(V,R)$ associated to the quadratic data $(V,R)$ is the coalgebra of the tensor algebra $T^c(V)$, defined on the vector space
$$
C(V,R)=K\oplus V\oplus R\oplus(V\otimes R\cap R\otimes V)\oplus...\oplus\left(\bigcap_{i+j=n-2}V^{\otimes i}\otimes R\otimes V^{\otimes j}\right)\oplus...=:\bigoplus_{n\geqslant0}C^{(n)}(V,R).
$$
\end{defi}\bigskip
The definition of quadratic coalgebra is dual to the definition of quadratic algebras, in the sense that if the vector space $V$ is finite dimensional then the coalgebra $A^*(V,R)\simeq C(V^*,R^\perp)$. Indeed,
\begin{align*}
{A^*}^{(n)}(V,R)&=\left(V^{\otimes n}\Big/\sum_{i+j=n-2}V^{\otimes i}\otimes R\otimes V^{\otimes j}\right)^*\\
&=\left(\sum_{i+j=n-2}V^{\otimes i}\otimes R\otimes V^{\otimes j}\right)^\perp
=V^{*\otimes i}\otimes R^{\perp}\otimes V^{*\otimes j}=C^{(n)}(V^*,R^\perp).
\end{align*}\bigskip
\begin{defi}
The {\it Koszul dual coalgebra} of a quadratic algebra $A(V,R)$ is
$$
A^{\text{!`}}(V,R):=C(sV,s^2R)
$$
\end{defi}
\begin{defi}
The {\it Koszul dual algebra} of a quadratic algebra $A(V,R)$ is
$$
{A^{\text{!}}}^{(n)}:=s^n(A^{\text{!`}*})^{(n)}
$$
\end{defi}
\begin{defi}
The {\it Koszul dual algebra} of a quadratic coalgebra $C(V,R)$ is
$$
C^{\text{!`}}(V,R):=A(s^{-1}V,s^{-2}R)
$$
\end{defi}
\begin{defi}
The {\it Koszul dual coalgebra} of a quadratic coalgebra $C(V,R)$ is
$$
{C^{\text{!}}}^{(n)}:=s^{-n}(C^{\text{!`}*})^{(n)}
$$
\end{defi}\bigskip
One can verify that $(A^{\text{!`}})^\text{!`}=A$ and $(C^{\text{!`}})^\text{!`}=C$ and in the finite dimensional case ${\left(A^{!}\right)}^{!}=A$, ${\left(C^{!}\right)}^{!}=C$ and
$A^{!}=A(V^*,R^\perp)$.

\begin{ex}
 If $V$ is a finite dimensional graded vector space and the vector space $R=\langle v\otimes w-(-1)^{|v||w|}w\otimes v\rangle$, then
\begin{enumerate}
\item $A(V,R)=T(V)/(R)=S(V)$,
\item $A^{\text{!`}}=C(sV,s^2R)= \Lambda^c(sV)$,
\item $A^{!}=\Lambda(V^*)$,
\item $C(V,R)=\Lambda^c(V)$,
\item $C^{\text{!`}}=A(s^{-1}V,s^{-2}R)=S(s^{-1}V)$,
\item $C^{!}=S^c(V^*)$.
\end{enumerate}
\end{ex}
\begin{ex}
 If $V$ is a finite dimensional graded vector space and the vector space $R=\langle v\otimes w+(-1)^{|v||w|}w\otimes v\rangle$, then
\begin{enumerate}
\item $A(V,R)=T(V)/(R)=\Lambda(V)$,
\item $A^{\text{!`}}=C(sV,s^2R)= S^c(sV)$,
\item $A^{!}=S(V^*)$,
\item $C(V,R)=S^c(V)$,
\item $C^{\text{!`}}=A(s^{-1}V,s^{-2}R)=\Lambda(s^{-1}V)$,
\item $C^{!}=\Lambda^c(V^*)$.

\end{enumerate}
\end{ex}
\subsubsection{Koszul algebras}
We now replace for quadratic algebras, under certain conditions, the `big' resolution $\Omega B A \stackrel{\sim}{\to} A$ by a smaller one, namely $\Omega A^{\text{!`}}\stackrel{\sim}{\to} A$, which is its minimal model.\\
We define the twisting morphism by
$$
\kappa:\xymatrix{
A^\text{!`}=C(sV,s^2R)\ar@{->>}^-{p}[r]&sV\ar@{->}^{s^{-1}}[r]&V\ar@{>->}^-{i}[r]&A(V,R)=A.
}
$$
Since the DGA coalgebra $A^\text{!`}$ and the DGA algebra $A$ have zero differential, the Maurer-Cartan equation will simplify:
$\kappa\star\kappa=0$. To verify that the defined above map $\kappa$ is a twisting morphism, it is sufficient to check that $\kappa\star\kappa$ is zero on $s^2R\subset sV^{\otimes 2}$:
$$
[(\kappa\star\kappa) sr_1\otimes sr_2]=[\kappa(sr_1)\kappa(sr_2)]=[r_1r_2]=0.
$$
All other term automatically goes to the zero, since the map $\kappa$ is not zero only on $sV$.
\begin{defi}
The twisted complex $A^{\text{!`}}\otimes_{\kappa}A$ is called {\it the Koszul complex}.
If the Koszul complex is acyclic then the algebra $A$ is called a {\it Koszul algebra}.
\end{defi}
If the algebra $A$ is Koszul then due to the fundamental theorem of twisting morphisms (Theorem~\ref{fundamentalTheoremTwistingMorphisms}) the projection $p:\Omega A^{\text{!`}}\to A$ is a quasi-isomorphism, that is the $\Omega A^{\text{!`}}$ is a resolution of $A$.
\subsection{Operads}
In this section we give the classical and functorial definitions of operad. The classical definition corresponds to our intuition. The functorial one is defined in a similar way as the associative algebra is defined. This helps to translate the Koszul duality theory for associative algebras to the operadic framework.
\subsubsection{Classical definition of operad}
In this section we work in the non graded framework. The action of  the symmetric group $S_n$ is defined in the following way:
\begin{itemize}
\item the left $S_n$ action:
$$\sigma (v_1\otimes...\otimes v_n)=v_{\sigma^{-1}(1)}\otimes...\otimes v_{\sigma^{-1}(n)},$$
\item the right $S_n$ action:
$$(v_1\otimes...\otimes v_n)\sigma=v_{\sigma(1)}\otimes...\otimes v_{\sigma(n)}.$$

\end{itemize}
\begin{defi}
An {\it $S$-module $P$} is a sequence $\{P(n)\},n\in\mathbb{N}$ of vector spaces endowed with right $S_n$-module structures.
\end{defi}
\begin{defi}
{\it A symmetric operad} consists of the following:
\begin{itemize}
\item $S$-module $P=\{P(n)\}$, whose elements are called {\it $n$-ary operations},
\item {\it an identity element} $\mathbf{1}\in P(1)$,
\item for all positive integers $n,k_1,...,k_n$ a composition function
\begin{align*}
\circ: P(n)\otimes P(k_1)\otimes...\otimes P(k_n)&\rightarrow P(k_1+...+k_n)\\
(\theta,\theta_1,...\theta_n)&\rightarrow\theta\circ(\theta_1,...,\theta_n),
\end{align*}
pictorially this composition looks as follows:
$$
\begin{tikzpicture}[baseline=(current bounding box.center), level distance=10mm, sibling distance=6mm,
					level 2/.style={sibling distance=12mm},
					level 3/.style={sibling distance=6mm}]
\node {} [grow'=up]
	child {node{$\scriptstyle{\theta}$}
		child {node{$\scriptstyle{\theta_1}$}
			child
			child
			child
		}
		child {node{$\scriptstyle{\cdots}$} edge from parent[draw=none]}
        child {node{$\scriptstyle{\theta_n}$}
			child
			child
		}
	};
\end{tikzpicture}
\quad=\quad
\begin{tikzpicture}[baseline=(current bounding box.center), level distance=14mm, sibling distance=7mm]
\node {} [grow'=up]
	child {node{$\scriptstyle{\theta\circ(\theta_1,...,\theta_n)}$}
		child
		child
		child {node{$\scriptstyle{\cdots}$}edge from parent[draw=none]}
        child
		child
		child
	};
\end{tikzpicture}
$$

satisfying the following coherence axioms:
\begin{enumerate}
\item identity: $\theta\circ(\mathbf{1},...,\mathbf{1})=\theta=\mathbf{1}\circ\theta,$
\item associativity:
$$
\theta\circ\Big(\theta_1\circ(\theta_{1,1},...,\theta_{1,k_1}),...,\theta_n\circ(\theta_{n,1},...,\theta_{n,k_n})\Big)=\Big(\theta\circ(\theta_1,...,\theta_n)\Big)\circ(\theta_{1,1},...,\theta_{1,k_1},...,\theta_{n,1},...,\theta_{n,k_n}),
$$
\item equivariance: given permutations $\sigma\in S_n$ and $\pi_i\in S_{k_i}$
$$
(\theta\sigma)\circ(\theta_1,...,\theta_n)=(\theta\circ(\theta_{\sigma^{-1}(1)},...,\theta_{\sigma^{-1}(n)}))\overline{\sigma}
$$
(where $\overline{\sigma}\in S_{k_1+...+k_n}$ on RHS is permuting $n$ groups in the same way as $\sigma\in S_n$ permutes $n$ elements),
$$
\theta\circ(\theta_1\pi_1,...,\theta_n\pi_n)=(\theta\circ(\theta_{1},...,\theta_{n}))(\pi_1,...,\pi_n).
$$

\end{enumerate}
\end{itemize}
\end{defi}\bigskip
\begin{ex}
The {\it endomorphism operad} $\mathcal{E}nd(V)$ over a vector space $V$ is made up by the vector spaces $\mathcal{E}nd(V)(n)=\mathrm{Hom}(V^{\otimes n},V)$ the usual composition and the identity map $\mathrm{id}_V$. The right $S_n$ module structure on $\mathrm{Hom}(V^{\otimes n},V)$ is given by
$$
(f\sigma)(v_1\otimes...\otimes v_n):=f(\sigma(v_1\otimes...\otimes v_n))=f(v_{\sigma^{-1}(1)}\otimes...\otimes v_{\sigma^{-1}(n)})
$$
\end{ex}
\begin{rem}
Note that the associativity and equivariance axioms do not mean that the $n$-ary operations that are encoded in the operad are associative and symmetric (whatever it means). You might see it in the previous example for the operad $\mathcal{E}nd(V)$, where these axioms hold tautologically. The existence of these axioms in the definition of the operad is the shadow of the fact, that we want to have a well defined notion of a representation of any operad $P$ in $\mathcal{E}nd(V)(n)$ (see the next section the precise definition of representation).
\end{rem}
\begin{ex}
The operad of {\it labeled planar trees} $\mathcal{T}$. The vector space $\mathcal{T}(n)$ is spanned by planar trees that have a root edge and $n$ leaves labeled by integers 1 through n. For instance:
$$
\begin{tikzpicture}[baseline=(current bounding box.center), level distance=7mm, sibling distance=6mm,
					level 2/.style={sibling distance=12mm},
					level 3/.style={sibling distance=6mm}]
\node {} [grow'=up]
	child {
		child {
			child{node{$\scriptstyle{2}$}}
			child{node{$\scriptstyle{5}$}}
			child{node{$\scriptstyle{1}$}}
		}
		child {
			child{node{$\scriptstyle{4}$}}
			child{node{$\scriptstyle{3}$}}
		}
	};
\end{tikzpicture}\in\mathcal{T}(5)
$$
The composition is given by grafting on the leaves. Where the order of grafting is defined by the labels of the leaves.
\end{ex}
\begin{ex}
The {\it little $n$-cubes operad} $C_n$. The elements of the vector space $C_n(j)$ are the  ordered collections of $j$ $n$-cubes linearly embedded in the standard $n$-dimensional unit cube $I^n$ with disjoint interiors and axes parallel to those of $I^n$. The compositions are given as indicated here:
$$
\parbox{3.2cm}{\begin{tikzpicture}
\draw [black] (0,0) rectangle (3,3);
\draw [black] (0.5,0.8) rectangle (2,2.3);
\node [below] at (2,0.8) {$\scriptstyle{2}$};
\draw [black] (2.5,1.2) rectangle (2.8,1.5);
\node [below] at (2.8,1.2) {$\scriptstyle{1}$};
\draw [black] (2.2,2.2) rectangle (2.8,2.8);
\node [below] at (2.8,2.2) {$\scriptstyle{3}$};
\end{tikzpicture}}\circ
\left(\quad\mathbf{1}\quad,\quad\parbox{3cm}{\begin{tikzpicture}
\draw [black] (0,0) rectangle (3,3);
\draw [black] (2,0.2) rectangle (2.8,1);
\node [left] at (2,0.2) {$\scriptstyle{2}$};
\draw [black] (0.1,1.5) rectangle (1.5,2.9);
\node [below] at (1.5,1.5) {$\scriptstyle{1}$};
\end{tikzpicture}}\quad,\quad\mathbf{1}\quad\right)=
\parbox{3.2cm}{\begin{tikzpicture}
\draw [black] (0,0) rectangle (3,3);
\draw [dotted,black] (0.5,0.8) rectangle (2,2.3);
\draw [black] (2.5,1.2) rectangle (2.8,1.5);
\node [below] at (2.8,1.2) {$\scriptstyle{1}$};
\draw [black] (2.2,2.2) rectangle (2.8,2.8);
\node [below] at (1.9,0.9) {$\scriptstyle{3}$};
\draw [black] (1.5,0.9) rectangle (1.9,1.3);
\draw [black] (0.55,1.55) rectangle (1.25,2.25);
\node [below] at (1.25,1.55) {$\scriptstyle{2}$};
\node [below] at (2.8,2.2) {$\scriptstyle{4}$};
\end{tikzpicture}}
$$
\end{ex}
\subsubsection{Morphisms and representations of operads}
Operads are important through their representations. In order to define representations of operads, we first have to define morphisms of operads.
\begin{defi}
An {\it operad morphism} $\varphi: P\to Q$ consists of the sequence of linear maps $\varphi_n: P(n)\to Q(n)$, which for all operations $\theta\in P(n),\theta_i\in P(k_i)$
\begin{itemize}
\item respect the composition:
$$ \varphi_n(\theta\circ_P(\theta_1,...,\theta_n))=\varphi_n(\theta)\circ_Q(\varphi_{k_1}(\theta_1),...,\varphi_{k_n}(\theta_n)),$$
\item respect the unit:
$$
\varphi_1(\mathbf{1}_P)=\mathbf{1}_Q,
$$
\item for any permutation $\sigma\in S(n)$
$$
\varphi_n(\theta\sigma)=\varphi_n(\theta)\sigma.
$$
\end{itemize}
\end{defi}
\begin{defi}
A {\it representation of operad $P$ on a vector space $V$} is an operad morphism $\rho:P\to\mathcal{E}nd(V)$
\end{defi}
 \begin{rem}\label{PalgebraRemark}
 Any representation $\rho$ is made up by a family of linear maps $\rho_n: P(n)\to\mathcal{E}nd(V)(n)=\mathrm{Hom}(V^{\otimes n},V)$. The map $\rho_n$ associates to abstract $n$-ary operations $\theta\in P(n)$ concrete $n$-ary operations on $V$, that is $\rho_n(\theta)\in \mathrm{Hom}(V^{\otimes n},V)$. Therefore one can actually get an algebraic structure on $V$. A representation of the operad $P$ on a vector space $V$ endows it with corresponding algebraic structure, that called {\it $P$-algebra} or {\it algebra over $P$.}\\
Note that the linear maps $\rho_n: P(n)\to\mathcal{E}nd(V)(n)$ can also be viewed as
\begin{align*}
&\rho_n\in\mathrm{Hom}(P(n),(\mathrm{Hom}(V^{\otimes n},V))=\mathrm{Hom}(P(n)\otimes V^{\otimes n},V)\\
&\rho_n: \theta\otimes v_1\otimes...\otimes v_n \to \theta(v_1,...,v_n).
\end{align*}
From the fact that the maps $\rho_n$ respect the action of the permutation group follows that the space may be restricted to the smaller one
\begin{equation}\label{representationSpace}
\rho_n\in\mathrm{Hom}(P(n)\otimes_{S_n}V^{\otimes n},V)\subset\mathrm{Hom}(P(n)\otimes V^{\otimes n},V).
\end{equation}
\end{rem}
\begin{defi}
The morphism between two $P$-algebras over $V$ and $W$ is a linear map $f:V\to W$, such that for any $n>0$ and for any operation $\theta\in P(n)$ the following identity holds:
$$
\theta(fv_1,...,fv_n)=f\theta(v_1,...,v_n).
$$
\end{defi}\bigskip
Now we are ready to consider other examples of operads
\begin{ex}
The operad $\mathcal{C}om$ is the operad associated with commutative associative algebras. The vector spaces $\mathcal{C}om(n)$ are one dimensional for each $n>0$ and $\mathcal{C}om(0)$=0. Pictorially:
$$
\begin{tikzpicture}[baseline=(current bounding box.center), level distance=8mm, sibling distance=5mm]
\node {} [grow'=up]
	child {
		child
        child
		child
		child
	};
\end{tikzpicture}\in \mathcal{C}om(4)
$$
The composition and the action of the symmetric group are trivial. Any algebra over operad
$\mathcal{C}om$ is a commutative associative algebra.
\end{ex}
\begin{ex}
The operad $\mathcal{A}ss$ associated with associative algebras, where $\mathcal{A}ss(n)=K[S_n]$ and $\mathcal{A}ss(0)=0$. For example for $n=3$ there are six basis elements in $\mathcal{A}ss(3)$:
$$
\begin{tikzpicture}[baseline=(current bounding box.center), level distance=7mm, sibling distance=7mm]
\node {} [grow'=up]
	child {
		child {node{$\scriptstyle{\sigma(1)}$}}
		child {node{$\scriptstyle{\sigma(2)}$}}
		child {node{$\scriptstyle{\sigma(3)}$}}
	};
\end{tikzpicture}=
\begin{tikzpicture}[baseline=(current bounding box.center), level distance=7mm, sibling distance=7mm]
\node {} [grow'=up]
	child { node{$\scriptstyle{\sigma}$}
		child
		child
		child
	};
\end{tikzpicture}
$$
The composition is given as indicated here:
$$
\begin{tikzpicture}[baseline=(current bounding box.center), level distance=10mm, sibling distance=6mm,
					level 2/.style={sibling distance=12mm},
					level 3/.style={sibling distance=6mm}]
\node {} [grow'=up]
	child {node{$\scriptstyle{\sigma}$}
		child {node{$\scriptstyle{\pi_1}$}
			child
			child
			child
		}
		child {node{$\scriptstyle{\cdots}$} edge from parent[draw=none]}
        child {node{$\scriptstyle{\pi_n}$}
			child
			child
		}
	};
\end{tikzpicture}
\quad=\quad
\begin{tikzpicture}[baseline=(current bounding box.center), level distance=14mm, sibling distance=7mm]
\node {} [grow'=up]
	child { node{$\scriptstyle{(\pi_1,...,\pi_n)\circ\sigma}$}
		child
		child
		child
        child {node{$\scriptstyle{\cdots}$}edge from parent[draw=none]}
		child
		child
	};
\end{tikzpicture}
$$
Any algebra over operad $\mathcal{A}ss$ is an associative algebra.
\end{ex}

\subsubsection{S-modules. Schur functor}
To the arbitrary $S$-module $P$ on can correspond the Schur functor $\widetilde{P}:\mathrm{Vect}\to \mathrm{Vect}$:
\begin{align*}
&\widetilde{P}(V)=\bigoplus_{n\in\mathbb{N}}P(n)\otimes_{S_n}V^{\otimes n},\\
&\widetilde{P}(\ell)=\bigoplus_{n\in\mathbb{N}}\mathrm{id}\otimes_{S_n}\ell^{\otimes n}:\widetilde{P}(V)\to\widetilde{P}(W).
\end{align*}
for any vector space $V$ and any linear map $\ell:V\to W$.\\

The {\it composite} of two $S$-modules $P$ and $Q$ is the $S$-module $P\circ Q$ defined by
\begin{align*}
(P\circ Q)(n)
&:=\bigoplus_{\substack{k\geqslant 0}}P(k)\otimes_{S_k}\left(\bigoplus_{i_1+...+i_k=n}\big(Q(i_1)\otimes...\otimes Q(i_k)\big)\otimes K[S_n/S_{i_1}\times...\times S_{i_k}]\right)\\
&=\bigoplus_{\substack{k\geqslant 0}}P(k)\otimes_{S_k}\left(\bigoplus_{i_1+...+i_k=n}\big(Q(i_1)\otimes...\otimes Q(i_k)\big)\otimes Sh(i_1,...,i_k)\right).\\
\end{align*}
Where the set $Sh(i_1,...,i_k)\simeq S_n/S_{i_1}\times...\times S_{i_k}$ of unshuffled permutations is defined in Definition \ref{unshuffledPermutationDef}. We denote the elements of $(P\circ Q)(n)$  by $(\mu;\nu_1,...,\nu_k;\sigma)$, where $\mu\in P(k)$, $\nu_j\in Q(i_j)$ and $\sigma\in Sh(i_1,...,i_k)$.\\
It is easy to check that the $S$-module $I=(0,K,0,0,...)$ is the identity for the defined above composition. Note that the composition is additive only on the left factor, that is\\
$$
(P_1\oplus P_2)\circ Q=(P_1\circ Q)\oplus (P_2\circ Q),$$$$
P\circ(Q_1\oplus Q_2)\neq (P\circ Q_1)\oplus (P\circ Q_2).
$$
\begin{prop}
The Schur functor is compatible with the composition of $S$-modules, that is
$$
\widetilde{P}\circ\widetilde{Q}=\widetilde{P\circ Q}.
$$
\begin{proof}
For an arbitrary vector space $V$ the following identities holds:
\begin{align*}
\widetilde{P\circ Q}(V)
&=\bigoplus_{\substack{k,n\geqslant 0}}P(k)\otimes_{S_k}\left(\bigoplus_{i_1+...i_k=n}\Big(Q(i_1)\otimes...\otimes Q(i_k)\Big)\otimes K[S_n/S_{i_1}\times...\times S_{i_k}]\right)\otimes_{S_n}V^{\otimes n}\\
&=\bigoplus_{\substack{k,n\geqslant 0}}P(k)\otimes_{S_k}\left(\bigoplus_{i_1+...i_k=n}\Big(Q(i_1)\otimes...\otimes Q(i_k)\Big)\otimes_{S_{i_1}\times...\times S_{i_k}} K[S_n]\right)\otimes_{S_n}V^{\otimes n}\\
&=\bigoplus_{\substack{k,n\geqslant 0}}P(k)\otimes_{S_k}\left(\bigoplus_{i_1+...i_k=n}\Big(Q(i_1)\otimes...\otimes Q(i_k)\Big)\otimes_{S_{i_1}\times...\times S_{i_k}} V^{\otimes {(i_1+...+i_k)}}\right)\\
&=\bigoplus_{\substack{k,n\geqslant 0}}P(k)\otimes_{S_k}\left(\bigoplus_{i_1+...i_k=n}\left(Q(i_1)\otimes_{S_{i_1}}V^{\otimes i_1}\right)\otimes...\otimes\left(Q(i_k)\otimes_{S_{i_k}}V^{\otimes i_k}\right) \right)\\
&=\bigoplus_{\substack{k\geqslant 0}}P(k)\otimes_{S_k}\widetilde{Q}(V)^{\otimes k}=\widetilde{P}(\widetilde{Q}(V))=\widetilde{P}\circ\widetilde{Q}(V).\\
\end{align*}
\end{proof}
\end{prop}
The composition of the Schur functors is associative, that is
$$
(\widetilde{P}\circ\widetilde{Q})\circ\widetilde{R}\simeq \widetilde{P}\circ(\widetilde{Q}\circ\widetilde{R})
$$
it implies that the composition of $S$-modules is associative too.
\begin{defi}
A $S$-module morphism  $f: P\to P'$ consists of a sequence of  linear maps $f_n: P(n)\xrightarrow{f} P'(n),n\geqslant 0$ which respect the action of the symmetric group:
$$
f_n\sigma=\sigma f_n,
$$
where $\sigma$ is an arbitrary permutation in $S_n$.
\end{defi}
Any morphism of $S$-modules $f: P\to P'$ gives rise to a natural transformation of Schur functors $\widetilde{f}: \widetilde{P}\to \widetilde{P'}$:
$$
\widetilde{f}(V): P(n)\otimes_{S_n}V^{\otimes n}\xrightarrow{f_n\otimes \mathrm{id}^{\otimes n}}P'(n)\otimes_{S_n}V^{\otimes n}.
$$
For any $S$-module maps $f:P\to P'$ and $g:Q\to Q'$ the composition $f\circ g: P\circ Q\to P'\circ Q'$ is defined as follows:
$$
(f\circ g)(P\circ Q)=fP\circ gQ
$$
The category of $S$-modules with a defined above composition $\circ$ and with a identity object $I=(0,K,0,0,...)$ form a monoidal category.\\

\subsubsection{Functorial definition of operads}
\begin{defi}
An {\it operad} $P=(P,\gamma,\eta)$ is an $S$-module P, endowed with morphsims of $S$-modules
$$
\gamma:P\circ P\to P,
$$
called {\it composition map}, and
$$
\eta: I\to P,
$$
called the {\it unit map}, such that the following diagrams commute:
\begin{equation}\label{AssiosativityUnitalityOperads}
\xymatrix@C=2pc@R=3pc{
&P\circ (P\circ P)\ar@{->}^-{\mathrm{id}\circ\gamma}[r]&P\circ P\ar@{->}^-{\gamma}[dd]\\
(P\circ P)\circ P\ar@{<->}^{\simeq}[ur]\ar@{->}^-{\gamma\circ \mathrm{id}}[d]&&\\
P\circ P\ar@{->}^-{\gamma}[rr]&&P
}\quad \mbox{and}\quad
\xymatrix@R=3pc{
I\circ P\ar@{->}^{\mathrm{id}\circ\gamma}[r]\ar@{<->}_{\simeq}[rd]&P\circ P\ar@{<-}^{\gamma\circ \mathrm{id}}[r]\ar@{->}^{\gamma}[d]& P\circ I \ar@{<->}^{\simeq}[ld]\\
&P&
}
\end{equation}
\end{defi}\bigskip
\begin{rem}
The functorial and the classical definition of operad are equivalent. Indeed, for any element $(\theta;\theta_1,...,\theta_n;\sigma)\in P\circ P$ the composition map $\gamma(\theta;\theta_1,...,\theta_n;\sigma)\in P$ corresponds to the classical composition, then acting by $\sigma$ from the right: $\theta\circ(\theta_1,...,\theta_n)\sigma$. The identity element in the classical definition is $\eta(I)\in P(1)$. The equivariance axiom in the classical definition is equivalent to the fact that the corresponding elements in the functorial definition belong to the same equivalence class in $P\circ P$. The associativity and identity axioms correspond to the commutative diagrams above.
\end{rem}
\begin{rem}
The functorial definition of operad is similar to the definition of the associative algebra. Indeed the associative algebra is a {\it monoid} in the monoidal category $(\mathrm{Vect},\otimes,K)$, that is an object $A\in\text{Vect}$ together with two morphisms: composition and identity morphism, that satisfy the associativity and unity requirements. The operad is a monoid in the monoidal category of $S$-modules $(\text{S-mod},\circ,I)$. The obstacle to make the full analogy between associative algebras and operads is the fact that the composition $\circ$ of $S$-modules is not additive on the right factor. Nevertheless it is possible transfer some results from the theory on associative algebras to operads.
\end{rem}
Now we give the functorial definition of operad morphism, which is equivalent to the classical one.
\begin{defi}
A {\it morphism of operads} from $P$ to $Q$ is a morphism of $S$-modules $\alpha:P\to Q$ which is compatible with the composition maps, that is
$$
\gamma(\alpha P\circ\alpha P)=\alpha\gamma(P\circ P),$$$$
\alpha\eta_P =\eta_Q.
$$
\end{defi}
\subsubsection{Functorial definition of $P$-algebras}
The notion of $P$-algebra can be given in the functorial language. Recall that in the classical framework a $P$-algebra is given by a representation $\rho_V$ of operad $P$ in $\mathcal{E}nd(V)$, which is given by the map $\bigoplus\limits_{n\geqslant0}P(n)\otimes_{S_n}V^{\otimes n}=\widetilde{P}(V)\xrightarrow{\rho_V} V$ (see Remark~\ref{PalgebraRemark}).
\begin{defi}
An {\it algebra over the operad $P$ (or $P$-algebra)} is a vector space $V$ equipped with a linear map $\rho_V:\widetilde{P}(V)\to V$ such that the following diagrams commute:
$$
\xymatrix@C=2pc@R=3pc{
&\widetilde{P}(\widetilde{P}(V))\ar@{->}^-{\widetilde{P}(\rho_V)}[r]&\widetilde{P}(V)\ar@{->}^-{\rho_V}[dd]\\
\widetilde{P\circ P}(V)\ar@{<->}^{\simeq}[ur]\ar@{->}^-{\widetilde{\gamma}(V)}[d]&&\\
\widetilde{P}(V)\ar@{->}^-{\rho_V}[rr]&&V
}\quad\text{and}\quad
\xymatrix{
\widetilde{I}(V)\ar@{->}[r]^{\widetilde{\eta}(V)}\ar@{->}[dr]_{\simeq}&\widetilde{P}(V)\ar@{->}[d]^{\rho_V}\\
&V
}
$$
\end{defi}\bigskip
The notion of $P$-algebra morphism easily can be translated in the functorial language.
\begin{defi}
A morphism of $P$-algebras $(V,\rho_V)$ and $(W,\rho_W)$ is a linear map $f:V\to W$ such that the following diagram commutes:
$$
\xymatrix{
\widetilde{P}(V)\ar@{->}^-{\rho_V}[r]\ar@{->}_-{\widetilde{P}(f)}[d]&V\ar@{->}^-{f}[d]\\
\widetilde{P}(W)\ar@{->}^-{\rho_W}[r]&W
}
$$
\end{defi}
\subsubsection{Free P-algebra}
\begin{defi}
In the category of $P$-algebras, a $P$-algebra $\mathcal{F}(V)$, equipped with a linear map $i:V\to\mathcal{F}(V)$ is said to be {\it free} over the vector space $V$ if for any $P$-algebra $A$ and any linear map $f:V\to A$ there is a unique $P$-algebra morphism $\widetilde{f}:\mathcal{F}(V)\to A$ such that the following diagram is commutative:
$$
\xymatrix{
V\ar@{->}^-{i}[r]\ar@{->}_-{f}[rd]& \mathcal{F}(V)\ar@{-->}^-{\widetilde{f}}[d]\\
&A
}
$$
\end{defi}\bigskip
Note that the free $P$-algebra is unique up to isomorphism. One can show that it is the vector space $\widetilde{P}(V)$ equipped with $\rho_{\widetilde{P}(V)}:\widetilde{P}(\widetilde{P}(V))\to \widetilde{P}(V)$ defined via the diagram:
$$
\xymatrix{
\widetilde{P}(\widetilde{P}(V))\ar@{->}^{\rho_{\widetilde{P}(V)}}[r]\ar@{->}_-{\simeq}[d]& \widetilde{P}(V)\\
\widetilde{P\circ P}(V)\ar@{->}_-{\widetilde{\gamma}(V)}[ur]&
}
$$
and the map $i:V\to \widetilde{P}(V)$, which is defined via the natural embedding
$$
\xymatrix{
V\ar@{->}^{i}[r]\ar@{->}_-{\simeq}[d]& \widetilde{P}(V)\\
\widetilde{I}(V)\ar@{->}_-{\widetilde{\eta}(V)}[ur]&
}
$$
\begin{ex}
For the operad $\mathcal{A}ss$, that controls associative algebras, the free $\mathcal{A}ss$-algebra over the vector space $V$ is
$$
\widetilde{\mathcal{A}ss}(V)=\bigoplus_{n>0} \mathcal{A}ss(n)\otimes_{S_n}V^{\otimes n}=\bigoplus_{n>0} K[S_n]\otimes_{S_n}V^{\otimes n}=\bigoplus_{n>0}V^{\otimes n}=\overline{T}(V),
$$
where the composition is given by concatenation.
\end{ex}
\begin{ex}
For the operad $\mathcal{C}om$, that controls commutative associative algebras, the free $\mathcal{C}om$-algebra over the vector space $V$ is
$$
\widetilde{\mathcal{C}om}(V)=\bigoplus_{n>0} K\otimes_{S_n}V^{\otimes n}=\overline{S}(V),
$$
where the composition is given by concatenation.
\end{ex}
\subsubsection{Free operad}
As for any free object, the free operad over an $S$-module $M$ is defined by means of a universal property. Namely, as being the operad $F(M)$ together with the $S$-module morphism $i:M\to F(M)$, such that for any operad $P$ and any $S$-module morphism $\varphi:M\to P$, there exists a unique morphism of operads $\overline{\varphi}:F(M)\to P$, such that the following diagram commutes:
$$
\xymatrix{
M\ar@{->}^-{i}[r]\ar@{->}_-{\varphi}[rd]& F(M)\ar@{-->}^-{\widetilde{\varphi}}[d]\\
&P
}
$$
In order to construct the free operad over $M$, we will define the sequence of $S-modules$ $\mathcal{T}_nM$ by
\begin{align*}
\mathcal{T}_0M&=I,\\
\mathcal{T}_1M&=I\oplus M,\\
\mathcal{T}_2M&=I\oplus (M\circ(I\oplus M))=I\oplus (M\circ\mathcal{T}_1M),\\
...\\
\mathcal{T}_nM&=I\oplus (M\circ\mathcal{T}_{n-1}M).
\end{align*}
Moreover, we recursively define a sequence of inclusions
\begin{align*}
&i_0:\mathcal{T}_{0}\to\mathcal{T}_{n}, \quad I\hookrightarrow I\oplus M;\\
&i_n:\mathcal{T}_{n-1}\to\mathcal{T}_{n} , \quad i_n=\mathrm{id}_I\oplus (\mathrm{id}_M\circ i_{n-1}).
\end{align*}
By definition, the {\it tree module} $\mathcal{T}M$ over the $S$-module $M$ is a union
$$
\mathcal{T}M:=\bigcup_n\mathcal{T}_nM.
$$
\begin{rem}
To get the intuition how the $S$-module $\mathcal{T}M$ is builded, it is useful to look at its elements as on formal operations:
$$
\theta(....)=\left\{a\Big(b\big(...\big)...c\big(....d(...)..\big)...\Big);\sigma\right\}.
$$
where the operations $a,b,c,d,...\in M$ and the permutation $\sigma$ permutes the elements in entries. If the `depth' of the formal operation is equal to $n$ then it belongs to $\mathcal{T}_nM$.
\end{rem}
The obvious composition of these formal operations and the obvious unit operation define a {\it free operad} $\mathcal{T}M$. \\

The construction of a free operad gives a conceptual approach for any type of algebras to define the operad that controls these algebras. Let us be more precise: any type of algebraic structure on a vector space $V$ is defined by the $n$-ry operations $A^{\otimes n}\to A$, called {\it generating operations}, that satisfy some specific multilinear relations . For example, associative algebras are given by binary multilinear map $\mu:A^{\otimes 2}\to A$ that satisfies $\mu\circ(\mu,\mathrm{id})-\mu\circ(\mu,\mathrm{id})=0$. Let $M$ be the $S$-module, whose arity $n$ spaces are generated by the $n$-ary generating operations. Since the relators are composites of these generating relations (and identity), they span a sub-$S$-module of the free operad $\mathcal{T}M$. Let $(R)$ denote the operadic ideal of $\mathcal{T}M$ generated by $R$. The precise definition of operadic ideals is given as follows:
\begin{defi}
An {\it operadic ideal} $I$ of an operad $P$ is a sub-$S$-module of $P$, such that for any family of operations $(\mu,\nu_1,...,\nu_k)$ of $P$, we have that if one of these operations is in $I$, then the composite $\mu\circ(\nu_1,...,\nu_k)$ is also in $I$.
\end{defi}
This way, we have naturally constructed the operad $\mathcal{T}M/(R)$, which controls the type of algebra defined above (operations in $M$ and relations $R$), that is any representation of $\mathcal{T}M/(R)$ in $\mathcal{E}nd(V)$ (any $\mathcal{T}M/(R)$-algebra on $V$) defines these types of algebras.
\subsection{Operadic homological algebra}
\subsubsection{Infinitesimal composite}
Recall that the composition of $S$-modules is not linear on the right factor. Here we introduce the infinitesimal composite, which is linear on two factors.\medskip

Consider the $S$-module $P\circ(I\oplus Q)$, its elements are of the form $(\mu;\nu_1,\mathrm{id},\mathrm{id},\nu_2,\mathrm{id},...;\sigma)$.
\begin{defi}
The {\it infinitesimal composite} $P\circ_{(1)}Q$ is the $S$-module, which is spanned by the elements of the `linear part' of $P\circ(I\oplus Q)$, that is made up by the terms containing $Q$ exactly once. The elements are of the form $(\mu;\mathrm{id},...,\mathrm{id},\nu,\mathrm{id},...,\mathrm{id};\sigma)$.
 \end{defi}
 Note that the infinitesimal composite $\circ_{(1)}$ is not associative.
\begin{defi}
The corresponding composite $f\circ_{(1)}g$ of two $S-module$ morphisms $f:P_1\to P_2$ and $g:Q_1\to Q_2$ is defined by
\begin{align*}
f\circ_{(1)}g: P_1\circ_{(1)}Q_1&\to P_2\circ_{(1)}Q_2\\
(\mu;\mathrm{id},...,\mathrm{id},\nu,\mathrm{id},...,\mathrm{id};\sigma)&\to (f(\mu);\mathrm{id},...,\mathrm{id},g(\nu),\mathrm{id},...,\mathrm{id};\sigma).
\end{align*}
\end{defi}\bigskip
Instead of `linearizing' the space $P\circ Q$, we can as well `linearize' the morphism $f\circ g$:
\begin{defi}
The {\it infinitesimal composite} $f\circ'g$ of two $S-module$ morphisms $f:P_1\to P_2$ and $g:Q_1\to Q_2$ is defined by
\begin{align*}
f\circ'g: P_1\circ Q_1&\to P_2\circ Q_2\\
(\mu;\nu_1,...,\nu_n;\sigma)&\to \sum_{i=1}^n(f(\mu);\nu_1,...,g(\nu_i),...,\nu_n;\sigma).
\end{align*}
\end{defi}
\subsubsection{Differential graded $S$-modules}
\begin{defi}
A {\it graded $S$-module $P$} is a $S$-module such that its components $P(n)$ are graded vector spaces and the $S_n$ action on them is degree preserving.
\end{defi}
The subspace of $P(n)$ of the elements of degree $k$ and we denote by $P_k(n)$.
\begin{defi}
A morphism $f:P\to Q$ of degree $r$ between graded $S$-modules $P$ and $Q$ is a sequence of $r$-degree $S_n$-equivariant maps $f_n:P_k(n)\to Q_{k+r}(n),\forall k$.
\end{defi}
\begin{rem}
The composite product $\circ$ can be extended to graded $S$-modules by
$$
(P\circ Q)_p(n)
:=\bigoplus_{\substack{k\geqslant 0}}P_r(k)\otimes_{S_k}\left(\bigoplus_{\substack{i_1+...+i_k=n\\s_1+...+s_k=p-r}}\big(Q_{s_1}(i_1)\otimes...\otimes Q_{s_k}(i_k)\big)\otimes K[S_n/S_{i_1}\times...\times S_{i_k}]\right).
$$
\end{rem}
\begin{defi}
A {\it DG (differential graded)  $S$-module} $(P,d)$ is a graded $S$-module $P$ endowed with a differential $d$, that is an
$S$-module morphism $d:P\to P$ of degree $-1$, such that $d^2=0$.
\end{defi}
\begin{defi}
A {\it morphism $f:(P,d_P)\to(Q,d_Q)$ of DG $S$-modules} is a graded $S$-module morphism $f:P\to Q$ of degree $0$, such that it commutes with the
differentials, that is
$$
d_Qf=fd_P.
$$
\end{defi}
\begin{defi}
The composite of two DG $S$-modules $(P,d_P)$ and $(Q,d_Q)$ is the DG $S$-module $(P\circ Q, d_{P\circ Q})$, where the differential
$$
d_{P\circ Q}:=d_{P}\circ \mathrm{id}_Q+\mathrm{id}_P\circ'd_Q.
$$
\end{defi}
\subsection{Differential graded operads and cooperads}
\begin{defi}
A {\it DG (differential graded) operad} is a DG $S$-module $(P,d_P)$ together with DG $S$-module morphisms: composition
$\gamma: P\circ P\to P$ and unit map $\eta:  I\to P$, which satisfy the associativity and unital relations~(\ref{AssiosativityUnitalityOperads}).
\end{defi}
\begin{rem}
The condition that the composition map $\gamma$ is a DG $S$-module morphism means that it is a morphism of degree 0, such that
$$
d_{P}\gamma=\gamma d_{P\circ P}=\gamma(\mathrm{id}_P\circ d_P+d_P\circ' \mathrm{id}_P).
$$
\end{rem}
\begin{defi}
A {\it DG (differential graded) cooperad} is a DG $S$-module $(C,d_C)$ together with DG $S$-module morphisms: decomposition
$\Delta: C\to C\circ C$ and counit map $\varepsilon: C\to I$, which satisfy the coassociativity and counital relations:
$$
\xymatrix@C=2pc@R=3pc{
&C\circ (C\circ C)\ar@{<-}^-{\mathrm{id}\circ\Delta}[r]&C\circ C\ar@{<-}^-{\Delta}[dd]\\
(C\circ C)\circ C\ar@{<->}^{\simeq}[ur]\ar@{<-}^-{\Delta\circ \mathrm{id}}[d]&&\\
C\circ C\ar@{<-}^-{\Delta}[rr]&&C
}\quad \mbox{and}\quad
\xymatrix@R=3pc{
I\circ C\ar@{<-}^{\mathrm{id}\circ\Delta}[r]\ar@{<->}_{\simeq}[rd]&C\circ C\ar@{->}^{\Delta\circ \mathrm{id}}[r]\ar@{<-}^{\Delta}[d]& C\circ I \ar@{<->}^{\simeq}[ld]\\
&P&
}
$$
\end{defi}
\begin{rem}
(DG) operads and cooperads are dual to each other, namely for any (DG) cooperad $C$ there exists a (DG) operad on the $S$-module $C^*$, where the
composition, unit, and differential are the transpositions of the corresponding `co-morphisms'. In the other way round, for any operad $P$, where each $P(n)$
is finite dimensional one can construct the cooperad on $P^*$ via the transposition of all structure maps. Note that to construct the operads from cooperads one not need to make the finite dimensional assumption. The explanation is similar to the explanation of the analogous statement in the algebraic world, see Remark~\ref{RemInfDimCoalgAlg}.
\end{rem}
\subsubsection{Operadic twisting morphisms}
To extend the theory of twisting morphisms to operads, we need the linearization of the composition map $\gamma:P\circ P\to P$ of an operad, and of the decomposition map $\Delta:C\to C\circ C$ of a cooperad.
\begin{defi}
The {\it infinitesimal composition map} of a (DG) operad is given by
$$
\gamma_{(1)}:\xymatrix{
P\circ_{(1)}P\ \ar@{>->}[r]&P\circ P\ar@{->}^-{\gamma}[r]&P.
}
$$
\end{defi}
\begin{defi}
The {\it infinitesimal decomposition map} of a (DG) cooperad $\Delta_{(1)}:C\to C\circ_{(1)}C$ is given by
$$
\Delta_{(1)}:\xymatrix{
C\ar@{->}^-{\Delta}[r]&C\circ C\ar@{->>}[r]&C\circ_{(1)}C.
}
$$
\end{defi}\medskip
From now on, we will require the DG operad $(P,d_P,\gamma,\eta)$ to be {\it augmented}, that is there exists a DG $S$-module morphism
$\varepsilon: P\to I$, that respects composition $\gamma$ and unit $\eta$. Similarly the DG cooperad $(C,d_C,\Delta,\varepsilon)$ to be {\it coaugmented}, that is there exists a DG $S$-module morphism $\eta: I\to C$, that respects decomposition $\Delta$ and counit $\varepsilon$.
\begin{defi}
The {\it differential convolution algebra} on the space of graded $S$-module morphisms $\mathrm{Hom}_{S}(C,P):=\bigoplus\limits_{n\geqslant 0}\mathrm{Hom}_{S_n}(C(n),P(n))$
is given by
$$
f\star g:\xymatrix{
C\ar@{->}^-{\Delta_{(1)}}[r]&C\circ_{(1)}C\ar@{->}^-{f\circ_{(1)}g}[r]& P\circ_{(1)}P\ar@{->}^-{\gamma_{(1)}}[r]&P.
}
$$
the unit of the product $\star$ is given by $\gamma\circ\varepsilon\in\mathrm{Hom}_S(C,P)$. And the differential
$$
\partial f=d_P\circ f-(-1)^{\overline{f}}f\circ d_C.
$$
\end{defi}\medskip
\begin{defi}
An operadic twisting morphism $\alpha\in \mathrm{Tw}(C,P)$ is a $-1$-degree morphism in $\mathrm{Hom}_{S}(C,P)$ which is the solution
of the Maurer-Cartan equation
$$
\partial \alpha+\alpha\star\alpha=0,
$$
and is null when composed with the augmentation of $P$ and also when with the coaugmentation of $C$.
\end{defi}\medskip
Consider the DG $S$-module $(C\circ P,d_{C\circ P})$, where the differential $d_{C\circ P}=d_C\circ \mathrm{id}_P+\mathrm{id}_C\circ' d_P$. Moreover, consider  graded $S$-module morphism $\alpha\in\mathrm{Hom}_{S}(C,P)$ and define $\overline{d}_\alpha:C\circ P\to C\circ P$ by the following diagram:
$$
\xymatrix@C=2pc{
C\circ P\ar@{->}^-{\Delta_{(1)}\circ \mathrm{id}_P}[rr]&&(C\circ_{(1)}C)\circ P\ar@{->}^-{(\mathrm{id}_C\circ_{(1)}\alpha)\circ \mathrm{id}_P}[rr]&&(C\circ_{(1)}P)\circ P\ \ar@{>->}[r]&C\circ P\circ P\ar@{->}^-{\mathrm{id}_C\circ\gamma}[rr]&&C\circ P.
}
$$
If $d_{\alpha}=d_{C\circ P}+\overline{d}_{\alpha}$ defines a differential, that is $d^2_{\alpha}=0$, which is the case if and only if $\alpha$ satisfies the Maurer-Cartan equation, then $C\circ_{\alpha}P:=(C\circ P,d_\alpha)$ is a DG S-module called {\it twisted composite complex}. The {\it comparison lemma} remains valid for twisted composite complexes.
\subsubsection{Operadic bar and cobar constructions}
These constructions are similar to the corresponding ones in the algebraic context. Let us detail the cobar construction. Consider an augmented DG cooperad $(C,\Delta,\varepsilon,d_C)$, that is, in particular, we have $C=I\oplus \overline{C}$. Consider the free operad $\mathcal{T}(s^{-1}C)$, with the differential on it given by the sum $\delta_1+\delta_2$, where $\delta_1$ extends the differential $d_C$ and $\delta_2$ extends the infinitesimal decomposition $\Delta_{(1)}$. More precisely,
$$
s^{-1}\overline{C}\xrightarrow{s}\overline{C}\xrightarrow{d_{C}}\overline{C}\xrightarrow{s^{-1}}s^{-1}\overline{C}\rightarrowtail\mathcal{T}(s^{-1}\overline{C})
$$
and
$$
s^{-1}\overline{C}\xrightarrow{s}\overline{C}\xrightarrow{\Delta_{(1)}}\overline{C}\circ_{(1)}\overline{C}\xrightarrow{s^{-1}\circ s^{-1}}s^{-1}\overline{C}\circ_{(1)}s^{-1}\overline{C}\rightarrowtail\mathcal{T}(s^{-1}\overline{C})
$$
uniquely extend, since $\mathcal{T}(s^{-1}C)$ is free, to derivations $\delta_1$ and $\delta_2$ of $\mathcal{T}(s^{-1}C)$ and form the cobar complex $\Omega C=(\mathcal{T}(s^{-1}C),\delta_1+\delta_2)$.
Similarly with the algebraic context the following isomorphisms holds:
$$
\mathrm{Hom}_{DG\ Op}(\Omega C,P)\simeq\mathrm Tw(C,P)\simeq\mathrm{Hom}_{DG\ CoOp}(C,BP).
$$
Under some weight-graded assumptions (similar to the one from associative algebras and coalgebras), we have:
$$
\mathrm{QIso}_{DG\ Op}(\Omega C,P)\simeq\mathrm{Kos}(C,P)\simeq\mathrm{QIso}_{DG\ CoOp}(C,BP),
$$
where $\alpha\in\mathrm{Kos}(C,P)\Leftrightarrow C\circ_{\alpha}P$ is acyclic.
And taking $C=BP$, we find that $\Omega BP\xrightarrow{\sim}P$.
\subsubsection{Koszul duality for operads}
We will adapt  the results of Koszul duality for algebras to operads. This will lead, for a quadratic Koszul operad $P$, to a model $P_\infty:=\Omega P^{\text{!`}}$, which then allows to define $P_\infty$-algebras (or homotopy P-algebras) as representations of this operad.
\begin{defi}
Operadic quadratic data $(E,R)$ consists of a graded $S$-module $E$ and a graded sub-$S$-module $R\in\mathcal{T}(E)^{(2)}$.
\end{defi}
Here $\mathcal{T}(E)^{(2)}$ refers to the weight 2 part of the free operad $\mathcal{T}(E)$, that is to the graded sub-$S$-module of $\mathcal{T}(E)$, which is spanned by composites of two elements of $E$.\medskip

We will use the same terminology as in the algebraic setting and refer to elements of $E$ as generating operations and to elements of $R$ as relations.
\begin{defi}
The \textit{quadratic operad} $P(E,R)$ associated to the operadic quadratic data
$(E,R)$ is the quotient operad ${\mathcal{T}(E)}/{(R)}$, where
$(R)$ denotes the operadic ideal generated by $R\subset \mathcal{T}(E)^{(2)}$.
\end{defi}

The quadratic operad $P(E,R)$ is the quotient operad of $\mathcal{T}(E)$ that is universal
among all quotient operads $\mathcal{P}$ of $\mathcal{T}(E)$, such that the composite
$$R\rightarrowtail \mathcal{T}(E) \twoheadrightarrow \mathcal{P}$$ vanishes. More precisely,
there exists a unique morphism of operads $P(E,R)\to \mathcal{P}$,
such that the following diagram commutes
$$
\xymatrix{
R\ \ar@{>->}[r]\ar@/^1.5pc/@{->}^{0}[rr]\ar@/_1pc/@{->}_{0}[rrd]&\mathcal{T}(E)\ar@{->>}[r]\ar@{->>}[dr]&P(E,R)\ar@{-->}[d] \\
 & & \mathcal{P}.
}
$$

\begin{defi}
The \textit{quadratic cooperad} $C(E,R)$ associated to the operadic quadratic data $(E,R)$ is
the subcooperad of the cofree cooperad $\mathcal{T}^c(E)$, that is universal among all subcooperads
$\mathcal{C}$ of $\mathcal{T}^c(E)$, such that the composite
$$\mathcal{C}\rightarrowtail \mathcal{T}^c(E)\twoheadrightarrow {\mathcal{T}^c(E)^{(2)}}/{R}$$
vanishes. More precisely,
there exists a unique morphism of cooperads $\mathcal{C}\to C(E,R)$,
such that the following diagram commutes
$$
\xymatrix{
C(E,R)\ \ar@{>->}[r]\ar@/^1.5pc/@{->}^{0}[rr] & \mathcal{T}^c(E) \ar@{->>}[r] & {\mathcal{T}^c(E)^{(2)}}/{R} \\
\mathcal{C}\ \ar@{-->}[u] \ar@{>->}[ur]\ar@/_2pc/@{->}_{0}[rru] & &
}.
$$

\end{defi}

Note that when we are working over graded $S$-modules, the above defined
quadratic operad (respectively cooperad) is not only endowed with an arity
grading and a weight grading (coming from the free, respectively, cofree operad),
but also with a degree.

\subsubsection{Koszul dual cooperad and operad of a quadratic operad}

\begin{defi}
The \textit{Koszul dual cooperad} of a quadratic operad $P=P(E,R)$ is
$$P^{\text{!`}} = C(sE, s^2R),$$
that is the quadratic cooperad associated to the shifted operadic quadratic data.
\end{defi}

Here $sE$ denotes the shifted $S$-module, 
obtained from $E$ by shifting the degree in each arity.

\medskip

In order to define the Koszul dual operad, we need some preliminary remarks.

First, the \textit{Hadamard product} $P \underset{\mathrm{H}}{\otimes} Q$ of
two $S$-modules is given by
$(P \underset{\mathrm{H}}{\otimes} Q) (n) = P(n) \otimes Q(n)$. The action of the symmetric group is given by the diagonal action,
i.e. $(\mu\otimes\nu)\cdot\sigma=(\mu\cdot\sigma)\otimes(\nu\cdot\sigma)$,
for any $\mu\in P(n)$, $\nu\in Q(n)$, $\sigma\in S_n$.
Moreover, the Hadamard product of operads has a natural operad structure.

Second, the suspension of an operad, obtained by suspending the underlying $S$-module,
is, in general, not an operad. Therefore, we will define an `operadic suspension'.
Let $\mathcal{S}:=\mathcal{E}nd(sK)$ be the endomorphism operad over the suspended
ground field. This means that $\mathcal{S}(n)=\mathrm{Hom}((sK)^{\otimes n}, sK)$; note that
this space contains morphisms of degree $-n+1$. The symmetric group action is given
by the signature action.
We also denote $\mathcal{S}^{-1}:=\mathcal{E}nd(s^{-1}K)$ and
$\mathcal{S}^c:=\mathcal{E}nd^c(sK)$, where $\mathcal{E}nd^c(sK)$
is the endomorphism cooperad, which is as $S$-module the same as the endomorphism
operad, but equipped with a decomposition map.

Finally, we define the \textit{operadic suspension} of an operad $P$ by
$\mathcal{S} \underset{\mathrm{H}}{\otimes} P$. The \textit{operadic desuspension}
is given by $\mathcal{S}^{-1} \underset{\mathrm{H}}{\otimes} P$.
For a cooperad $C$, the \textit{cooperadic suspension} is given by
$\mathcal{S}^c \underset{\mathrm{H}}{\otimes} C$.
The operadic suspension has the property that a vector space $V$ is
equipped with a $P$-algebra structure, if and only if the suspended
vector space $sV$ is equipped with a
$\mathcal{S} \underset{\mathrm{H}}{\otimes} P$-algebra structure.

\begin{defi}
The \textit{Koszul dual operad} of a quadratic operad $P=P(E,R)$ is defined by
$$P^{!} = {(\mathcal{S}^c \underset{\mathrm{H}}{\otimes} P^{\text{!`}})}^*.$$
\end{defi}

The dual means here that we take the linear dual in each arity.

Let us mention that the $P^{!}$ is quadratic in a certain case. More precisely,
\begin{prop}
Let $P=P(E,R)$ be a quadratic operad, generated by a reduced $S$-module $E$
which is of finite dimension in each arity. Then the Koszul dual
operad $P^{!}$ admits the quadratic presentation
$P^{!}=P(s^{-1}\mathcal{S}^{-1} \underset{\mathrm{H}}{\otimes} E^*, R^\perp)$.
\end{prop}

Moreover, we have that, under the assumptions of the previous proposition,
${(P^{!})}^{!}=P$.

\subsubsection{Koszul operads and infinity algebras}

For given operadic quadratic data $(E,R)$, we have that $P(E,R)^{(1)}=E$ and $C(E,R)^{(1)}=E$,
and we can define the morphism $\kappa$ by
$$\kappa: C(sE,s^2R)\twoheadrightarrow sE \stackrel{s^{-1}}{\rightarrow} E \rightarrowtail P(E,R).$$
This morphism is clearly of degree $-1$, and verifies (for the same reasons as in the algebraic case)
$\kappa\star\kappa=0$. Therefore, $\kappa\in\mathrm{Tw}$ is an operadic twisting morphism.

This defines a \textit{Koszul complex} $P^{\text{!`}}\circ_\kappa P :=(P^{\text{!`}}\circ P, d_\kappa)$.
We thus have a sequence of chain complexes of $S_n$-modules $((P^{\text{!`}}\circ P)(n), d_\kappa)$,
called Koszul complexes in arity $n$.

A quadratic operad $P$ is called a \textit{Koszul operad} if the corresponding Koszul complex
$P^{\text{!`}}\circ_\kappa P$ is acyclic.

\medskip

Let us mention that there exists many Koszul operads, in particular
$\mathcal{A}\!ss$, $\mathcal{C}\!om$, $\mathcal{L}ie$, $\mathcal{L}eib$ and $\mathcal{P}ois$ are Koszul operads.

\medskip

Just as we have for Koszul algebras $A$, a resolution $\Omega A^{\text{!`}} \stackrel{\sim}{\rightarrow} A$,
we obtain, for Koszul operads $P$, a resolution $\Omega P^{\text{!`}} \stackrel{\sim}{\rightarrow} P$.
The operad $\Omega P^{\text{!`}}$ is the $P_\infty$-operad.
Hence, to a $P$-algebra structure on a vector space $V$, given by $P\to \mathcal{E}nd(V)$,
corresponds via
\[
\xymatrix{
P_\infty := \Omega P^{\text{!`}} \ar[r]^-{\sim} \ar[dr] & P \ar[d]\\
 & \mathcal{E}nd(V)
}
\]
a $P_\infty$-algebra (also called homotopy $P$-algebra) structure on $V$.
\begin{thm}[Ginzburg-Kapranov]\cite{GK94}.\label{GinzburgKapranovThm}
Let $P$ be a quadratic Koszul operad. A $P_\infty$-structure on a graded vector space $V$, in the sense of a representation on $V$ of the DG operad $P_{\infty}:=\Omega P^{\text{!`}}$, is equivalent (in the finite-dimensional setting) to a square-zero derivation of degree $-1$ on the free $P^{\text{!}}$-algebra over $s^{-1}V^*$
$$
P_{\infty}-\text{algebra on }V\quad\Leftrightarrow\quad d\in \mathrm{Der}_{-1}(\mathrm{Free}_{P^\text{!}}(s^{-1}V^*)),\ d^2=0.
$$
Similarly $P_\infty$-structure on $V$ (here, no finite-dimensional requirement is needed) is equivalent to the square-zero coderivation of degree $-1$ on the cofree $P^{\text{!}}$-coalgebra over $sV$
$$
P_{\infty}-\text{algebra on }V\quad\Leftrightarrow\quad D\in \mathrm{CoDer}_{-1}(\mathrm{CoFree}_{P^\text{!}}(sV)),\ D^2=0.
$$
\end{thm}\medskip
\subsection{Lie infinity algebras}
 Usual differential graded Lie  algebras can be considered as a very special case of $L_\infty$ algebras.  One may consider them as objects of the category of $L_\infty$-algebras. Thus we radically increase morphisms. This has very important applications in mathematical physics, for example it was essentially used in the famous Kontsevich's proof of the existence of deformation quantization for any Poisson manifold.\\
Lie infinity algebras can be constructed from the general operadic approach. Let us be more precise: the quadratic operad $\mathcal{L}ie$, which controls graded Lie algebras is Koszul. Its Koszul dual operad ${\mathcal{L}ie}^{\text{!}}=\mathcal{C}om$ controls graded commutative associative algebras. Lie infinity operad $\mathcal{L}ie_\infty$, by definition, is the cobar construction $\Omega\mathcal{L}ie^{\text{!`}}$. The representation of the operad $\mathcal{L}ie_\infty$ on a vector space $V$ defines $L_{\infty}$ ($\mathrm{Lie}_\infty$) algebra on $V$. Due to the Ginzburg-Kapranov theorem~(\ref{GinzburgKapranovThm}) $L_{\infty}$ algebra on $V$ is equivalent to a square-zero coderivation $D$ of degree $-1$ on the free ${\mathcal{L}ie}^{\text{!}}=\mathcal{C}om$ coalgebra over $sV$, that is the reduced symmetric coalgebra $\overline{S}^c(sV)$.\\
\subsubsection{Definition of Lie infinity algebra}
 Here we give the classical definition of $L_\infty$ algebra and then show that it coincides with the operadic one.
\begin{defi}\label{LieInftyAlgebraDefinition}
{\it $L_\infty$ algebra on a graded vector space $V$} is given by the family of multilinear maps
$l_i: V^{\otimes i}\to V$ of degrees $(i-2)$ s.t. the following conditions holds:\\
- {\it the graded antisymmetry}
\begin{equation}\label{antisimmetryConditionLieInfinity}
l_i(v_1,...,v_i)=\mathrm{sign}(\sigma)\cdot\varepsilon(\sigma)\cdot l_i(v_{\sigma(1)},...,v_{\sigma(i)}),
\end{equation}
- {\it higher Jacobi identities.} For any $n>0$:
\begin{equation}\label{higherJacobiIdentities}
\sum\limits_{i+j-1=n}\sum\limits_{\sigma\in Sh(i,j-1)}(-1)^{i(j-1)}\mathrm{sign}(\sigma)\cdot\varepsilon(\sigma)\cdot l_j(l_i(v_{\sigma(1)},...,v_{\sigma(i)}),v_{\sigma(i+1)},...,v_{\sigma(i+j-1)}).
\end{equation}
\end{defi}\bigskip
\begin{rem}
One can reformulate the definition of $L_\infty$ algebra and say that the structure maps $\{l_i\}$ are defined on the on the reduced exterior algebra $\overline{\Lambda}(V)$ and $l_i:\Lambda^i(V)\to V$. Then the graded antisymmetric conditions (\ref{antisimmetryConditionLieInfinity})  hold automatically.
\end{rem}

\begin{prop}
The operadic definition of $L_\infty$ algebra is equivalent to the  classical one. Namely $L_\infty$ algebra (in the classical sense) over a finite dimensional graded vector space $V$ is given by the differential
$d$ on the reduced symmetric algebra $\overline{S}(s^{-1}V^*)$ or dually by the codifferential $D=d^*$ on the reduced symmetric coalgebra $\overline{S}^c(sV)$. And the condition
$d^2=0$ (dually $D^2=0$) encodes the bunch of higher Jacobi identities (\ref{higherJacobiIdentities}).
\end{prop}
\begin{proof}
Consider an arbitrary finite dimensional graded vector space $W$ and the differential $d$ on the reduced symmetric algebra $\overline{S}(W)$. The derivation $d$ is given by its action on the generators, that is by the action on the vector space $W$. By applying the Leibniz rule one can get the action of the derivation on the whole space $\overline{S}(W)$ knowing only its action on $W$. For any element $w_1...w_p\in S^p(W)$:
\begin{equation}\label{actionOfDerivationOnSymmetricAlgebra}
\begin{array}{l}
d(w_1...w_p)=\sum\limits_{i=1}^p(-1)^{\overline{w_1}+...+\overline{w_{i-1}}}(w_1...dw_i...w_p)=\\[4mm]
=(d\odot \underbrace{\mathrm{id}\odot...\odot \mathrm{id}}_{p-1})(w_1...w_p)=(d\odot \mathrm{id}^{\odot (p-1)})(w_1...w_p).
\end{array}
\end{equation}
If for an arbitrary -1-degree derivation $d$ on $\overline{S}(W)$ the action of $d^2$ on the generators is zero then it is zero on the whole space $\overline{S}(W)$, and therefore the derivation becomes a differential. It follows from the following identity:
$$
d^2(w\cdot v)=d^2w\cdot v+w\cdot d^2 v=0.
$$
The action of the differential $d$ on $\overline{S}(W)$ is given by its action on generators, that is the  series of linear maps
$$d_p:W\to S^p(W),\ d|_W=d_1+d_2+...$$

and the condition that $d^2=0$ reads as follows:
\begin{equation*}
\begin{array}{l}
d(dw)=\sum\limits_{p=1}^{\infty}d(d_pw)\stackrel{(\ref{actionOfDerivationOnSymmetricAlgebra})}{=}\sum\limits_{p=1}^{\infty}(d\odot \mathrm{id}^{\odot (p-1)})d_pw=\\
=\sum\limits_{p=1}^{\infty}\sum\limits_{k=1}^{\infty}(d_k\odot \mathrm{id}^{\odot (p-1)})d_pw=\sum\limits_{n=1}^{\infty}\sum\limits_{p+k-1=n}(d_k\odot \mathrm{id}^{\odot (p-1)})d_pw=0.
\end{array}
\end{equation*}
In the dual language, the transposed map $D=d^*$ is a codifferential on the reduced symmetric coalgebra $\overline{S}^c(W^*)$. And the transposition of the last identity will encode that $D^2=0$:
\begin{equation}\label{equationForCodiffereintal}
\begin{array}{l}
\sum\limits_{n=1}^{\infty}\sum\limits_{p+k-1=n}D_p(D_k\odot \mathrm{id}^{\odot (p-1)})=0,
\end{array}
\end{equation}
where the transposed maps $D_p=d^*_p: S^p(W^*)\to W^*$ called the corestrictions of the codifferential $D$.\medskip

In equation (\ref{equationForCodiffereintal}) the weight of the operator inside the sum $\sum\limits_{n=1}^{\infty}$ is equal to $2-p-k=1-n$, so it depends only on $n$, that is the last equation splits into the series of equations:
\begin{equation}\label{operatorForLieInfinity}
(D^2)_{n}=\sum\limits_{p+k-1=n}D_p(D_k\odot \mathrm{id}^{\odot (p-1)})=0,\quad\mbox{where }n>0.
\end{equation}
The condition $D^2=0$ on the reduced symmetric coalgebra $\overline{S}^c(sV)$ reads as follows:
$$
\sum\limits_{p+k-1=n}D_p(D_k\odot \mathrm{id}^{\odot (p-1)})(sv_1...sv_n)=0,\quad n>0.\\
$$
We choose the representative element $sv_1\otimes...\otimes sv_1\in sv_1...sv_n$ then the last identity becomes
$$
\sum\limits_{p+k-1=n}D_p\circ(D_k\otimes \mathrm{id}^{\otimes (p-1)})\circ\left(\sum\limits_{\sigma\in Sh(k,p-1)}\sigma\right)(sv_1\otimes...\otimes sv_n)=0,\quad n>0.\\
$$
Now we insert identities of the type $(-1)^{\frac{i(i-1)}{2}}s^{\otimes i}(s^{-1})^{\otimes i}=\mathrm{id}^{\otimes i}$ in two places:
\begin{equation}\label{LieInfinityBeforeEvaluation}
\begin{array}{l}
\sum\limits_{p+k-1=n}\overbrace{s^{-1}D_ps^{\otimes p}}^{l_p} (-1)^{\frac{p(p-1)}{2}} \overbrace{\left(s^{-1}\right)^{\otimes p}(D_k\otimes \mathrm{id}^{\otimes (p-1)})s^{\otimes n}}^{\pm l_k\otimes \mathrm{id}^{\otimes(p-1)}}\circ\\
\circ(-1)^{\frac{n(n-1)}{2}} \underbrace{\left(s^{-1}\right)^{\otimes n}\left(\sum\limits_{\sigma\in Sh(k,p-1)}\sigma\right)s^{\otimes n}}_{\sum\limits_{\sigma\in Sh(k,p-1)}\pm\sigma}=0,\quad n>0.\\
\end{array}
\end{equation}
The precise signs in the formula above are the following:
\begin{enumerate}
\item ${s^{-1}D_ps^{\otimes p}}=l_p,$
\item ${\left(s^{-1}\right)^{\otimes p}(D_k\otimes \mathrm{id}^{\otimes (p-1)})s^{\otimes (p+k-1)}}={(-1)^{\frac{p(p-1)}{2}+k(p-1)} l_k\otimes \mathrm{id}^{\otimes(p-1)}},$
\item ${\left(s^{-1}\right)^{\otimes n}\left(\sum\limits_{\sigma\in Sh(k,p-1)}\sigma\right)s^{\otimes n}}={\sum\limits_{\sigma\in Sh(k,p-1)}(-1)^{\frac{n(n-1)}{2}}\mathrm{sign}(\sigma)\cdot\sigma}.$
\end{enumerate}
We evaluate the operators (\ref{LieInfinityBeforeEvaluation}) on the elements $v_1\otimes...\otimes v_n$ and get
$$
\sum\limits_{k+p-1=n}\sum\limits_{\sigma\in Sh(k,p-1)}(-1)^{k(p-1)}\mathrm{sign}(\sigma)\cdot\varepsilon(\sigma)\cdot l_p(l_k(v_{\sigma(1)},...,v_{\sigma(k)}),v_{\sigma(k+1)},...,v_{\sigma(k+p-1)}).
$$
The structure maps $l_i=D^{susp}_i$ have a degree $i-2$, and they act on the reduced exterior coalgebra $\overline{\Lambda}^c(V)$ (as a vector space is isomorphic to the reduced exterior algebra $\overline{\Lambda}(V)$). That guarantees that the graded antisymmetric conditions~(\ref{antisimmetryConditionLieInfinity}) hold true.
\end{proof}
\begin{defi}
If the degrees of the graded vector space $V$ are concentrated from $0$ to $n-1$:
$V=V_0\oplus...\oplus V_{n-1}$ then the $L_\infty$ structure on $V$ is called {\it $n$-term $L_\infty$-algebra}.
\end{defi}
\subsubsection{Lie infinity algebra morphism}
\begin{defi}
The morphism between two $L_\infty$ algebras over $V$ and $W$ is given by the family of multilinear maps $\varphi_i:\Lambda^i(V)\to W$ of degree $(i-1)$ which for any $n>0$ satisfy the following identities:\\
\begin{equation}\label{lieInfinityAlgebraMorphismIdentities}
\begin{array}{l}
\sum\limits_{p=1}^n\ \sum\limits_{k_1+...+k_p=n}\ \sum\limits_{\sigma\in Sh(k_1,...,k_p)}(-1)^{\frac{p(p-1)}{2}+\sum\limits_{i=1}^{p-1}(p-i)k_i}\cdot(-1)^{\sum\limits_{i=2}^p(k_i-1)(\overline{v}_{\sigma(1)}+...+\overline{v}_{\sigma(k_1+...+k_{i-1})})
}\times\\
\times \mathrm{sign}(\sigma)\cdot\varepsilon(\sigma)\cdot l_p(\varphi_{k_1}(v_{\sigma(1)},...,v_{\sigma(k_1)}),...,\varphi_{k_p}(v_{\sigma(k_1+...+k_{p-1}+1)},...,v_{\sigma(k_1+...+k_p)}))=
\\[3mm]
=\sum\limits_{k+p-1=n}\sum\limits_{\sigma\in Sh(k,p-1)}(-1)^{k(p-1)}\cdot\mathrm{sign}(\sigma)\cdot\varepsilon(\sigma)\cdot \varphi_p(l_k(v_{\sigma(1)},...,v_{\sigma(k)}),v_{\sigma(k+1)},...,v_{\sigma(k+p-1)}).
\end{array}
\end{equation}
\end{defi}
\begin{prop}
$L_\infty$-algebra morphism between $L_\infty$-algebras over $V$ and $W$ is given by the differential algebra morphism $f: S(s^{-1}W^*)\to S(s^{-1}V^*)$, or dually by the codifferential coalgebra morphism $F=f^*:S^c(sV)\to S^c(sW)$. The morphism condition $fd-df=0$ (or dually $FD-DF=0$) encodes the bunch of higher identities~(\ref{lieInfinityAlgebraMorphismIdentities}).
\end{prop}
\begin{proof}
Consider an arbitrary graded vector spaces $U$ and $U'$. An arbitrary symmetric algebra morphism $f:S(U)\to S(U')$ is given by its action on generators:
\begin{equation}\label{actionOfMorphismOnSymmetricAlgebra}
f(u_1...u_p)=f(u_1)\cdot...\cdot f(u_p)=\underbrace{(f\odot...\odot f)}_{p}(u_1...u_p).
\end{equation}
If for the algebra morphism $f$ the condition $df-fd=0$ holds on the generators then it also holds on the whole space $S(U)$, it follows from the fact that
$$
(fd-df)(u\cdot w)=(fd-df)u\cdot w +(-1)^{\overline{u}}u\cdot (fd-df)w.
$$
The action of the differential algebra morphism is given by the series of maps $f_p: U\to S^p(U')$, and the condition that the algebra morphism $f$ is a differential algebra morphism reads as follows:
$$
\begin{array}{l}
(fd-df)u=f\sum\limits^\infty_{p=1}d_pu-d\sum\limits^\infty_{p=1}f_pu\stackrel{(\ref{actionOfDerivationOnSymmetricAlgebra}), (\ref{actionOfMorphismOnSymmetricAlgebra})}{=}\sum\limits^\infty_{p=1}f^{\odot p}d_pu-\sum\limits^\infty_{p=1}(d\odot \mathrm{id}^{\odot (p-1)})f_pu=\\
=\sum\limits_{n=1}^\infty\ \sum\limits^n_{p=1}\ \sum\limits_{k_1+...+k_p=n}(f_{k_1}\odot...\odot f_{k_p})d_pu-\sum\limits_{n=1}^\infty\sum\limits^\infty_{p+k-1=n}(d_k\odot \mathrm{id}^{\odot (p-1)})f_pu=0.
\end{array}
$$
In the dual language, the transposed map $F=f^*$ is a morphism of the reduced symmetric coalgebras $F:\overline{S}^{c}(U'^*)\to\overline{S}^{c}(U^*)$. And the transposition of the last identity  encodes that $FD-DF=0$:
\begin{equation}\label{equationForSymmetricCoalgebraMorphism}
\sum\limits_{n=1}^\infty\ \left(\sum\limits^n_{p=1}\ \sum\limits_{k_1+...+k_p=n}D_p(F_{k_1}\odot...\odot F_{k_p})-\sum\limits^\infty_{p+k-1=n}F_p(D_k\odot \mathrm{id}^{\odot (p-1)})\right)=0,
\end{equation}
where the transposed maps $F_k=f^*_k: S^k(U'^*)\to U^*$ are called the corestrictions of the morphism  $F$.\medskip

In equation (\ref{equationForSymmetricCoalgebraMorphism}) the weight of the operator inside the sum $\sum\limits_{n=1}^{\infty}$ is equal to $1-n$, so it depends only on $n$, i.e. the last equation splits into the series of equations:
\begin{equation}\label{operatorForLieInfinityMorphism}
 \sum\limits^n_{p=1}\ \sum\limits_{k_1+...+k_p=n}D_p(F_{k_1}\odot...\odot F_{k_p})-\sum\limits^\infty_{p+k-1=n}F_p(D_k\odot \mathrm{id}^{\odot (p-1)})=0,\mbox{ where }n>0.
\end{equation}
We choose the representative element such that $[sv_1\otimes...\otimes sv_1]= sv_1...sv_n\in S^n(sV)$. And
the condition $FD-DF=0:\overline{S}^c(sV)\to\overline{S}^c(sW)$  reads as follows:
$$
 \begin{array}{l}
 \sum\limits^n_{p=1}\ \sum\limits_{k_1+...+k_p=n}D_p(F_{k_1}\otimes...\otimes F_{k_p})\circ\left(\sum\limits_{\sigma\in Sh(k_1,...,k_p)}\sigma\right)(sv_1\otimes...\otimes sv_n)=\\
 \sum\limits^\infty_{p+k-1=n}F_p\circ(D_k\otimes \mathrm{id}^{\otimes (p-1)})\circ\left(\sum\limits_{\sigma\in Sh(k,p-1)}\sigma\right)(sv_1\otimes...\otimes sv_n),\ n>0.
 \end{array}
$$
Now we insert in certain places the identities of the type $(-1)^{\frac{i(i-1)}{2}}s^{\otimes i}(s^{-1})^{\otimes i}=\mathrm{id}^{\otimes i}$:
\begin{equation}\label{LieInfinityMorphismBeforeEvaluation}
 \begin{array}{l}
 \sum\limits^n_{p=1}\ \sum\limits_{k_1+...+k_p=n}\overbrace{s^{-1}D_ps^{\otimes p}}^{l_p}(-1)^{\frac{p(p-1)}{2}}\overbrace{(s^{-1})^{\otimes p}(F_{k_1}\otimes...\otimes F_{k_p})s^{\otimes n}}^{\pm \varphi_{k_1}\otimes...\otimes\varphi_{k_n}}\circ\\
 (-1)^{\frac{n(n-1)}{2}}\circ\overbrace{(s^{-1})^{\otimes n}\left(\sum\limits_{\sigma\in Sh(k_1,...,k_p)}\sigma\right)s^{\otimes n}}^{\sum\limits_{\sigma\in Sh(k_1,...,k_p)}\pm\sigma}=\\[7mm]
= \sum\limits^\infty_{p+k-1=n}\overbrace{s^{-1}F_ps^{\otimes p}}^{\varphi_p}(-1)^{\frac{p(p-1)}{2}}\overbrace{(s^{-1})^{\otimes p}(D_k\otimes \mathrm{id}^{\otimes (p-1)})s^{\otimes n}}^{\pm l_k\otimes \mathrm{id}^{\otimes(p-1)}}\circ\\
\circ(-1)^{\frac{n(n-1)}{2}}\overbrace{(s^{-1})^{\otimes n}\left(\sum\limits_{\sigma\in Sh(k,p-1)}\sigma\right)s^{\otimes n}}^{\sum\limits_{\sigma\in Sh(k,p-1)}\pm\sigma},\ n>0.
 \end{array}
\end{equation}
The precise signs in the formula above are the following:
\begin{enumerate}
\item ${s^{-1}D_ps^{\otimes p}}=l_p,$
\item $(s^{-1})^{\otimes p}(F_{k_1}\otimes...\otimes F_{k_p})s^{\otimes (k_1+...+k_p)}=(-1)^{\sum\limits_{i=1}^{p-1}(p-i)k_i}(\varphi_{k_1}\otimes...\otimes \varphi_{k_p}),$
\item ${s^{-1}F_ps^{\otimes p}}=\varphi_p,$
\item ${\left(s^{-1}\right)^{\otimes p}(D_k\otimes \mathrm{id}^{\otimes (p-1)})s^{\otimes (p+k-1)}}={(-1)^{\left[\frac{p(p-1)}{2}+k(p-1)\right]} \cdot l_k\otimes \mathrm{id}^{\otimes(p-1)}},$
\item ${\left(s^{-1}\right)^{\otimes n}\sigma s^{\otimes n}}=(-1)^{\frac{n(n-1)}{2}}\mathrm{sign}(\sigma)\cdot\sigma$.
\end{enumerate}
We evaluate the operators (\ref{LieInfinityMorphismBeforeEvaluation}) on the elements $v_1\otimes...\otimes v_n$ and get identities (\ref{lieInfinityAlgebraMorphismIdentities}).
\end{proof}

\newpage
\section{Higher categorified algebras versus bounded
homotopy algebras}\label{HigherCatAlgVSBoundHomAlg}
The following research paper was published in `Theory and Applications of Categories', 25(10) (2011), 251-275 (joint work with Ashis Mandal and Norbert Poncin).

\subsection{Introduction}

Higher structures -- infinity algebras and other objects up to
homotopy, higher categories, ``oidified'' concepts, higher Lie
theory, higher gauge theory... -- are currently intensively
investigated. In particular, higher generalizations of Lie
algebras have been conceived under various names, e.g. Lie
infinity algebras, Lie $n$-algebras, quasi-free differential
graded commutative associative algebras ({\small qfDGCAs} for
short), $n$-ary Lie algebras, see e.g. \cite{Dzh05}, crossed
modules \cite{MP09} ...\medskip

More precisely, there are essentially two ways to increase the
flexibility of an algebraic structure: homotopification and
categorification.\medskip

Homotopy, sh or infinity algebras \cite{Sta63} are homotopy
invariant extensions of differential graded algebras. This
property explains their origin in BRST of closed string field
theory. One of the prominent applications of Lie infinity algebras
\cite{LS93} is their appearance in Deformation Quantization of
Poisson manifolds. The deformation map can be extended from
differential graded Lie algebras ({\small DGLA}s) to
$L_{\infty}$-algebras and more precisely to a functor from the
category {\tt L}$_{\infty}$ to the category {\tt Set}. This
functor transforms a weak equivalence into a bijection. When
applied to the {\small DGLA}s of polyvector fields and
polydifferential operators, the latter result, combined with the
formality theorem, provides the 1-to-1 correspondence between
Poisson tensors and star products.\medskip

On the other hand, categorification \cite{CF94}, \cite{Cra95} is
characterized by the replacement of sets (resp. maps, equations)
by categories (resp. functors, natural isomorphisms). Rather than
considering two maps as equal, one details a way of identifying
them. Categorification is a sharpened viewpoint that leads to
astonishing results in TFT, bosonic string theory... Categorified
Lie algebras, i.e. Lie 2-algebras (alternatively, semistrict Lie
2-algebras) in the category theoretical sense, have been
introduced by J. Baez and A. Crans \cite{BC04}. Their
generalization, weak Lie 2-algebras (alternatively, Lie
2-algebras), has been studied by D. Roytenberg
\cite{Roy07}.\medskip

It has been shown in \cite{BC04} that categorification and
homotopification are tightly connected. To be exact, Lie
2-algebras and 2-term Lie infinity algebras form equivalent
2-categories. Due to this result, Lie $n$-algebras are often
defined as sh Lie algebras concentrated in the first $n$ degrees
\cite{Hen08}. However, this `definition' is merely a
terminological convention, see e.g. Definition 4 in
\cite{SS07Structure}. On the other hand, Lie infinity algebra
structures on an $\mathbb{N}$-graded vector space $V$ are in 1-to-1
correspondence with square 0 degree -1 (with respect to the
grading induced by $V$) coderivations of the free reduced graded
commutative associative coalgebra $S^c(s V)$, where $s$ denotes
the suspension operator, see e.g. \cite{SS07Structure} or
\cite{GK94}. In finite dimension, the latter result admits a
variant based on {\small qfDGCAs} instead of coalgebras. Higher
morphisms of free {\small DGCAs} have been investigated under the
name of derivation homotopies in \cite{SS07Structure}. Quite a
number of examples can be found in \cite{SS07Zoo}.\medskip

Besides the proof of the mentioned correspondence between Lie
2-algebras and 2-term Lie infinity algebras, the seminal work
\cite{BC04} provides a classification of all Lie infinity
algebras, whose only nontrivial terms are those of degree 0 and
$n-1$, by means of a Lie algebra, a representation and an
$(n+1)$-cohomology class; for a possible extension of this
classification, see \cite{Bae07}.

In this paper, we give an explicit categorical definition of Lie
3-algebras and prove that these are in 1-to-1 correspondence with
the 3-term Lie infinity algebras, whose bilinear and trilinear
maps vanish in degree $(1,1)$ and in total degree 1, respectively.
Note that a `3-term' Lie infinity algebra implemented by a
4-cocycle \cite{BC04} is an example of a Lie 3-algebra in the
sense of the present work.\medskip

The correspondence between categorified and bounded homotopy
algebras is expected to involve classical functors and chain maps,
like e.g. the normalization and Dold-Kan functors, the (lax and
oplax monoidal) Eilenberg-Zilber and Alexander-Whitney chain maps,
the nerve functor... We show that the challenge ultimately resides
in an incompatibility of the cartesian product of linear
$n$-categories with the monoidal structure of this category, thus
answering a question of \cite{Roy07}.\medskip

The paper is organized as follows. Section~\ref{KMPSection2} contains all relevant
higher categorical definitions. In Section~\ref{KMPSection3}, we define Lie
3-algebras. Section~\ref{KMPSection4} contains the proof of the mentioned
1-to-1 correspondence between categorified algebras and truncated
sh algebras -- the main result of this paper. A specific aspect of
the monoidal structure of the category of linear $n$-categories is
highlighted in Section~\ref{KMPSection5}. In Section~\ref{KMPSection6}, we show that this
feature is an obstruction to the use of the Eilenberg-Zilber map
in the proof of the correspondence ``bracket functor -- chain
map''.

\subsection{Higher linear categories and bounded chain complexes of vector spaces}\label{KMPSection2}

Let us emphasize that notation and terminology used in the present
work originate in \cite{BC04}, \cite{Roy07}, as well as in
\cite{Lei04}. For instance, a linear $n$-category will be an (a
strict) $n$-category \cite{Lei04} in {\tt Vect}. Categories in
{\tt Vect} have been considered in \cite{BC04} and also called
internal categories or 2-vector spaces. In \cite{BC04}, see
Sections 2 and 3, the corresponding morphisms (resp. 2-morphisms)
are termed linear functors (resp. linear natural transformations),
and the resulting 2-category is denoted by {\tt VectCat} and also
by {\tt 2Vect}. Therefore, the $(n+1)$-category made up by linear
$n$-categories ($n$-categories in {\tt Vect} or $(n+1)$-vector
spaces), linear $n$-functors... will be denoted by {\tt Vect}
$n$-{\tt Cat} or {\tt$(n+1)$Vect}.\medskip

The following result is known. We briefly explain it here as its
proof and the involved concepts are important for an easy reading
of this paper.

\begin{prop}\label{EquivCat} The categories {\tt Vect} $n$-{\tt Cat} of linear
$n$-categories and linear $n$-functors and {\tt C}$^{n+1}(${\tt Vect}$)$ of
$(n+1)$-term chain complexes of vector spaces and linear chain maps
are equivalent.\end{prop}

We first recall some definitions.

\begin{defi} An {\bf $\mathbf n$-globular vector space} $L$, $n\in \mathbb{N}$, is a
sequence
\begin{equation}
\label{GlobVs}L_n\stackrel{s,t}{\rightrightarrows}L_{n-1}\stackrel{s,t}{\rightrightarrows}\ldots\stackrel{s,t}{\rightrightarrows}L_0\rightrightarrows
0,
\end{equation}
of vector spaces $L_m$ and linear maps $s,t$ such that
\begin{equation}
s(s(a))=s(t(a))\;\;\mbox{and}\;\;t(s(a))=t(t(a)),\label{GlobCond}
\end{equation}
for any $a\in L_m,$ $m\in\{1,\ldots,n\}$. The maps $s,t$ are
called {\bf source map} and {\bf target map}, respectively, and
any element of $L_m$ is an {\bf m-cell}.\end{defi}

By higher category we mean in this text a {\bf strict} higher
category. Roughly, a linear $n$-category, $n\in\mathbb{N}$, is an
$n$-globular vector space endowed with compositions of $m$-cells,
$0<m\le n$, along a $p$-cell, $0\le p<m$, and an identity
associated to any $m$-cell, $0\le m<n$. Two $m$-cells $(a,b)\in
L_m\times L_m$ are composable along a $p$-cell, if
$t^{m-p}(a)=s^{m-p}(b)$. The composite $m$-cell will be denoted by
$a\circ_p b$ (the cell that `acts' first is written on the left)
and the vector subspace of $L_m\times L_m$ made up by the pairs of
$m$-cells that can be composed along a $p$-cell will be denoted by
$L_m\times_{L_p}L_m$. The following figure schematizes the
composition of two 3-cells along a 0-, a 1-, and a
2-cell.\vspace{5mm}
$$
\xymatrix@C=1.2pc@R=2pc{
&&\ar@/^1.6pc/@{->}[dd]\ar@/_1.6pc/@{->}[dd]&&&&\ar@/^1.6pc/@{->}[dd]\ar@/_1.6pc/@{->}[dd]&&\\
\bullet\ar@/^2.6pc/@{->}[rrrr]\ar@/_2.6pc/@{->}[rrrr]&\ar@{->}[rr]&&&\bullet\ar@/^2.6pc/@{->}[rrrr]\ar@/_2.6pc/@{->}[rrrr]&\ar@{->}[rr]&&&\bullet\\
&&&&&&}
\hspace{1cm}
\xymatrix@C=1.2pc@R=1pc{
&&\ar@/^1.6pc/@{->}[dd]\ar@/_1.6pc/@{->}[dd]&&\\
&\ar@{->}[rr]&&&\\
\bullet\ar@{->}[rrrr]\ar@/^3.2pc/@{->}[rrrr]\ar@/_3.2pc/@{->}[rrrr]&&\ar@/^1.6pc/@{->}[dd]\ar@/_1.6pc/@{->}[dd]&&\bullet\\
&\ar@{->}[rr]&&&\\
&&&&&&
}
\hspace{0cm}
\xymatrix@C=1.2pc@R=2pc{
&&\ar@{->}[dd]\ar@/^1.6pc/@{->}[dd]\ar@/_1.6pc/@{->}[dd]&&\\
\bullet\ar@/^2.6pc/@{->}[rrrr]\ar@/_2.6pc/@{->}[rrrr]&\ar@{->}[r]&\ar@{->}[r]&&\bullet\\
&&&&&&
}
$$

\begin{defi} A {\bf linear n-category}, $n\in\mathbb{N}$, is an $n$-globular vector space $L$ (with source and target maps $s,t$) together
with, for any $m\in\{1,\ldots,n\}$ and any $p\in\{0,\ldots,m-1\}$,
a linear composition map $\circ_p:L_m\times_{L_p}L_m\to L_m$ and,
for any $m\in\{0,\ldots,n-1\}$, a linear identity map $1:L_m\to
L_{m+1}$, such that the properties

\begin{itemize}\item for $(a,b)\in L_m\times_{L_p}L_m$, $$\mbox{if }\; p=m-1, \mbox{ then }\, s(a\circ_p b)=s(a) \mbox{ and }\; t(a\circ_pb)=t(b),$$
$$\mbox{if }\; p\le m-2, \mbox{ then }\, s(a\circ_p b)=s(a)\circ_ps(b) \mbox{ and }\; t(a\circ_pb)=t(a)\circ_pt(b),$$
\item $$s(1_a)=t(1_a)=a,$$ \item for any $(a,b),(b,c)\in
L_m\times_{L_p} L_m$, $$(a\circ_pb)\circ_p
c=a\circ_p(b\circ_pc),$$ \item
$$1^{m-p}_{s^{m-p}a}\circ_pa=a\circ_p1^{m-p}_{t^{m-p}a}=a$$
\end{itemize} are verified, as well as the compatibility
conditions
\begin{itemize}\item for $q<p$, $(a,b),(c,d)\in
L_m\times_{L_p}L_m$ and $(a,c),(b,d)\in L_m\times_{L_q}L_m$,
$$(a\circ_pb)\circ_q(c\circ_pd)=(a\circ_qc)\circ_p(b\circ_qd),$$
\item for $m<n$ and $(a,b)\in L_m\times_{L_p}L_m$,
$$1_{a\circ_pb}=1_a\circ_p1_b.$$
\end{itemize}
\end{defi}

The morphisms between two linear $n$-categories are the linear
$n$-functors.

\begin{defi} A {\bf linear n-functor} $F:L\to L'$ between two linear
$n$-categories is made up by linear maps $F:L_m\to L'_m$,
$m\in\{0,\ldots,n\}$, such that the categorical structure --
source and target maps, composition maps, identity maps -- is
respected.
\end{defi}

Linear $n$-categories and linear $n$-functors form a category {\tt
Vect} $n$-{\tt Cat}, see Proposition \ref{EquivCat}. To disambiguate this proposition, let us specify that the objects of {\tt C}$^{n+1}(${\tt Vect}$)$ are the complexes whose underlying vector space $V=\oplus_{i=0}^nV_i$ is made up by $n+1$ terms $V_i$.\medskip

The proof of Proposition \ref{EquivCat} is based upon the following result.

\begin{prop}\label{UniqueCatStr}
Let $L$ be any $n$-globular vector space with linear identity maps. If $s_m$ denotes the restriction of the source map to $L_m$, the vector spaces
$L_m$ and $L'_m:=\oplus_{i=0}^mV_i$, $V_i:=\mathrm{ker}s_i$,
$m\in\{0,\ldots,n\}$, are isomorphic. Further, the $n$-globular
vector space with identities can be completed in a unique way by
linear composition maps so to form a linear $n$-category. If we
identify $L_m$ with $L_m'$, this unique linear $n$-categorical
structure reads
\begin{equation}
\label{source}
s(v_0,\ldots,v_m)=(v_{0},\ldots,v_{m-1}),
\end{equation}
\begin{equation}
\label{target}t(v_0,\ldots,v_m)=(v_0,\ldots,v_{m-1}+tv_m),
\end{equation}
\begin{equation}\label{identity}
1_{(v_0,\ldots,v_m)}=(v_{0},\ldots,v_m,0),
\end{equation}
\begin{equation}\label{composition}
(v_0,\ldots,v_m)\circ_p(v'_0,\ldots,v'_m)=(v_0,\ldots,v_p,v_{p+1}+v_{p+1}',\ldots,v_m+v_m'),
\end{equation}
where the two $m$-cells in Equation (\ref{composition}) are
assumed to be composable along a $p$-cell.
\end{prop}

\begin{proof} As for the first part of this proposition, if $m=2$ e.g., it suffices to observe that the linear
maps
$$\alpha_L: L_2'=V_0\oplus V_1\oplus V_2\ni (v_0,v_1,v_2)\mapsto
1^2_{v_0}+1_{v_1}+v_2\in L_2$$ and $$\beta_L: L_2\ni a\mapsto
(s^2a,s(a-1^2_{s^2a}),a-1_{s(a-1^2_{s^2a})}-1^2_{s^2a})\in
V_0\oplus V_1\oplus V_2=L_2'$$ are inverses of each other. For arbitrary $m\in\{0,\ldots,n\}$ and $a\in L_m$, we set $$\beta_La=\left(s^ma,\ldots, s^{m-i}(a-\sum_{j=0}^{i-1}1^{m-j}_{p_{j}\beta_La}),\ldots,a-\sum_{j=0}^{m-1}1^{m-j}_{p_{j}\beta_La}\right)\in V_0\oplus\ldots\oplus V_i\oplus\ldots\oplus V_m= L'_m,$$ where $p_j$ denotes the projection $p_j:L_m'\to V_j$ and where the components must be computed from left to right.\medskip

For the second claim, note that when reading the source, target
and identity maps through the detailed isomorphism, we get
$s(v_0,\ldots,v_m)=(v_{0},\ldots,v_{m-1})$,
$t(v_0,\ldots,v_m)=(v_0,\ldots,v_{m-1}+tv_m)$, and
$1_{(v_0,\ldots,v_m)}=(v_{0},\ldots,v_m,0)$. Eventually, set
$v=(v_0,\ldots,v_m)$ and let $(v,w)$ and $(v',w')$ be two pairs of
$m$-cells that are composable along a $p$-cell. The composability
condition, say for $(v,w)$, reads
$$(w_0,\ldots,w_p)=(v_0,\ldots,v_{p-1},v_p+tv_{p+1}).$$ It follows
from the linearity of $\circ_p:L_m\times_{L_p}L_m\to L_m$ that
$(v+v')\circ_p(w+w')=(v\circ_pw)+(v'\circ_pw')$. When taking
$w=1^{m-p}_{t^{m-p}v}$ and $v'=1^{m-p}_{s^{m-p}w'}$, we find
$$(v_0+w'_0,\ldots,v_p+w'_p,v_{p+1},\ldots,v_m)\circ_p(v_0+w'_0,\ldots,v_p+w'_p+tv_{p+1},w'_{p+1},\ldots,w'_m)$$
$$=(v_0+w'_0,\ldots,v_m+w'_m),$$ so that $\circ_p$ is necessarily the composition given by Equation
(\ref{composition}). It is easily seen that, conversely, Equations
(\ref{source}) -- (\ref{composition}) define a linear
$n$-categorical structure. \end{proof}

\begin{proof}[Proof of Proposition \ref{EquivCat}] We define functors $\mathfrak{N}:$ {\tt Vect} $n$-{\tt Cat} $\to $ {\tt C}$^{n+1}(${\tt Vect}$)$ and
$\mathfrak{G}:$ {\tt C}$^{n+1}(${\tt Vect}$)$ $\to$ {\tt Vect} $n$-{\tt Cat} that
are inverses up to natural isomorphisms.\medskip

If we start from a linear $n$-category $L$, so in particular from
an $n$-globular vector space $L$, we define an $(n+1)$-term chain
complex ${\mathfrak{N}}(L)$ by setting $V_m=\mathrm{ker}s_m\subset L_m$ and
$d_m=t_m|_{V_m}:V_m\to V_{m-1}$. In view of the globular space
conditions (\ref{GlobCond}), the target space of $d_m$ is actually
$V_{m-1}$ and we have $d_{m-1}d_mv_m=0.$\medskip

Moreover, if $F:L\to L'$ denotes a linear $n$-functor, the value
${\mathfrak{N}}(F):V\to V'$ is defined on $V_m\subset L_m$ by ${\mathfrak
{N}}(F)_m=F_m|_{V_{m}}:V_m\to V'_{m}$. It is obvious that ${\mathfrak{N}}(F)$ is a linear chain map.\medskip

It is obvious that ${\mathfrak{N}}$ respects the categorical structures
of {\tt Vect} $n$-{\tt Cat} and {\tt C}$^{n+1}$({\tt
Vect}).\medskip

As for the second functor $\mathfrak{G}$, if $(V,d)$,
$V=\oplus_{i=0}^nV_i$, is an $(n+1)$-term chain complex of vector
spaces, we define a  linear $n$-category ${\mathfrak{G}}(V)=L$,
$L_m=\oplus_{i=0}^mV_i$, as in Proposition \ref{UniqueCatStr}: the
source, target, identity and composition maps are defined by
Equations (\ref{source}) -- (\ref{composition}), except that
$tv_m$ in the {\small RHS} of Equation (\ref{target}) is replaced
by $dv_m$.\medskip

The definition of $\mathfrak{G}$ on a linear chain map $\phi:V\to V'$
leads to a linear $n$-functor ${\mathfrak{G}}(\phi):L\to L'$, which is
defined on $L_m=\oplus_{i=0}^mV_i$ by ${\mathfrak{G}}(\phi)_m=\oplus_{i=0}^m\phi_i$. Indeed, it is readily checked that
${\mathfrak{G}}(\phi)$ respects the linear $n$-categorical structures of
$L$ and $L'$.\medskip

Furthermore, ${\mathfrak{G}}$ respects the categorical structures of
{\tt C}$^{{n+1}}$({\tt Vect}) and {\tt Vect} $n$-{\tt
Cat}.\medskip

Eventually, there exists natural isomorphisms
$\alpha:\mathfrak{NG}\Rightarrow\mathrm{id}$ and
$\gamma:\mathfrak{GN}\Rightarrow\mathrm{id}$.\medskip

To define a natural transformation
$\alpha:\mathfrak{NG}\Rightarrow\mathrm{id}$, note that $L'=(\mathfrak{NG})(L)$
is the linear $n$-category made up by the vector spaces
$L'_m=\oplus_{i=0}^mV_i$, $V_i=\mathrm{ker}s_i$, as well as by the
source, target, identities and compositions defined from $V={\mathfrak{N}}(L)$ as in the above definition of ${\mathfrak{G}}(V)$, i.e. as in
Proposition \ref{UniqueCatStr}. It follows that $\alpha_L:L'\to L$,
defined by $\alpha_L:L_m'\ni(v_0,\ldots,v_m)\mapsto 1^m_{v_0}+\ldots
+1_{v_{m-1}}+v_m\in L_m$, $m\in\{0,\ldots,n\}$, which pulls the
linear $n$-categorical structure back from $L$ to $L'$, see
Proposition \ref{UniqueCatStr}, is an invertible linear
$n$-functor. Moreover $\alpha$ is natural in $L$.\medskip

It suffices now to observe that the composite $\mathfrak{GN}$ is the
identity functor.\end{proof}

Next we further investigate the category {\tt Vect} $n$-{\tt Cat}.

\begin{prop} The category {\tt Vect} $n$-{\tt Cat} admits finite products.\end{prop}

Let $L$ and $L'$ be two linear $n$-categories. The product linear
$n$-category $L\times L'$ is defined by $(L\times L')_m=L_m\times
L'_m$, $S_m=s_m\times s'_m$, $T_m=t_m\times t'_m$, $I_m=1_m\times
1'_m$, and $\bigcirc_p=\circ_p\times \circ'_p$. The compositions
$\bigcirc_p$ coincide with the unique compositions that complete
the $n$-globular vector space with identities, thus providing a
linear $n$-category. It is straightforwardly checked that the
product of linear $n$-categories verifies the universal property
for binary products.

\begin{prop}\label{3-CatStr} The category {\tt Vect}\,$2$-{\tt Cat} admits a 3-categorical
structure. More precisely, its 2-cells are the linear natural
2-transformations and its 3-cells are the linear 2-modifications.
\end{prop}

This proposition is the linear version (with similar proof) of the
well-known result that the category $2$-{\tt Cat} is a 3-category
with 2-categories as 0-cells, 2-functors as 1-cells, natural
2-transformations as 2-cells, and 2-modifications as 3-cells. The
definitions of $n$-categories and $2$-functors are similar to
those given above in the linear context (but they are formulated
without the use of set theoretical concepts). As for (linear)
natural $2$-transformations and (linear) 2-modifications, let us
recall their definition in the linear setting:

\begin{defi}
A {\bf linear natural 2-transformation} $\theta: F\Rightarrow G$ between two
linear 2-functors $F,G:{\cal C}\to{\cal D}$, between the same two
linear 2-categories, assigns to any $a\in{\cal C}_0$ a unique
$\theta_a:F(a)\to G(a)$ in ${\cal D}_1$, linear with respect to
$a$ and such that for any $\alpha:f\Rightarrow g$ in ${\cal C}_2$,
$f,g:a\to b$ in ${\cal C}_1,$ we have
\begin{equation}
F(\alpha)\circ_0
1_{\theta_b}=1_{\theta_a}\circ_0G(\alpha)\;.
\end{equation}
\end{defi}

If ${\cal C}=L\times L$ is a product linear 2-category, the last
condition reads $$F(\alpha,\beta)\circ_0
1_{\theta_{t^2\alpha,t^2\beta}}=1_{\theta_{s^2\alpha,s^2\beta}}\circ_0G(\alpha,\beta),$$
for all $(\alpha,\beta)\in L_2\times L_2$. As functors respect
composition, i.e. as
$$F(\alpha,\beta)=F(\alpha\circ_01^2_{t^2\alpha},1^2_{s^2\beta}\circ_0\beta)=F(\alpha,1^2_{s^2\beta})\circ_0
F(1^2_{t^2\alpha},\beta),$$ this naturality condition is verified if
and only if it holds true in case all but one of the 2-cells are
identities $1^2_{-}$, i.e. if and only if the transformation is
natural with respect to all its arguments separately.

\begin{defi} Let $\cal{C,D}$ be two linear
2-categories. A {\bf linear 2-modification} $\mu:\eta\Rrightarrow
\varepsilon$ between two linear natural 2-transformations
$\eta,\varepsilon:F\Rightarrow G$, between the same two linear 2-functors
$F,G:\cal C\to \cal D$, assigns to any object $a\in{\cal C}_0$ a
unique $\mu_a:\eta_a\Rightarrow\varepsilon_a$ in ${\cal D}_2$, which is
linear with respect to $a$ and such that, for any
$\alpha:f\Rightarrow g$ in ${\cal C}_2$, $f,g:a\to b$ in ${\cal
C}_1$, we have \begin{equation} \label{2-modification}
F(\alpha)\circ_0\mu_b=\mu_a\circ_0G(\alpha).\end{equation}\end{defi}

If ${\cal C}=L\times L$ is a product linear 2-category, it
suffices again that the preceding modification property be
satisfied for tuples $(\alpha,\beta),$ in which all but one 2-cells are
identities $1^2_{-}$. The explanation is the same as for natural
transformations.\medskip

Beyond linear $2$-functors, linear natural $2$-transformations,
and linear $2$-modifications, we use below multilinear cells.
Bilinear cells e.g., are cells on a product linear 2-category,
with linearity replaced by bilinearity. For instance,

\begin{defi} Let $L$, $L'$, and $L''$ be linear $2$-categories. A
{\bf bilinear 2-functor} $F:L\times L'\to L''$ is a $2$-functor
such that $F:L_m\times L'_m\to L''_m$ is bilinear for all
$m\in\{0,1,2\}$.\end{defi}

Similarly,

\begin{defi}
Let $L$, $L'$, and $L''$ be linear $2$-categories. A {\bf bilinear natural 2-transformation} $\theta: F\Rightarrow G$ between two
bilinear 2-functors $F,G:L\times L'\to L''$, assigns to any
$(a,b)\in L_0\times L'_0$ a unique $\theta_{(a,b)}:F(a,b)\to
G(a,b)$ in $L''_1$, which is bilinear with respect to $(a,b)$ and
such that for any $(\alpha,\beta):(f,h)\Rightarrow (g,k)$ in $L_2\times
L_2'$, $(f,h),(g,k):(a,b)\to (c,d)$ in $L_1\times L'_1,$ we have
\begin{equation}\label{2-naturality}
F(\alpha,\beta)\circ_0
1_{\theta_{(c,d)}}=1_{\theta_{(a,b)}}\circ_0G(\alpha,\beta)\;.
\end{equation}
\end{defi}

\subsection{Homotopy Lie algebras and categorified Lie algebras}\label{KMPSection3}

We now recall the definition of a Lie infinity (strongly homotopy
Lie, sh Lie, $L_{\infty}-$) algebra and specify it in the case of
a 3-term Lie infinity algebra.

\begin{defi} A {\bf Lie infinity algebra} is an $\mathrm{N}$-graded vector space
$V=\oplus_{i\in\mathrm{N}}V_i$ together with a family
$(\ell_i)_{i\in\mathrm{N}^*}$ of graded antisymmetric $i$-linear weight
$i-2$ maps on $V$, which verify the sequence of conditions
\begin{equation}\label{LieInftyCond} \sum_{i+j=n+1}\sum_{(i,n-i)\mbox{ --
shuffles }
\sigma}\chi(\sigma)(-1)^{i(j-1)}\ell_j(\ell_i(a_{\sigma_1},\ldots,a_{\sigma_i}),a_{\sigma_{i+1}},\ldots,a_{\sigma_n})=0,\end{equation}
where $n\in\{1,2,\ldots\}$, where $\chi(\sigma)$ is the product of
the signature of $\sigma$ and the Koszul sign defined by $\sigma$ and
the homogeneous arguments $a_1,\ldots,a_n\in V$.\end{defi}

For $n=1$, the $L_{\infty}$-condition (\ref{LieInftyCond}) reads
$\ell_1^2=0$ and, for $n=2$, it means that $\ell_1$ is a graded
derivation of $\ell_2$, or, equivalently, that $\ell_2$ is a chain
map from $(V\otimes V,\ell_1\otimes \mathrm{id}+\mathrm{id}\otimes\;
\ell_1)$ to $(V,\ell_1).$\medskip

In particular,

\begin{defi} A {\bf 3-term Lie infinity algebra} is a 3-term graded vector
space $V=V_0\oplus V_1\oplus V_2$ endowed with graded
antisymmetric $p$-linear maps $\ell_p$ of weight $p-2$,
\begin{equation}\begin{array}{ll} \ell_1:V_{i}\to V_{i-1}&(1\le i\le
2),\\\ell_2:V_i\times V_j\to V_{i+j}&(0\le i+j\le 2),\\
\ell_3:V_i\times V_j\times V_k\to V_{i+j+k+1}&(0\le i+j+k\le
1),\\\ell_4:V_0\times V_0\times V_0\times V_0\to V_2\end{array}\end{equation}
(all structure maps $\ell_p$, $p>4$, necessarily vanish), which
satisfy $L_{\infty}$-condition (\ref{LieInftyCond}) (that is
trivial for all $n>5$).\end{defi}

In this 3-term situation, each $L_{\infty}$-condition splits into
a finite number of equations determined by the various
combinations of argument degrees, see below.\medskip

On the other hand, we have the

\begin{defi}\label{Lie3Alg} A {\bf Lie 3-algebra} is a {\bf linear 2-category} $L$ endowed
with a {\bf bracket}, i.e. an antisymmetric bilinear 2-functor
$[-,-]:L\times L\to L$, which verifies the Jacobi identity up to a
{\bf Jacobiator}, i.e. a skew-symmetric trilinear natural
2-transformation
\begin{equation}\label{Jacobiator}J_{xyz}:[[x,y],z]\rightarrow[[x,z],y]+[x,[y,z]],\end{equation}
$x,y,z\in L_0$, which in turn satisfies the Baez-Crans Jacobiator
identity up to an {\bf Identiator}, i.e. a skew-symmetric
quadrilinear 2-modification $$\mu_{xyzu}:
[J_{x,y,z},1_u]\circ_0(J_{[x,z],y,u}+J_{x,[y,z],u})\circ_0([J_{xzu},1_y]+1)\circ_0([1_x,J_{yzu}]+1)$$
\begin{equation}\label{Identiator}\Rightarrow
J_{[x,y],z,u}\circ_0([J_{xyu},1_z]+1)\circ_0(J_{x,[y,u],z}+J_{[x,u],y,z}+J_{x,y,[z,u]}),\end{equation}
$x,y,z,u\in L_0$, required to verify the {\bf coherence law}
\begin{equation}\label{CohLaw0}\alpha_1+\alpha_4^{-1}=\alpha_3+\alpha_2^{-1},\end{equation} where
$\alpha_1$ -- $\alpha_4$ are explicitly given in Definitions
\ref{Alpha1} -- \ref{Alpha4} and where superscript $-1$ denotes the inverse for composition along a 1-cell.\end{defi}

Just as the {\it Jacobiator} is a natural transformation between
the two sides of the {\it Jacobi} identity, the {\it Identiator}
is a modification between the two sides of the Baez-Crans
Jacobiator {\it identity}.\medskip

In this definition ``skew-symmetric 2-transformation'' (resp.
``skew-symmetric 2-modification'') means that, if we identify
$L_m$ with $\oplus_{i=0}^m V_i$, $V_i=\ker s_i$, as in Proposition
\ref{UniqueCatStr}, the $V_1$-component of $J_{xyz}\in L_1$ (resp.
the $V_2$-component of $\mu_{xyzu}\in L_2$) is antisymmetric.
Moreover, the definition makes sense, as the source and target in
Equation (\ref{Identiator}) are quadrilinear natural
2-transformations between quadrilinear 2-functors from $L^{\times
4}$ to $L.$ These 2-functors are simplest obtained from the
{\small RHS} of Equation (\ref{Identiator}). Further, the
mentioned source and target actually are natural
2-transformations, since a 2-functor composed (on the left or on
the right) with a natural 2-transformation is again a
2-transformation.

\subsection{Lie 3-algebras in comparison with 3-term Lie infinity algebras}\label{KMPSection4}

\begin{rem}In the following, we systematically identify the vector spaces $L_m$, $m\in\{0,\ldots,n\}$, of a
linear $n$-category with the spaces $L_m'=\oplus_{i=0}^mV_i,$
$V_i=\ker s_i$, so that the categorical structure is given by
Equations (\ref{source}) -- (\ref{composition}). In addition, we
often substitute common, index-free notations $($e.g. $\alpha=(x,{\bf
f},{\bf a})$$)$ for our notations $($e.g. $v=(v_0,v_1,v_2)\in
L_2$$)$. \label{BasicIdentification}\end{rem}

The next theorem is the main result of this paper.

\begin{thm}\label{MainTheo} There exists a 1-to-1 correspondence between Lie 3-algebras and 3-term Lie infinity algebras $(V,\ell_p),$
whose structure maps $\ell_2$ and $\ell_3$ vanish on $V_1\times
V_1$ and on triplets of total degree 1, respectively.
\end{thm}

\begin{ex} There exists a 1-to-1 correspondence between $(n+1)$-term Lie infinity algebras $V=V_0\oplus V_n$ (whose intermediate terms vanish), $n\ge 2$, and $(n+2)$-cocycles of Lie
algebras endowed with a linear representation, see \cite{BC04}, Theorem 6.7. A 3-term Lie infinity algebra implemented by a 4-cocycle can therefore be viewed as a special case of a Lie 3-algebra.\end{ex}

The proof of Theorem \ref{MainTheo} consists of five lemmas.

\subsubsection{Linear 2-category -- three term chain complex of vector spaces}

First, we recall the correspondence between the underlying structures of a Lie 3-algebra and a 3-term Lie infinity algebra.

\begin{lem} There is a bijective correspondence between linear 2-categories $L$ and 3-term chain complexes of vector spaces $(V,\ell_1)$.\end{lem}

\begin{proof} In the proof of Proposition \ref{EquivCat}, we associated to any linear $2$-category $L$ a unique 3-term chain complex of vector spaces
${\mathfrak{N}}(L)=V$, whose spaces are given by $V_m=\ker s_m$,
$m\in\{0,1,2\}$, and whose differential $\ell_1$ coincides on
$V_m$ with the restriction $t_m|_{V_m}$. Conversely, we assigned
to any such chain complex $V$ a unique linear 2-category ${\mathfrak{G}}(V)=L$, with spaces $L_m=\oplus_{i=0}^mV_i$, $m\in \{0,1,2\}$
and target $t_0(x)=0,$ $t_1(x,{\bf f})=x+\ell_1{\bf f}$,
$t_2(x,{\bf f},{\bf a})=(x,{\bf f}+\ell_1{\bf a})$. In view of
Remark \ref{BasicIdentification}, the maps $\mathfrak{N}$ and $\mathfrak{G}$
are inverses of each other.\end{proof}

\begin{rem} The globular space condition is the categorical counterpart of $L_{\infty}$-condition $n=1$.\end{rem}

\subsubsection{Bracket -- chain map}

We assume that we already built $(V,\ell_1)$ from $L$ or $L$ from $(V,\ell_1)$.

\begin{lem}\label{Lem2} There is a bijective correspondence between antisymmetric bilinear 2-functors $[-,-]$ on $L$ and graded antisymmetric chain maps $\ell_2:(V\otimes V,\ell_1\otimes\mathrm{id}+\mathrm{id}\otimes \ell_1)\to (V,\ell_1)$ that vanish on $V_1\times V_1$.\end{lem}

\begin{proof} Consider first an antisymmetric bilinear ``2-map'' $[-,-]:L\times L\to L$ that verifies all functorial requirements except as concerns composition. This bracket then respects the compositions, i.e., for each pairs $(v,w), (v',w')\in L_m\times L_m$, $m\in\{1,2\}$, that are composable along a $p$-cell, $0\le p<m$, we have \begin{equation}\label{FunCompGen}
[v\circ_pv',w\circ_pw']=[v,w]\circ_p[v',w'],
\end{equation}
if and only if the following conditions hold true, for any ${\bf f,g}\in V_1$ and any ${\bf a,b}\in V_2$: \begin{equation}\label{FunComp1}
[{\bf f},{\bf g}]=[1_{t{\bf f}},{\bf g}]=[{\bf f},1_{t{\bf g}}],
\end{equation}
\begin{equation}\label{FunComp2}
[{\bf a},{\bf b}]=[1_{t{\bf a}},{\bf b}]=[{\bf a},1_{t{\bf b}}]=0,
\end{equation}
\begin{equation}\label{FunComp3}
[1_{{\bf f}},{\bf b}]=[1^2_{t{\bf f}},{\bf b}]=0.
\end{equation}
To prove the first two conditions, it suffices to compute $[{\bf f}\circ_01_{t{\bf f}},1_0\circ_0{\bf g}]$, for the next three conditions, we consider $[{\bf a}\circ_11_{t{\bf a}},1_0\circ_1{\bf b}]$ and $[{\bf a}\circ_01^2_{0},1^2_0\circ_0{\bf b}]$, and for the last two, we focus on $[1_{\bf f}\circ_01^2_{t{\bf f}},1^2_0\circ_0{\bf b}]$ and $[1_{\bf f}\circ_0(1^2_{t{\bf f}}+1_{\bf f'}),{\bf b}\circ_0{\bf b'}].$ Conversely, it can be straightforwardly checked that Equations (\ref{FunComp1}) -- (\ref{FunComp3}) entail the general requirement (\ref{FunCompGen}).\medskip

On the other hand, a graded antisymmetric bilinear weight 0 map
$\ell_2:V\times V\to V$ commutes with the differentials $\ell_1$
and $\ell_1\otimes\mathrm{id}+\mathrm{id}\otimes \ell_1$, i.e., for all
$v,w\in V$, we have
\begin{equation}\label{ChainGen}
\ell_1(\ell_2(v,w))=\ell_2(\ell_1v,w)+(-1)^v\ell_2(v,\ell_1w)
\end{equation}
(we assumed that $v$ is homogeneous and denoted its degree by $v$ as well),
if and only if, for any $y\in V_0$, ${\bf f,g}\in V_1$, and ${\bf
a}\in V_2$,
\begin{equation}\label{Chain1}
\ell_1(\ell_2({\bf f},y))=\ell_2(\ell_1{\bf
f},y),\end{equation} \begin{equation}\label{Chain2} \ell_1(\ell_2({\bf f},{\bf
g}))=\ell_2(\ell_1{\bf f},{\bf g})-\ell_2({\bf f},\ell_1{\bf g}),
\end{equation}
\begin{equation}\label{Chain3}
\ell_1(\ell_2({\bf a},y))=\ell_2(\ell_1{\bf a},y),
\end{equation}
\begin{equation}\label{Chain4}
0=\ell_2(\ell_1{\bf f},{\bf b})-\ell_2({\bf f},\ell_1{\bf
b}).
\end{equation}

\begin{rem} Note that, in the correspondence $\ell_1\leftrightarrow t$ and $\ell_2\leftrightarrow [-,-]$, Equations
(\ref{Chain1}) and (\ref{Chain3}) read as compatibility
requirements of the bracket with the target and that Equations
(\ref{Chain2}) and (\ref{Chain4}) correspond to the second
conditions of Equations (\ref{FunComp1}) and (\ref{FunComp3}),
respectively.\end{rem}

{\it Proof of Lemma \ref{Lem2} (continuation)}. To prove the
announced 1-to-1 correspondence, we first define a graded
antisymmetric chain map ${\mathfrak{N}}([-,-])=\ell_2$, $\ell_2:V\otimes
V\to V$ from any antisymmetric bilinear 2-functor $[-,-]:L\times
L\to L$.

Let $x,y\in V_0$, ${\bf f,g}\in V_1$, and ${\bf a,b}\in V_2$. Set
$\ell_2(x,y)=[x,y]\in V_0$ and $\ell_2(x,{\bf g})=[1_{x},{\bf g}]\in
V_1$. However, we must define $\ell_2({\bf f},{\bf g})\in V_2$,
whereas $[{\bf f},{\bf g}]\in V_1.$ Moreover, in this case, the
antisymmetry properties do not match. The observation
$$[{\bf f},{\bf g}]=[1_{t{\bf f}},{\bf g}]=[{\bf f},1_{t{\bf g}}]=\ell_2(\ell_1{\bf f},{\bf g})=\ell_2({\bf f},\ell_1{\bf g})$$ and Condition (\ref{Chain2})
force us to {\it define $\ell_2$ on $V_1\times V_1$ as a symmetric
bilinear map valued in $V_2\cap\ker \ell_1.$} We further set
$\ell_2(x,{\bf b})=[1^2_{x},{\bf b}]\in V_2,$ and, as $\ell_2$ is
required to have weight 0, we must set $\ell_2({\bf f},{\bf b})=0$
and $\ell_2({\bf a},{\bf b})=0.$ It then follows from the functorial
properties of $[-,-]$ that the conditions (\ref{Chain1}) --
(\ref{Chain3}) are verified. In view of Equation (\ref{FunComp3}),
Property (\ref{Chain4}) reads
$$0=[1^2_{t{\bf f}},{\bf b}]-\ell_2({\bf f},\ell_1{\bf b})=-\ell_2({\bf f},\ell_1{\bf b}).$$ In other
words, in addition to the preceding requirement, we must {\it
choose $\ell_2$ in a way that it vanishes on $V_1\times V_1$ if
evaluated on a 1-coboundary.} These conditions are satisfied if we
choose $\ell_2=0$ on $V_1\times V_1$.\medskip

Conversely, from any graded antisymmetric chain map $\ell_2$ that
vanishes on $V_1\times V_1$, we can construct an antisymmetric
bilinear 2-functor ${\mathfrak{G}}(\ell_2)=[-,-]$. Indeed, using
obvious notations, we set
$$[x,y]=\ell_2(x,y)\in L_0,\,[1_{x},1_{y}]=1_{[x,y]}\in
L_1,\,[1_{x},{\bf g}]=\ell_2(x,{\bf g})\in V_1\subset L_1.$$ Again
$[{\bf f},{\bf g}]\in L_1$ cannot be defined as $\ell_2({\bf f},{\bf
g})\in V_2$. Instead, if we wish to build a 2-functor, we must set
$$[{\bf f},{\bf g}]=[1_{t{\bf f}},{\bf g}]=[{\bf f},1_{t{\bf
g}}]=\ell_2(\ell_1{\bf f},{\bf g})=\ell_2({\bf f}, \ell_1{\bf g})\in
V_1\subset L_1,$$ which is possible in view of Equation
(\ref{Chain2}), {\it if $\ell_2$ is on $V_1\times V_1$ valued in
2-cocycles} (and in particular if it vanishes on this subspace).
Further, we define
$$
[1^2_x,1^2_y]=1^2_{[x,y]}\in L_2,\,[1^2_x,1_{{\bf
g}}]=1_{[1_x,{\bf g}]}\in L_2,\,[1^2_x,{\bf b}]=\ell_2(x,{\bf b})\in
V_2\subset L_2,\,[1_{{\bf f}},1_{{\bf g}}]=1_{[{\bf f},{\bf
g}]}\in L_2.$$ Finally, we must set $$[1_{{\bf f}},{\bf
b}]=[1^2_{t{\bf f}},{\bf b}]=\ell_2(\ell_1{\bf f},{\bf b})=0,
$$ which
is possible in view of Equation (\ref{Chain4}), {\it if $\ell_2$
vanishes on $V_1\times V_1$ when evaluated on a 1-coboundary} (and
especially if it vanishes on the whole subspace $V_1\times V_1$),
and
$$
[{\bf a},{\bf b}]=[1_{t{\bf a}},{\bf b}]=[{\bf a},1_{t{\bf
b}}]=0,
$$
which is possible.

It follows from these definitions that the bracket of $\alpha=(x,{\bf
f},{\bf a})=1^2_x+1_{{\bf f}}+{\bf a}\in L_2$ and $\beta=(y,{\bf
g},{\bf b})=1^2_y+1_{{\bf g}}+{\bf b}\in L_2$ is given by
{\begin{equation}\label{BracketFromEll2}
[\alpha,\beta]=(\ell_2(x,y),\ell_2(x,{\bf
g})+\ell_2({\bf f},tg),\ell_2(x,{\bf b})+\ell_2({\bf a},y))\in
L_2,\end{equation}}
where $g=(y,{\bf g})$. The brackets of two elements of
$L_1$ or $L_0$ are obtained as special cases of the latter result.

We thus defined an antisymmetric bilinear map $[-,-]$ that assigns
an $i$-cell to any pair of $i$-cells, $i\in\{0,1,2\}$, and that
respects identities and sources. Moreover, since Equations
(\ref{FunComp1}) -- (\ref{FunComp3}) are satisfied, the map
$[-,-]$ respects compositions provided it respects targets. For
the last of the first three defined brackets, the target condition
is verified due to Equation (\ref{Chain1}). For the fourth
bracket, the target must coincide with $[t{\bf f},t{\bf
g}]=\ell_2(\ell_1{\bf f},\ell_1{\bf g})$ and it actually coincides with
$t[{\bf f},{\bf g}]=\ell_1\ell_2(\ell_1{\bf f},{\bf g})=\ell_2(\ell_1{\bf
f},\ell_1{\bf g})$, again in view of (\ref{Chain1}). As regards the
seventh bracket, the target $t[1^2_x,{\bf b}]=\ell_1\ell_2(x,{\bf
b})=\ell_2(x,\ell_1{\bf b}),$ due to (\ref{Chain3}), must coincide
with $[1_x,t{\bf b}]=\ell_2(x,\ell_1{\bf b})$. The targets of the two
last brackets vanish and $[{\bf f},t{\bf b}]=\ell_2({\bf
f},\ell_1\ell_1{\bf b})=0$ and $[t{\bf a},t{\bf b}]=\ell_2(\ell_1{\bf
a},\ell_1\ell_1{\bf b})=0.$\medskip

It is straightforwardly checked that the maps $\mathfrak{N}$ and $\mathfrak{G}$ are inverses.\end{proof}

Note that ${\mathfrak{N}}$ actually assigns to any antisymmetric
bilinear 2-functor a class of graded antisymmetric chain maps that
coincide outside $V_1\times V_1$ and whose restrictions to
$V_1\times V_1$ are valued in 2-cocycles and vanish when evaluated
on a 1-coboundary. The map $\mathfrak{N}$, with values in chain maps,
is well-defined thanks to a canonical choice of a representative
of this class. Conversely, the values on $V_1\times V_1$ of the
considered chain map cannot be encrypted into the associated
2-functor, only the mentioned cohomological conditions are of
importance. Without the canonical choice, the map $\mathfrak{G}$ would
not be injective.

\begin{rem} The categorical counterpart of $L_{\infty}$-condition $n=2$ is the functor condition on compositions.\end{rem}

\begin{rem} A 2-term Lie infinity algebra (resp. a Lie 2-algebra) can be viewed as a 3-term Lie infinity algebra (resp. a Lie 3-algebra). The preceding correspondence then of course reduces to the correspondence of \cite{BC04}.\end{rem}

\subsubsection{Jacobiator -- third structure map}

We suppose that we already constructed $(V,\ell_1,\ell_2)$ from
$(L,[-,-])$ or $(L,[-,-])$ from $(V,\ell_1,\ell_2)$.

\begin{lem}\label{Lem3}
There exists a bijective correspondence between skew-symmetric trilinear natural 2-transformations $J:[[-,-],\bullet]\Rightarrow [[-,\bullet],-]+
[-,[-,\bullet]]$ and graded antisymmetric trilinear weight 1 maps
$\ell_3:V^{\times 3}\to V$ that verify $L_{\infty}$-condition
$n=3$ and vanish in total degree 1.
\end{lem}

\begin{proof} A skew-symmetric trilinear natural
2-transformation $J:[[-,-],\bullet]\Rightarrow [[-,\bullet],-]+[-,[-,\bullet]]$
is a map that assigns to any $(x,y,z)\in L_0^{\times 3}$ a unique $J_{xyz}:[[x,y],z]\to [[x,z],y]+[x,[y,z]]$
in $L_1$, such that for any $\alpha=(z,{\bf f},{\bf a})\in L_2$, we have
$$[[1^2_x,1^2_y],\alpha]\circ_01_{J_{x,y,\,t^2\alpha}}=1_{J_{x,y,\,s^2\alpha}}\circ_0\left([[1^2_x,\alpha],1^2_y]+[1^2_x,[1^2_y,\alpha]]\right)$$
(as well as similar equations pertaining to naturality with respect
to the other two variables). A short computation shows that the
last condition decomposes into the following two requirements
on the $V_1$- and the $V_2$-component:
\begin{equation}
\label{3a}{\bf
J}_{x,y,\,t{\bf  f}}+[1_{[x,y]},{\bf f}]=[[1_x,{\bf
f}],1_y]+[1_x,[1_y,{\bf f}]],\end{equation}
\begin{equation}\label{3b}[1^2_{[x,y]},{\bf a}]=[[1^2_x,{\bf a}],1^2_y]+[1^2_x,[1^2_y,{\bf a}]].
\end{equation}\smallskip

A graded antisymmetric trilinear weight $1$ map $\ell_3: V^{\times 3}\to V$ verifies $L_{\infty}$-condition $n=3$ if
\begin{equation}\label{LieInfty3}
\ell_1(\ell_3(u,v,w))+\ell_2(\ell_2(u,v),w)-(-1)^{vw}\ell_2(\ell_2(u,w),v)+(-1)^{u(v+w)}\ell_2(\ell_2(v,w),u)+
$$ $$+\ell_3(\ell_1(u),v,w)
-(-1)^{uv}\ell_3(\ell_1(v),u,w)+(-1)^{w(u+v)}\ell_3(\ell_1(w),u,v)=0,
\end{equation} for any homogeneous $u,v,w\in V$.
This condition is trivial for any arguments of total degree $d=u+v+w>2$.
For $d=0$, we write $(u,v,w)=(x,y,z)\in V_0^{\times 3}$, for $d=1$,
we consider $(u,v)=(x,y)\in V_0^{\times 2}$ and $w={\bf f}\in V_1$, for $d=2$,
either $(u,v)=(x,y)\in V_0^{\times 2}$ and $w={\bf a}\in V_2$, or
$u=x\in V_0$ and $(v,w)=({\bf f},{\bf g})\in V_1^{\times 2}$, so that Equation (\ref{LieInfty3}) reads
\begin{equation}\label{31}\ell_1(\ell_3(x,y,z))+\ell_2(\ell_2(x,y),z)-\ell_2(\ell_2(x,z),y)+\ell_2(\ell_2(y,z),x)=0,\end{equation}
\begin{equation}\label{32}\ell_1(\ell_3(x,y,{\bf f}))+\ell_2(\ell_2(x,y),{\bf f})-\ell_2(\ell_2(x,{\bf f}),y)+\ell_2(\ell_2(y,{\bf f}),x)
+\ell_3(\ell_1({\bf f}),x,y)=0,\end{equation}
\begin{equation}\ell_2(\ell_2(x,y),{\bf a})-\ell_2(\ell_2(x,{\bf a}),y)+\ell_2(\ell_2(y,{\bf a}),x)
+\ell_3(\ell_1({\bf a}),x,y)=0,\label{33}\end{equation}
\begin{equation}\ell_2(\ell_2(x,{\bf f}),{\bf g})+\ell_2(\ell_2(x,{\bf g}),{\bf f})+\ell_2(\ell_2({\bf f},{\bf g}),x) -\ell_3(\ell_1({\bf f}),x,{\bf g})-\ell_3(\ell_1({\bf g}),x,{\bf f})=0.\label{34}\end{equation}\smallskip

It is easy to associate to any such map $\ell_3$ a unique Jacobiator ${\mathfrak{G}}(\ell_3)=J$: it suffices to set $J_{xyz}:=([[x,y],z],\ell_3(x,y,z))\in L_1$, for
any $x,y,z\in L_0$. Equation (\ref{31}) means that $J_{xyz}$ has the correct
target. Equations (\ref{3a}) and (\ref{3b}) exactly correspond to
Equations (\ref{32}) and (\ref{33}), respectively, {\it if we
assume that in total degree $d=1$, $\ell_3$ is valued
in 2-cocycles and vanishes when evaluated on a 1-coboundary}. These conditions are verified if we start from a structure map $\ell_3$ that vanishes on any arguments of total degree 1.

\begin{rem} Remark that the values $\ell_3(x,y,{\bf f})\in V_2$ cannot be encoded
in a natural 2-transformation $J:L_0^{\times 3}\ni (x,y,z)\to
J_{xyz}\in L_1$ (and that the same holds true for Equation (\ref{34}), whose first three terms are zero, since we started from a map $\ell_2$ that vanishes on $V_1\times
V_1$).\end{rem}

{\it Proof of Lemma \ref{Lem3} (continuation)}. Conversely, to any Jacobiator $J$ corresponds a unique map ${\mathfrak{N}}(J)=\ell_3$. Just set $\ell_3(x,y,z):={\bf J}_{xyz}\in V_1$ and $\ell_3(x,y,{\bf f})=0$, for all $x,y,z\in V_0$ and ${\bf f}\in V_1$ (as $\ell_3$ is required to have weight 1, it must vanish if evaluated on elements of degree $d\ge 2$).\smallskip

Obviously the composites $\mathfrak{NG}$ and $\mathfrak{GN}$ are identity maps.\end{proof}

\begin{rem} The naturality condition is, roughly speaking, the categorical analogue of the $L_{\infty}$-condition $n=3$.\end{rem}

\subsubsection{Identiator -- fourth structure map}

For $x,y,z,u\in L_0$, we set
\begin{equation}\label{Eta}\eta_{xyzu}:=[J_{x,y,z},1_u]\circ_0(J_{[x,z],y,u}+J_{x,[y,z],u})\circ_0([J_{xzu},1_y]+1)\circ_0([1_x,J_{yzu}]+1)\in L_1\end{equation} and
\begin{equation}\label{Epsilon}\varepsilon_{xyzu}:=
J_{[x,y],z,u}\circ_0([J_{xyu},1_z]+1)\circ_0(J_{x,[y,u],z}+J_{[x,u],y,z}+J_{x,y,[z,u]})\in L_1,\end{equation} see Definition \ref{Lie3Alg}. The identities $1$ are uniquely determined by the sources of the involved factors. The quadrilinear natural 2-transformations $\eta$ and $\varepsilon$ are actually the left and right hand composites of the Baez-Crans octagon that pictures the coherence law of a Lie 2-algebra, see \cite{BC04}, Definition 4.1.3. They connect the quadrilinear 2-functors $F,G:L\times L\times L\times L\to L$, whose values at $(x,y,z,u)$ are given by the source and the target of the 1-cells $\eta_{xyzu}$ and $\varepsilon_{xyzu}$, as well as by the top and bottom sums of triple brackets of the mentioned octagon.

\begin{lem} The skew-symmetric quadrilinear 2-modifications $\mu:\eta\Rrightarrow \varepsilon$ are in 1-to-1 correspondence with the graded antisymmetric quadrilinear weight 2 maps $\ell_4:V^{\times 4}\to V$ that verify the $L_{\infty}$-condition $n=4$.\end{lem}

\begin{proof}  A skew-symmetric quadrilinear 2-modification
$\mu:\eta\Rrightarrow \varepsilon$ maps every tuple $(x,y,z,u)\in
L_0^{\times 4}$ to a unique $\mu_{xyzu}:\eta_{xyzu}\Rightarrow
\varepsilon_{xyzu}$ in $L_2$, such that, for any
$\alpha=(u,{\bf f},{\bf a})\in L_2$, we have
\begin{equation} F(1^2_x,1^2_y,1^2_z,\alpha)\circ_0\mu_{x,y,z,u+t{\bf f}}=\mu_{xyzu}\circ_0G(1^2_x,1^2_y,1^2_z,\alpha)\label{ModCond}\end{equation} (as well as similar
results concerning naturality with respect to the three other
variables). If we decompose $\mu_{xyzu}\in L_2=V_0\oplus V_1\oplus V_2$,
$$\mu_{xyzu}=(F(x,y,z,u),{\bf h}_{xyzu},{\bf m}_{xyzu})=1_{\eta_{xyzu}}+{\bf m}_{xyzu},$$ Condition (\ref{ModCond}) reads
\begin{equation}\label{41}F(1_x,1_y,1_z,{\bf f})+{\bf h}_{x,y,z,u+t{\bf f}}={\bf h}_{xyzu}+G(1_x,1_y,1_z,{\bf f}),\end{equation}
\begin{equation}\label{42}F(1^2_x,1^2_y,1^2_z,{\bf a})+{\bf m}_{x,y,z,u+t{\bf f}}={\bf m}_{xyzu}+G(1^2_x,1^2_y,1^2_z,{\bf a}).\end{equation}\smallskip

On the other hand, a graded antisymmetric quadrilinear weight 2 map $\ell_4:V^{\times 4}\to V$, and more
precisely $\ell_4:V_0^{\times 4}\to V_2$, verifies $L_{\infty}$-condition $n=4$, if
$$\ell_1(\ell_4(a,b,c,d))$$
$$-\ell_2(\ell_3(a,b,c),d)+(-1)^{cd}\ell_2(\ell_3(a,b,d),c)-(-1)^{b(c+d)}\ell_2(\ell_3(a,c,d),b)$$
$$+(-1)^{a(b+c+d)}\ell_2(\ell_3(b,c,d),a)+
\ell_3(\ell_2(a,b),c,d)-(-1)^{bc}\ell_3(\ell_2(a,c),b,d)$$
$$+(-1)^{d(b+c)}\ell_3(\ell_2(a,d),b,c)+(-1)^{a(b+c)}\ell_3(\ell_2(b,c),a,d)$$
$$-(-1)^{ab+ad+cd}\ell_3(\ell_2(b,d),a,c)+(-1)^{(a+b)(c+d)}\ell_3(\ell_2(c,d),a,b)$$
$$-\ell_4(\ell_1(a),b,c,d)+(-1)^{ab}\ell_4(\ell_1(b),a,c,d)$$
\begin{equation}\label{LieInfty4}-(-1)^{c(a+b)}\ell_4(\ell_1(c),a,b,d)+(-1)^{d(a+b+c)}\ell_4(\ell_1(d),a,b,c)=0,\end{equation} for all homogeneous $a,b,c,d\in V$.
The condition is trivial for $d\ge 2$. For $d=0$, we write
$(a,b,c,d)=(x,y,z,u)\in V_0^{\times 4}$, and, for $d=1$, we take
$(a,b,c,d)=(x,y,z,{\bf f})\in V_0^{\times 3}\times V_1$, so that -- since $\ell_2$ and $\ell_3$ vanish on $V_1\times
V_1$ and for $d=1$, respectively -- Condition (\ref{LieInfty4}) reads
\begin{equation}\label{KMP4a}
\ell_1(\ell_4(x,y,z,u))-{\bf h}_{xyzu}+{\bf e}_{xyzu}=0,
\end{equation}
\begin{equation}\label{4b} \ell_4(\ell_1({\bf f}),x,y,z)=0,\end{equation} where
${\bf h}_{xyzu}$ and ${\bf e}_{xyzu}$ are the $V_1$-components of $\eta_{xyzu}$ and $\varepsilon_{xyzu}$, see Equations (\ref{Eta}) and (\ref{Epsilon}).\medskip

We can associate to any such map $\ell_4$ a unique 2-modification ${\mathfrak{G}}(\ell_4)=\mu$, $\mu:\eta\Rrightarrow \varepsilon$. It suffices to set, for $x,y,z,u\in
L_0$, $$\mu_{xyzu}=(F(x,y,z,u),{\bf h}_{xyzu},-\ell_4(x,y,z,u))\in
L_2.$$ In view of Equation (\ref{KMP4a}), the target of this 2-cell is
$$t\mu_{xyzu}=(F(x,y,z,u),{\bf h}_{xyzu}-\ell_1(\ell_4(x,y,z,u)))=\varepsilon_{xyzu}\in L_1.$$ Note now that the
2-naturality equations (\ref{3a}) and (\ref{3b}) show that 2-naturality of $\eta:F\Rightarrow G$ means that
$$F(1_x,1_y,1_z,{\bf f})+{\bf h}_{x,y,z,u+t{\bf f}}={\bf h}_{xyzu}+G(1_x,1_y,1_z,{\bf f}),$$
$$F(1^2_x,1^2_y,1^2_z,{\bf a})=G(1^2_x,1^2_y,1^2_z,{\bf a}).$$
When comparing with Equations (\ref{41}) and (\ref{42}), we conclude
that $\mu$ is a 2-modification if and only if $\ell_4(\ell_1({\bf f}),x,y,z)=0,$ which is exactly Equation (\ref{4b}).\medskip

Conversely, if we are given a skew-symmetric quadrilinear 2-modification
$\mu:\eta\Rrightarrow\varepsilon$, we
define a map ${\mathfrak{N}}(\mu)=\ell_4$ by setting
$\ell_4(x,y,z,u)=-{\bf m}_{xyzu}$, with self-explaining notations. $L_{\infty}$-condition $n=4$ is equivalent with Equations
(\ref{KMP4a}) and (\ref{4b}). The first means that $\mu_{xyzu}$ must
have the target $\varepsilon_{xyzu}$ and the second requires that ${\bf m}_{t{\bf f},x,y,z}$ vanish -- a consequence of the 2-naturality of
$\eta$ and of Equation (\ref{42}).\medskip

The maps $\mathfrak{N}$ and $\mathfrak{G}$ are again inverses.\end{proof}

\subsubsection{Coherence law -- $L_{\infty}$-condition $n=5$}

\begin{lem} Coherence law (\ref{CohLaw0}) is equivalent to $L_{\infty}$-condition $n=5$.\end{lem}

\begin{proof} The sh Lie condition $n=5$ reads,
$$\ell_2(\ell_4(x,y,z,u),v)-\ell_2(\ell_4(x,y,z,v),u)+\ell_2(\ell_4(x,y,u,v),z)-\ell_2(\ell_4(x,z,u,v),y)+\ell_2(\ell_4(y,z,u,v),x)$$
$$+\ell_4(\ell_2(x,y),z,u,v)-\ell_4(\ell_2(x,z),y,u,v)+\ell_4(\ell_2(x,u),y,z,v)-\ell_4(\ell_2(x,v)y,z,u)+\ell_4(\ell_2(y,z),x,u,v)$$
$$-\ell_4(\ell_2(y,u),x,z,v)+\ell_4(\ell_2(y,v),x,z,u)+\ell_4(\ell_2(z,u),x,y,v)-\ell_4(\ell_2(z,v),x,y,u)+\ell_4(\ell_2(u,v),x,y,z)$$
\begin{equation}=0,\label{51}\end{equation} for any $x,y,z,u,v\in V_0$. It is trivial in degree $d\ge 1$. Let us mention that it follows from Equation (\ref{31}) that $(V_0,\ell_2)$ is
a Lie algebra up to homotopy, and from Equation (\ref{33}) that $\ell_2$ is a
representation of $V_0$ on $V_2$. Condition (\ref{51}) then requires
that $\ell_4$ be a Lie algebra 4-cocycle of $V_0$ represented upon
$V_2$.\medskip

The coherence law for the 2-modification $\mu$ corresponds to four different ways
to rebracket the expression $F([x,y],z,u,v)=[[[[x,y],z],u],v]$ by
means of $\mu$, $J$, and $[-,-]$. More precisely, we define, for any tuple
$(x,y,z,u,v)\in L_0^{\times 5}$, four 2-cells
$$\alpha_i:\sigma_i\Rightarrow \tau_i,$$ $i\in\{1,2,3,4\}$, in $L_2$, where
$\sigma_i,\tau_i:A_i\to B_i$. Dependence on the considered tuple is
understood. We omit temporarily also index $i$. Of course, $\sigma$ and $\tau$ read $\sigma=(A,{\bf s})\in L_1$ and $\tau=(A,{\bf t})\in L_1$. $$\mbox{If }\,\alpha=(A,{\bf s},{\bf a})\in L_2,\,\mbox{ we set }\,\alpha^{-1}=(A,{\bf t},-{\bf a})\in L_2,$$ which is,
as easily seen, the inverse of $\alpha$ for composition along
1-cells.

\begin{defi} The {\bf coherence law} for the 2-modification $\mu$ of a Lie 3-algebra $(L,[-,-],J,\mu)$ reads
\begin{equation}\label{CohLaw}\alpha_1+\alpha_4^{-1}=\alpha_3+\alpha_2^{-1},\end{equation} where $\alpha_1$ -- $\alpha_4$ are detailed in the next definitions.\end{defi}

\begin{defi}\label{Alpha1} The {\bf first 2-cell} $\alpha_1$ is given by \begin{equation}\label{CohLaw11} \alpha_1=
1_{11}\circ_0\left(\mu_{x,y,z,[u,v]}+[\mu_{xyzv},1^2_u]\right)\circ_01_{12}\circ_0\left(\mu_{[x,v],y,z,u}+\mu_{x,[y,v],z,u}+\mu_{x,y,[z,v],u}+1^2\right),\end{equation}
where \begin{equation}\label{CohLaw12}
1_{11}=1_{J_{[[x,y],z],u,v}},1_{12}=1_{[J_{x,[z,v],y},1_{u}]+[J_{[x,v],z,y},1_{u}]+[J_{x,z,[y,v]},1_{u}]}+1^2,\end{equation}
and where the $1^2$ are the identity 2-cells associated with
the elements of $L_0$ provided by the composability
condition.\end{defi}

For instance, the squared target of the second factor of $\alpha_1$
is $G(x,y,z,[u,v])+[G(x,y,z,v),u]$, whereas the squared source of
the third factor is
$$[[[x,[z,v]],y],{u}]+[[[[x,v],z],y],{u}]+[[[x,z],[y,v]],{u}]+\ldots.$$
As the three first terms of this sum are three of the six terms of
$[G(x,y,z,v),u]$, the object ``$\ldots$'', at which $1^2$ in $1_{12}$ is evaluated,
is the sum of the remaining terms and $G(x,y,z,[u,v]).$

\begin{defi}\label{Alpha2} The {\bf fourth 2-cell} $\alpha_4$ is equal to
\begin{equation}\label{CohLaw21}\alpha_4=[\mu_{xyzu},1^2_v]\circ_01_{41}\circ_0\left(\mu_{[x,u],y,z,v}+\mu_{x,[y,u],z,v}+\mu_{x,y,[z,u],v}\right)\circ_01_{42},\end{equation}
where $$
1_{41}=1_{[J_{[x,u],z,y},1_v]+[J_{x,z,[y,u]},1_v]+[J_{x,[z,u],y},1_v]}+1^2,$$
\begin{equation}\label{CohLaw22}
1_{42}=1_{[[J_{xuv},1_z],1_y]+[J_{xuv},1_{[y,z]}]+[1_x,[J_{yuv},1_z]]+[[1_x,J_{zuv}],1_y]+[1_x,[1_y,J_{zuv}]]+[1_{[x,z]},J_{yuv}]}+1^2.\end{equation}\end{defi}

\begin{defi}\label{Alpha3} The {\bf third 2-cell} $\alpha_3$ reads \begin{equation}\label{CohLaw31}
\alpha_3=\mu_{[x,y],z,u,v}\circ_01_{31}\circ_0\left([\mu_{xyuv},1^2_z]+1^2\right)\circ_01_{32}\circ_01_{33},\end{equation}
where $$ 1_{31}=1_{[J_{[x,y],v,u},1_z]}+1^2,$$
$$\hspace{-15mm} 1_{32}=1_{[J_{xyv},1_{[z,u]}]+J_{x,y,[[z,v],u]}+J_{x,y,[z,[u,v]]}+J_{[[x,v],u],y,z}+J_{[x,v],[y,u],z}+J_{[x,u],[y,v],z}+
J_{x,[[y,v],u],z}+J_{[x,[u,v]],y,z}+J_{x,[y,[u,v]],z}+[J_{xyu},1_{[z,v]}]},$$
\begin{equation}\label{CohLaw32}
1_{33}=1_{J_{x,[y,v],[z,u]}+J_{[x,v],y,[z,u]}+J_{x,[y,u],[z,v]}+J_{[x,u],y,[z,v]}}+1^2.\end{equation}\end{defi}

\begin{defi}\label{Alpha4} The {\bf second 2-cell} $\alpha_2$ is defined as \begin{equation}\label{CohLaw41}
\alpha_2=1_{21}\circ_0\left(\mu_{[x,z],y,u,v}+\mu_{x,[y,z],u,v}\right)\circ_01_{22}\circ_0\left([1^2_x,\mu_{yzuv}]+[\mu_{xzuv},1^2_y]+1^2\right)\circ_01_{23},\end{equation}
where \begin{equation}\label{CohLaw42} 1_{21}=1_{[[J_{xyz},1_u],1_v]},
1_{22}=1_{[1_x,J_{[y,z],v,u}]+[J_{[x,z],v,u},1_y]}+1^2,$$$$
1_{23}=1_{[J_{xzv},1_{[y,u]}]+[J_{xzu},1_{[y,v]}]+[1_{[x,v]},J_{yzu}]+[1_{[x,u]},J_{yzv}]}+1^2.\end{equation}\end{defi}

To get the component expression
\begin{equation}\label{CohLawComp}(A_1+A_4,{\bf s}_1+{\bf t}_4,{\bf a}_1-{\bf a}_4)=(A_3+A_2,{\bf s}_3+{\bf t}_2,{\bf a}_3-{\bf a}_2)\end{equation} of the coherence law (\ref{CohLaw}), we now comment on the computation of the components
$(A_i,{\bf s}_i,{\bf a}_i)$ (resp. $(A_i,{\bf t}_i,-{\bf a}_i)$) of $\alpha_i$ (resp. $\alpha_i^{-1}$).\medskip

As concerns $\alpha_1$, it is straightforwardly seen that all compositions make sense, that its
$V_0$-component is $$A_1=F([x,y],z,u,v),$$ and that the
$V_2$-component is
$${\bf a}_1=$$
$$-\ell_4(x,y,z,\ell_2(u,v))-\ell_2(\ell_4(x,y,z,v),u)-\ell_4(\ell_2(x,v),y,z,u)-\ell_4(x,\ell_2(y,v),z,u)-\ell_4(x,y,\ell_2(z,v),u).$$
When actually examining the composability conditions, we find that
$1^2$ in the fourth factor of $\alpha_1$ is $1^2_{G(x,y,z,[u,v])}$
and thus that the target $t^2\alpha_1$ is made up by the 24 terms
$$G([x,v],y,z,u)+G(x,[y,v],z,u)+G(x,y,[z,v],u)+G(x,y,z,[u,v]).$$
The computation of the $V_1$-component ${\bf s}_1$ is tedious but
simple -- it leads to a sum of 29 terms of the type
``$\ell_3\ell_2\ell_2,$ $\ell_2\ell_3\ell_2$, or
$\ell_2\ell_2\ell_3$''. We will comment on it in the case of
$\alpha_4^{-1}$, which is slightly more interesting.\medskip

The $V_0$-component of $\alpha_4^{-1}$ is
$$A_4=[F_{xyzu},v]=F([x,y],z,u,v)$$ and its $V_2$-component is
equal to
$$-{\bf a}_4=\ell_2(\ell_4(x,y,z,u),v)+\ell_4(\ell_2(x,u),y,z,v)+\ell_4(x,\ell_2(y,u),z,v)+\ell_4(x,y,\ell_2(z,u),v).$$
The $V_1$-component ${\bf t}_4$ of $\alpha_4^{-1}$ is the
$V_1$-component of the target of $\alpha_4$. This target is the
composition of the targets of the four factors of $\alpha_4$ and its
$V_1$-component is given by
$${\bf t}_4=[{\bf e}_{xyzu},1_v]+[{\bf J}_{[x,u],z,y},1_v]+[{\bf J}_{x,z,[y,u]},1_v]+[{\bf J}_{x,[z,u],y},1_v]+{\bf e}_{[x,u],y,z,v}+{\bf e}_{x,[y,u],z,v}+{\bf e}_{x,y,[z,u],v}$$
$$+[[{\bf J}_{xuv},1_z],1_y]+[{\bf J}_{xuv},1_{[y,z]}]+[1_x,[{\bf J}_{yuv},1_z]]+[[1_x,{\bf J}_{zuv}],1_y]+[1_x,[1_y,{\bf J}_{zuv}]]+[1_{[x,z]},{\bf J}_{yuv}].$$ The definition (\ref{Epsilon}) of $\varepsilon$ immediately provides its
$V_1$-component $\bf e$ as a sum of 5 terms of the type ``$\ell_3\ell_2$ or
$\ell_2\ell_3$''. The preceding $V_1$-component ${\bf t}_4$ of
$\alpha_4^{-1}$ can thus be explicitly written as a sum of $29$ terms
of the type ``$\ell_3\ell_2\ell_2,$ $\ell_2\ell_3\ell_2$, or
$\ell_2\ell_2\ell_3$''. It can moreover be checked that the target
$t^2\alpha_4^{-1}$ is again a sum of 24 terms -- the same as for $t^2\alpha_1$.\medskip

The $V_0$-component of $\alpha_3$ is $$A_3=F([x,y],z,u,v),$$ the $V_1$-component
${\bf s}_3$ can be computed as before and is a sum of 25 terms of
the usual type ``$\ell_3\ell_2\ell_2,$ $\ell_2\ell_3\ell_2$, or
$\ell_2\ell_2\ell_3$'', whereas the $V_2$-component is equal to
$${\bf a}_3=-\ell_4(\ell_2(x,y),z,u,v)-\ell_2(\ell_4(x,y,u,v),z).$$
Again $t^2\alpha
_3$ is made up by the same 24 terms as
$t^2\alpha_1$ and $t^2\alpha_4^{-1}$.\medskip

Eventually, the $V_0$-component of $\alpha_2^{-1}$ is $$A_2=F([x,y],z,u,v),$$ the
$V_1$-component ${\bf t}_2$ is straightforwardly obtained as a sum
of 27 terms of the form ``$\ell_3\ell_2\ell_2,$
$\ell_2\ell_3\ell_2$, or $\ell_2\ell_2\ell_3$'', and the
$V_2$-component reads
$$-{\bf a}_2=\ell_4(\ell_2(x,z),y,u,v)+\ell_4(x,\ell_2(y,z),u,v)+\ell_2(x,\ell_4(y,z,u,v))+\ell_2(\ell_4(x,z,u,v),y).$$
The target $t^2\alpha_2^{-1}$ is the same as in the
preceding cases.\medskip

Coherence condition (\ref{CohLaw}) and its component expression (\ref{CohLawComp}) can now be understood. The condition on the $V_0$-components is obviously trivial. The condition on the $V_2$-components is nothing but
$L_{\infty}$-condition $n=5$, see Equation (\ref{51}). The verification of triviality of the condition on the $V_1$-components is lengthy: 6 pairs (resp. 3 pairs) of terms of the {\small LHS} ${\bf s}_1+{\bf t}_4$ (resp. {\small RHS} ${\bf s}_3+{\bf t}_2$) are opposite and cancel out, 25 terms of
the {\small LHS} coincide with terms of the {\small RHS}, and, finally, 7 triplets of {\small LHS}-terms combine with triplets of {\small RHS}-terms and provide 7 sums of 6 terms, e.g.
$$\ell_3(\ell_2(\ell_2(x,y),z),u,v)+\ell_2(\ell_3(x,y,z),\ell_2(u,v))+\ell_2(\ell_2(\ell_3(x,y,z),v),u)$$ $$-\ell_2(\ell_2(\ell_3(x,y,z),u),v)-\ell_3(\ell_2(\ell_2(x,z),y),u,v)-\ell_3({\ell_2(x,\ell_2(y,z)),u,v}).$$
Since, for ${\bf f}=\ell_3(x,y,z)\in V_1$, we have $$\ell_1({\bf f})=t{\bf J}_{xyz}=\ell_2(\ell_2(x,z),y)+\ell_2(x,\ell_2(y,z))-\ell_2(\ell_2(x,y),z),$$ the preceding sum vanishes in view of Equation
(\ref{32}). Indeed, if we associate a Lie 3-algebra to a 3-term Lie infinity algebra, we started from a homotopy algebra whose term $\ell_3$ vanishes in total degree 1, and if we build an sh algebra from a categorified algebra,
we already constructed an $\ell_3$-map with that property. Finally, the condition on $V_1$-components is really trivial and the coherence law (\ref{CohLaw}) is actually equivalent to $L_{\infty}$-condition $n=5$.\end{proof}

\subsection{Monoidal structure of the category {\tt Vect} $n$-{\tt
Cat}}\label{KMPSection5}

In this section we exhibit a specific aspect of the natural
monoidal structure of the category of linear $n$-categories.

\begin{prop}\label{LinnFun} If $L$ and $L'$ are linear $n$-categories, a family $F_m:L_m\to L_m'$ of linear maps that respects sources, targets, and identities, commutes
automatically with compositions and thus defines a linear
$n$-functor $F:L\to L'$.\end{prop}

\begin{proof} If $v=(v_0,\ldots, v_m), w=(w_0,\ldots,w_m)\in L_m$ are composable along a $p$-cell, then $F_mv=(F_0v_0, \ldots,$ $ F_mv_m)$ and $F_mw=(F_0w_0,
\ldots, F_mw_m)$ are composable as well, and
$F_m(v\circ_pw)=(F_mv)\circ_p(F_mw)$ in view of Equation
(\ref{composition}).\end{proof}

\begin{prop} The category {\tt Vect} $n$-{\tt Cat} admits a canonical symmetric monoidal structure $\boxtimes$.\end{prop}

\begin{proof} We first define the product $\boxtimes$ of two linear $n$-categories $L$ and $L'$. The
$n$-globular vector space that underlies the linear $n$-category
$L\boxtimes L'$ is defined in the obvious way, $(L\boxtimes
L')_m=L_m\otimes L_m'$, $S_m=s_m\otimes s'_m$, $T_m=t_m\otimes
t'_m.$ Identities are clear as well, $I_m=1_m\otimes 1'_m.$ These
data can be completed by the unique possible compositions
$\square_p$ that then provide a linear $n$-categorical structure.

If $F:L\to M$ and $F':L'\to M'$ are two linear $n$-functors, we
set $$(F\boxtimes F')_m=F_m\otimes F'_m\in\mathrm{Hom}_{\mathbb{K}}(L_m\otimes
L'_m,M_m\otimes M'_m),$$ where $\mathbb{K}$ denotes the ground field. Due
to Proposition {\ref{LinnFun}}, the family $(F\boxtimes F')_m$
defines a linear $n$-functor $F\boxtimes F':L\boxtimes L'\to
M\boxtimes M'$.

It is immediately checked that $\boxtimes$ respects composition
and is therefore a functor from the product category ({\tt Vect}
$n$-{\tt Cat})$^{\times 2}$ to {\tt Vect} $n$-{\tt Cat}. Further,
the linear $n$-category $K$, defined by $K_m=\mathbb{K}$,
$s_m=t_m=\mathrm{id}_{\mathbb{K}}$ ($m>0$), and $1_m=\mathrm{id}_{\mathbb{K}}$ ($m<n$),
acts as identity object for $\boxtimes$. Its is now clear that
$\boxtimes$ endows {\tt Vect} $n$-{\tt Cat} with a symmetric
monoidal structure.\end{proof}

\begin{prop}\label{PreUP} Let $L$, $L'$, and $L''$ be linear $n$-categories. For any bilinear
$n$-functor $F:L\times L'\to L''$, there exists a unique linear
$n$-functor $\tilde F:L\boxtimes L'\to L''$, such that
$\boxtimes\, \tilde F=F.$ Here $\boxtimes:L\times L'\to L\boxtimes
L'$ denotes the family of bilinear maps $\boxtimes_m:L_m\times
L_m'\ni (v,v')\mapsto v\otimes v'\in L_m\otimes L_m'$, and
juxtaposition denotes the obvious composition of the first with
the second factor.\end{prop}

\begin{proof} The result is a straightforward consequence of the universal property of the tensor product of vector spaces.\end{proof}

The next remark is essential.

\begin{rem} Proposition \ref{PreUP} is not a Universal Property for the tensor product $\boxtimes$ of {\tt Vect} $n$-{\tt Cat},
since $\boxtimes:L\times L'\to L\boxtimes L'$ is not a bilinear
$n$-functor. It follows that bilinear $n$-functors on a product
category $L\times L'$ cannot be identified with linear
$n$-functors on the corresponding tensor product category
$L\boxtimes L'$.\end{rem}

The point is that the family $\boxtimes_m$ of bilinear maps
respects sources, targets, and identities, but not compositions
(in contrast with a similar family of linear maps, see Proposition
\ref{LinnFun}). Indeed, if $(v,v'),(w,w')\in L_m\times L'_m$ are
two $p$-composable pairs (note that this condition is equivalent
with the requirement that $v,w\in L_m$ and $v',w'\in L_m'$ be
$p$-composable), we have
\begin{equation}\label{CartProd}\boxtimes_m((v,v')\circ_p(w,w'))=(v\circ_pw)\otimes(v'\circ_pw')\in
L_m\otimes L_m',\end{equation} and
\begin{equation}\label{TensProd}\boxtimes_m(v,v')\;\circ_p\;\boxtimes_m(w,w')=(v\otimes
v')\circ_p(w\otimes w')\in L_m\otimes L_m'.\end{equation} As the elements
(\ref{CartProd}) and (\ref{TensProd}) arise from the compositions
in $L_m\times L'_m$ and $L_m\otimes L'_m$, respectively, -- which
are forced by linearity and thus involve the completely different
linear structures of these spaces -- it can be expected that the
two elements do not coincide.\medskip

Indeed, when confining ourselves, to simplify, to the case $n=1$
of linear categories, we easily check that
\begin{equation}\label{NonFunUP}(v\circ w)\otimes (v'\circ w')=(v\otimes v')\circ(w\otimes
w')+(v-1_{tv})\otimes {\bf w'}+{\bf w}\otimes (v'-1_{{tv'}}).\end{equation}

Observe also that the source spaces of the linear maps
$$\circ_L\otimes\circ'_{L'}:(L_1\times_{L_0}L_1)\otimes
(L'_1\times_{L'_0}L'_1)\ni (v,w)\otimes(v',w')\mapsto (v\circ
w)\otimes (v'\circ w')\in L_1\otimes L_1'$$ and
$$\circ_{L\boxtimes L'}:(L_1\otimes L_1')\times_{L_0\otimes L_0'}(L_1\otimes L_1')\ni ((v\otimes v'),(w\otimes w'))\mapsto (v\otimes v')\circ(w\otimes w')\in
L_1\otimes L_1'$$ are connected by
\begin{equation} \ell_2: (L_1\times_{L_0}L_1)\otimes (L'_1\times_{L'_0}L'_1)\ni
(v,w)\otimes(v',w')\mapsto (v\otimes v',w\otimes w')\in
(L_1\otimes L_1')\times_{L_0\otimes L_0'}(L_1\otimes L_1')\label{NonFun}\end{equation} -- a
linear map with nontrivial kernel.

\subsection{Discussion}\label{KMPSection6}

We continue working in the case $n=1$ and
investigate a more conceptual approach to the construction of a
chain map $\ell_2:{\mathfrak{N}}(L)\otimes {\mathfrak{N}}(L)\to {\mathfrak{N}}(L)$ from a bilinear functor $[-,-]:L\times L\to L$.\medskip

When denoting by $[-,-]:L\boxtimes L\to L$ the induced linear functor, we get a chain map ${\mathfrak
N}([-,-]):{\mathfrak N}(L\boxtimes L)\to {\mathfrak N}(L)$, so that it is
natural to look for a second chain map $$\phi:{\mathfrak N}(L)\otimes
{\mathfrak N}(L)\to {\mathfrak N}(L\boxtimes L).$$

The informed reader may skip the following subsection.

\subsubsection{Nerve and normalization functors, Eilenberg-Zilber chain map}

The objects of the simplicial category $\Delta$ are the finite
ordinals $n=\{0,\ldots, n-1\}$, $n\ge 0$. Its morphisms $f:m\to n$ are the
order respecting functions between the sets $m$ and $n$. Let
$\delta_i:n\rightarrowtail n+1$ be the injection that omits image
$i$, $i\in\{0,\ldots,n\}$, and let $\sigma_i:n+1\twoheadrightarrow n$
be the surjection that assigns the same image to $i$ and $i+1$,
$i\in\{0,\ldots,n-1\}$. Any order respecting function $f:m\to n$
reads uniquely as
$f=\sigma_{j_1}\ldots\sigma_{j_h}\delta_{i_1}\ldots\delta_{i_k}$, where the
$j_r$ are decreasing and the $i_s$ increasing. The application of this epi-monic decomposition to binary composites $\delta_i\delta_j$, $\sigma_i\sigma_j$, and $\delta_i\sigma_j$ yields
three basic commutation relations.\medskip

A simplicial object in the category {\tt Vect} is a functor
$S\in[\Delta^{+\mathrm{op}},\mbox{\tt Vect}]$, where $\Delta^{+}$ denotes
the full subcategory of $\Delta$ made up by the nonzero finite
ordinals. We write this functor $n+1\mapsto S(n+1)=:{S_n}$, $n\ge 0$,
($S_n$ is the vector space of $n$-simplices), $\delta_i\mapsto
S(\delta_i)=:{d_i:S_n\to S_{n-1}}$, $i\in\{0,\ldots,n\}$ ($d_i$ is a
face operator), $\sigma_i\mapsto S(\sigma_i)=:{s_i:S_{n}\to S_{n+1}}$,
$i\in\{0,\ldots,n\}$ ($s_i$ is a degeneracy operator). The $d_i$
and $s_j$ verify the duals of the mentioned commutation rules.
The simplicial data ($S_n,d^n_i,s^n_i$) (we added
superscript $n$) of course completely determine the functor $S$.
Simplicial objects in {\tt Vect} form themselves a category,
namely the functor category $s(\mbox{\tt
Vect}):=[\Delta^{+\mathrm{op}},\mbox{\tt Vect}]$, for which the
morphisms, called simplicial morphisms, are the natural
transformations between such functors. In view of the epi-monic
factorization, a simplicial map $\alpha
:S\to T$ is exactly a
family of linear maps $\alpha_n:S_n\to T_n$ that commute with the
face and degeneracy operators. \medskip

The nerve functor $${\cal N}:\mbox{\tt VectCat}\to s(\mbox{\tt
Vect})$$ is defined on a linear category $L$ as the sequence
$L_0,L_1,L_2:=L_1\times_{L_0}L_1,
L_3:=L_1\times_{L_0}L_1\times_{L_0}L_1\ldots$ of vector spaces of
$0,1,2,3\ldots$ simplices, together with the face operators
``composition'' and the degeneracy operators ``insertion of
identity'', which verify the simplicial commutation rules.
Moreover, any linear functor $F:L\to L'$ defines linear maps
$F_n:L_n\ni (v_1,\ldots,v_n)\to (F(v_1),\ldots, F(v_n))\in L'_n$
that implement a simplicial map.\medskip

The normalized or Moore chain complex of a simplicial vector space
$S=(S_n, d_i^n, s_i^n)$ is given by
$N(S)_n=\cap_{i=1}^n\mathrm{ker}d_i^n\subset S_n$ and
$\partial_n=d_0^n.$ Normalization actually provides a functor
$$ N:s(\mbox{\tt Vect}) \leftrightarrow \mbox{\tt C}^+(\mbox{\tt
Vect}):\Gamma$$ valued in the category of nonnegatively graded chain
complexes of vector spaces. Indeed, if $\alpha:S\to T$ is a
simplicial map, then $\alpha_{n-1}d_{i}^n=d_{i}^n\alpha
_{n}$. Thus,
$N(\alpha):N(S)\to N(T$), defined on $c_n\in N(S)_n$ by
$N(\alpha)_n(c_n)=\alpha_n(c_n)$, is valued in $N(T)_n$ and is further a
chain map. Moreover, the Dold-Kan correspondence claims that the
normalization functor $N$ admits a right adjoint $\Gamma$ and that
these functors combine into an equivalence of categories.\medskip

It is straightforwardly seen that, for any linear category $L$, we have \begin{equation}
\label{CompClassFun} N({\cal N}(L))={\mathfrak N}(L).\end{equation}

The categories $s(\mbox{\tt Vect})$ and {\tt C}$^+$({\tt Vect}) have well-known monoidal structures (we denote the unit objects by $I_s$ and $I_{\tt C}$, respectively). The normalization functor $N:s(\mbox{{\tt Vect}})\rightarrow
\mbox{\tt C}^+(\mbox{{\tt Vect}})$ is lax monoidal, i.e. it respects the tensor products and unit objects up to coherent chain maps $\varepsilon:I_{{\tt C}}\to N(I_s)$ and
$$EZ_{S,T}:N(S)\otimes N(T)\to N(S\otimes T)$$ (functorial in $S,T\in s(\mbox{\tt Vect})$), where $EZ_{S,T}$ is the
Eilenberg-Zilber map. Functor $N$ is lax comonoidal or oplax monoidal as well, the chain morphism
being here the Alexander-Whitney map $AW_{S,T}.$ These chain maps are inverses of
each other up to chain homotopy, $EZ\;AW=1,\;AW\,EZ\sim 1.$\medskip

The Eilenberg-Zilber map is defined as follows. Let $a\otimes b\in N(S)_p\otimes
N(T)_q\subset S_p\otimes T_q$ be an element of degree $p+q$. The chain map $EZ_{S,T}$ sends $a\otimes b$ to an
element of $N(S\otimes T)_{p+q}\subset (S\otimes
T)_{p+q}=S_{p+q}\otimes T_{p+q}$. We have $$EZ_{S,T}(a\otimes
b)=\sum_{(p,q)-\mbox{shuffles }(\mu,\nu)}\mathrm{sign}(\mu,\nu)\;\;s_{\nu_q}(\ldots
(s_{\nu_1}a))\,\otimes\, s_{\mu_p}(\ldots (s_{\mu_1}b))\in
S_{p+q}\otimes T_{p+q},$$ where the shuffles are permutations of
$(0,\ldots, p+q-1)$ and where the $s_i$ are the degeneracy operators.

\subsubsection{Monoidal structure and obstruction}

We now come back to the construction of a chain map $\phi:{\mathfrak
N}(L)\otimes {\mathfrak N}(L)\to {\mathfrak N}(L\boxtimes L)$.\medskip

For $L'=L$, the linear map (\ref{NonFun}) reads
$$\ell_2: ({\cal N}(L)\otimes {\cal N}(L))_2\ni
(v,w)\otimes(v',w')\mapsto (v\otimes v',w\otimes w')\in
{\cal N}(L\boxtimes L)_2.$$ If its obvious extensions $\ell_n$ to all other spaces $({\cal N}(L)\otimes {\cal N}(L))_n$ define a simplicial map $\ell:{\cal N}(L)\otimes {\cal N}(L)\to {\cal N}(L\boxtimes L),$ then $$N(\ell):N({\cal N}(L)\otimes {\cal N}(L))\to N({\cal N}(L\boxtimes L))$$ is a chain map. Its composition with the Eilenberg-Zilber chain map $$EZ_{{\cal N}(L),{\cal N}(L)}:N({\cal N}(L))\otimes N({\cal N}(L))\to N({\cal N}(L)\otimes {\cal N}(L))$$ finally provides the searched chain map $\phi$, see Equation (\ref{CompClassFun}).

However, the $\ell_n$ do not commute with all degeneracy and face operators. Indeed, we have for instance $$\ell_2((d_2^3\otimes d_2^3)((u,v,w)\otimes(u',v',w')))=(u\otimes u',(v\circ w)\otimes (v'\circ w')),$$ whereas $$d_2^3(\ell_3((u,v,w)\otimes(u',v',w')))=(u\otimes u',(v\otimes v')\circ(w\otimes w')).$$ Equation (\ref{NonFunUP}), which means that $\boxtimes:L\times L'\to L\boxtimes L'$ is not a functor, shows that these results do not coincide.\medskip

A natural idea would be to change the involved monoidal structures
$\boxtimes$ of {\tt VectCat} or $\otimes$ of {\tt C}$^{+}$({\tt
Vect}). However, even if we substitute the Loday-Pirashvili tensor
product $\otimes_{\mathrm{LP}}$ of 2-term chain complexes of vector
spaces, i.e. of linear maps \cite{LP98}, for the usual tensor
product $\otimes$, we do not get ${\mathfrak N}(L)\otimes_{\mathrm{LP}}
{\mathfrak N}(L)={\mathfrak N}(L\boxtimes L).$\bigskip\bigskip

\newpage
\section{The Supergeometry of Loday Algebroids}\label{SupergeomLodAlebroids}
The next research paper was published in `Journal of Geometric Mechanics', 5(2) (2013), 185-213 (joint work with Janusz Grabowski and Norbert Poncin).
\subsection{Introduction}
The concept of {\it Dirac structure}, proposed by Dorfman \cite{Do} in the Hamiltonian framework of integrable
evolution equations and defined in \cite{Co} as an isotropic subbundle of the Whitney sum ${\cal T} M= T
M\oplus_M T^\ast M$ of the tangent and the cotangent bundles and satisfying some additional conditions,
provides a geometric setting for Dirac's theory of constrained mechanical systems. To formulate the
integrability condition defining the Dirac structure, Courant \cite{Co} introduced a natural
skew-symmetric bracket operation on sections of \ ${\cal T} M$. The Courant bracket does not satisfy the Leibniz
rule with respect to multiplication by functions nor the Jacobi identity. These defects disappear upon
restriction to a Dirac subbundle because of the isotropy condition. Particular cases of Dirac structures are
graphs of closed 2-forms and  Poisson bivector fields on the manifold $M$.

The nature of the Courant bracket itself remained unclear until several years later when it was observed by
Liu, Weinstein and Xu \cite{LWX} that ${\cal T} M$ endowed with the Courant bracket plays the role of a `double'
object, in the sense of Drinfeld \cite{Dr}, for a pair of Lie algebroids (see \cite{Mac}) over $M$. Let us
recall that, in complete analogy with Drinfeld's Lie bialgebras, in the category of Lie algebroids there also
exist `bi-objects', Lie bialgebroids, introduced by Mackenzie and Xu \cite{MX} as linearizations of Poisson
groupoids. On the other hand, every Lie bialgebra has a double which is a Lie algebra. This is not so for
general Lie bialgebroids. Instead, Liu, Weinstein and Xu \cite{LWX} showed that the double of a Lie
bialgebroid is a more complicated structure they call a {\it Courant algebroid}, ${\cal T} M$ with the Courant
bracket being a special case.

There is also another way of viewing Courant algebroids as a generalization of Lie algebroids. This requires a
change in the definition of the Courant bracket and considering an analog of  the non-antisymmetric Dorfman
bracket \cite{Do}, so that the traditional Courant bracket becomes the skew-symmetrization of the new one
\cite{Roy}. This change replaces one of the defects with another one: a version of the Jacobi identity is
satisfied, while the bracket is no longer skew-symmetric. Such algebraic structures have been introduced by
Loday \cite{Lo} under the name {\it Leibniz algebras}, but they are nowadays also often called {\it Loday
algebras}. Loday algebras, like their skew-symmetric counterparts -- Lie algebras -- determine certain
cohomological complexes, defined on tensor algebras instead of Grassmann algebras. {Canonical examples of Loday algebras arise often as {\em derived brackets} introduced by Kosmann-Schwarzbach \cite{K-S,YKS}.}

Since Loday brackets, like the Courant-Dorfman bracket, appear naturally in Geometry and Physics in the form
of `algebroid brackets', i.e. brackets on sections of vector bundles, there were several attempts to formalize
the concept of {\it Loday} (or {\it Leibniz}) {\em algebroid} (see e.g. \cite{Ba,BV,G1,GM,ILMP,Ha,HM,KS,MM,SX,Wa}).
We prefer the terminology {\em Loday algebroid} to distinguish them from other {\em general algebroid}
brackets with both anchors (see \cite{GU}), called sometimes {\em Leibniz algebroids} or {\em Leibniz
brackets} and used recently in Physics, for instance, in the context of nonholonomic constraints
\cite{GG,GG1,GGU,GLMM,OPB}. Note also that a Loday algebroid is the horizontal categorification of a Loday algebra; vertical categorification would lead to Loday $n$-algebras, which are tightly related to truncated Loday infinity algebras, see \cite{AP10}, \cite{KMP11}.

The concepts of Loday algebroid we found in the literature do not seem to be exactly appropriate. The notion
in \cite{G1}, which assumes the existence of both anchor maps, is too strong and admits no real new examples,
except for Lie algebroids and bundles of Loday algebras. The concept introduced in \cite{SX} requires a
pseudo-Riemannian metric on the bundle, so it is too strong as well and does not reduce to a Loday algebra
when we consider a bundle over a single point, while the other concepts \cite{Ha,HM,ILMP,KS,MM,Wa}, assuming
only the existence of a left anchor, do not put any differentiability requirements for the first variable, so
that they are not geometric and too weak (see Example \ref{e1}). Only in \cite{Ba} one considers some Leibniz algebroids with local brackets.

The aim of this work is to propose a modified concept of Loday algebroid in terms of an operation on
sections of a vector bundle, as well as in terms of a homological vector field of a supercommutative manifold.
We put some minimal requirements that a proper concept of Loday algebroid should satisfy. Namely, the
definition of Loday algebroid, understood as a certain operation on sections of a vector bundle $E$,
\begin{itemize}
\item should reduce to the definition of Loday algebra in the case when $E$ is just a vector space;

\item should contain the Courant-Dorfman bracket as a particular example;

\item should be as close to the definition of Lie algebroid as possible.
\end{itemize}

We propose a definition satisfying all these requirements and
including all main known examples of Loday brackets with geometric
origins. Moreover, we can interpret our Loday algebroid structures
as homological vector fields on a supercommutative manifold; this
opens, like in the case of Lie algebroids, new horizons for a
geometric understanding of these objects and of their possible
`higher generalizations' \cite{BP12}.
{This supercommutative manifold is associated with a superalgebra of differential operators{, whose multiplication is a supercommutative shuffle product}.

Note that we cannot work with the supermanifold {$\Pi E$} like in the case of a Lie algebroid on $E$, since the Loday coboundary operator rises the degree of {a differential operator,} even for Lie algebroid brackets. For instance, the Loday differential associated with the standard bracket of vector fields produces the Levi-Civita connection out of a Riemannian metric \cite{Ldd1}. However, the Levi-Civita connection $\nabla_XZ$ is no longer a tensor, as it is of the first-order with respect to $Z$. Therefore, instead of the Grassmann algebra $\mathrm{Sec}(\wedge E^*)$ of `differential forms', which are zero-degree skew-symmetric multidifferential operators on $E$, we are forced to consider, not just the  tensor algebra of sections of $\oplus_{k=0}^\infty(E^*)^{\otimes k}$, but the algebra $\mathcal{D}^\bullet(E)$ spanned by all multidifferential operators
$$D:\mathrm{Sec}(E)\times\cdots\times\mathrm{Sec}(E)\to C^\infty(M)\,.$$
However, to retain the supergeometric flavor, we can reduce ourselves to a smaller subspace $\mathrm{D}^\bullet(E)$ of $\mathcal{D}^\bullet(E)$, which is a subalgebra with respect to the canonical supercommutative shuffle product and is closed under the Loday coboundary operators associated with the Loday algebroids we introduce. This interesting observation deserves further investigations that we postpone to a next paper.

We should also make clear that, although the algebraic structures in question have their roots in Physics (see the papers on Geometric Mechanics mentioned above), we do not propose in this paper new applications to Physics, but focus on finding a proper framework unifying all these structures. Our work seems to be technically complicated enough and applications to Mechanics will be the subject of a separate work.
}

\medskip
The paper is organized as follows. We first recall, in Section~\ref{GKPSection2}, needed results on differential
operators and derivative endomorphisms. In Section~\ref{GKPSection3} we investigate, under the name of pseudoalgebras,
algebraic counterparts of algebroids requiring varying differentiability properties for the two entries of the
bracket. The results of Section~\ref{GKPSection4} show that we should relax our traditional understanding of the right anchor
map. A concept of Loday algebroid satisfying all the above requirements is proposed in Definition \ref{d1}
and further detailed in Theorem \ref{LodAld}. In Section~\ref{GKPSection5} we describe a number of new Loday algebroids
containing main canonical examples of Loday brackets on sections of a vector bundle. A natural reduction a Loday pseudoalgebra to a Lie pseudoalgebra is studied in Section~\ref{GKPSection6}. For the standard Courant bracket it corresponds to its reduction to the Lie bracket of vector fields.  We then define Loday
algebroid cohomology, Section~\ref{GKPSection7}, and interpret in Section~\ref{GKPSection8} our Loday algebroid structures in terms of
homological vector fields of the graded ringed space given by the shuffle multiplication of multidifferential
operators, see Theorem \ref{GeoIntKLA2}. We introduce also the corresponding Cartan calculus.

\subsection{Differential operators and derivative endomorphisms}\label{GKPSection2}
All geometric objects, like manifolds, bundles, maps, sections, etc. will be smooth throughout
this paper.

\begin{defi}
A {\it Lie algebroid} structure on a vector bundle $\tau:E\to M$ is a Lie algebra bracket $[\cdot,\cdot]$ on the real vector space $\mathcal{E}=\mathrm{Sec}(E)$ of sections of $E$ which satisfies the following compatibility condition related to the $\mathcal{A}=C^\infty(M)$-module structure in $\mathcal{E}$:
\begin{equation}\label{anchor}
\forall \ X,Y\in\mathcal{E}\ \forall f\in\mathcal{A}\quad [X,fY]-f[X,Y]=\rho(X)(f)Y\,,
\end{equation}
for some vector bundle morphism $\rho:E\to T M$ covering the identity on $M$ and called the {\it anchor map}.
Here, $\rho(X)=\rho\circ X$ is the vector field on $M$ associated {\it via} $\rho$ with the section $X$.
\end{defi}
Note that the bundle morphism $\rho$ is uniquely determined by the bracket of the Lie algebroid.
What differs a general Lie algebroid bracket from just a Lie module bracket on the $C^\infty(M)$-module
$\mathrm{Sec}(E)$  of sections of $E$ is the fact that it is not $\mathcal{A}$-bilinear but a certain first-order
bidifferential operator: the adjoint operator $\mathrm{ad}_X=[X,\cdot]$ is a {\it derivative endomorphism}, i.e., the
{\it Leibniz rule}
\begin{equation}\label{lr}\mathrm{ad}_X(fY)=f\mathrm{ad}_X(Y)+\widehat{X}(f)Y
\end{equation}
is satisfied for each $Y\in\mathcal{E}$ and $f\in\mathcal{A}$, where $\widehat{X}=\rho(X)$ is the vector field on $M$ assigned to
$X$, the {\it anchor} of $X$. Moreover, the assignment $X\mapsto\widehat{X}$ is a differential operator of order 0,
as it comes from a bundle map $\rho:E\mapsto  T M$.

Derivative endomorphisms (also called {\it quasi-derivations}), like differential operators in general, can be
defined for any module $\mathcal{E}$ over an associative commutative ring $\mathcal{A}$. Also an extension to superalgebras is
straightforward. These natural ideas go back to Grothendieck and Vinogradov \cite{Vi}.
On the module $\mathcal{E}$ we have namely a distinguished family $\mathcal{A}_\mathcal{E}=\{ f_\mathcal{E}:f\in\mathcal{A}\}$ of linear operators
provided by the module structure: $f_\mathcal{E}(Y)=fY$.
\begin{defi} Let $\mathcal{E}_i$, $i=1,2$, be modules over the same ring $\mathcal{A}$. We say that an additive operator $D:\mathcal{E}_1\to\mathcal{E}_2$ is a {\it differential operator of order 0}, if it intertwines $f_{\mathcal{E}_1}$ with $f_{\mathcal{E}_2}$, i.e.
\begin{equation}\label{cm}\delta(f)(D):=D\circ f_{\mathcal{E}_1}-f_{\mathcal{E}_2}\circ D\,, \end{equation}
vanishes for all $f\in\mathcal{A}$. Inductively, we say that $D$ is a {\it differential operator of order} $\le k+1$, if the
commutators (\ref{cm}) are differential operators of order $\le k$. In other words, $D$ is a differential
operator of order $\le k$ if and only if \begin{equation}\label{do} \forall \ f_1,\dots,f_{k+1}\in\mathcal{A}\quad
\delta({f_1})\delta({f_2})\cdots\delta({f_{k+1}})(D)=0\,. \end{equation} The corresponding set of differential operators of order
$\le k$ will be denoted by $\mathcal{D}_k(\mathcal{E}_1;\mathcal{E}_2)$ (shortly, $\mathcal{D}_k(\mathcal{E})$, if $\mathcal{E}_1=\mathcal{E}_2=\mathcal{E}$) and the set of
differential operators of arbitrary order (filtered by $\left(\mathcal{D}_k(\mathcal{E}_1;\mathcal{E}_2)\right)_{k=0}^\infty$) by
$\mathcal{D}(\mathcal{E}_1;\mathcal{E}_2)$ (resp., $\mathcal{D}(\mathcal{E})$). We will say that $D$ {\it is of order $k$} if it is of order $\le k$
and not of order $\le k-1$.
\end{defi}
In particular, $\mathcal{D}_0(\mathcal{E}_1;\mathcal{E}_2)=\mathrm{Hom}_\mathcal{A}(\mathcal{E}_1;\mathcal{E}_2)$ is made up by module homomorphisms. Note that in the
case when $\mathcal{E}_i=\mathrm{Sec}(E_i)$ is the module of sections of a vector bundle $E_i$, $i=1,2$, the concept of
differential operators defined above coincides with the standard understanding. As this will be our standard
geometric model, to reduce algebraic complexity we will assume that $\mathcal{A}$ is an associative commutative
algebra with unity $1$ over a field $\mathbb{K}$ of characteristic 0 and all the $\mathcal{A}$-modules are faithful. In this
case, $\mathcal{D}(\mathcal{E}_1;\mathcal{E}_2)$ is a (canonically filtered) vector space over $\mathbb{K}$ and, since we work with fields of
characteristic 0, condition (\ref{do}) is equivalent to a simpler condition (see \cite{G}) \begin{equation}\label{do1}
\forall \ f\in\mathcal{A}\quad \delta({f})^{k+1}(D)=0\,. \end{equation} If $\mathcal{E}_1=\mathcal{E}_2=\mathcal{E}$, then $\delta(f)(D)=[D,f_\mathcal{E}]_c$, where
$[\cdot,\cdot]_c$ is the commutator bracket, and elements of $\mathcal{A}_\mathcal{E}$ are particular 0-order operators.
Therefore, we can canonically identify $\mathcal{A}$ with the subspace $\mathcal{A}_\mathcal{E}$ in $\mathcal{D}_0(\mathcal{E})$ and use it to
distinguish a particular set of first-order differential operators on $\mathcal{E}$ as follows.
\begin{defi} {\it Derivative endomorphisms} (or {\it quasi-derivations}) $D:\mathcal{E}\to\mathcal{E}$ are particular first-order differential operators distinguished by the condition
\begin{equation}\label{anchor1}\forall \ f\in\mathcal{A}\quad \exists\ \widehat{f}\in\mathcal{A}\quad [D,f_\mathcal{E}]_c=\widehat{f}_\mathcal{E}\,.\end{equation}
\end{defi}
Since the commutator bracket satisfies the Jacobi identity, one can immediately conclude that
$\widehat{f}_\mathcal{E}=\widehat{D}(f)_\mathcal{E}$ which holds for some derivation $\widehat{D}\in\mathrm{Der}(\mathcal{A})$ and an arbitrary $f\in\mathcal{A}$
\cite{G1}. Derivative endomorphisms form a submodule $\mathrm{Der}(\mathcal{E})$ in the $\mathcal{A}$-module $\mathrm{End}_{\mathbb{K}}(\mathcal{E})$
of $\mathbb{K}$-linear endomorphisms of $\mathcal{E}$ which is simultaneously a Lie subalgebra over $\mathbb{K}$ with respect
to the commutator bracket. The linear map,
$$\mathrm{Der}(\mathcal{E})\ni D\mapsto\widehat{D}\in\mathrm{Der}(\mathcal{A})\,,$$
called the {\it universal anchor map}, is a differential operator of order 0, $\widehat{fD}=f\widehat{D}$. The Jacobi
identity for the commutator bracket easily implies (see \cite[Theorem 2]{G1})
\begin{equation}\label{anchor2} {\widehat{[D_1,D_2]_c}}=[\widehat{D}_1,\widehat{D}_1]_c\,.
\end{equation}
It is worth remarking (see \cite{G1}) that also $\mathcal{D}(\mathcal{E})$ is a Lie subalgebra in $\mathrm{End}_\mathbb{K}(\mathcal{E})$, as
\begin{equation}\label{qP}[\mathcal{D}_k(\mathcal{E}),\mathcal{D}_l(\mathcal{E})]_c\subset\mathcal{D}_{k+l-1}(\mathcal{E})\,,
\end{equation}
and an associative subalgebra, as
\begin{equation}\label{qP1}\mathcal{D}_k(\mathcal{E})\circ\mathcal{D}_l(\mathcal{E})\subset\mathcal{D}_{k+l}(\mathcal{E})\,,
\end{equation}
that makes $\mathcal{D}(\mathcal{E})$ into a canonical example of a {\it quantum Poisson algebra} in the terminology of
\cite{GP}.

It was pointed out in \cite{KM} that the concept of derivative endomorphism can be traced back to N.~Jacobson
\cite{Ja1,Ja2} as a special case of his {\it pseudo-linear endomorphism}. It has appeared also in \cite{Ne}
under the name {\it module derivation} and was used to define linear connections in the algebraic setting. In
the geometric setting of Lie algebroids it has been studied in \cite{Mac} under the name {\it covariant
differential operator}. For more detailed history and recent development we refer to \cite{KM}.

Algebraic operations in differential geometry have usually a local character in order to be treatable with
geometric methods. On the pure algebraic level we should work with differential (or multidifferential)
operations, as tells us the celebrated Peetre Theorem \cite{Pe,Pe1}. The algebraic concept of a
multidifferential operator is obvious. For a $\mathbb{K}$-multilinear operator $D:\mathcal{E}_1\times\cdots\times\mathcal{E}_p\to\mathcal{E}$ and
each $i=1,\dots,p$, we say that  $D$ is a {\it differential operator of order $\le k$ with respect to the
$i$th variable}, if, for all $y_j\in\mathcal{E}_j$, $j\ne i$,
$$D(y_1,\dots,y_{i-1},\,\cdot\,,y_{i+1},\dots, y_p):\mathcal{E}_i\to\mathcal{E}$$
is a differential operator of order $\le k$. In other words,
\begin{equation}\label{mdo}
\forall \ f\in\mathcal{A}\quad \delta_i({f})^{k+1}(D)=0\,,
\end{equation}
where
\begin{equation}
\delta_i(f)D(y_1,\dots,y_p)=D(y_1,\dots,fy_i,\dots,y_p)-fD(y_1,\dots,y_p)\,.
\end{equation}
Note that the operations $\delta_i(f)$ and $\delta_j(g)$ commute. We say that the operator $D$ {\it is a
multidifferential operator of order $\le n$}, if it is of order $\le n$ with respect to each variable
separately. This means that, fixing any $p-1$ arguments, we get a differential operator of order $\le n$. A
similar, but stronger, definition is the following
\begin{defi} We say that a multilinear operator $D:\mathcal{E}_1\times\cdots\times\mathcal{E}_p\to\mathcal{E}$ is a {\it multidifferential operator of total order $\le k$}, if
\begin{equation}\label{mdo1}
\forall \ f_1,\dots,f_{k+1}\in\mathcal{A}\ \forall\ i_1,\dots,i_{k+1}=1,\dots,p\quad
\left[\delta_{i_1}({f_1})\delta_{i_2}({f_2})\cdots\delta_{i_{k+1}}({f_{k+1}})(D)=0\right]\,.
\end{equation}
\end{defi}
\noindent Of course, a multidifferential operator of total order $\le k$ is a multidifferential operator of
order $\le k$. It is also easy to see that a $p$-linear differential operator of order $\le k$ is a
multidifferential operator of total order $\le pk$. In particular, the Lie bracket of vector fields (in fact,
any Lie algebroid bracket) is a bilinear differential operator of total order $\le 1$.

\subsection{Pseudoalgebras}\label{GKPSection3}
Let us start this section with recalling that Loday, while studying relations between Hochschild and cyclic
homology in the search for obstructions to the periodicity of algebraic K-theory, discovered that one can skip
the skew-symmetry assumption in the definition of Lie algebra, still having a possibility to define an
appropriate (co)homology (see \cite{Lo1,LP} and \cite[Chapter 10.6]{Lo}). His Jacobi identity for such
structures was formally the same as the classical Jacobi identity in the form
\begin{equation}\label{JI} [x,[y,z]]=[[x,y],z]+[y,[x,z]]. \end{equation} This time,
however, this is no longer equivalent to \begin{equation}\label{JI1} [[x,y],z]=[[x,z],y]+[x,[y,z]], \end{equation} nor to
\begin{equation}\label{JI2} [x,[y,z]]+[y,[z,x]]+[z,[x,y]]=0, \end{equation} since we have no skew-symmetry. Loday called such
structures {\it Leibniz algebras}, but to avoid collision with another concept of {\it Leibniz brackets} in
the literature, we shall call them {\it Loday algebras}. This is in accordance with the terminology of {\cite{K-S}},
where analogous structures in the graded case are defined. Note that the identities (\ref{JI}) and (\ref{JI1})
have an advantage over the identity (\ref{JI2}) obtained by cyclic permutations, since they describe the
algebraic facts that the left-regular (resp., right-regular) actions are left (resp., right) derivations. This
was the reason to name the structure `Leibniz algebra'.

Of course, there is no particular reason not to define Loday algebras by means of (\ref{JI1}) instead of
(\ref{JI}) (and in fact, it was the original definition by Loday), but this is not a substantial difference,
as both categories are equivalent via transposition of arguments. We will use the form (\ref{JI}) of the
Jacobi identity.

Our aim is to find a proper generalization of the concept of Loday algebra in a way similar to that in which
Lie algebroids generalize Lie algebras. If one thinks about a generalization of a concept of Lie algebroid as
operations on sections of a vector bundle including operations (brackets) which are non-antisymmetric or which
do not satisfy the Jacobi identity, and are not just $\mathcal{A}$-bilinear, then it is reasonable, on one hand, to
assume differentiability properties of the bracket as close to the corresponding properties of Lie algebroids
as possible and, on the other hand, including all known natural examples of such brackets. This is not an
easy task, since, as we will see soon, some natural possibilities provide only few new examples.

To present a list of these possibilities, we propose the following definitions serving in the pure algebraic
setting.
\begin{defi}\label{def} Let $\mathcal{E}$ be a faithful module over an associative commutative algebra $\mathcal{A}$ over a field $\mathbb{K}$ of characteristic 0. A a $\mathbb{K}$-bilinear bracket $B=[\cdot,\cdot] :\mathcal{E}\times\mathcal{E}\to\mathcal{E}$ on the module $\mathcal{E}$
\begin{enumerate}
\item is called a {\it {faint} pseudoalgebra bracket}, if  $B$ is a bidifferential operator; \item is called
a {\it weak pseudoalgebra bracket}, if $B$ is a bidifferential operator of degree $\le 1$; \item is called a
{\it quasi pseudoalgebra bracket}, if $B$ is a bidifferential operator of total degree $\le 1$; \item is
called a {\it pseudoalgebra bracket}, if $B$ is a bidifferential operator of total degree $\le 1$ and the {\it
adjoint map} $\mathrm{ad}_X=[X,\cdot]:\mathcal{E}\to\mathcal{E}$ is a derivative endomorphism for each $X\in\mathcal{E}$; \item is called a
{\it QD-pseudoalgebra bracket}, if the {\it adjoint maps} $\mathrm{ad}_X,\mathrm{ad}_X^r:\mathcal{E}\to\mathcal{E}$, \begin{equation}\label{ad}
\mathrm{ad}_X=[X,\cdot]\,,\quad \mathrm{ad}_X^r=[\cdot,X]\,\quad(X\in\mathcal{E})\,,\end{equation} associated with $B$ are derivative
endomorphisms (quasi-derivations); \item is called a {\it strong pseudoalgebra bracket}, if $B$ is a
bidifferential operator of total degree $\le 1$ and the {\it adjoint maps} $\mathrm{ad}_X,\mathrm{ad}_X^r:\mathcal{E}\to\mathcal{E}$,
\begin{equation}\label{ad1} \mathrm{ad}_X=[X,\cdot]\,,\quad \mathrm{ad}_X^r=[\cdot,X]\,\quad(X\in\mathcal{E})\,, \end{equation} are derivative endomorphisms.
\end{enumerate}
We call the module $\mathcal{E}$ equipped with such a bracket, respectively, a {\it {faint} pseudoalgebra}, {\it weak
pseudoalgebra} etc. If the bracket is symmetric (skew-symmetric), we speak about {faint}, weak, etc., {\it
symmetric} ({\it skew}) {\it pseudoalgebras}. If the bracket satisfies the Jacobi identity (\ref{JI}), we
speak about local, weak, etc.,  {\it Loday pseudoalgebras}, and if the bracket is a Lie algebra bracket, we
speak about local, weak, etc., {\it Lie pseudoalgebras}. If $\mathcal{E}$ is the $\mathcal{A}=C^\infty(M)$ module of sections
of a vector bundle $\tau:E\to M$, we refer to the above pseudoalgebra structures as to {\it algebroids}.
\end{defi}
\begin{thm} If \ $[\cdot,\cdot]$ is a pseudoalgebra bracket, then the map
$$\rho:\mathcal{E}\to \mathrm{Der}(\mathcal{A})\,,\quad \rho(X)=\widehat{\mathrm{ad}_X}\,,$$
called the {\em anchor map}, is $\mathcal{A}$-linear, $\rho(fX)=f\rho(X)$, and
\begin{equation}\label{aanchor}[X,fY]=f[X,Y]+\rho(X)(f)Y
\end{equation}
for all $X,Y\in\mathcal{E}$, $f\in\mathcal{A}$. Moreover, if \ $[\cdot,\cdot]$ satisfies additionally the Jacobi identity,
i.e., we deal with a Loday pseudoalgebra, then the anchor map is a homomorphism into the commutator bracket,
\begin{equation}\label{anhom}
\rho\left([X,Y]\right)=[\rho(X),\rho(Y)]_c\,.
\end{equation}
\end{thm}
\begin{proof}
Since the bracket $B$ is a bidifferential operator of total degree $\le 1$, we have $\delta_1(f)\delta_2(g)B=0$ for
all $f,g\in\mathcal{A}$. On the other hand, as easily seen,
\begin{equation}\label{QD}(\delta_1(f)\delta_2(g)B)(X,Y)=\left(\rho(fX)-f\rho(X)\right)(g)Y\,,\end{equation}
and the module is faithful, it follows $\rho(fX)=f\rho(X)$. The identity (\ref{anhom}) is a direct implication
of the Jacobi identity combined with (\ref{aanchor}).

\end{proof}
\begin{thm} If \ $[\cdot,\cdot]$ is a QD-pseudoalgebra bracket, then it is a weak pseudoalgebra bracket and admits two {\it anchor maps}
$$\rho,\rho^r:\mathcal{E}\to \mathrm{Der}(\mathcal{A})\,,\quad \rho(X)=\widehat{\mathrm{ad}_X}\,,\ \rho^r=-\widehat{\mathrm{ad}^r}\,,$$
for which we have
\begin{equation}\label{anchors}[X,fY]=f[X,Y]+\rho(X)(f)Y\,,\quad [fX,Y]=f[X,Y]-\rho^r(X)(f)Y\,,
\end{equation}
for all $X,Y\in\mathcal{E}$, $f\in\mathcal{A}$. If the bracket is skew-symmetric, then both anchors coincide, and if the
bracket is a strong QD-pseudoalgebra bracket, they are $\mathcal{A}$-linear. Moreover, if \ $[\cdot,\cdot]$ satisfies
additionally the Jacobi identity, i.e., we deal with a Loday QD-pseudoalgebra, then, for all $X,Y\in\mathcal{E}$,
\begin{equation}\label{anhomn}
\rho\left([X,Y]\right)=[\rho(X),\rho(Y)]_c\,.
\end{equation}
\end{thm}
\begin{proof}
Similarly as above,
$$(\delta_2(f)\delta_2(g)B)(X,Y)=\rho(X)(g)fY-f\rho(X)(g)Y=0\,,$$
so $B$ is a first-order differential operator with respect to the second argument. The same can be done for
the first argument.

Next, as for any QD-pseudoalgebra bracket $B$ we have, analogously to (\ref{QD}),
\begin{equation}\label{QD1}(\delta_1(f)\delta_2(g)B)(X,Y)=\left(\rho(fX)-f\rho(X)\right)(g)Y=\left(\rho^r(gY)-g\rho^r(Y)\right)(f)X\,,\end{equation}
both anchor maps are $\mathcal{A}$-linear if and only if $D$ is of total order $\le 1$. The rest follows analogously
to the previous theorem.

\end{proof}

The next observation is that quasi pseudoalgebra structures on an $\mathcal{A}$-module $\mathcal{E}$ have certain analogs of
anchor maps, namely $\mathcal{A}$-module homomorphisms $b=b^r,b^l:\mathcal{E}\to\mathrm{Der}(\mathcal{A})\otimes_\mathcal{A}\mathrm{End}(\mathcal{E})$. For every
$X\in\mathcal{E}$ we will view $b(X)$ as an $\mathcal{A}$-module homomorphism $b(X):\Omega^1\otimes_\mathcal{A}\mathcal{E}\to\mathcal{E}$, where $\Omega^1$
is the $\mathcal{A}$-submodule of $\mathrm{Hom}_\mathcal{A}(\mathrm{Der}(\mathcal{A});\mathcal{A})$ generated by $\mathrm{d}\mathcal{A}=\{\mathrm{d} f:f\in\mathcal{A}\}$ and $\langle\mathrm{d}
f,D\rangle=D(f)$. Elements of $\,\mathrm{Der}(\mathcal{A})\otimes_\mathcal{A}\mathrm{End}(\mathcal{E})$ act on elements of $\Omega^1\otimes_\mathcal{A} \mathcal{E}$ in the
obvious way: $(V\otimes\Phi)(\omega\otimes X)=\langle V,\omega\rangle\Phi(X)$.

\begin{thm}\label{T1} A $\,\mathbb{K}$-bilinear bracket $B=[\cdot,\cdot]$ on an $\mathcal{A}$-module $\mathcal{E}$ defines a quasi pseudoalgebra structure if
and only if there are $\mathcal{A}$-module homomorphisms \begin{equation}\label{Lqa} b^r,b^l:\mathcal{E}\to\mathrm{Der}(\mathcal{A})\otimes_\mathcal{A}\mathrm{End}(\mathcal{E})\,,
\end{equation}
called {\em generalized anchor maps, right and left}, such that, for all $X,Y\in\mathcal{E}$ and all $f\in\mathcal{A}$,
\begin{equation}\label{aLqa} [X,fY]=f[X,Y]+b^l(X)(\mathrm{d} f\otimes Y)\,,\quad [fX,Y]=f[X,Y]-b^r(Y)(\mathrm{d} f\otimes X)\,. \end{equation} The
generalized anchor maps are actual anchor maps if they take values in $\mathrm{Der}(\mathcal{A})\otimes_\mathcal{A}\{\mathrm{Id}_\mathcal{E}\}$.
\end{thm}
\begin{proof} Assume first that the bracket $B$ is a bidifferential operator of total degree $\le 1$ and define a
three-linear map of vector spaces  $A:\mathcal{E}\times\mathcal{A}\times\mathcal{E}\to\mathcal{E}$ by
$$A(X,g,Y)=(\delta_2(g)B)(X,Y)=[X,gY]-g[X,Y]\,.
$$
It is easy to see that $A$ is $\mathcal{A}$-linear with respect to the first and the third argument, and a derivation
with respect to the second. Indeed, as
$$
(\delta_1(f)\delta_2(g)B)(X,Y)=A(fX,g,Y)-fA(X,g,Y)=0\,,
$$
we get $\mathcal{A}$-linearity with respect to the first argument. Similarly, from $\delta_2(f)\delta_2(g)B=0$, we get the
same conclusion for the third argument. We have also
\begin{eqnarray}\nonumber A(X,fg,Y)&=&[X,fgY]-fg[X,Y]=[X,fgY]-f[X,gY]+f[X,gY]-fg[X,Y]\\
&=&A(X,f,gY)+fA(X,g,Y)=gA(X,f,Y)+fA(X,g,Y)\,,\label{w1}
\end{eqnarray}
thus the derivation property. This implies that
$A$ is represented by an $\mathcal{A}$-module homomorphism $b^l:\mathcal{E}\to\mathrm{Der}(\mathcal{A})\otimes_\mathcal{A}\mathrm{End}(\mathcal{E})$. Analogous
considerations give us the right generalized anchor map $b^r$.

Conversely, assume the existence of both generalized anchor maps. Then, the map $A$ defined as above reads
$A(X,f,Y)=b^l(X)(\mathrm{d} f\otimes Y)$, so is $\mathcal{A}$-linear with respect to $X$ and $Y$. Hence,
$$(\delta_1(f)\delta_2(g)B)(X,Y)=A(fX,g,Y)-fA(X,g,Y)=0$$ and
$$(\delta_2(f)\delta_2(g)B)(X,Y)=A(X,g,fY)-fA(X,g,Y)=0\,.$$
A similar reasoning for $b^r$ gives $(\delta_1(f)\delta_1(g)B)(X,Y)=0$, so the bracket is a bidifferential operator
of total order $\le 1$.
\end{proof}

In the case when we deal with a quasi algebroid, i.e., $\mathcal{A}=C^\infty(M)$ and $\mathcal{E}=\mathrm{Sec}(E)$ for a vector bundle
$\tau:E\to M$, the generalized anchor maps (\ref{Lqa}) are associated with vector bundle maps that we denote
(with some abuse of notations) also by $b^r,b^l$,
$$b^r,b^l:E\to T M\otimes_M\mathrm{End}(E)\,,$$
covering the identity on $M$. Here, $\mathrm{End}(E)$ is the endomorphism bundle of $E$, so $\mathrm{End}(E)\simeq E^\ast\otimes_M
E$. The induced maps of sections produce from sections of $E$ sections of $ T M\otimes_M\mathrm{End}(E)$ which, in
turn, act on sections of $ T^\ast M\otimes_M E$ in the obvious way. An algebroid version of Theorem \ref{T1}
is the following.

\begin{thm}\label{T1a} An $\mathbb{R}$-bilinear bracket $B=[\cdot,\cdot]$ on the real space $\mathrm{Sec}(E)$ of sections of a vector
bundle $\tau:E\to M$ defines a quasi
algebroid structure if and only if there are vector bundle morphisms
\begin{equation}\label{Lqa1} b^r,b^l:E\to T
M\otimes_M\mathrm{End}(E) \end{equation} covering the identity on $M$, called {\em generalized anchor maps, right and left}, such
that, for all $X,Y\in\mathrm{Sec}(E)$ and all $f\in C^\infty(M)$, (\ref{aLqa}) is satisfied. The generalized anchor maps are
actual anchor maps, if they take values in $ T M\otimes\langle \mathrm{Id}_E\rangle\simeq T M$.
\end{thm}

\subsection{Loday algebroids}\label{GKPSection4}
Let us isolate and specify the most important particular cases of Definition \ref{def}.
\begin{defi}\

\begin{enumerate}
\item A {\it {faint} Loday algebroid} (resp., {\it {faint} Lie algebroid}) on a vector bundle $E$ over a
base manifold $M$ is a Loday bracket (resp., a Lie bracket) on the $C^\infty(M)$-module $\mathrm{Sec}(E)$ of smooth
sections of $E$ which is a bidifferential operator.

\item A {\it weak Loday algebroid} (resp., {\it weak Lie algebroid}) on a vector bundle $E$ over a base
manifold $M$ is a Loday bracket (resp., a Lie bracket) on the $C^\infty(M)$-module $\mathrm{Sec}(E)$ of smooth
sections of $E$ which is a bidifferential operator of degree $\le 1$ with respect to each variable separately.

\item A {\it Loday quasi algebroid} (resp., {\it Lie quasi algebroid}) on a vector bundle $E$ over a base
manifold $M$ is a Loday bracket (resp., Lie bracket) on the $C^\infty(M)$-module $\mathrm{Sec}(E)$ of smooth sections
of $E$ which is a bidifferential operator of total degree $\le 1$.

\item A {\it QD-algebroid} (resp., {\it skew QD-algebroid, Loday QD-algebroid, Lie QD-algebroid}) on a vector
bundle $E$ over a base manifold $M$ is an $\mathbb{R}$-bilinear bracket (resp., skew bracket, Loday bracket, Lie
bracket) on the $C^\infty(M)$-module $\mathrm{Sec}(E)$ of smooth sections of $E$ for which the adjoint operators
$\mathrm{ad}_X$ and $\mathrm{ad}_X^r$ are derivative endomorphisms.
\end{enumerate}
\end{defi}
\begin{rem} {\it Lie pseudoalgebras} appeared first in the paper of Herz  \cite{He}, but one can find similar concepts
under more than a dozen of names in the literature  (e.g. {\it Lie modules,  $(R,A)$-Lie   algebras,
Lie-Cartan pairs, Lie-Rinehart algebras, differential algebras}, etc.). Lie algebroids were introduced  by
Pradines \cite{Pr} as infinitesimal parts of differentiable groupoids. In the same year a book by Nelson was
published where a general theory of Lie modules, together with a big part of the corresponding differential
calculus, can be found. We also refer to a survey article by Mackenzie \cite{Ma}. QD-algebroids, as well as
Loday QD-algebroids  and Lie QD-algebroids, have been introduced in \cite{G1}. In \cite{GU,GGU} Loday strong
QD-algebroids have been called Loday algebroids and strong QD-algebroids have been called just {\it
algebroids}. The latter served as geometric framework for generalized Lagrange and Hamilton formalisms.

In the case of line bundles, $\mathrm{rk} E=1$, Lie QD-algebroids are exactly {\it local Lie algebras} in the sense of
Kirillov \cite{Ki}. They are just {\it Jacobi brackets}, if the bundle is trivial, $\mathrm{Sec}(E)=C^\infty(M)$. Of course,
Lie QD-algebroid brackets are first-order bidifferential operators by definition, while Kirillov has
originally started with considering Lie brackets on sections of line bundles determined by local operators and
has only later discovered that these operators have to be bidifferential operators of first order. A purely
algebraic version of Kirillov's result has been proven in \cite{G}, Theorems 4.2 and 4.4, where bidifferential
Lie brackets on associative commutative algebras containing no nilpotents have been considered.
\end{rem}

\begin{ex}\label{e1} Let us consider a Loday algebroid bracket in the sense of \cite{Ha,HM,ILMP,KS,MM,Wa}, i.e., a Loday algebra bracket $[\cdot,\cdot]$ on the $C^\infty(M)$-module $\mathcal{E}=\mathrm{Sec}(E)$ of sections of a vector bundle $\tau:E\to M$ for which there is a vector bundle morphism $\rho:E\to T M$ covering the identity on $M$ (the left anchor map) such that (\ref{anchor}) is satisfied. Since, due to (\ref{anhom}), the anchor map is necessarily a homomorphism of the Loday bracket into the Lie bracket of vector fields, our Loday algebroid is just a Lie algebroid in the case when $\rho$ is injective.
In the other cases the anchor map does not determine the Loday algebroid structure, in particular does not
imply any locality of the bracket with respect to the first argument. Thus, this concept of Loday algebroid is
not geometric.

For instance, let us consider a Whitney sum bundle $E=E_1\oplus_M E_2$ with the canonical projections $p_i:E\to
E_i$ and any $\mathbb{R}$-linear map $\varphi:\mathrm{Sec}(E_1)\to C^\infty(M)$. Being only $\mathbb{R}$-linear, $\varphi$ can be chosen very strange
non-geometric and non-local. Define now the following bracket on $\mathrm{Sec}(E)$:
$$[X,Y]=\varphi(p_1(X))\cdot p_2(Y)\,.$$
It is easy to see that this is a Loday bracket which admits the trivial left anchor, but the bracket is
non-local and non-geometric as well.
\end{ex}

\begin{ex} A standard example of a weak Lie algebroid bracket is a Poisson (or, more generally, Jacobi) bracket $\{\cdot,\cdot\}$ on $C^\infty(M)$ viewed as a $C^\infty(M)$-module of section of the trivial line bundle $M\times\mathbb{R}$.
It is a bidifferential operator of order $\le 1$ and the total order $\le 2$. It is actually a Lie QD-algebroid
bracket, as $\mathrm{ad}_f$ and $\mathrm{ad}_f^r$ are, by definition, derivations (more generally, first-order differential
operators). Both anchor maps coincide and give the corresponding Hamiltonian vector fields, $\rho(f)(g)=\{
f,g\}$. The map $f\mapsto \rho(f)$ is again a differential operator of order 1, so is not implemented by a
vector bundle morphism $\rho:M\times\mathbb{R}\to T M$. Therefore, this weak Lie algebroid is not a Lie algebroid. {This has a straightforward generalization to {\em Kirillov brackets} being local Lie brackets on sections of a line bundle \cite{Ki}.}
\end{ex}
\begin{ex}
Various brackets are associated with a volume form $\omega$ on a manifold $M$ of dimension $n$ (see e.g. \cite{Li}). Denote with $\mathcal{X}^k(M)$ (resp., $\Omega^k(M)$) the spaces of $k$-vector fields (resp., $k$-forms) on $M$. As
the contraction maps $\mathcal{X}^k(M)\ni K\mapsto i_K\omega\in\Omega^{n-k}(M)$ are isomorphisms of $C^\infty(M)$-modules, to the
de Rham cohomology operator $\mathrm{d}:\Omega^{n-k-1}(M)\to\Omega^{n-k}(M)$ corresponds a homology operator
$\delta:\mathcal{X}^k(M)\to\mathcal{X}^{k-1}(M)$. The skew-symmetric bracket $B$ on $\mathcal{X}^2(M)$ defined in \cite{Li} by
$B(t,u)=-\delta(t)\wedge\delta(u)$ is not a Lie bracket, since its Jacobiator $B(B(t,u),v)+c.p.$ equals
$\delta(\delta(t)\wedge\delta(u)\wedge\delta(v))$. A solution proposed in \cite{Li} depends on considering the algebra $N$ of bivector
fields modulo $\delta$-exact bivector fields for which the Jacobi anomaly disappears, so that $N$ is a Lie
algebra.

Another option is to resign from skew-symmetry and define the corresponding {faint} Loday algebroid. In view
of the duality between $\mathcal{X}^2(M)$ and $\Omega^{n-2}$, it is possible to work with $\Omega^{n-2}(M)$ instead. For
$\gamma\in\Omega^{n-2}(M)$ we define the vector field $\widehat{\gamma}\in\mathcal{X}(M)$ from the formula $i_{\widehat{\gamma}}\omega=\mathrm{d}\gamma$.
The bracket in $\Omega^{n-2}(M)$ is now defined by {(see \cite{Lo})}
$$\{\gamma,\beta\}_\omega=\mathcal{L}_{\widehat{\gamma}}\beta=i_{\widehat{\gamma}}i_{\widehat{\beta}}\omega+\mathrm{d}
i_{\widehat{\gamma}}\beta\,.$$ Since we have
$$i_{[\widehat{\gamma},\widehat{\beta}]_{vf}}\omega=\mathcal{L}_{\widehat{\gamma}}i_{\widehat{\beta}}\omega-i_{\widehat{\beta}}\mathcal{L}_{\widehat{\gamma}}\omega=
\mathrm{d} i_{\widehat{\gamma}}i_{\widehat{\beta}}\omega=\mathrm{d} \{\gamma,\beta\}_\omega\,,$$ it holds
$$\{\gamma,\beta\}_\omega^{\widehat{}}=[\widehat{\gamma},\widehat{\beta}]_{vf}\,.$$
Therefore,
$$\{\{\gamma,\beta\}_\omega,\eta\}_\omega=\mathcal{L}_{\{\gamma,\beta\}_\omega^{\widehat{}}}\eta=
\mathcal{L}_{\widehat{\gamma}}\mathcal{L}_{\widehat{\beta}}\eta-\mathcal{L}_{\widehat{\beta}}\mathcal{L}_{\widehat{\gamma}}\eta=\{\gamma,\{\beta,\eta\}_\omega\}_\omega
-\{\beta,\{\gamma,\eta\}_\omega\}_\omega\,,$$
so the Jacobi identity is satisfied and we deal with a Loday algebra. This
is in fact a {faint} Loday algebroid structure on $\wedge^{n-2} T^\ast M$ with the left anchor
$\rho(\gamma)=\widehat{\gamma}$. This bracket is a bidifferential operator which is first-order with respect to the second
argument and second-order with respect to the first one.
\end{ex}

Note that Lie QD-algebroids are automatically Lie algebroids, if the rank of the bundle $E$ is $>1$
\cite[Theorem 3]{G1}. Also some other of the above concepts do not produce qualitatively new examples.
\begin{thm} (\cite{G1,GM,GM1})\

\begin{description}
\item{(a)} Any Loday bracket on $C^\infty(M)$ (more generally, on sections of a line bundle) which is a
bidifferential operator is actually a Jacobi bracket (first-order and skew-symmetric). \item{(b)} Let
$[\cdot,\cdot]$ be a Loday bracket on sections of a vector bundle $\tau:E\to M$, admitting anchor maps
$\rho,\rho^r:\mathrm{Sec}(E)\to\mathcal{X}(M)$ which assign vector fields to sections of $E$ and such that (\ref{anchors}) is
satisfied (Loday QD-algebroid on $E$). Then, the anchors coincide, $\rho=\rho^r$, and the bracket is
skew-symmetric at points $p\in M$ in the support of $\rho=\rho^r$. Moreover, if the rank of $E$ is $>1$, then
the anchor maps are $C^\infty(M)$-linear, i.e. they come from a vector bundle morphism $\rho=\rho^r:E\to M$. In other words,
any Loday QD-algebroid is actually, around points where one anchor does not vanish, a Jacobi bracket if
$\mathrm{rk}(E)=1$, or Lie algebroid bracket if $\mathrm{rk}(E)>1$.
\end{description}
\end{thm}
The above results show that relaxing skew-symmetry and considering Loday brackets on $C^\infty(M)$ or $\mathrm{Sec}(E)$ does
not lead to new structures (except for just bundles of Loday algebras), if we assume differentiability in the
first case and the existence of both (possibly different) anchor maps in the second. Therefore, a definition
of Loday algebroids that admits a rich family of new examples, must resign from the traditionally understood
right anchor map.

\medskip
\noindent The definition of the main object of our studies can be formulated as follows.
\begin{defi}\label{d1}
A {\it Loday algebroid} on a vector bundle $E$ over a base manifold $M$ is a Loday bracket on the
$C^\infty(M)$-module $\mathrm{Sec}(E)$ of smooth sections of $E$ which is a bidifferential operator of total degree
$\le 1$ and for which the adjoint operator $\mathrm{ad}_X$ is a derivative endomorphism.
\end{defi}
{Of course, the above definition of Loday algebroid is stronger than those known in the literature (e.g. \cite{Ha,HM,ILMP,KS,MM,Wa}), which assume only the existence of a left anchor and put no differentiability requirements for the first variable.
}
\begin{thm}\label{LodAld} A Loday bracket $[\cdot,\cdot]$ on the real space $\mathrm{Sec}(E)$ of sections of a vector bundle $\tau:E\to M$
defines a Loday algebroid structure
if and only if there are vector bundle morphisms
\begin{equation}\label{Lqa1a} \rho:E\to T M\,, \quad \alpha:E\to T
M\otimes_M\mathrm{End}(E)\,, \end{equation} covering the identity on $M$, such that, for all $X,Y\in\mathrm{Sec}(E)$ and all $f\in
C^\infty(M)$, \begin{equation}\label{aLqa1} [X,fY]=f[X,Y]+\rho(X)(f)Y\,,\quad [fX,Y]=f[X,Y]-\rho(Y)(f)X+\alpha(Y)(\mathrm{d} f\otimes
X)\,. \end{equation} If this is the case, {the anchors are uniquely determined and} the left anchor induces a homomorphism of the Loday bracket into the bracket
$[\cdot,\cdot]_{vf}$ of vector fields,
$$\rho([X,Y])=[\rho(X),\rho(Y)]_{vf}\,.$$
\end{thm}
\begin{proof} This is a direct consequence of Theorem \ref{T1a} and the fact that an algebroid bracket has the left anchor map. We
just write the generalized right anchor map as $b^r=\rho\otimes I-\alpha$.
\end{proof}

To give a local form of a Loday algebroid bracket, let us recall that sections $X$ of the vector bundle $E$
can be identified with linear (along fibers) functions $\iota_X$ on the dual bundle $E^\ast$. Thus, fixing local
coordinates $(x^a)$ in $M$ and a basis of local sections $e_i$ of $E$, we have a corresponding system
$(x^a,\xi_i=\iota_{e_i})$ of affine coordinates in $E^\ast$. As local sections of $E$ are identified with linear
functions $\sigma=\sigma^i(x)\xi_i$, the Loday bracket is represented by a bidifferential operator $B$ of total
order $\le 1$:
$$B(\sigma^i_1(x)\xi_i,\sigma^j_2(x)\xi_j)=c_{ij}^k(x)\sigma^i_1(x)\sigma^j_2(x)\xi_k+\beta_{ij}^{ak}(x)\frac{\partial\sigma^i_1}
{\partial x^a}(x)\sigma^j_2(x)\xi_k+\gamma_{ij}^{ak}(x)\sigma^i_1(x)\frac{\partial\sigma^j_2}{\partial x^a}(x)\xi_k\,.$$
Taking into account the existence of the left anchor, we have
\begin{eqnarray}\label{loc} B(\sigma^i_1(x)\xi_i,\sigma^j_2(x)\xi_j)&=&c_{ij}^k(x)\sigma^i_1(x)\sigma^j_2(x)\xi_k+\alpha_{ij}^{ak}(x)
\frac{\partial\sigma^i_1}{\partial x^a}(x)\sigma^j_2(x)\xi_k\\
&&+\rho_{i}^{a}(x)\left(\sigma^i_1(x)\frac{\partial\sigma^j_2}{\partial x^a}(x)-\frac{\partial\sigma^j_1}{\partial
x^a}(x)\sigma^i_2(x)\right)\xi_j\,. \nonumber
\end{eqnarray} Since sections of $\mathrm{End}(E)$ can be written in the form of linear
differential operators, we can rewrite (\ref{loc}) in the form
\begin{equation}\label {loc1}
B=c_{ij}^k(x)\xi_k\partial_{\xi_i}\otimes\partial_{\xi_j}+\alpha_{ij}^{ak}(x)\xi_k\partial_{x^a}\partial_{\xi_i}\otimes\partial_{\xi_j}
+\rho_{i}^{a}(x)\partial_{\xi_i}\wedge\partial_{x^a}\,.
\end{equation}
Of course, there are additional relations between coefficients of $B$ due to the fact that the Jacobi identity
is satisfied.

\subsection{Examples}\label{GKPSection5}
\subsubsection{Leibniz algebra}

Of course, a finite-dimensional Leibniz algebra is a Leibniz algebroid over a point.

\subsubsection{Courant-Dorfman bracket}

The {\it Courant bracket} is defined on sections of ${\cal T} M= T M\oplus_M T^*M$ as follows:
\begin{equation}\label{CB}[X+\omega,Y+\eta]=[X,Y]_{vf}+\mathcal{L}_X\eta-\mathcal{L}_Y\omega-\frac{1}{2}\left(d\,i_X\eta-d\,i_Y\omega\right)\,.
\end{equation} This bracket is antisymmetric, but it does not satisfy the
Jacobi identity; the Jacobiator is an exact 1-form. It is, as easily seen, given by a bidifferential operator
of total order $\le 1$, so it is a skew quasi algebroid.

\medskip
The {\it Dorfman bracket} is defined on the same module of sections. Its definition is the same as for
Courant, except that the corrections and the exact part of the second Lie derivative disappear:
\begin{equation}\label{CD}[X+\omega,Y+\eta]=[X,Y]_{vf}+\mathcal{L}_X\eta-i_Y\,\mathrm{d}\omega=[X,Y]_{vf}+i_X\,\mathrm{d}\eta-i_Y\,\mathrm{d}\omega+\mathrm{d}\,i_X\eta\,.
\end{equation}
This bracket is visibly non skew-symmetric, but it is a Loday bracket which is bidifferential of total order
$\le 1$. Moreover, the Dorfman bracket admits the classical left anchor map
\begin{equation}\label{CDa}\rho:{\cal T} M= T M\oplus_M T^*M\to  T M
\end{equation}
which is the projection onto the first component. Indeed,
$$[X+\omega,f(Y+\eta)]=[X,fY]_{vf}+\mathcal{L}_Xf\eta-i_{fY}\,\mathrm{d}\omega=f[X+\omega,Y+\eta]+X(f)(Y+\eta)\,.$$
For the right generalized anchor we have
\begin{eqnarray*}[f(X+\omega),Y+\eta]&=&[fX,Y]_{vf}+i_{fX}\,\mathrm{d}\eta-i_Y\,\mathrm{d}(f\omega)+\mathrm{d}\,i_{fX}\eta\\
&=&f[X+\omega,Y+\eta]-Y(f)(X+\omega)+\mathrm{d} f\wedge (i_X\eta+i_Y\omega)\,,\end{eqnarray*} so that
 $$\alpha(Y+\eta)(\mathrm{d} f\otimes(X+\omega))=\mathrm{d} f\wedge (i_X\eta+i_Y\omega)=2\langle X+\omega,Y+\eta\rangle_+\cdot\mathrm{d} f\,,$$
where
$$\langle
X+\omega,Y+\eta\rangle_+=\frac{1}{2}\left(i_X\eta+i_Y\omega\right)=\frac{1}{2}\left(\langle X,\eta\rangle+\langle
Y,\omega\rangle\right)\,,$$ is a symmetric nondegenerate bilinear form on ${\cal T} M$ (while $\langle\cdot,\cdot\rangle$ is the
canonical pairing). We will refer to it, though it is not positively defined, as the {\it scalar product} in
the bundle ${\cal T} M$.

Note that $\alpha(Y+\eta)$ is really a section of $ T M\otimes_M\mathrm{End}( T M\oplus_M T^\ast M)$ that in local
coordinates reads
$$\alpha(Y+\eta)=\sum_k\partial_{x^k}\otimes(\mathrm{d} x^k\wedge (i_{\eta}+i_Y))\,.$$
Hence, the Dorfman bracket is a Loday algebroid bracket.\medskip

It is easily checked that the Courant bracket is the antisymmetrization of the Dorfman bracket, and that the
Dorfman bracket is the Courant bracket plus $\mathrm{d}\langle X+\omega,Y+\eta\rangle_+$
\subsubsection{Twisted Courant-Dorfman bracket}
The Courant-Dorfman bracket can be twisted by adding a term associated with a 3-form $\Theta$ \cite{YKS1,SW}:
\begin{equation}\label{TCD}[X+\omega,Y+\eta]=[X,Y]_{vf}+\mathcal{L}_X\eta-i_Y\,\mathrm{d}\omega+i_{X\wedge Y}\Theta\,.
\end{equation}
It turns out that this bracket is still a Loday bracket if the 3-form $\Theta$ is closed. As the added term is
$C^\infty(M)$-linear with respect to $X$ and $Y$, the anchors remain the same, thus we deal with a Loday algebroid.

\subsubsection{Courant algebroid}
Courant algebroids -- structures generalizing the Courant-Dorfman bracket on ${\cal T} M$ -- were introduced as as
double objects for Lie bialgebroids by Liu, Weinstein and Xu \cite{LWX} in a bit complicated way. It was shown
by Roytenberg \cite{Roy0} that a Courant algebroid can be equivalently defined as a vector bundle $\tau:E\to M$
with a Loday bracket on $\mathrm{Sec}(E)$, an anchor $\rho:E\to  T M$, and a symmetric nondegenerate inner product
$(\cdot,\cdot)$ on $E$, related by a set of four additional properties. It was further observed \cite{Uch,
GM2} that the number of independent conditions can be reduced.

\begin{defi} A \textit{Courant algebroid} is a vector bundle $\tau:E\to M$ equipped with a Leibniz
bracket $[\cdot,\cdot]$ on $\mathrm{Sec}(E)$, a vector bundle map (over the identity) $\rho:E\to  T M$, and a
nondegenerate symmetric bilinear form (scalar product) $({\cdot}|{\cdot})$ on $E$ satisfying the identities
\begin{eqnarray}\label{4} &\rho(X)( Y|Y)=2( X|[Y,Y]),\\ &\rho(X)( Y|Y)=2( [X,Y]|Y).\label{5} \end{eqnarray}
\end{defi}

\medskip\noindent Note that (\ref{4}) is equivalent to
\begin{equation}\label{4a}
\rho(X)( Y|Z)=( X|[Y,Z]+[Z,Y]).
\end{equation}
Similarly, (\ref{5}) easily implies the invariance of the pairing $({\cdot},{\cdot})$ with respect to the
adjoint maps
\begin{equation}\label{6}\rho(X)(Y|Z)=( [X,Y]|Z)+( Y|[X,Z]),
\end{equation}
which in turn shows that $\rho$ is the anchor map for the left multiplication: \begin{equation}\label{zr}
[X,fY]=f[X,Y]+\rho(X)(f)Y\,. \end{equation} Twisted Courant-Dorfman brackets are examples of Courant algebroid brackets
with $({\cdot},{\cdot})=\langle\cdot,\cdot\rangle_+$ as the scalar product. Defining a derivation
$\mathrm{D}:C^\infty(M)\to\mathrm{Sec}(E)$ by means of the scalar product
\begin{equation}\label{D}(\mathrm{D}(f)|X)=\frac{1}{2}\rho(X)(f)\,, \end{equation} we get out of
(\ref{4a}) that \begin{equation}\label{4c}[Y,Z]+[Z,Y]=2\mathrm{D}(Y|Z)\,. \end{equation} This, combined with (\ref{zr}), implies in turn
\begin{equation}\label{LAX} \alpha(Z)(\mathrm{d} f\otimes Y)=2(Y|Z)\mathrm{D}(f)\,, \end{equation} so any Courant algebroid is a Loday algebroid.
{
\subsubsection{Brackets associated with contact structures}
In \cite{G2}, contact (super)manifolds have been studied as symplectic principal $\mathbb{R}^\times$-bundles $(P,\omega)$; the symplectic form being homogeneous with respect to the $\mathbb{R}^\times$-action. Similarly, Kirillov brackets on line bundles have been regarded as Poisson principal $\mathbb{R}^\times$-bundles. Consequently, {\em Kirillov algebroids} and {\em contact Courant algebroids} have been introduced, respectively, as homogeneous Lie algebroids and Courant algebroids on vector bundles equipped with a compatible
$\mathbb{R}^\times$-bundle structure. The corresponding brackets are therefore particular Lie algebroid and Courant algebroid brackets, thus Loday algebroid brackets. In other words, Kirillov and contact Courant algebroids are examples of Loday algebroids equipped additionally with some extra geometric structures.

As a canonical example of a contact Courant algebroid, consider the contact 2-manifold represented by the symplectic principal $\mathbb{R}^\times$-bundle $ T^*[2] T[1](\mathbb{R}^\times\times M)$, for a purely even manifold $M$ \cite{G2}. As the cubic Hamiltonian $H$ associated with the canonical vector field on $ T[1](\mathbb{R}^\times\times M)$ being the de Rham derivative is 1-homogeneous, we obtain a homogeneous Courant bracket on the linear principal $\mathbb{R}^\times$-bundle $P= T(\mathbb{R}^\times\times M)\oplus_{\mathbb{R}^\times\times M} T^*(\mathbb{R}^\times\times M)$. It can be reduced to the vector bundle $E=(\mathbb{R}\times T M)\oplus_M(\mathbb{R}^*\times T^*M)$ whose sections are $(X,f)+(\alpha,g)$, where $f,g\in C^\infty(M)$, $X$ is a vector field, and $\alpha$ is a one-form on $M$,
which is a Loday algebroid bracket of the form
\begin{eqnarray}\label{edirac}
&[(X_1,f_1)+(\alpha_1,g_1),(X_2,f_2)+(\alpha_2,g_2)]=
\left([X_1,X_2]_{vf},X_1(f_2)-X_2(f_1)\right)\\
&+\left(\mathcal{L}_{X_1}\alpha_2-i_{X_2}\mathrm{d}\alpha_1+f_1\alpha_2- f_2\alpha_1
+f_2\mathrm{d} g_1+g_2\mathrm{d} f_1,
X_1(g_2)-X_2(g_1)+i_{X_2}\alpha_1+f_1g_2\right)\,.\nonumber
\end{eqnarray}
This is the Dorfman-like version of the bracket whose skew-symmetrization gives exactly the bracket introduced by Wade \cite{Wa0} to define so called {\it $\mathcal{E}^1(M)$-Dirac structures} and considered also in \cite{GM2}.
The full contact Courant algebroid structure on $E$ consists additionally \cite{G2} of
the symmetric pseudo-Euclidean product
$$\langle(X,f)+(\alpha,g),(X,f)+(\alpha,g)\rangle=\langle X,\alpha\rangle+fg\,,$$
and
the vector bundle morphism $\rho^{\! 1}:E\to T M\times\mathbb{R}$, corresponding to a map assigning to sections of $E$ first-order differential operators on $M$, of the form
$$\rho^{\! 1}\left((X,f)+(\alpha,g)\right)=X+f\,.$$

}

\subsubsection{Grassmann-Dorfman bracket}
The Dorfman bracket (\ref{CD}) can be immediately generalized to a bracket on sections of ${\cal T}^\wedge M= T
M\oplus_M\wedge T^\ast M$, where
$$\wedge T^\ast M=\bigoplus_{k=0}^\infty\wedge^k T^\ast M\,,$$
so that the module of sections, $\mathrm{Sec}(\wedge T^\ast M)=\Omega(M)=\bigoplus_{k=0}^\infty\Omega^k(M)$, is the
Grassmann algebra of differential forms. The bracket, {\it Grassmann-Dorfman bracket}, is formally given by
the same formula $(\ref{CD})$ and the proof that it is a Loday algebroid bracket is almost the same. The left
anchor is the projection on the summand $ T M$,
\begin{equation}\label{CDa1}\rho: T M\oplus_M \wedge T^*M\to  T M\,,
\end{equation}
and
$$\alpha(Y+\eta)(\mathrm{d} f\otimes(X+\omega))=\mathrm{d} f\wedge (i_X\eta+i_Y\omega)=2\,\mathrm{d} f\wedge\langle X+\omega,Y+\eta\rangle_+\,,$$
where
$$\langle
X+\omega,Y+\eta\rangle_+=\frac{1}{2}\left(i_X\eta+i_Y\omega\right)\,,$$ is a symmetric nondegenerate bilinear form on
${\cal T}^\wedge M$, this time with values in $\Omega(M)$. Like for the classical Courant-Dorfman bracket, the graph of a
differential form $\beta$ is an isotropic subbundle in ${\cal T}^\wedge M$ which is involutive (its sections are
closed with respect to the bracket) if and only if $\mathrm{d} \beta=0$. The Grassmann-Dorfman bracket induces Loday
algebroid brackets on all bundles $ T M\oplus_M\wedge^k T^\ast M$, $k=0,1,\dots,\infty$. These brackets have
been considered in \cite{Sh} and called there {\it higher-order Courant brackets} {(see also \cite{Za})}. Note that this is exactly
the bracket derived from the bracket of first-order (super)differential operators on the Grassmann algebra
$\Omega(M)$: we associate with $X+\omega$ the operator $S_{X+\omega}=i_X+\omega\wedge\,$ and compute the super-commutators,
$$[[S_{X+\omega},\mathrm{d}]_{sc},S_{Y+\eta}]_{sc}=S_{[X+\omega,Y+\eta]}\,.$$

\subsubsection{Grassmann-Dorfman bracket for a Lie algebroid}
All the above remains valid when we replace $ T M$ with a Lie algebroid $(E,[\cdot,\cdot]_E,\rho_E)$, the de
Rham differential $\mathrm{d}$ with the Lie algebroid cohomology operator $\mathrm{d}^E$ on $\mathrm{Sec}(\wedge E^\ast)$, and the Lie
derivative along vector fields with the Lie algebroid Lie derivative $\mathcal{L}^E$. We define a bracket on sections
of $E\oplus_M\wedge E^\ast$ with formally the same formula
\begin{equation}\label{CDA}[X+\omega,Y+\eta]=[X,Y]_{E}+\mathcal{L}_X^E\eta-i_Y\,\mathrm{d}^E\omega\,.
\end{equation}
This is a Loday algebroid bracket with the left anchor
$$\rho:E\oplus_M\wedge E^\ast\to T M\,,\quad \rho(X+\omega)=\rho_E(X)$$
and
$$\alpha(Y+\eta)(\mathrm{d} f\otimes(X+\omega))=\mathrm{d}^E f\wedge (i_X\eta+i_Y\omega)\,.$$

\subsubsection{Lie derivative bracket for a Lie algebroid}
The above Loday bracket on sections of $E\oplus_M\wedge E^\ast$ has a simpler version. Let us put simply
\begin{equation}\label{CDA1}[X+\omega,Y+\eta]=[X,Y]_{E}+\mathcal{L}_X^E\eta\,.
\end{equation}
This is again a Loday algebroid bracket with the same left anchor and and
$$\alpha(Y+\eta)(\mathrm{d} f\otimes(X+\omega))=\mathrm{d}^E f\wedge i_X\eta+\rho_E(Y)(f)\omega\,.$$
In particular, when reducing to 0-forms, we get a Leibniz algebroid structure on $E\times \mathbb{R}$, where the
bracket is defined by $[X+f,Y+g]=[X,Y]_E+\rho_E(X)g$, the left anchor by $\rho(X,f)=\rho_E(X)$, and the generalized right
anchor by
$$b^r(Y,g)(\mathrm{d} h\otimes (X+f))=-\rho_E(Y)(h)X\,.$$
In other words,
$$\alpha(Y,g)(\mathrm{d} h\otimes (X+f))=\rho_E(Y)(h)f\,.$$

\subsubsection{Loday algebroids associated with a Nambu-Poisson structure}

In the following $M$ denotes a smooth $m$-dimensional manifold and $n$ is an integer such that $3\le n\le
m$. An almost Nambu-Poisson structure of order $n$ on $M$ is an $n$-linear bracket $\{\cdot,\ldots,\cdot\}$ on
$C^\infty(M)$ that is skew-symmetric and has the Leibniz property with respect to the point-wise multiplication. It
corresponds to an $n$-vector field $\Lambda\in\mathrm{Sec}(\wedge^n T M)$. Such a structure is Nambu-Poisson if it verifies
the {\em Filippov identity} ({\em generalized Jacobi identity}): \begin{eqnarray}\label{J} &\{ f_1,\dots,f_{n-1},\{
g_1,\dots,g_n\}\}=
\{\{ f_1,\dots,f_{n-1},g_1\},g_2,\dots,g_n\}+\\
&\{ g_1,\{ f_1,\dots,f_{n-1},g_2\},g_3,\dots,g_n\}+\dots+ \{ g_1,\dots,g_{n-1},\{
f_1,\dots,f_{n-1},g_n\}\}\,,\nonumber
\end{eqnarray} i.e., if the Hamiltonian vector fields $X_{f_1\ldots
f_{n-1}}=\{f_1,\ldots,f_{n-1},\cdot\}$ are derivations of the bracket. Alternatively, an almost Nambu-Poisson
structure is Nambu-Poisson if and only if
$$\mathcal{L}_{X_{f_1,\ldots,f_{n-1}}}\Lambda=0\,,$$
for all functions $f_1,\dots,f_{n-1}$.\medskip

Spaces equipped with skew-symmetric brackets satisfying the above identity have been introduced by Filippov
\cite{Fi} under the name {\it $n$-Lie algebras}.\medskip

The concept of Leibniz (Loday) algebroid used in \cite{ILMP} is the usual one, without differentiability
condition for the first argument. Actually, this example is a Loday algebroid in our sense as well. The
bracket is defined for $(n-1)$-forms by
$$[\omega,\eta]=\mathcal{L}_{\rho(\omega)}\eta+(-1)^n(i_{\mathrm{d}\omega}\Lambda)\eta\,,$$
where
$$\rho:\wedge^{n-1} T^*M\ni\omega\mapsto i_{\omega}\Lambda\in  T M$$
is actually the left anchor. Indeed,
$$[\omega,f\eta]=\mathcal{L}_{\rho(\omega)}f\eta+(-1)^n(i_{\mathrm{d}\omega}\Lambda)f\eta=f[\omega,\eta]+\rho(\omega)(f)\eta\,.$$
For the generalized right anchor we get
$$[f\omega,\eta]=\mathcal{L}_{\rho(f\omega)}\eta+(-1)^n(i_{\mathrm{d}(f\omega)}\Lambda)\eta=f[\omega,\eta]-i_{\rho(\omega)}(\mathrm{d} f\wedge \eta)\,,$$
so
$$\alpha(\eta)(\mathrm{d} f\otimes\omega)=\rho(\eta)(f)\,\omega-\rho(\omega)(f)\,\eta+\mathrm{d} f\wedge i_{\rho(\omega)}\eta\,.$$
Note that $\alpha$ is really a bundle map $\alpha:\wedge^{n-1} T^\ast M\to  T M\otimes_M\mathrm{End}(\wedge^{n-1} T^\ast M)$, since
it is obviously $C^{\infty}(M)$-linear in $\eta$ and $\omega$, as well as a derivation with respect to
$f.$\medskip

In \cite{Ha,HM}, another Leibniz algebroid associated with the Nambu-Poisson structure $\Lambda$ is proposed. The
vector bundle is the same, $E=\wedge^{n-1} T^\ast M$, the left anchor map is the same as well,
$\rho(\omega)=i_{\omega}\Lambda$, but the Loday bracket reads
$$[\omega,\eta]'=\mathcal{L}_{\rho(\omega)}\eta-i_{\rho(\eta)}\mathrm{d}\omega\,.$$
Hence,
\begin{eqnarray*}[f\omega,\eta]'&=&\mathcal{L}_{\rho(f\omega)}\eta-i_{\rho(\eta)}\mathrm{d}(f\omega)\\
&=&f[\omega,\eta]'-\rho(\eta)(f)\,\omega+\mathrm{d} f\wedge(i_{\rho(\omega)}\eta+i_{\rho(\eta)}\omega)\,, \end{eqnarray*} so that for the generalized
right anchor we get
$$\alpha(\eta)(\mathrm{d} f\otimes\omega)=\mathrm{d} f\wedge(i_{\rho(\omega)}\eta+i_{\rho(\eta)}\omega)\,.$$
This Loday algebroid structure is clearly the one obtained from the Grassmann-Dorfman bracket on the graph of
$\Lambda$,
$$\operatorname{graph}(\Lambda)=\{ \rho(\omega)+\omega:\omega\in\Omega^{n-1}(M)\}\,.$$
Actually, an $n$-vector field $\Lambda$ is a Nambu-Poisson tensor if and only if its graph is closed with respect
to the Grassmann-Dorfman bracket {\cite{Sh,Ha}}.

\subsection{The Lie pseudoalgebra of a Loday algebroid}\label{GKPSection6}
Let us fix a Loday pseudoalgebra bracket $[\cdot,\cdot]$ on an $\mathcal{A}$-module $\mathcal{E}$. Let $\rho:\mathcal{E}\to \mathrm{Der}(\mathcal{A})$
be the left anchor map, and let
$$b^r=\rho-\alpha:\mathcal{E}\to\mathrm{Der}(\mathcal{A})\otimes_\mathcal{A}\mathrm{End}(\mathcal{E})$$
be the generalized right anchor map. For every $X\in\mathcal{E}$ we will view $\alpha(X)$ as a $\mathcal{A}$-module homomorphism
$\alpha(X):\Omega^1\otimes_\mathcal{A}\mathcal{E}\to\mathcal{E}$, where $\Omega^1$ is the $\mathcal{A}$-submodule of $\mathrm{Hom}_\mathcal{A}(\mathcal{E};\mathcal{A})$ generated by
$\mathrm{d}\mathcal{A}=\{\mathrm{d} f:f\in\mathcal{A}\}$ and $\mathrm{d} f(D)=D(f)$.

It is a well-known fact that the subspace $\mathfrak{g}^{0}$ generated in a Loday algebra $\mathfrak{g}$ by
the symmetrized brackets $X\diamond Y=[X,Y]+[Y,X]$ is a two-sided ideal and that $\mathfrak{g}/\mathfrak{g}^0$
is a Lie algebra. Putting
$$\mathcal{E}^0=\operatorname{span}\{ [X,X]: X\in\mathcal{E}\}\,,$$
we have then \begin{equation}\label{ce} [\mathcal{E}^0,\mathcal{E}]=0\,,\quad [\mathcal{E},\mathcal{E}^0]\subset\mathcal{E}^0\,. \end{equation} Indeed, symmetrized brackets
are spanned by squares $[X,X]$, so, due to the Jacobi identity,
$$[[X,X],Y]=[X,[X,Y]]-[X,[X,Y]]=0$$
and \begin{equation}\label{dia}[Y,[X,X]]=[[Y,X],X]+[X,[Y,X]]=[X,Y]\diamond Y\,.
\end{equation} However, working with $\mathcal{A}$-modules, we would like to have an
$\mathcal{A}$-module structure on $\mathcal{E}/\mathcal{E}^0$. Unfortunately, $\mathcal{E}^0$ is not a submodule in general. Let us consider
therefore the $\mathcal{A}$-submodule ${\bar{\mathcal{E}}^0}$ of $\mathcal{E}$ generated by $\mathcal{E}^0$, i.e., ${\bar{\mathcal{E}}^0}=\mathcal{A}\cdot\mathcal{E}^0$.
\begin{lem} For all $f\in\mathcal{A}$ and $X,Y,Z\in\mathcal{E}$ we have
\begin{eqnarray}\label{fid0}\alpha(X)(\mathrm{d} f\otimes Y)&=&X\diamond(fY)-f(X\diamond Y)\,,\\
{[\alpha(X)(\mathrm{d} f\otimes Y),Z]}&=&{\rho(Z)(f)(X\diamond Y)-\alpha(Z)(\mathrm{d} f\otimes(X\diamond Y))\,.}\label{fid1}
\end{eqnarray} In
particular,
\begin{equation}\label{sy} [\alpha(X)(\mathrm{d} f\otimes Y),Z]=[\alpha(Y)(\mathrm{d} f\otimes X),Z]\,. \end{equation}
\end{lem}
\begin{proof} To prove (\ref{fid0}) it suffices to combine the identity $[X,fY]=f[X,Y]+\rho(X)(f)(Y)$ with
$$[fY,X]=f[Y,X]-\rho(X)(f)Y+\alpha(X)(\mathrm{d} f\otimes Y)\,.$$
Then, as $[\mathcal{E}^0,\mathcal{E}]=0$,
$$[\alpha(X)(\mathrm{d} f\otimes Y),Z]=-[f(X\diamond Y),Z]={\rho(Z)(f)(X\diamond Y)-\alpha(Z)(\mathrm{d} f\otimes(X\diamond Y))\,.}$$
\end{proof}
\begin{cor}\label{cor1} For all $f\in\mathcal{A}$ and $X,Y\in\mathcal{E}$,
\begin{equation}\label{fid} \alpha(X)(\mathrm{d} f\otimes Y)\in{\bar{\mathcal{E}}^0}\,, \end{equation} and the left
anchor vanishes on ${\bar{\mathcal{E}}^0}$,
\begin{equation}\label{qan}
\rho({\bar{\mathcal{E}}^0})=0\,.
\end{equation}
Moreover, ${\bar{\mathcal{E}}^0}$ is a two-sided Loday ideal in
$\mathcal{E}$ and the Loday bracket induces on the $\mathcal{A}$-module
${\bar{\mathcal{E}}}=\mathcal{E}\slash{\bar{\mathcal{E}}^0}$  a Lie pseudoalgebra structure with
the anchor
\begin{equation}\label{qan1}\bar\rho([X])=\rho(X)\,, \end{equation} where $[X]$ denotes the
coset of $X$.
\end{cor}
\begin{proof}
The first statement follows directly from (\ref{fid0}). As $[\mathcal{E}^0,\mathcal{E}]=0$, the anchor vanishes on $\mathcal{E}^0$ and
thus on ${\bar{\mathcal{E}}^0}=\mathcal{A}\cdot\mathcal{E}^0$. From
$$[Z,f(X\diamond Y)]=f[Z,X\diamond Y]+\rho(X)(f)(X\diamond Y)\in{\bar{\mathcal{E}}^0}$$
and
$$[f(X\diamond Y),Z]=f[(X\diamond Y),Z]-\rho(Z)(f)(X\diamond Y)+\alpha(Z)(\mathrm{d} f\otimes(X\diamond Y))\in{\bar{\mathcal{E}}^0}\,,$$
we conclude that ${\bar{\mathcal{E}}^0}$ is a two-sided ideal. As ${\bar{\mathcal{E}}^0}$ contains all elements $X\diamond Y$, The Loday bracket
induces on $\mathcal{E}/{\bar{\mathcal{E}}^0}$ a skew-symmetric bracket with the anchor (\ref{qan1}) and satisfying the Jacobi
identity, thus a Lie pseudoalgebra structure.
\end{proof}
\begin{defi}
The Lie pseudoalgebra ${\bar{\mathcal{E}}}=\mathcal{E}/{\bar{\mathcal{E}}^0}$ we will call the {\it Lie pseudoalgebra of the Loday pseudoalgebra $\mathcal{E}$}.
If $\mathcal{E}=\mathrm{Sec}(E)$ is the Loday pseudoalgebra of a Loday algebroid on a vector bundle $E$ and the module ${\bar{\mathcal{E}}^0}$ is the module
of sections of a vector subbundle $\bar E$ of $E$, we deal with the {\it Lie algebroid of the Loday algebroid $E$}.
\end{defi}
\begin{ex} The Lie algebroid of the Courant-Dorfman bracket is the canonical Lie algebroid $ T M$.
\end{ex}
\begin{thm} For any Loday pseudoalgebra structure on an $\mathcal{A}$-module $\mathcal{E}$ there is a short exact sequence of morphisms of
Loday pseudoalgebras over $\mathcal{A}$,
\begin{equation}\label{es}
0\longrightarrow{\bar{\mathcal{E}}^0}\longrightarrow\mathcal{E}\longrightarrow\bar{\mathcal{E}}\longrightarrow 0\,, \end{equation} where ${\bar{\mathcal{E}}^0}$ -- the
$\mathcal{A}$-submodule in $\mathcal{E}$ generated by $\{[X,X]:X\in\mathcal{E}\}$ -- is a Loday pseudoalgebra with the trivial left
anchor and $\bar{\mathcal{E}}=\mathcal{E}/{\bar{\mathcal{E}}^0}$ is a Lie pseudoalgebra.
\end{thm}
Note that the Loday ideal $\mathcal{E}^0$ is clearly commutative, while the modular ideal ${\bar{\mathcal{E}}^0}$ is no longer
commutative in general.

\subsection{Loday algebroid cohomology}\label{GKPSection7}

We first recall the definition of the Loday cochain complex associated to a bi-module over a Loday algebra
\cite{LP}.\medskip

Let $\mathbb{K}$ be a field of nonzero characteristic and $V$ a $\mathbb{K}$-vector space endowed with a
(left) Loday bracket $[\cdot,\cdot]$. A {\it bimodule} over a Loday algebra $(V,[\cdot,\cdot])$ is a $\mathbb{K}$-vector space
$W$ together with a left (resp., right) {\it action} $\mu^l\in\mathrm{Hom}(V\otimes W,W)$ (resp.,
$\mu^r\in\mathrm{Hom}(W\otimes V,W)$) that verify the following requirements \begin{equation}\label{VVW}
\mu^r[x,y]=\mu^r(y)\mu^r(x)+\mu^l(x)\mu^r(y),\end{equation} \begin{equation}\label{WVV}
\mu^r[x,y]=\mu^l(x)\mu^r(y)-\mu^r(y)\mu^l(x),\end{equation} \begin{equation}\label{VWV}
\mu^l[x,y]=\mu^l(x)\mu^l(y)-\mu^l(y)\mu^l(x),\end{equation} for all $x,y\in V.$\medskip

The {\it Loday cochain complex} associated to the Loday algebra $(V,[\cdot,\cdot])$ and the bimodule $(W,\mu^l,\mu^r)$, shortly -- to $B=([\cdot,\cdot],\mu^r,\mu^l)$,
is made up by the cochain space $$\mathrm{Lin}^\bullet(V,W)=\bigoplus_{p\in \mathbb{N}}\, \mathrm{Lin}^p(V,W)=\bigoplus_{p\in
\mathbb{N}}\,\mathrm{Hom}(V^{\otimes p},W),$$ where we set $\mathrm{Lin}^0(V,W)=W$, and the coboundary operator $\partial_B$
defined, for any $p$-cochain $c$ and any vectors $x_1,\ldots,x_{p+1}\in V$, by
\begin{eqnarray}\nonumber (\partial_B
c)(x_1,\ldots,x_{p+1})&=&(-1)^{p+1}\mu^r(x_{p+1})c(x_1,\ldots,x_p)+\sum_{i=1}^p(-1)^{i+1}\mu^l(x_i)c(x_1,\ldots\hat{\imath}\ldots,x_{p+1})\\
&&+\sum_{i<j}(-1)^{i}c(x_1,\ldots\hat{\imath}\ldots,\stackrel{(j)}{\overbrace{[x_i,x_j]}},\ldots,x_{p+1})\;\;.\label{LodCohOp}
\end{eqnarray}

Let now $\rho$ be a {\it representation} of the Loday algebra $(V,[\cdot,\cdot])$ on a $\mathbb{K}$-vector space $W$,
i.e. a Loday algebra homomorphism $\rho: V\to \mathrm{End}(W)$. It is easily checked that $\mu^l:=\rho$ and
$\mu^r:=-\rho$ endow $W$ with a bimodule structure over $V$. Moreover, in this case of a bimodule induced by a
representation, the Loday cohomology operator reads

\begin{eqnarray}\label{LodCohOpRepr}(\partial_B
c)(x_1,\ldots,x_{p+1})&=&\sum_{i=1}^{p+1}(-1)^{i+1}\rho(x_i)c(x_1,\ldots\hat{\imath}\ldots,x_{p+1})\\
&&
+\sum_{i<j}(-1)^ic(x_1,\ldots\hat{\imath}\ldots,\stackrel{(j)}{\overbrace{[x_i,x_j]}},\ldots,x_{p+1})\;\;.\nonumber
\end{eqnarray}

Note that the above operator $\partial_B$ is well defined if only the {\it
map} $\rho:V\to \mathrm{End}(W)$ and the {\it bracket}
$[\cdot,\cdot]:V\otimes V\to V$ are given. We will refer to it as
to the {\it Loday operator} associated with $B=([\cdot,\cdot],\rho)$.
The point is that $\partial_B^2=0$ if and only if $[\cdot,\cdot]$ is a
Loday bracket and $\rho$ is its representation. Indeed, the Loday
algebra homomorphism property of $\rho$ (resp., the Jacobi identity
for $[\cdot,\cdot]$) is encoded in $\partial_B^2=0$ on
$\mathrm{Lin}^0(V,W)=W$ (resp., $\mathrm{Lin}^1(V,W)$), at least if
$W\neq \{0\}$, what we assume).\medskip

{
Let now $E$ be a vector bundle over a manifold $M$ and $B=([\cdot,\cdot],\rho)$ be an {\it anchored
faint algebroid} structure on $E$, where $[\cdot,\cdot]$ is a
faint pseudoalgebra bracket (bidifferential operator) and $\rho:E\to  T M$ is
a vector bundle morphism covering the identity, so inducing a
module morphism $\rho:\mathrm{Sec}(E)\to\mathrm{Der}(C^\infty(M))=\mathcal{X}(M)$.
It is easy to see that, unlike in the case of a Lie algebroid,
the tensor algebra of sections of $\oplus_{k=0}^\infty(E^*)^{\otimes k}$ is, in general, not
invariant under the Loday cohomology operator $\partial_B$ associated with
$B=([\cdot,\cdot],\rho)$. Actually, $\partial_B$ rises the degree of a multidifferential operator by one, even when the Loday bracket is skew-symmetric (see e.g. \cite{Ldd,Ldd1}).

\begin{ex} \cite{Ldd1}
Suppose that $M$ is a Riemannian manifold with metric tensor
$g$ and let $\partial_B$ be the Loday coboundary operator associated with the canonical
bracket of vector fields $B=([\cdot,\cdot]_{\text{vf}},\text{id}_{ T M})$ on $E= T M$. {When adopting the conventions of \cite{Ldd1}, where the Loday differential associated to right Loday algebras is considered, we then get,} for all $X,Y,Z\in\mathcal{X}(M)$,
\begin{equation}\label{LC}(\partial_B g)(X,Y,Z)=2g(Y,\nabla_XZ)\,,
\end{equation}
where $\nabla$ is the Levi-Civita connection on M. One can say that the Loday differential of a Riemannian metric defines the corresponding Levi-Civita connection, which clearly is no longer a tensor on $M$.
\end{ex}

The above observation suggests to consider in $\mathrm{Lin}^\bullet(\mathrm{Sec}(E),C^{\infty}(M))$, instead of {$\mathrm{Sec}(\otimes^{\bullet}E^*)$},
the subspace
$$\mathcal{D}^\bullet(\mathrm{Sec}(E),C^{\infty}(M))\subset\mathrm{Lin}^\bullet(\mathrm{Sec}(E),C^{\infty}(M))$$
consisting of all multidifferential operators. If now $B=([\cdot,\cdot],\rho)$ is an {\it anchored
{faint} algebroid} structure on $E${, see above}, then it is clear that the space
{$\mathcal{D}^\bullet(E):=\mathcal{D}^\bullet(\mathrm{Sec}(E),C^{\infty}(M))$} is stable under the
Loday operator $\partial_B$ associated with
$B=([\cdot,\cdot],\rho)$.}

\medskip
In particular, if $([\cdot,\cdot],\rho,\alpha)$ is a Loday algebroid structure on
$E$, its left anchor $\rho:\mathrm{Sec}(E)\to
\mathrm{Der}(C^{\infty}(M))\subset \mathrm{End}(C^{\infty}(M))$ is a
representation of the Loday algebra $(\mathrm{Sec}(E),[\cdot,\cdot])$
by derivations on $C^{\infty}(M)$ and $\partial_B^2=0$, so $\partial_B$ is a
coboundary operator.

\begin{defi} Let $(E,[\cdot,\cdot],\rho,\alpha)$ be a Loday algebroid over a manifold $M$. We call {\it
Loday algebroid cohomology}, the cohomology of the Loday cochain
subcomplex {$(\mathcal{D}^\bullet(E),\partial_B)$}
associated with $B=([\cdot,\cdot],\rho)$, i.e. the Loday algebra
structure $[\cdot,\cdot]$ on $\mathrm{Sec}(E)$ represented by $\rho$
on $C^{\infty}(M)$.\end{defi}

\subsection{Supercommutative geometric interpretation}\label{GKPSection8}

Let $E$ be a vector bundle over a manifold $M$.
{Looking for a canonical superalgebra structure in $\mathcal{D}^\bullet(E)$, a natural candidate is the {\it shuffle (super)product}, introduced by Eilenberg and Mac Lane \cite{EML}
(see also \cite{R1,R2}). It is known that a shuffle algebra on a free associative algebra is a free commutative algebra with the Lyndon words as its free generators \cite{Ra}. A similar result is valid in the supercommutative case \cite{ZM}. In this sense the free shuffle superalgebra represents a supercommutative space.
}
\begin{defi} For any $\ell'\in
\mathcal{D}^p(\mathrm{Sec}(E),C^{\infty}(M))$ and $\ell''\in \mathcal{D}^q(\mathrm{Sec}(E),C^{\infty}(M))$,
$p,q\in\mathbb{N}$, we define the {\it shuffle product}
$$(\ell'\pitchfork
\ell'')(X_1,\ldots,X_{p+q}):=\sum_{\sigma\in\mathrm{Sh}(p,q)}\mathrm{sign}\sigma\,\;
\ell'(X_{\sigma_1},\ldots,X_{\sigma_p})\;\ell''(X_{\sigma_{p+1}},\ldots,X_{\sigma_{p+q}}),$$ where the $X_i$-s denote
sections in $\mathrm{Sec}(E)$ and where $\mathrm{Sh}(p,q)\subset \mathbb{S}_{p+q}$ is the subset of the symmetric
group $\mathbb{S}_{p+q}$ made up by all $(p,q)$-shuffles.\end{defi}

The next proposition is well-known.

\begin{prop} The space {$\mathcal{D}^\bullet(E)$}, together with the shuffle multiplication $\pitchfork$, is a graded commutative associative unital $\mathbb{R}$-algebra.\end{prop}

We refer to this algebra as the {\it shuffle algebra of the vector bundle} $E\to M$, or simply, of $E$.

\medskip
Let $B=([\cdot,\cdot],\rho)$ be an anchored {faint} algebroid
structure on $E$ and let $\partial_B$ be the associated Loday operator
in $\mathcal{D}^\bullet(E)$. Note that we would have $\partial_B^2=0$ if we had assumed that we deal with a Loday
algebroid.\medskip

Denote now by $\mathrm{D}^k(E)$ those $k$-linear multidifferential
operators from $\mathcal{D}^k(E)$ which are of degree 0 with respect to
the last variable and of total degree $\le k-1$, and set
{$\mathrm{D}^\bullet(E)=\bigoplus_{k=0}^\infty\mathrm{D}^k(E)$}. By convention,
$\mathrm{D}^0(E)=\mathcal{D}^0(E)=C^\infty(M)$. Moreover,
$\mathrm{D}^1(E)=\mathrm{Sec}(E^\ast)$. It is easy to see that $\mathrm{D}^\bullet(E)$
is stable for the shuffle multiplication. We will call the
subalgebra $(\mathrm{D}^\bullet(E),\pitchfork)$, the {\it reduced shuffle
algebra}, and refer to the corresponding graded ringed space as
{\it supercommutative manifold}. Let us emphasize that this
denomination is in the present text merely a terminological
convention. The graded ringed spaces of the considered type are
being investigated in a separate work.

\begin{thm}\label{GeoIntKLA1} The coboundary operator $\partial_B$ is a degree 1
graded derivation of the shuffle algebra of $E$, i.e.
\begin{equation}\partial(\ell'\pitchfork\ell'')=(\partial \ell')\pitchfork \ell''+(-1)^p
\ell'\pitchfork (\partial\ell''),\label{DerShuffle}\end{equation} for any
$\ell'\in \mathcal{D}^p(E)$ and $\ell''\in \mathcal{D}^q(E)$. Moreover, if
$[\cdot,\cdot]$ is a pseudoalgebra bracket, i.e., if it is of
total order $\le 1$ and $\rho$ is the left anchor for
$[\cdot,\cdot]$, then $\partial_B$ leaves invariant the reduced shuffle
algebra $\mathrm{D}^\bullet(E)\subset\mathcal{D}^\bullet(E)$.\end{thm}

The claim is easily checked on low degree examples. The general proof is as follows.

\begin{proof} The value of the {\small LHS} of Equation (\ref{DerShuffle}) on sections $X_1,\ldots,X_{p+q+1}\in\mathrm{Sec}(E)$
is given by $S_1 +\ldots + S_4$, where
$$S_1=\sum_{k=1}^{p+1}\sum_{\tau\in\mathrm{Sh}(p,q)}(-1)^{k+1}\mathrm{sign}\tau\; \rho(X_k)\left(\ell'(X_{\tau_1},\ldots,
\widehat{X}_{\tau_k},\ldots, X_{\tau_{p+1}})\;\ell''(X_{\tau_{p+2}},\ldots, X_{\tau_{p+q+1}})\right)$$
and
$$S_3=\sum_{1\le k< m\le p+q+1}\sum_{\tau\in \mathrm{Sh}(p,q)}(-1)^k\mathrm{sign}\tau\;\,\ell'(X_{\tau_1},\ldots,[X_k,X_m],\ldots)
\;\ell''(X_{\tau_{-}},\ldots).$$
In the sum $S_2$, which is similar to $S_1$, the index $k$ runs through $\{p+2,\ldots,p+q+1\}$ ($X_{\tau_k}$ is
then missing in $\ell''$). The sum $S_3$ contains those shuffle permutations of $1\ldots \hat{k}\ldots p+q+1$
that send the argument $[X_k,X_m]$ with index $m=:\tau_r$ into $\ell'$, whereas $S_4$ is taken over the shuffle
permutations that send $[X_k,X_m]$ into $\ell''.$\medskip

Analogously, the value of $(\partial\ell')\pitchfork\ell''$ equals $T_1+T_2$ with
$$T_1=\sum_{\sigma\in\mathrm{Sh}(p+1,q)}\sum_{i=1}^{p+1}\mathrm{sign}\sigma\,(-1)^{i+1}\left(\rho(X_{\sigma_i})\,\ell'(X_{\sigma_1},
\ldots,\widehat{X}_{\sigma_i},\ldots,X_{\sigma_{p+1}})\right)\ell''(X_{\sigma_{p+2}},\ldots,X_{\sigma_{p+q+1}})$$
and
$$T_2=\sum_{\sigma\in\mathrm{Sh}(p+1,q)}\sum_{1\le i<j\le p+1}\mathrm{sign}\sigma\,(-1)^{i}\,\ell'(X_{\sigma_1},\ldots,[X_{\sigma_i},
X_{\sigma_j}],\ldots)\;\ell''(X_{\sigma_{p+2}},\ldots,X_{\sigma_{p+q+1}})$$
(whereas the value $T_3+T_4$ of $(-1)^p\;\ell'\pitchfork(\partial\ell'')$, which is similar, is not (really) needed
in this (sketch of) proof).\medskip

Let us stress that in $S_3$ and $T_2$ the bracket is in its natural position determined by the index $\tau_r=m$
or $\sigma_j$ of its second argument, that, since $\mathrm{Sh}(p,q)\simeq \mathbb{S}_{p+q}/(\mathbb{S}_p\times
\mathbb{S}_q)$, the number of $(p,q)$-shuffles equals ${(p+q)!}/{(p!\,q!)}\,$, and that in $S_1$ the vector field
$\rho(X_k)$ acts on a product of functions according to the Leibniz rule, so that each term splits. It is now
easily checked that after this splitting the number of different terms in $\rho(X_{-})$ (resp. $[X_-,X_-]$) in
the {\small LHS} and the {\small RHS} of Equation (\ref{DerShuffle}) is equal to $2 (p+q+1)!/(p!\,q!)$ (resp.
$(p+q)(p+q+1)!/(2\,p!\,q!)$). To prove that both sides coincide, it therefore suffices to show that any term of the
{\small LHS} can be found in the {\small RHS}.\medskip

We first check this for any split term of $S_1$ with vector field action on the value of $\ell'$ (the proof is
similar if the field acts on the second function and also if we choose a split term in $S_2$),
$$(-1)^{k+1}\mathrm{sign}\tau\; \left(\rho(X_k)\ell'(X_{\tau_1},\ldots, \widehat{X}_{\tau_k},\ldots,
X_{\tau_{p+1}})\right)\,\ell''(X_{\tau_{p+2}},\ldots, X_{\tau_{p+q+1}}),$$ where $k\in\{1,\ldots,p+1\}$ is fixed,
as well as $\tau\in\mathrm{Sh(p,q)}$ -- which permutes $1\ldots\hat{k}\ldots p+q+1$. This term exists also in
$T_1$. Indeed, the shuffle $\tau$ induces a unique shuffle $\sigma\in\mathrm{Sh}(p+1,q)$ and a unique
$i\in\{1,\ldots,p+1\}$ such that $\sigma_i=k.$ The corresponding term of $T_1$ then coincides with the chosen
term in $S_1$, since, as easily seen, $\mathrm{sign}\sigma\;(-1)^{i+1}=(-1)^{k+1}\mathrm{sign}\tau$.\medskip

Consider now a term in $S_3$ (the proof is analogous for the terms of $S_4$),
$$(-1)^k\mathrm{sign}\tau\;\,\ell'(X_{\tau_1},\ldots,[X_k,X_m],\ldots)\;\ell''(X_{\tau_{-}},\ldots),$$
where $k<m$ are fixed in $\{1,\ldots, p+q+1\}$ and where $\tau\in \mathrm{Sh}(p,q)$ is a fixed permutation of
$1\ldots\hat{k}\ldots p+q+1$ such that the section $[X_k,X_m]$ with index $m=:\tau_r$ is an argument of
$\ell'$. The shuffle $\tau$ induces a unique shuffle $\sigma\in\mathrm{Sh}(p+1,q)$. Set $k=:\sigma_i$ and $m=:\sigma_j$. Of
course $1\le i<j\le p+1$. This means that the chosen term reads
$$(-1)^k\mathrm{sign}\tau\;\,\ell'(X_{\sigma_1},\ldots,[X_{\sigma_i},X_{\sigma_j}],\ldots,X_{\sigma_{p+1}})\;\ell''(X_{\sigma_{p+2}},\ldots,X_{\sigma_{p+q+1}}).$$
Finally this term is a term of $T_2$, as it is again clear that
$(-1)^k\mathrm{sign}\tau=\mathrm{sign}\sigma\,(-1)^i$.\medskip

That $\mathrm{D}^\bullet(E)$ is invariant under $\partial_B$ in the case of a
pseudoalgebra bracket is obvious. This completes the
proof.\end{proof}

Note that the derivations $\partial_B$ of the reduced shuffle algebra
(in the case of pseudoalgebra brackets on $\mathrm{Sec}(E)$) are, due to
formula (\ref{LodCohOpRepr}), completely determined by their
values on $\mathrm{D}^0(E)\oplus\mathrm{D}^1(E)$. More precisely,
$B=([\cdot,\cdot],\rho)$ can be easily reconstructed from $\partial_B$
thanks to the formulae \begin{equation}\label{anchor-reconstruction}
\rho(X)(f)=\langle X,\partial_Bf\rangle \end{equation} and
\begin{equation}\label{bracket-reconstruction} \langle\mathfrak{l},[X,Y]\rangle=\langle
X,\partial_B\langle\mathfrak{l},Y\rangle\rangle-\langle
Y,\partial_B\langle\mathfrak{l},X\rangle\rangle-\partial_B\mathfrak{l}(X,Y)\,, \end{equation}
where $X,Y\in\mathrm{Sec}(E)$, $\mathfrak{l}\in\mathrm{Sec}(E^\ast)$, and $f\in
C^\infty(M)$.

\begin{thm}\label{th:LO} If $\partial$ is a derivation of the reduced shuffle algebra $\mathrm{D}^\bullet(E)$, then on $\mathrm{D}^0(E)\oplus\mathrm{D}^1(E)$ the
derivation $\partial$ coincides with $\partial_B$ for a certain uniquely determined
$B=([\cdot,\cdot]_\partial,\rho_\partial)$ associated with a pseudoalgebra
bracket $[\cdot,\cdot]_\partial$ on $\mathrm{Sec}(E)$.
\end{thm}
\begin{proof}
Let us define $\rho=\rho_\partial$ and $[\cdot,\cdot]=[\cdot,\cdot]_\partial$ out of formulae
(\ref{anchor-reconstruction}) and (\ref{bracket-reconstruction}),
i.e., \begin{equation}\label{anchor-reconstruction1} \rho(X)(f)=\langle X,\partial
f\rangle \end{equation} and \begin{equation}\label{bracket-reconstruction1}
\langle\mathfrak{l},[X,Y]\rangle=\langle
X,\partial\langle\mathfrak{l},Y\rangle\rangle-\langle
Y,\partial\langle\mathfrak{l},X\rangle\rangle-\partial\mathfrak{l}(X,Y)\,. \end{equation} The
fact that $\rho(X)$ is a derivation of $C^{\infty}(M)$ is a
direct consequence of the shuffle algebra derivation property of
$\partial$. Eventually, the map $\rho$ is visibly associated with
a bundle map $\rho:E\to TM$.\smallskip

The bracket $[\cdot,\cdot]$ has $\rho$ as left anchor.
Indeed, since $\partial\mathfrak{l}(X,Y)$ is of order 0 with respect to
$Y$, we get from (\ref{bracket-reconstruction1})
$$[X,fY]-f[X,Y]= \langle X,\partial f\rangle Y=\rho(X)(f)Y\,.$$
Similarly, as $\partial\mathfrak{l}(X,Y)$ is of order 1 with respect to $X$ and of order 0 with respect to $Y$, the operator
$$\delta_1(f)\left(\partial\mathfrak{l}\right)(X,Y)=\partial\mathfrak{l}(fX,Y)-f\partial\mathfrak{l}(X,Y)$$
is $C^\infty(M)$-bilinear, so that the LHS of
$$\langle\mathfrak{l},[fX,Y]-f[X,Y]\rangle=-\langle Y,\partial f\rangle\langle\mathfrak{l},X\rangle -\delta_1(f)\left(\partial\mathfrak{l}\right)(X,Y),$$ see (\ref{bracket-reconstruction1}),
is $C^\infty(M)$-linear with respect to $X$ and $Y$ and a
derivation with respect to $f$. The bracket $[\cdot,\cdot]$ is
therefore of total order $\le 1$ with the generalized right anchor
$b\,^r=\rho-\alpha$, where $\alpha$ is determined by
the identity \begin{equation}\label{za}\langle\mathfrak{l},\alpha(Y)(\mathrm{d} f\otimes
X)\rangle=\delta_1(f)\left(\partial\mathfrak{l}\right)(X,Y)\,. \end{equation} This
corroborates that $\alpha$ is a bundle map from $E$ to
$TM\otimes_M\mathrm{End}(E)$.\end{proof}

\begin{defi}
Let ${\mathrm{Der}}_1(\mathrm{D}^\bullet(E),\pitchfork)$ be the space of degree $1$ graded
derivations $\partial$ of the reduced shuffle algebra that verify, for any $c\in \mathrm{D}^2(E)$ and any
$X_i\in\mathrm{Sec}(E)$, $i=1,2,3$,
\begin{eqnarray}\label{EncodingJacobi}(\partial c)(X_1,X_2,X_3)&=&\sum_{i=1}^3(-1)^{i+1}\langle\partial
(c(X_1,\ldots\hat{\imath}\ldots,X_3)),X_i\rangle\\
&&+\sum_{i<j}(-1)^i\,c(X_1,\ldots\hat{\imath}\ldots,\stackrel{(j)}
{[X_i,X_j]_\partial},\ldots,X_3)\,.\nonumber
\end{eqnarray}
A {\it homological vector field} of the supercommutative manifold $(M,\mathrm{D}^\bullet(E))$ is a square-zero
derivation in $\mathrm{{Der}}_1(\mathrm{D}^\bullet(E),\pitchfork)$. Two homological vector fields of $(M,\mathrm{D}^\bullet(E))$ are {\it equivalent}, if they coincide on $C^{\infty}(M)$ and on
$\mathrm{Sec}(E^*)$.
\end{defi}

\medskip
Observe that Equation (\ref{EncodingJacobi}) implies that two equivalent homological fields also coincide on
$\mathrm{D}^2(E)$.
We are now prepared to give the main theorem of this section.

\begin{thm}\label{GeoIntKLA2} Let $E$ be a vector bundle. There exists a 1-to-1 correspondence between equivalence classes of homological vector
fields
$$\partial\in\mathrm{{Der}}_1(\mathrm{D}^\bullet(E),\pitchfork),\;\partial^2=0$$
and Loday algebroid structures on $E$.\end{thm}

\begin{rem} This theorem is a kind of a non-antisymmetric counterpart of
the well-known similar correspondence between homological vector fields of split supermanifolds and Lie
algebroids. Furthermore, it may be viewed as an analogue for Loday algebroids of the celebrated
Ginzburg-Kapranov correspondence for quadratic Koszul operads \cite{GK94}. According to the latter result,
homotopy Loday structures on a graded vector space $V$ correspond bijectively to degree 1 differentials of the
Zinbiel algebra $(\bar{\otimes}sV^*,\star)$, where $s$ is the suspension operator and where
$\bar{\otimes}sV^*$ denotes the reduced tensor module over $sV^*.$ However, in our geometric setting scalars,
or better functions, must be incorporated (see the proof of Theorem \ref{GeoIntKLA2}), which turns out to be
impossible without passing from the Zinbiel multiplication or half shuffle $\star$ to its symmetrization
$\pitchfork$. Moreover, it is clear that the algebraic structure on the function sheaf should be
associative.\end{rem}

\begin{proof} Let $([\cdot,\cdot],\rho,\alpha)$ be a Loday algebroid
structure on the given vector bundle $E\to M.$ According to
Theorem \ref{GeoIntKLA1}, the corresponding coboundary operator
$\partial_B$ is a square 0 degree 1 graded derivation of the
reduced shuffle algebra and (\ref{EncodingJacobi}) is satisfied by
definition, as $[\cdot,\cdot]_{\partial_B}=[\cdot,\cdot]$.\medskip

Conversely, let $\partial$ be such a homological vector field.
According to Theorem \ref{th:LO}, the derivation $\partial$ coincides
on $\mathrm{D}^0(E)\oplus\mathrm{D}^1(E)$ with $\partial_B$ for a certain
pseudoalgebra bracket $[\cdot,\cdot]=[\cdot,\cdot]_\partial$ on
$\mathrm{Sec}(E)$. Its left anchor is $\rho=\rho_\partial$ and the generalized
right anchor $b^r=\rho-\alpha$ is determined by means of formula
(\ref{za}), where $\mathfrak{l}$ runs through all sections of
$E^\ast$.\medskip

To prove that the triplet $([\cdot,\cdot],\rho,\alpha)$ defines a Loday algebroid structure on $E,$ it now suffices to
check that the Jacobi identity holds true. It follows from (\ref{bracket-reconstruction1}) that
$$\langle\mathfrak{l},[X_1,[X_2,X_3]]\rangle=-\langle\partial\langle \mathfrak{l},X_1\rangle, [X_2,X_3]\rangle+\langle\partial\langle
\mathfrak{l},[X_2,X_3]\rangle, X_1\rangle-(\partial\mathfrak{l})(X_1,[X_2,X_3]).$$ Since the first term of the
{\small RHS} is (up to sign) the evaluation of $[X_2,X_3]$ on the section $\partial\langle \mathfrak{l},X_1\rangle$ of
$E^*$, and a similar remark is valid for the contraction $\langle \mathfrak{l},[X_2,X_3]\rangle$ in the second
term, we can apply (\ref{bracket-reconstruction1}) also to these two brackets. If we proceed analogously for
$[[X_1,X_2],X_3]$ and $[X_2,[X_1,X_3]]$, and use (\ref{anchor-reconstruction1}) and the homological property
$\partial^2=0$, we find, after simplification, that the sum of the preceding three double brackets equals
$$\sum_{i=1}^3(-1)^{i+1}
\rho(X_i)(\partial\mathfrak{l})(X_1,\ldots\hat{\imath}\ldots,X_3)+\sum_{i<j}(-1)^i\,(\partial\mathfrak{l})(X_1,\ldots\hat{\imath}\ldots,\stackrel{(j)}
{\overbrace{[X_i,X_j]}},\ldots,X_3)\;.
$$
In view of (\ref{EncodingJacobi}), the latter expression coincides with $(\partial^2\mathfrak{l})(X_1,X_2,X_3)=0$, so that
the Jacobi identity holds.\medskip

It is clear that the just detailed assignment of a Loday algebroid
structure to any homological vector field can be viewed as a map
on equivalence classes of homological vector fields.\end{proof}

Having a homological vector field $\partial$ associated with a Loday
algebroid structure $([\cdot,\cdot],\rho,\alpha)$ on $E$, we can
easily develop the corresponding Cartan calculus for the shuffle
algebra $\mathcal{D}^\bullet(E)$.

\begin{prop} For any $X\in\mathrm{Sec}(E)$, the contraction
$$\mathcal{D}^p(E)\ni\ell\mapsto i_X\ell\in\mathcal{D}^{p-1}(E)\,,\quad (i_X\ell)(X_1,\dots,X_{p-1})=\ell(X,X_1,\dots,X_{p-1})\,,$$
is a degree $-1$ graded derivation of the shuffle algebra $({\cal
D}^{\bullet}(E),\pitchfork)$.\end{prop}

\begin{proof} Using usual notations, our definitions, as well as a separation of
the involved shuffles $\sigma$ into the $\sigma$-s that verify
${\sigma_1}=1$ and those for which ${\sigma_{p+1}=1}$, we get
$$\left(i_{X_1}(\ell'\pitchfork\ell'')\right)(X_2,\ldots,X_{p+q})= \sum_{\sigma:\sigma_1=1}\mathrm{sign}\sigma\;(i_{X_1}\ell')(X_{\sigma_2},
\ldots,X_{\sigma_{p}}) \ell''(X_{\sigma_{p+1}},\ldots,X_{\sigma_{p+q}})$$
$$+\sum_{\sigma:\sigma_{p+1}=1}\mathrm{sign}\sigma\;\ell'(X_{\sigma_1},\ldots,X_{\sigma_{p}})
(i_{X_1}\ell'')(X_{\sigma_{p+2}},\ldots,X_{\sigma_{p+q}}).$$ Whereas a
$(p,q)$-shuffle of the type $\sigma_1=1$ is a $(p-1,q)$-shuffle with
same signature, a $(p,q)$-shuffle such that $\sigma_{p+1}=1$ defines
a $(p,q-1)$-shuffle with signature $(-1)^p\mathrm{sign}\sigma$.
Therefore, we finally get
$$i_{X_1}(\ell'\pitchfork\ell'')=(i_{X_1}\ell')\pitchfork
\ell''+(-1)^p\ell'\pitchfork(i_{X_1}\ell'').$$\end{proof}

Observe that the supercommutators
$[i_X,i_Y]_{\mathrm{sc}}=i_Xi_Y+i_Yi_X$ do not necessarily vanish, so
that the derivations $i_X$ of the shuffle algebra generate a Lie
superalgebra of derivations with negative degrees. Indeed,
$[i_X,i_Y]_{\mathrm{sc}}=:i_{X\Box Y}$,
$[[i_X,i_Y]_{\mathrm{sc}},i_Z]_{\mathrm{sc}}=:i_{(X\Box Y)\Box Z},...$
are derivations of degree $-2$, $-3,...$ given on any
$\ell\in{\cal D}^p(E)$ by
$$(i_{X\Box Y}\ell)(X_1,\dots,X_{p-2})=\ell(Y,X,X_1,\dots,X_{p-2})+\ell(X,Y,X_1,\dots,X_{p-2})\,,$$
$$(i_{(X\Box X)\Box
Y}\ell)(X_1,\dots,X_{p-3})=2\ell(Y,X,X,X_1,\dots,X_{p-3})-2\ell(X,X,Y,X_1,\dots,X_{p-3})\,,...$$\smallskip

The next proposition is obvious.

\begin{prop} The supercommutator $\mathcal{L}_X:=[\partial,i_X]_{\mathrm{sc}}=\partial
i_X+i_X\partial$, $X\in\mathrm{Sec}(E)$, is a degree 0 graded derivation of the
shuffle algebra. Explicitly, for any $\ell\in\mathcal{D}^p(E)$ and
$X_1,\ldots,X_p\in\mathrm{Sec}(E)$, \begin{equation}\label{Ld}
(\mathcal{L}_X\ell)(X_1,\dots,X_p)=\rho(X)\left(\ell(X_1,\dots,X_p)\right)-\sum_i\ell(X_1,\dots,\stackrel{(i)}
{\overbrace{[X,X_i]}},\dots,X_p)\,.\end{equation}\end{prop}

We refer to the derivation $\mathcal{L}_X$ as the Loday algebroid {\it Lie
derivative along $X$}.\medskip

If we define the Lie derivative on the tensor algebra
$T_\mathbb{R}(E)=\bigoplus_{p=0}^\infty\mathrm{Sec}(E)^{\otimes_\mathbb{R} p}$ in the obvious
way by
$$\mathcal{L}_X(X_1\otimes_\mathbb{R}\cdots\otimes_\mathbb{R} X_p)=\sum_iX_1\otimes_\mathbb{R}\dots\otimes_\mathbb{R}\stackrel{(i)}
{\overbrace{[X,X_i]}}\otimes_\mathbb{R}\dots\otimes_\mathbb{R} X_p\,,$$ and if we use the
canonical pairing
$$\langle\ell,X_1\otimes_\mathbb{R}\dots\otimes_\mathbb{R} X_p\rangle=\ell(X_1,\dots,X_p)$$ between $\mathcal{D}^\bullet(E)$ and
$T_\mathbb{R}(E)$, we get \begin{equation}
\label{Ld1} \mathcal{L}_X\langle\ell,X_1\otimes_\mathbb{R}\dots\otimes_\mathbb{R}
X_p\rangle=\langle\mathcal{L}_X\ell,X_1\otimes_\mathbb{R}\dots\otimes_\mathbb{R}
X_p\rangle+\langle\ell,\mathcal{L}_X(X_1\otimes_\mathbb{R}\dots\otimes_\mathbb{R}
 X_p)\rangle\,.\end{equation}

The following theorem is analogous to the results in the standard
case of a Lie algebroid $E= T M$ and operations on the Grassmann
algebra $\Omega(M)\subset \mathcal{D}^\bullet( T M)$ of differential forms.

\begin{thm} The graded derivations $\partial$, $i_X$, and $\mathcal{L}_X$ on $\mathcal{D}^\bullet(E)$ satisfy the following identities:
\begin{itemize}
\item[(a)] $2\partial^2=[\partial,\partial]_{\mathrm{sc}}=0$\;, \item[(b)]
$\mathcal{L}_X=[\partial,i_X]_{\mathrm{sc}}=\partial i_X+i_X\partial$\;, \item[(c)]
$\partial\mathcal{L}_X-\mathcal{L}_X\partial=[\partial,\mathcal{L}_X]_{\mathrm{sc}}=0$\;, \item[(d)]
$\mathcal{L}_Xi_Y-i_Y\mathcal{L}_X=[\mathcal{L}_X,i_Y]_{\mathrm{sc}}=i_{[X,Y]}$\;, \item[(e)]
$\mathcal{L}_X\mathcal{L}_Y-\mathcal{L}_Y\mathcal{L}_X=[\mathcal{L}_X,\mathcal{L}_Y]_{\mathrm{sc}}=\mathcal{L}_{[X,Y]}$ .
\end{itemize}
\end{thm}\smallskip

\begin{proof} The results (a), (b), and (c) are obvious. Identity (d) is immediately checked by direct computation. The last
equality is a consequence of (c), (d), and the Jacobi identity
applied to $[\mathcal{L}_X,[\partial,i_Y]_{\mathrm{sc}}]_{\mathrm{sc}}$.\end{proof}

Note that we can easily calculate the Lie derivatives of negative
degrees, $\mathcal{L}_{X\Box Y}:=[\partial,i_{X\Box Y}]_{\mathrm{sc}}$,
$\mathcal{L}_{(X\Box Y)\Box Z}:=[\partial,i_{(X\Box Y)\Box Z}]_{\mathrm{sc}}$, ...
with the help of the graded Jacobi identity. \medskip

Observe finally that Item (d) of the preceding theorem actually
means that
$$i_{[X,Y]}=[\![i_X,i_Y]\!]_{\partial},$$ where the {\small RHS} is the restriction to interior products of the derived bracket on
$\mathrm{Der}(\mathcal{D}^{\bullet}(E),\pitchfork\nolinebreak)$ defined by
the graded Lie bracket $[\cdot,\cdot]_{\mathrm{sc}}$ and the interior
Lie algebra derivation $[\partial,\cdot]_{\mathrm{sc}}$ of
$\mathrm{Der}(\mathcal{D}^{\bullet}(E),\pitchfork)$ induced by the
homological vector field $\partial.$

\newpage
\section{On the infinity category of homotopy Leibniz algebras}\label{InfCatHomLeibAlg}
The following research work is a cooperation with Prof. Dr. Norbert Poncin and Dr. Jian Qiu, which is being published in the ArXiv preprint database and submitted for publication in a peer-reviewed international journal.
\subsection{Introduction}
\subsubsection{General background}

Homotopy, sh, or infinity algebras \cite{Sta63} are homotopy invariant extensions of differential graded algebras. They are of importance, e.g. in {\small BRST} of closed string field theory, in Deformation Quantization of Poisson manifolds ... Another technique to increase the flexibility of algebraic structures is categorification \cite{CF94}, \cite{Cra95} -- a sharpened viewpoint that leads to astonishing results in {\small TFT}, bosonic string theory ... Both methods, homotopification and categorificiation are tightly related: the 2-categories of 2-term Lie (resp., Leibniz) infinity algebras and of Lie (resp., Leibniz) 2-algebras turned out to be equivalent \cite{BC04}, \cite{SL10} (for a comparison of 3-term Lie infinity algebras and Lie 3-algebras, as well as for the categorical definition of the latter, see \cite{KMP11}). However, homotopies of $\infty$-morphisms and their compositions are far from being fully understood. In \cite{BC04}, $\infty$-homotopies are obtained from categorical homotopies, which are God-given. In \cite{SS07Structure}, (higher) $\infty$-homotopies are (higher) derivation homotopies, a variant of infinitesimal concordances, which seems to be the wrong concept \cite{DP12}. In \cite{BS07}, the author states that $\infty$-homotopies of sh Lie algebra morphisms can be composed, but no proof is given and the result is actually not true in whole generality. The objective of this work is to clarify the concept of (higher) $\infty$-homotopies, as well as the problem of their compositions.

\subsubsection{Structure and main results}

In Section~\ref{KPQSection2}, we provide explicit formulae for Leibniz infinity algebras and their morphisms. Indeed, although a category of homotopy algebras is simplest described as a category of quasi-free {\small DG} coalgebras, its original nature is its manifestation in terms of brackets and component maps.\medskip

We report, in Section~\ref{KPQSection3}, on the notions of homotopy that are relevant for our purposes: concordances, i.e. homotopies for morphisms between quasi-free {\small DG} (co)algebras, gauge and Quillen homotopies for Maurer-Cartan ({\small MC} for short) elements of pronilpotent Lie infinity algebras, and $\infty$-homotopies, i.e. gauge or Quillen homotopies for $\infty$-morphisms viewed as {\small MC} elements of a complete convolution Lie infinity algebra.\medskip

Section~\ref{KPQSection4} starts with the observation that vertical composition of $\infty$-homotopies of {\small DG} algebras is well-defined. However, this composition is not associative and cannot be extended to the $\infty$-algebra case -- which suggests that $\infty$-algebras actually form an $\infty$-category. To allow independent reading of the present paper, we provide a short introduction to $\infty$-categories, see Subsection \ref{InftyCatIntro}. In Subsection \ref{InftyCatInftyAlg}, the concept of $\infty$-$n$-homotopy is made precise and the class of $\infty$-algebras is viewed as an $\infty$-category. Since we apply the proof of the Kan property of the nerve of a nilpotent Lie infinity algebra to the 2-term Leibniz infinity case, a good understanding of this proof is indispensable: we detail the latter in Subsection \ref{Kan property}.\medskip

To be complete, we give an explicit description of the category of 2-term Leibniz infinity algebras at the beginning of Section~\ref{KPQSection5}. We show that composition of $\infty$-homotopies in the nerve-$\infty$-groupoid, which is defined and associative only up to higher $\infty$-homotopy, projects to a well-defined and associative vertical composition in the 2-term case -- thus obtaining the Leibniz counterpart of the strict 2-category of 2-term Lie infinity algebras \cite{BC04}, see Subsection \ref{KanHomComp}, Theorem \ref{KanHom1} and Theorem \ref{KomComp2}.\medskip

Eventually, we provide, in Section~\ref{KPQSection6}, the definitions of the strict 2-category of Leibniz 2-algebras, which is 2-equivalent to the preceding 2-category.\medskip

An $\infty$-category structure on the class of $\infty$-algebras over a quadratic Koszul operad is being investigated independently of \cite{Get09} in a separate paper.


\subsection{Category of Leibniz infinity algebras}\label{KPQSection2}

Let $P$ be a quadratic Koszul operad. Surprisingly enough, $P_{\infty}$-structures on a graded vector space $V$ (over a field $\mathbb{K}$ of characteristic zero), which are essentially sequences $\ell_n$ of $n$-ary brackets on $V$ that verify a sequence $R_n$ of defining relations, $n\in\{1,2,\ldots\}$, are 1:1 \cite{GK94} with codifferentials
\begin{equation}\label{GKCoalg}D\in\mathrm{CoDer}^{1}({\cal F}^{\mathrm{gr,c}}_{P^{\text{!`}}}(s^{-1}V))\quad (|\ell_n|=2-n)\quad \text{
or }\quad D\in\mathrm{CoDer}^{-1}({\cal F}^{\mathrm{gr,c}}_{P^{\text{!`}}}(sV))\quad (|\ell_n|=n-2)\;,\end{equation} or, also, (if $V$ is finite-dimensional) 1:1 with differentials \begin{equation}\label{GKAlgFin}d\in
\mathrm{Der}^1({\cal F}^{\mathrm{gr}}_{P^{!}}(sV^*))\quad(|\ell_n|=2-n)\quad \text{ or }\quad d\in \mathrm{Der}^{-1}({\cal F}^{\mathrm{gr}}_{P^{!}}(s^{-1}V^*))\quad(|\ell_n|=n-2)\;.\end{equation} Here $\mathrm{Der}^1({\cal F}^{\mathrm{gr}}_{P^{!}}(sV^*))$ (resp., $\mathrm{CoDer}^{1}({\cal F}^{\mathrm{gr,c}}_{P^{\text{!`}}}(s^{-1}V))$), for instance, denotes the space of endomorphisms of the free graded algebra over the Koszul dual operad $P^{!}$ of $P$ on the suspended linear dual $sV^*$ of $V$, which have degree 1 (with respect to the grading of the free algebra that is induced by the grading of $V$) and are derivations for each binary operation in $P^{!}$ (resp., the space of endomorphisms of the free graded coalgebra over the Koszul dual cooperad $P^{\text{!`}}$ on the desuspended space $s^{-1}V$ that are coderivations) (by differential and codifferential we mean of course a derivation or coderivation that squares to 0).\medskip

Although the original nature of homotopified or oidified algebraic objects is their manifestation in terms of brackets \cite{BP12}, the preceding coalgebraic and algebraic settings are the most convenient contexts to think about such higher structures.

\subsubsection{Zinbiel (co)algebras}

Since we take an interest mainly in the case where $P$ is the operad $\text{\sf Lei}$ (resp., the operad $\text{\sf Lie}$) of Leibniz (resp., Lie) algebras, the Koszul dual $P^{!}$ to consider is the operad $\text{\sf Zin}$ (resp., $\text{\sf Com}$) of Zinbiel (resp., commutative) algebras. We now recall the relevant definitions and results.

\begin{defi} A {\em graded Zinbiel algebra} $(${\em\small GZA}$)$ $($resp., {\em graded Zinbiel coalgebra} $($\em{\small GZC}$)$$)$ is a $\mathds{Z}$-graded vector space $V$ endowed with a multiplication, i.e. a degree 0 linear map $m:V\otimes V\to V$ $($resp., a comultiplication, i.e. a degree 0 linear map $\Delta:V\to V\otimes V$$)$ that verifies the relation
\begin{equation}
m(\mathrm{id}\otimes m)=m(m\otimes \mathrm{id})+m(m\otimes \mathrm{id})(\tau\otimes \mathrm{id})\quad (\text{resp.,}\; (\mathrm{id}\otimes \Delta)\Delta=(\Delta\otimes\mathrm{id})\Delta+(\tau\otimes\mathrm{id})(\Delta\otimes\mathrm{id})\Delta)\;,
\end{equation} where $\tau:V\otimes V\ni u\otimes v\mapsto (-1)^{|u||v|}v\otimes u\in V\otimes V$.
\end{defi}

Of course, when evaluated on homogeneous vectors $u,v,w\in V$, the Zinbiel relation for the multiplication $m(u,v)=:u\cdot v$ reads, $$ u\cdot(v\cdot w)=(u\cdot v)\cdot w+(-1)^{|u||v|}(v\cdot u)\cdot w\;.$$

\begin{ex} \label{propFreeZinbielProducFormulaOnTVarticle}
The multiplication $\cdot$ on the reduced tensor module $\overline{T}(V):=\oplus_{n\ge 1} V^{\otimes n}$ over a $\mathbb{Z}$-graded vector space $V$, defined, for homogeneous $v_i\in V$, by
\begin{equation}
\begin{array}{l}
(v_1...v_p)\cdot(v_{p+1}...v_{p+q})=\sum\limits_{\sigma\in \mathrm{Sh}(p,q-1)}(\sigma^{-1}\otimes \mathrm{id})(v_1...v_{p+q})=\\
=\sum\limits_{\sigma\in \mathrm{Sh}(p,q-1)}\varepsilon(\sigma^{-1})v_{\sigma^{-1}(1)}v_{\sigma^{-1}(2)}... v_{\sigma^{-1}(p+q-1)}v_{p+q}\;,\\
\end{array}\;
\end{equation}
where we wrote tensor products of vectors by simple juxtaposition, where $\mathrm{Sh}(p,q-1)$ is the set of $(p,q-1)$-shuffles, and where $\varepsilon(\sigma^{-1})$ is the Koszul sign, endows $\overline{T}(V)$ with a {\em\small GZA} structure.\medskip

Similarly, the comultiplication $\Delta$ on $\overline{T}(V)$, defined, for homogeneous $v_i\in V$, by
\begin{equation}\label{FreeZinbielCoProducFormulaOnTVarticle}
\Delta(v_1...v_p)=\sum_{k=1}^{p-1}\sum\limits_{\sigma\in \mathrm{Sh}(k,p-k-1)}\varepsilon(\sigma)\left( v_{\sigma(1)}... v_{\sigma(k)}\right)\bigotimes\left( v_{\sigma(k+1)}... v_{\sigma(p-k-1)}v_{p}\right)\;,
\end{equation} is a {\em\small GZC} structure on $\overline{T}(V)$.
\end{ex}
\noindent As for the {\small GZA} multiplication on $\overline{T}(V),$ we have in particular
$$ v_1\cdot v_2=v_1 v_2\;;\quad (v_1v_2)\cdot v_3=v_1v_2v_3\;;$$ $$\quad v_1\cdot(v_2v_3)=v_1v_2v_3+(-1)^{|v_1||v_2|}v_2v_1v_3;\;\quad
(((v_1\cdot v_2)\cdot v_3)...)\cdot v_k=v_1v_2...v_k\;.$$

\begin{prop}
 The {\em\small GZA} $(\overline{T}(V),\cdot)$ $($resp., the {\em\small GZC} $(\overline{T}(V),\Delta)$$)$ defined in Example~\ref{propFreeZinbielProducFormulaOnTVarticle} is the {{\em free} {\small \em GZA} $($resp., {\em free} {\small \em GZC}$)$ over $V$}. We will denote it by $\mathrm{Zin}(V)$ $($resp., $\mathrm{Zin}^{\mathrm{c}}(V)$$)$.
\end{prop}

\begin{defi} A {\em differential graded Zinbiel algebra} $(${\em\small DGZA}$)$ $($resp., a {\em differential graded Zinbiel coalgebra}$)$ $(${\em\small DGZC}$)$ is a {\em\small GZA} $(V,m)$ $($resp., {\em\small GZC} $(V,\Delta)$$)$ together with a degree $1$ $($$-1$ in the homological setting$)$ derivation $d$ $($resp., coderivation $D$$)$ that squares to 0. More precisely, $d$ $($resp., $D$$)$ is a degree $1$ $($$-1$ in the homological setting$)$ linear map $d:V\to V$ $($resp., $D:V\to V$$)$, such that $$d\,m=m\left( d\otimes \mathrm{id}+\mathrm{id}\otimes d\right)\quad (\text{resp.,}\quad \Delta\,D=\left( D\otimes\mathrm{id} +\mathrm{id}\otimes D\right)\Delta)\;$$ and $d^2=0$ $($resp., $D^2=0$$)$.
\end{defi}

Since the {\small GZA} $\mathrm{Zin}(V)$ (resp., {\small GZC} $\mathrm{Zin}^{\mathrm{c}}(V)$) is free, any degree $1$ linear map $d:V\rightarrow \mathrm{Zin}(V)$ (resp., $D:\mathrm{Zin}^{\mathrm{c}}(V)\to V$) uniquely extends to a derivation ${d}:\mathrm{Zin}(V)\to\mathrm{Zin}(V)$ (resp., coderivation $D:\mathrm{Zin}^{\mathrm{c}}(V)\to\mathrm{Zin}^{\mathrm{c}}(V)$).

\begin{defi} A {\em quasi-free} {\small DGZA} (resp., a {\em quasi-free} {\small DGZC}) over $V$ is a {\small DGZA} (resp., {\small DGZC}) of the type $(\mathrm{Zin}(V),d)$ $($resp., $(\mathrm{Zin}^{\mathrm{c}}(V),D)\,$$)$.\end{defi}

\subsubsection{Leibniz infinity algebras}

In the present text we use homological ($i$-ary map of degree $i-2$) and cohomological ($i$-ary map of degree $2-i$) infinity algebras. Let us recall the definition of homological Leibniz infinity algebras.

\begin{defi}\label{LeibInftyAlg} A (homological) {\em Leibniz infinity algebra} is a graded vector space $V$ together with a sequence of linear maps $l_i: V^{\otimes i}\to V$ of degree $i-2$, $i\ge 1$, such that for any $n\ge 1$, the following {\em higher Jacobi identity} holds:
\begin{equation}
\begin{array}{l}
\displaystyle\sum\limits_{i+j=n+1}\sum\limits_{\substack{j\leqslant k\leqslant n}}\sum\limits_{\sigma\in \mathrm{Sh}(k-j,j-1)}(-1)^{(n-k+1)(j-1)}\,(-1)^{j(v_{\sigma(1)}+...+v_{\sigma(k-j)})}\,\varepsilon(\sigma)\, \mathrm{sign}(\sigma) \\[4ex]\quad\quad
{l_i(v_{\sigma(1)},...,v_{\sigma(k-j)},l_j(v_{\sigma(k-j+1)},...,v_{\sigma(k-1)},v_k),v_{k+1},...,v_{n})}=0\;,
\end{array}
\end{equation} where $\mathrm{sign}{\sigma}$ is the signature of $\sigma$ and where we denoted the degree of the homogeneous $v_i\in V$ by $v_i$ instead of $|v_i|$.
\end{defi}

\begin{thm}\label{LeibInftyAlg1:1} There is a 1:1 correspondence between Leibniz infinity algebras, in the sense of Definition \ref{LeibInftyAlg}, over a graded vector space $V$ and quasi-free {\em\small DGZC}-s $(\mathrm{Zin}^{\mathrm{c}}(sV),D)$ $($resp., in the case of a finite-dimensional graded vector space $V$, quasi-free {\small\em DGZA}-s $(\mathrm{Zin}(s^{-1}V^*),d)$$)$.
\end{thm}

In the abovementioned 1:1 correspondence between infinity algebras over a quadratic Koszul operad $P$ and quasi-free {\small DG$P^{\text{!`}}$C} (resp., quasi-free {\small DG$P^!$A}) (self-explaining notation), a $P_{\infty}$-algebra structure on a graded vector space $V$ is viewed as a representation on $V$ of the {\small DG} operad $P_{\infty}$ -- which is defined as the cobar construction $\Omega P^{\text{!`}}$ of the Koszul dual cooperad $P^{\text{!`}}$. Theorem \ref{LeibInftyAlg1:1} makes this correspondence concrete in the case $P={\sf Lei}$; a proof can be found in \cite{AP10} and in the Appendix of this thesis (see page~\pageref{LeibnizInftyAlgebra}).

\subsubsection{Leibniz infinity morphisms}

\begin{defi}\label{LeibInftyAlgMorph}
A {\em morphism between Leibniz infinity algebras} $(V,l_i)$ and $(W,m_i)$ is a sequence of linear maps $\varphi_i:V^{\otimes i}\to W$ of degree $i-1$, $i\ge 1$, which satisfy, for any $n\ge 1,$ the condition

\begin{equation}
\begin{array}{l}
\sum\limits_{i=1}^n\hspace{1mm}\sum\limits_{\substack{k_1+...+k_i=n}}\hspace{1mm}\sum\limits_{\sigma\in \mathrm{\mathfrak{Sh}}(k_1,...,k_i)}(-1)^{\sum\limits_{r=1}^{i-1}(i-r)k_r+\frac{i(i-1)}{2}}\,(-1)^{
\sum\limits_{r=2}^i(k_r-1)(v_{\sigma(1)}+...+v_{\sigma(k_1+...+k_{r-1})})}\,\varepsilon(\sigma)\,\mathrm{sign}(\sigma)\\[0.5cm]
\;m_i\left(\varphi_{k_1}(v_{\sigma(1)},...,v_{\sigma(k_1)}),\varphi_{k_2}(v_{\sigma(k_1+1)},...,v_{\sigma(k_1+k_2)}),...,\varphi_{k_i}(v_{\sigma(k_1+...+k_{i-1}+1)},...,v_{\sigma(k_1+...+k_i)})\right)\\[0,5cm]
=\\\quad\quad\quad\sum\limits_{i+j=n+1}\hspace{1mm}\sum\limits_{j\leqslant k \leqslant n}\sum\limits_{\sigma\in \mathrm{Sh}(k-j,j-1)}(-1)^{k+(n-k+1)j}\,(-1)^{j(v_{\sigma(1)}+...+v_{\sigma(k-j)})}\,\varepsilon(\sigma)\,\mathrm{sign}(\sigma)\\[0.5cm]
\quad\quad\quad\quad\quad\quad{\varphi_i(v_{\sigma(1)},...,v_{\sigma(k-j)},l_j(v_{\sigma(k-j+1)},...,v_{\sigma(k-1)},v_k),v_{k+1},...,v_{n})}\;,
\end{array}
\end{equation} where $\mathrm{\mathfrak S\mathfrak h}(k_1,\ldots,k_i)$ denotes the set of shuffles $\sigma\in\mathrm{Sh}(k_1,\ldots,k_i)$, such that $\sigma(k_1)<\sigma(k_1+k_2)<\ldots<\sigma(k_1+k_2+\ldots+k_i)$.
\end{defi}

\begin{thm}\label{LeibInftyAlgMorph1:1}
There is a 1:1 correspondence between Leibniz infinity algebra morphisms from $(V,l_i)$ to $(W,m_i)$ and {\small \em DGC} morphisms $\mathrm{Zin}^{\mathrm{c}}(sV)\to \mathrm{Zin}^{\mathrm{c}}(sW)$ $($resp., in the finite-dimensional case, {\em\small DGA} morphisms $\mathrm{Zin}(s^{-1}W^*)\to \mathrm{Zin}(s^{-1}V^*)$$)$, where the quasi-free {\small\em DGZC}-s $($resp., the quasi-free {\small\em DGZA}-s$)$ are endowed with the codifferentials $($resp., differentials$)$ that encode the structure maps $l_i$ and $m_i$.
\end{thm}

In literature, infinity morphisms of $P_{\infty}$-algebras are usually defined as morphisms of quasi-free {\small DG$P^{\text{!`}}$C}-s. However, no explicit formulae seem to exist for the Leibniz case. A proof of Theorem \ref{LeibInftyAlgMorph1:1} can be found in the Appendix (see page~\pageref{LeinbizInftyAlgebraMorph}). Let us also stress that the concept of infinity morphism of $P_{\infty}$-algebras does not coincide with the notion of morphism of algebras over the operad $P_{\infty}$.

\subsubsection{Composition of Leibniz infinity morphisms}\label{LeibInftyMorphComp}

Composition of infinity morphisms between $P_{\infty}$-algebras corresponds to composition of the corresponding morphisms between quasi-free {\small DG$P^{\text{!`}}$C}-s: the categories {\tt $P_{\infty}$-Alg} and {\tt qfDG$P^{\text{!`}}$CoAlg} (self-explaining notation) are isomorphic. Explicit formulae can easily be computed.

\subsection{Leibniz infinity homotopies}\label{KPQSection3}

\subsubsection{Concordances and their compositions}

Let us first look for a proper concept of homotopy in the category {\tt qfDG$P^{\text{!`}}$CoAlg}, or, dually, in {\tt qfDG$P^{\,!}$Alg}.

\paragraph{Definition and characterization}

The following concept of homotopy -- referred to as concordance -- first appeared in an unpublished work by Stasheff and Schlessinger, which was based on ideas of Bousfield and Gugenheim. It can also be found in \cite{SSS07}, for homotopy algebras over the operad {\sf Lie} (algebraic version), as well as in \cite{DP12}, for homotopy algebras over an arbitrary operad $P$ (coalgebraic version).\medskip

It is well-known that a $C^{\infty}$-homotopy $\eta:I\times X\rightarrow Y$, $I=[0,1]$, connecting two smooth maps $p,q$ between two smooth manifolds $X,Y$, induces a cochain homotopy between the pullbacks $p^*,q^*.$ Indeed, in the algebraic category, $$
\eta^*:\Omega(Y)\rightarrow \Omega(I)\otimes\Omega(X)\;,
$$
and $\eta^*(\omega)$, $\omega\in\Omega(Y),$ reads \begin{equation}\label{ConDecomp}\eta^*(\omega)(t)=\varphi(\omega)(t)+dt\,\rho(\omega)(t)\;.\end{equation} It is easily checked (see below for a similar computation) that, since $\eta^*$ is a cochain map, we have
$$
{d_t\varphi}=d_X\rho(t)+\rho(t) d_Y\;,
$$
where $d_X,d_Y$ are the de Rham differentials. When integrating over $I$, we thus obtain
$$
q^*-p^*=d_Xh+hd_Y\;,
$$
where $\displaystyle h=\smallint_I \rho(t) dt$ has degree $-1$.

Before developing a similar approach to homotopies between morphisms of quasi-free {\small DGZA}-s, let us recall that tensoring an `algebra' (resp., `coalgebra') with a {\small DGCA} (resp., {\small DGCC}) does not change the considered type of algebra (resp., coalgebra); let us also introduce the `evaluation' maps $$\varepsilon_1^i:\Omega(I)=C^{\infty}(I)\oplus dt\,C^{\infty}(I)\ni f(t)+dt\,g(t)\mapsto f(i)\in\mathbb{K},\quad i\in\{0,1\}\;.$$

In the following -- in contrast with our above notation -- we omit stars. Moreover -- although the `algebraic' counterpart of a Leibniz infinity algebra over $V$ is ($\mathrm{Zin}(s^{-1}V^*),d_V)$ -- we consider Zinbiel algebras of the type $(\mathrm{Zin}(V),d_V)$.

\begin{defi}If $p,q:\mathrm{Zin}(W)\to \mathrm{Zin}(V)$ are two {\em\small DGA} morphisms, a {\em homotopy} or {\em concordance} $\eta:p\Rightarrow q$ from $p$ to $q$ is a {\em\small DGA} morphism
$\eta:\mathrm{Zin}(W)\to \Omega(I)\otimes\mathrm{Zin}(V)$, such that $$\varepsilon_1^0\eta=p\quad\text{and}\quad \varepsilon_1^1\eta=q\;.$$\end{defi}

The following proposition is basic.

\begin{prop}\label{CharConcord} Concordances
$$\eta:\mathrm{Zin}(W)\to\Omega(I)\otimes\mathrm{Zin}(V)$$ between {\em\small DGA} morphisms $p$, $q$ can be identified with 1-parameter families
$$\varphi:I\to\mathrm{Hom}_{\mathrm{DGA}}(\mathrm{Zin}(W),\mathrm{Zin}(V))\;$$and $$\rho:I\to \varphi\!\mathrm{Der}(\mathrm{Zin}(W),\mathrm{Zin}(V))\;$$of (degree 0) {\em\small DGA} morphisms and of degree $1$ $\varphi$-Leibniz
morphisms, respectively, such that
{\begin{equation}\label{DECon}d_t\varphi=[d,\rho(t)]\;\end{equation}}and $\varphi(0)=p$, $\varphi(1)=q$. The {\em\small RHS} of the differential equation (\ref{DECon}) is defined by $$[d,\rho(t)]:=d_V\rho(t)+\rho(t)d_W\;,$$ where $d_V,d_W$ are the differentials of the quasi-free {\em\small DGZA}-s $\mathrm{Zin}(V),\mathrm{Zin}(W)$. \end{prop}

The notion of $\varphi$-derivation or $\varphi$-Leibniz morphism appeared for instance in \cite{BKS04}: for $w,w'\in\mathrm{Zin}(W)$, $w$ homogeneous, $$\rho(w\cdot w')=\rho(w)\cdot\varphi(w')+(-1)^{w}\varphi(w)\cdot\rho(w')\;,$$ where we omitted the dependence of $\rho$ on $t$.

\begin{proof} As already mentioned in Equation (\ref{ConDecomp}), $\eta(w)$, $w\in\mathrm{Zin}(W)$, reads $$\eta(w)(t)=\varphi(w)(t)+dt\,\rho(w)(t)\;,$$ where $\varphi(t):\mathrm{Zin}(W)\to\mathrm{Zin}(V)$ and $\rho(t):\mathrm{Zin}(W)\to\mathrm{Zin}(V)$ have degrees $0$ and $1$, respectively (the grading of $\mathrm{Zin}(V)$ is induced by that of $V$ and the grading of $\Omega(I)$ is the homological one). Let us now translate the remaining properties
of $\eta$ into properties of $\varphi$ and $\rho$. We denote by $d_I=dt\,d_t$ the de Rham differential of $I$. Since $\eta$ is a chain map,
$$dt\,d_t\varphi+d_V\varphi-dt\,d_V\rho=(d_I\otimes\mathrm{id}+\mathrm{id}\otimes d_V)\eta=\eta d_W=\varphi d_W+dt\, \rho d_W\;,$$ so that $$d_{V}\varphi=\varphi d_{W}\quad \text{and}\quad d_t\varphi=d_{V}\rho+\rho d_{W}=[d,\rho]\;.$$ As $\eta$
is also an algebra morphism, we have, for $w,w'\in\mathrm{Zin}(W),$ $$
\varphi(w\cdot w')+dt\,\rho(w\cdot w')=
(\varphi(w)+dt\,\rho(w))\cdot(\varphi(w')+dt\,\rho(w')) $$ $$= \varphi(w)\cdot\varphi(w')+(-1)^{w}dt\,(\varphi(w)\cdot\rho(w')) +dt\,(\rho(w)\cdot\varphi(w'))\;,$$ and $\varphi$ (resp., $\rho$) is a family of {\small DGA} morphisms (resp., of degree $1$ $\varphi$-Leibniz maps) from $\mathrm{Zin}(W)$ to $\mathrm{Zin}(V)$. Eventually, $$p=\varepsilon_1^0\eta=\varphi(0)\quad\text{and}\quad q=\varepsilon_1^1\eta=\varphi(1)\;.$$\end{proof}

\paragraph{Horizontal and vertical compositions}

{\em In literature, the `categories' of Leibniz (resp., Lie) infinity algebras over $V$ (finite-dimensional) and of quasi-free {\small DGZA}-s (resp., quasi-free {\small DGCA}-s) over $s^{-1}V^*$ are (implicitly or explicitly) considered equivalent}. This conjecture is so far corroborated by the results of this paper. Hence, let us briefly report on compositions of concordances.\medskip

Let $\eta:p\Rightarrow q$ and $\eta':p'\Rightarrow q'$, 

\begin{equation}\label{horizontal_composition_diagramm}
\xymatrix
@R=1pc{
&\ar@{=>}^{\eta}[dd]&&\ar@{=>}^{\eta'}[dd]&\\
(\mathrm{Zin}(W),d_W)\ar@/^2pc/@{->}[rr]^-{p}\ar@/_2pc/@{->}[rr]_-{q}&&(\mathrm{Zin}(V),d_V)\ar@/_2pc/@{->}[rr]_-{q'}\ar@/^2pc/@{->}[rr]^-{p'}&&(\mathrm{Zin}(U),d_U)\;,\\
&&&&
}
\end{equation}
be concordances between {\small DGA} morphisms. Their horizontal composite $\eta'\circ_0\eta:p'\circ p\Rightarrow q'\circ q$,
$$
\xymatrix
@R=1pc{
&\ar@{=>}^{\eta'\circ_0\eta}[dd]\\
(\mathrm{Zin}(W),d_W)\ar@/^2pc/@{->}[rr]^-{p'\circ p}\ar@/_2pc/@{->}[rr]_-{q'\circ q}&&(\mathrm{Zin}(U),d_U)\;,\\
&&&&
}
$$
is defined by
\begin{equation}
(\eta'\circ_0\eta)(t)=(\varphi'(t)\circ\varphi(t))+dt\,(\varphi'(t)\circ\rho(t)+\rho'(t)\circ\varphi(t))\;,
\end{equation}
with self-explaining notation. It is easily checked that the first term of the {\small RHS} and the coefficient of $dt$ in the second term have the properties needed to make $\eta'\circ_0\eta$ a concordance between $p'\circ p$ and $q'\circ q$.

As for the vertical composite $\eta'\circ_1\eta:p\Rightarrow r$ of concordances $\eta:p\Rightarrow q$ and $\eta':q\Rightarrow r$,

$$\xymatrix@C=3pc@R=3pc
{
&\ar@{=>}[d]^{\eta}&\\
(\mathrm{Zin}(W),d_W)\ar@/^4pc/@{->}^-p[rr]\ar@{->}^-(0.43)q[rr]\ar@/_4pc/@{->}_-r[rr]&\ar@{=>}^{\eta'}[d]&(\mathrm{Zin}(V),d_V)\;,\\
&&
}$$ note that the composability condition $\varphi(1)=q=t(\eta)=s(\eta')=q=\varphi'(0)$, where $s,t$ denote the source and target maps, does not encode any information about $\rho(1),\rho'(0)$. Hence, the usual `half-time' composition cannot be applied.

\begin{rem} The preceding observation is actually the shadow of the fact that the `category' of Leibniz infinity algebras is an infinity category.\end{rem}

\subsubsection{Infinity homotopies}\label{sec_Getzler_Shoiket}

Some authors addressed directly or indirectly the concept of homotopy of Lie infinity algebras ($L_{\infty}$-algebras). As aforementioned, in the (equivalent) `category' of quasi-free {\small DGCA}-s, the classical picture of homotopy leads to concordances. In the `category' of $L_{\infty}$-algebras itself, morphisms can be viewed as Maurer-Cartan ({\small MC}) elements of a specific $L_{\infty}$-algebra \cite{Dol07},\cite{BS07}, which yields the notion of `gauge homotopy' between $L_{\infty}$-morphisms. Additional notions of homotopy between {\small MC} elements do exist: Quillen and cylinder homotopies. On the other hand, Markl \cite{Mar02} uses colored operads to construct homotopies for $\infty$-morphisms in a systematic way. The concepts of concordance, operadic homotopy, as well as Quillen, gauge, and cylinder homotopies are studied in detail in \cite{DP12}, for homotopy algebras over any Koszul operad, and they are shown to be equivalent, essentially due to homotopy transfer.\medskip

In this subsection, we focus on the Leibniz infinity case and provide a brief account on the relationship between concordances, gauge homotopies, and Quillen homotopies (in the next section, we explain why the latter concept is the bridge to Getzler's \cite{Get09} (and Henriques' \cite{Hen08}) work, as well as to the infinity category structure on the set of Leibniz infinity algebras).\medskip

Let us stress that all series in this section converge under some local finiteness or nilpotency conditions (for instance pronilpotency or completeness).

\paragraph{Gauge homotopic Maurer-Cartan elements}\label{Gauge}

Lie infinity algebras over ${\mathfrak g}$ are in bijective correspondence with quasi-free {\small DGCC}-s $(\mathrm{Com}^c(s{\mathfrak g}),D)$, see Equation (\ref{GKCoalg}). Depending on the definition of the $i$-ary brackets $\ell_i$, $i\ge 1$, from the corestrictions $D_i:(s{\mathfrak g})^{\odot i}\to s{\mathfrak g}$, where $\odot$ denotes the graded symmetric tensor product, one obtains various sign conventions in the defining relations of a Lie infinity algebra. When setting $\ell_i:=D_i$ (resp., $\ell_i:=s^{-1}D_i\, s^{\,i}$ (our choice in this paper), $\ell_i:=(-1)^{i(i-1)/2}s^{-1}D_i\, s^{\,i}$), we get a Voronov $L_{\infty}$-antialgebra \cite{Vor05} (resp., a Lada-Stasheff $L_{\infty}$-algebra \cite{LS93}, a Getzler $L_{\infty}$-algebra \cite{Get09}) made up by graded symmetric multilinear maps $\ell_i:(s{\mathfrak g})^{\times i}\to s{\mathfrak g}$ of degree $-1$ (resp., by graded antisymmetric multilinear maps $\ell_i:{\mathfrak g}^{\times i}\to {\mathfrak g}$ of degree $i-2$, idem), which verify the conditions
\begin{equation}\sum_{i+j=r+1}\sum_{\sigma\in\mathrm{Sh}(i,j-1)}\varepsilon(\sigma)\ell_{j}(\ell_i(sv_{\sigma_1},\ldots,sv_{\sigma_i}),sv_{\sigma_{i+1}},\ldots,sv_{\sigma_r})=0\;,\label{VLInfty}\end{equation} for all homogeneous $sv_k\in s{\mathfrak g}$ and all $r\ge 1$ (resp., the same higher Jacobi identities (\ref{VLInfty}), except that the sign $\varepsilon(\sigma)$ is replaced by $(-1)^{i(j-1)}\varepsilon(\sigma)\mathrm{sign}(\sigma)$, by $(-1)^{i}\varepsilon(\sigma)\mathrm{sign}(\sigma)$).\medskip

As the {\small MC} equation of a Lie infinity algebra $(\mathfrak g,\ell_i)$ must correspond to the {\small MC} equation given by the $D_i$, it depends on the definition of the operations $\ell_i$. For a Lada-Stasheff $L_{\infty}$-algebra (resp., a Getzler $L_{\infty}$-algebra), we obtain that the set $\mathrm{MC}(\mathfrak g)$ of {\small MC} elements of $\mathfrak g$ is the set of solutions $\alpha\in\mathfrak g_{-1}$ of the {\small MC} equation \begin{equation}\label{MCLInfty}{\sum_{i=1}^{\infty}\frac{1}{i!}(-1)^{i(i-1)/2}\ell_i(\alpha,\ldots,\alpha)=0}\quad (\text{resp.}, \sum_{i=1}^{\infty}\frac{1}{i!}\ell_i(\alpha,\ldots,\alpha)=0)\;.\end{equation}

\begin{rem} Since we prefer the original definition of homotopy Lie \cite{LS93} and homotopy Leibniz \cite{AP10} algebras, but use the results of \cite{Get09}, we adopt -- to facilitate the comparison with \cite{Get09} -- Getzler's sign convention whenever his work is involved, and change signs appropriately later, when applying the results in our context.\end{rem}

Hence, we now consider the second {\small MC} equation (\ref{MCLInfty}). Further, for any $\alpha\in\mathfrak g_{-1}$, the twisted brackets {$$\ell^{\alpha}_i(v_1,\ldots,v_i)=\sum_{k=0}^{\infty}\frac{1}{k!}\ell_{k+i}(\alpha^{\otimes
k},v_1,\ldots,v_i)\;,$$} $v_1,\ldots,v_i\in {\mathfrak g}$, are a sequence of graded antisymmetric multilinear maps of degree $i-2$. It is well-known that the $\ell_i^{\alpha}$ endow $\mathfrak g$ with a new Lie infinity structure, if $\alpha\in\mathrm{MC}({\mathfrak g}).$ Finally, any vector $r\in{\mathfrak g}_0$ gives rise to a vector field \begin{equation} V_r:{\mathfrak g}_{-1}\ni\alpha\mapsto {V_r(\alpha):=-\ell_1^{\alpha}(r)}=-\sum_{k=0}^{\infty}\frac{1}{k!}\ell_{k+1}(\alpha^{\otimes
k},r)=\sum_{k=1}^{\infty}\frac{(-1)^{k}}{(k-1)!}\ell_{k}(r,\alpha^{\otimes (k-1)})\in{\mathfrak g}_{-1}\;.\label{SpecVectfield}\end{equation} This field restricts to a vector field of the Maurer-Cartan quadric $\mathrm{MC}(\mathfrak g)$ \cite{DP12}. It follows that the integral curves {\begin{equation}\label{FlowEq}d_t\alpha=V_r(\alpha(t))\;,\end{equation}} starting from points in the quadric, are located inside $\mathrm{MC}(\mathfrak g)$. Hence, the

\begin{defi}\label{DolSho} $($\cite{Dol07}, \cite{BS07}$)$ Two {\small MC} elements $\alpha,\beta\in\mathrm{MC}({\mathfrak g})$ of a Lie infinity algebra $\mathfrak g$ are {\em gauge homotopic} if there exists $r\in{\mathfrak g}_0$ and an integral curve $\alpha(t)$ of $V_r$, such that $\alpha(0)=\alpha$ and $\alpha(1)=\beta$.\end{defi}

This gauge action is used to define the deformation functor $\mathrm{Def}:{\tt L}_{\infty}\to{\tt Set}$ from the category of Lie infinity algebras to the category of sets. Moreover, it will provide a concept of homotopy between Leibniz infinity morphisms.\medskip

Let us first observe that Equation (\ref{FlowEq}) is a 1-variable ordinary differential equation ({\small ODE}) and can be solved via an iteration procedure. The integral curve with initial point $\alpha\in\mathrm{MC}(\mathfrak g)$ is computed in \cite{Get09}. When using our sign convention in the defining relations of a Lie infinity algebra, we get an {\small ODE} that contains different signs and the solution of the corresponding Cauchy problem reads \begin{equation} \alpha(t)=\sum_{k=0}^\infty\frac{t^k}{k!}e_{\alpha}^k(r)\;,\label{GetzlerGaugeHomotopyExplicit}\end{equation} where the $e_{\alpha}^k(r)$ admit a nice combinatorial description in terms of rooted trees. Moreover, they can be obtained inductively:
\begin{equation}\label{Getzler induction formula}
\left\{\begin{array}{l}
\displaystyle e_{\alpha}^{i+1}(r)=\sum\limits_{n=0}^{\infty}\frac{1}{n!}(-1)^{\frac{n(n+1)}{2}}\sum\limits_{k_1+...+k_n=i}\frac{i!}{k_1!...k_n!}\ell_{n+1}(e_{\alpha}^{k_1}(r),...,e_{\alpha}^{k_n}(r),r)\;,\\
e^0_{\alpha}(r)=\alpha\;.
\end{array}\right.
\end{equation}
It follows that $\alpha,\beta\in\mathrm{MC}({\mathfrak g})$ are gauge homotopic if \begin{equation}\beta-\alpha=\sum_{k=1}^{\infty}\frac{1}{k!}e_{\alpha}^k(r)\;,\label{FundaHomotopyEquation}\end{equation} for some $r\in{\mathfrak g}_0.$

\paragraph{Simplicial de Rham algebra}\label{SimpicialDR}

\newcommand*{\EnsQuot}[2]%
{\ensuremath{%
    #1/\!\raisebox{-.65ex}{\ensuremath{{#2}}}}}

We first fix the notation.\medskip

Let $\Delta$ be the {\it simplicial category} with objects the nonempty finite ordinals $[n]=\{0,\ldots, n\}$, $n\ge 0$, and morphisms the order-respecting functions $f:[m]\to [n]$. Denote by $\delta_n^i:[n-1]\rightarrowtail [n]$ the injection that omits the image $i$ and by $\sigma_n^i:[n+1]\twoheadrightarrow [n]$ the surjection that assigns the same image to $i$ and $i+1$, $i\in\{0,\ldots,n\}$.

A {\it simplicial object} in a category {\tt C} is a functor $X \in [\Delta^{\mathrm{op}},{\tt C}\,]\;.$ It is completely determined by the simplicial data $(X_n,d\,^n_i,s\,^n_i)$, ${n\ge 0,\, i\in\{0,\dots,n\}},$ where $X_n=X[n]$ ($n$-simplices), $d\,^n_i=X(\delta^i_n)$ (face maps), and $s\,^n_i=X(\sigma^i_n)$ (degeneracy maps). We denote by ${\tt SC}$ the functor category $[\Delta^{\mathrm{op}},{\tt C}]$ of simplicial objects in ${\tt C}$.

The simplicial category is embedded in its Yoneda dual category: $$h_{\ast}: \Delta\ni[n]\mapsto \mathrm{Hom}_{\Delta}(-,[n])\in [\Delta^{\mathrm{op}},{\tt Set}]={\tt SSet}\;.$$ We refer to the functor of points of $[n]$, i.e. to the simplicial set $\Delta[n]:=\mathrm{Hom}_{\Delta}(-,[n]),$ as the {\it standard simplicial $n$-simplex}. Moreover, the Yoneda lemma states that $$\mathrm{Hom}_{\Delta}([n],[m])\simeq \mathrm{Hom}(\mathrm{Hom}_{\Delta}(-,[n]),\mathrm{Hom}_{\Delta}(-,[m]))=\mathrm{Hom}(\Delta[n],\Delta[m]) \;.$$ This bijection sends $f:[n]\to [m]$ to $\varphi$ defined by $\varphi_{[k]}(\bullet)= f\circ \bullet$ and $\varphi$ to $\varphi_{[n]}(\mathrm{id}_{[n]})$. In the following we identify $[n]$ (resp., $f$) with $\Delta[n]$ (resp., $\varphi$).

The set $S_n$ of $n$-simplices of a simplicial set $S$ is obviously given by $S_n\simeq\mathrm{Hom}(\mathrm{Hom}_{\Delta}(-,[n]),S)=\mathrm{Hom}(\Delta[n],S).$

Let us also recall the adjunction $$|-|:{\tt SSet}\rightleftarrows {\tt Top}:\mathrm{Sing}$$ given by the `geometric realization functor' $|-|$ and the `singular complex functor' Sing. To define $|-|$, we first define the realization $|\Delta[n]|$ of the standard simplicial $n$-simplex to be the {\it standard topological $n$-simplex} $$\Delta^n=\{(x_0,\ldots,x_n)\in \mathbb{R}^{{n+1}}: x_i\ge 0, \sum_ix_i=1\}\;.$$ We can view $|-|$ as a functor $|-|\in[\Delta,{\tt Top}]\;$. Indeed, if $f:[n]\to [m]$ is an order-respecting map, we can define a continuous map $|f|:\Delta^n\to \Delta^m$ by $$|f|[x_0,\ldots,x_n]=[y_0,\ldots,y_m]\;,$$ where $y_i=\sum_{j\in f^{-1}\{i\}}x_j$.\medskip

Let $\wedge_{\mathbb{K}}^{\star}(x_0,\ldots,x_n,dx_0,\ldots,dx_n)$ be the free graded commutative algebra generated over $\mathbb{K}$ by the degree 0 (resp., degree 1) generators $x_i$ (resp., $dx_i$). If we divide the relations $\sum_ix_i=1$ and $\sum_idx_i=0$ out and set $d(x_i)=dx_i$ and $d(dx_i)=0$, we obtain a quotient {\small DGCA} $$\Omega_n^{\star}=\EnsQuot{\wedge_{\mathbb{K}}^{\star}(x_0,\ldots,x_n,dx_0,\ldots,dx_n)}{\left(\sum_ix_i-1,\sum_idx_i\right)}$$ that can be identified, for $\mathbb{K}=\mathbb{R}$, with the algebra of polynomial differential forms $\Omega^{\star}(\Delta^n)$ of the standard topological $n$-simplex $\Delta^n$. When defining $\Omega^{\star}:\Delta^{\mathrm{op}}\to {\tt DGCA}$ by $\Omega^{\star}[n]:=\Omega^{\star}_n$ and, for $f:[n]\to [m]$, by $\Omega^{\star}(f):=|f|^*:\Omega^{\star}_m\to \Omega^{\star}_n$ (use the standard pullback formula for differential forms given by $y_i=\sum_{j\in f^{-1}\{i\}}x_j$), we obtain a simplicial differential graded commutative algebra $\Omega^{\star}\in {\tt SDGCA}$. Hence, the face maps $d\,^n_i:\Omega^{\star}_n\to \Omega^{\star}_{n-1}$ are the pullbacks by the $|\delta^i_n|:\Delta^{n-1}\to \Delta^n$, and similarly for the degeneracy maps. In particular, $d^2_0=|\delta_2^0|^*:\Omega^{\star}_2\to \Omega^{\star}_{1}$ is induced by $y_0=0,y_1=x_0,y_2=x_1$. Let us eventually introduce, for $0\le i\le n$, the vertex $e_i$ of $\Delta^n$ whose $(i+1)$-th component is equal to $1$, as well as the evaluation map $\varepsilon_n^i:\Omega^{\star}_n\to \mathbb{K}$ at $e_i$ (`pullback' induced by $(y_0,\ldots,y_n)=e_i$).

\paragraph{Quillen homotopic Maurer-Cartan elements}\label{QuillenDefSect}

We already mentioned that, if $({\mathfrak g},\ell_i)$ is an $L_{\infty}$-algebra and $(A,\cdot,d)$ a {\small DGCA}, their tensor product ${\mathfrak g}\otimes A$ has a canonical $L_{\infty}$-structure $\bar\ell_i$. It is given by $$\bar\ell_1(v\otimes a)=(\ell_1\otimes\mathrm{id}+\mathrm{id}\otimes d)(v\otimes a)=\ell_1(v)\otimes a+(-1)^vv\otimes d(a)$$ and, for $i\ge 2$, by $$\bar\ell_i(v_1\otimes a_1,\ldots,v_i\otimes a_i)=\pm\ell_i(v_1,\ldots,v_i)\otimes(a_1\cdot \ldots \cdot a_i)\;,$$ where $\pm$ is the Koszul sign generated by the commutation of the variables.\medskip

The following concept originates in Rational Homotopy Theory.

\begin{defi} Two {\small MC} elements $\alpha,\beta\in\mathrm{MC}({\mathfrak g})$ of a Lie infinity algebra $\mathfrak g$ are {\em Quillen homotopic} if there exists a {\small MC} element $\bar\gamma\in\mathrm{MC}(\bar{\mathfrak g}_1)$ of the Lie infinity algebra $\bar{\mathfrak g}_1:={\mathfrak g}\otimes\Omega^{\star}_1$, such that $\varepsilon_1^0\bar\gamma=\alpha$ and $\varepsilon_1^1\bar\gamma=\beta$ $($where the $\varepsilon_1^i$ are the natural extensions of the evaluation maps$)$.\end{defi}

From now on, we accept, in the definition of gauge equivalent {\small MC} elements, vector fields $V_{r(t)}$ induced by time-dependent $r=r(t)\in{\mathfrak g}_0$. The next result is proved in \cite{Can99} (see also \cite{Man99}). A proof sketch will be given later.

\begin{prop}
Two {\small MC} elements of a Lie infinity algebra are Quillen homotopic if and only if they are gauge homotopic.
\end{prop}

\paragraph{Infinity morphisms as Maurer-Cartan elements and infinity homotopies}

The possibility to view morphisms in $\mathrm{Hom}_{\mathrm{DGP^{\text{!`}}C}}(C,{\cal F}^{\mathrm{gr,c}}_{P^{\text{!`}}}(sW))$ as {\small MC} elements is known from the theory of the bar and cobar constructions of algebras over an operad. In \cite{DP12}, the authors show that the fact that infinity morphisms between $P_{\infty}$-algebras $V$ and $W$, i.e. morphisms in $$\mathrm{Hom}_{\mathrm{DGP^{\text{!`}}C}}({\cal F}^{\mathrm{gr,c}}_{P^{\text{!`}}}(sV), {\cal F}^{\mathrm{gr,c}}_{P^{\text{!`}}}(sW))\;,$$ are 1:1 with Maurer-Cartan elements of an $L_{\infty}$-structure on $$\mathrm{Hom}_{\mathbb{K}}({\cal F}^{\mathrm{gr,c}}_{P^{\text{!`}}}(sV),W)\;,$$ is actually a consequence of a more general result based on the encoding of two $P_{\infty}$-algebras and an infinity morphism between them in a {\small DG} colored free operad. In the case $P={\sf Lie}$, one recovers the fact \cite{BS07} that \begin{equation}\label{NatBijCoAlg}\mathrm{Hom}_{\mathrm{DGCC}}(C,\mathrm{Com}^c(sW))\simeq \mathrm{MC}(\mathrm{Hom}_{\mathbb{K}}(C,W))\;,\end{equation} where $C$ is any locally conilpotent {\small DGCC}, where $W$ is an $L_{\infty}$-algebra, and where the {\small RHS} is the set of {\small MC} elements of some convolution $L_{\infty}$-structure on $\mathrm{Hom}_{\mathbb{K}}(C,W)$.\medskip

In the sequel we detail the case $P={\sf Lei}$. Indeed, when interpreting infinity morphisms of Leibniz infinity algebras as {\small MC} elements of a Lie infinity algebra, the equivalent notions of gauge and Quillen homotopies provide a concept of homotopy between Leibniz infinity morphisms.

\begin{prop}\label{proposition_about_lie_inf_on_space_of_homomorphisms}
Let $(V,\ell_i)$ and $(W, m_i)$ be two Leibniz infinity algebras and let $(\mathrm{Zin}^c(sV),D)$ be the quasi-free {\small\em DGZC} that corresponds to $(V,\ell_i)$. The graded vector space $$L(V,W):=\mathrm{Hom}_{\mathbb{K}}(\mathrm{Zin}^c(sV),W)$$ carries a convolution Lie infinity structure given by
\begin{equation}\label{lieInfinityStr1}
\mathcal{L}_1f=m_1\circ f+(-1)^ff\circ D
\end{equation}
and, as for $\mathcal{L}_p(f_1,...,f_p)$, $p\ge 2$, by
\begin{equation}\label{lieInfinityStrp}
\xymatrix@C=3.5pc{
\mathrm{Zin}^c(sV)\ar@{->}^-{\Delta^{p-1}}[r]&(\mathrm{Zin}^c(sV))^{\otimes p}\ar@{->}^-{\sum\limits_{\sigma\in S(p)}\varepsilon(\sigma)\, \mathrm{sign}(\sigma)\, f_{\sigma(1)}\otimes...\otimes f_{\sigma(p)}}[rr]&\qquad\quad&W^{\otimes p}\ar@{->}^-{m_p}[r]&W\;,
}
\end{equation}
where $f,f_1,\ldots,f_p\in L(V,W)$, $\Delta^{p-1}=(\Delta\otimes \mathrm{id}^{\otimes(p-2)})...(\Delta\otimes \mathrm{id})\Delta\,$, where $S(p)$ denotes the symmetric group on $p$ symbols, and where the central arrow is the graded antisymmetrization operator.
\end{prop}
\begin{proof}
See \cite{Laa02} or \cite{DP12}. A direct verification is possible as well.
\end{proof}

\begin{prop}\label{MorphMC}
Let $(V,\ell_i)$ and $(W, m_i)$ be two Leibniz infinity algebras. There exists a 1:1 correspondence between the set of infinity morphisms from $(V,\ell_i)$ to $(W,m_i)$ and the set of {\small MC} elements of the convolution Lie infinity algebra structure ${\cal L}_i$ on $L(V,W)$ defined in Proposition~\ref{proposition_about_lie_inf_on_space_of_homomorphisms}.
\end{prop}

Observe that the considered {\small MC} series converges pointwise. Indeed, the evaluation of $\mathcal{L}_p(f_1,...,f_p)$ on a tensor in $\mathrm{Zin}^c(sV)$ vanishes for $p\gg$, in view of the local conilpotency of $\mathrm{Zin}^c(sV)$. Moreover, convolution $L_{\infty}$-algebras are complete, so that their {\small MC} equation converges in the topology induced by the filtration (a descending filtration $F^iL$ of the space $L$ of an $L_{\infty}$-algebra $(L,{\cal L}_k)$ is \emph{compatible} with the $L_{\infty}$-structure ${\cal L}_k$, if ${\cal L}_k(F^{i_1}L,\ldots,F^{i_k}L)\subset F^{i_1+\ldots+i_k}L$, and it is \emph{complete}, if, in addition, the `universal' map $L\to  \varprojlim L/F^iL$ from $L$ to the (projective) limit of the inverse system $L/F^iL$ is an isomorphism) \cite{DP12}.\medskip

Note also that Proposition \ref{MorphMC} is a specification, in the case $P={\sf Lei}$, of the abovementioned 1:1 correspondence between infinity morphisms of $P_{\infty}$-algebras and {\small MC} elements of a convolution $L_{\infty}$-algebra. To increase the readability of this text, we give nevertheless a sketchy proof.

\begin{proof} An {\small MC} element is an $\alpha\in \mathrm{Hom}_{\mathbb{K}}(\mathrm{Zin}^c(sV),W)$ of degree $-1$ that verifies the {\small MC} equation. Hence, $s\alpha:\mathrm{Zin}^c(sV)\to sW$ has degree 0 and, since $\mathrm{Zin}^c(sW)$ is free as {\small GZC}, $s\alpha$ coextends uniquely to $\widehat{s\alpha}\in\mathrm{Hom}_{\mathrm{GZC}}(\mathrm{Zin}^c(sV),\mathrm{Zin}^c(sW))$. The fact that $\alpha$ is a solution of the {\small MC} equation exactly means that $\widehat{s\alpha}$ is a {\small DGZC}-morphism, i.e. an infinity morphism between the Leibniz infinity algebras $V$ and $W$. Indeed, when using e.g. the relations $\ell_i= s^{-1} D_i\, s^{\,i}$ and $m_i=s^{-1} {\mathfrak D}_i\,s^{\,i}$, and the corresponding version of the {\small MC} equation, we get
$$
\begin{array}{l}
\sum\limits^\infty_{p=1}\frac{1}{p!}(-1)^{\frac{p(p-1)}{2}}\mathcal{L}_p(\alpha,...,\alpha)=0 \Leftrightarrow\\
\sum\limits^\infty_{p=1} (-1)^{\frac{p(p-1)}{2}}m_p(\alpha\otimes...\otimes\alpha)\Delta^{p-1}+(-1)^\alpha\alpha D=0\Leftrightarrow\\
\sum\limits^\infty_{p=1}(-1)^{\frac{p(p-1)}{2}}s^{-1}{\mathfrak D}_p\,s^{\, p}(\alpha\otimes...\otimes\alpha)\Delta^{p-1}+(-1)^\alpha\alpha D=0 \Leftrightarrow\\
\sum\limits^\infty_{p=1}s^{-1}{\mathfrak D}_p(s\alpha\otimes...\otimes s\alpha)\Delta^{p-1}-s^{-1}s\alpha D=0 \Leftrightarrow\\
{\mathfrak D}\widehat{(s\alpha)}=\widehat{(s\alpha)}D\;.
\end{array}
$$
\end{proof}

Hence, the

\begin{defi} Two infinity morphisms $f,g$ between Leibniz infinity algebras $(V,\ell_i)$, $(W,m_i)$ are {\em infinity homotopic}, if the corresponding {\small MC} elements $\alpha=\alpha(f)$ and $\beta=\beta(g)$ of the convolution Lie infinity structure ${\cal L}_i$ on $L=L(V,W)$ are Quillen (or gauge) homotopic. In other words, $f$ and $g$ are infinity homotopic, if there exists $\bar\gamma\in\mathrm{MC}_1(\bar{L})$, i.e. an {\small MC} element $\bar\gamma$ of the Lie infinity structure $\bar{\cal L}_i$ on $\bar{L}=L\otimes\Omega^{\star}_1$ -- obtained from the convolution structure ${\cal L}_i$ on $L=\mathrm{Hom}_{\mathbb{K}}(\mathrm{Zin}^c(sV),W)$ via extension of scalars --, such that $\varepsilon_1^0\bar\gamma=\alpha$ and $\varepsilon_1^1\bar\gamma=\beta$.\end{defi}

\paragraph{Comparison of concordances and infinity homotopies}

Since, according to the prevalent philosophy, the `categories' {\tt qfDG$P^{\text{!`}}$CoAlg} and {\tt $P_{\infty}$-Alg} are `equivalent', appropriate concepts of homotopy in both categories should be in 1:1 correspondence. It can be shown \cite{DP12} that, for any type of algebras, the concepts of concordance and of Quillen homotopy are equivalent (at least if one defines concordances in an appropriate way); as, Quillen homotopies as already known to be equivalent to gauge homotopies, the wanted result follows in whole generality.\medskip

To accommodate the reader who is not interested in (nice) technicalities, we provide now a sketchy explanation of both relationships, defining concordances dually and assuming for simplicity that $\mathbb{K}=\mathbb{R}$.\medskip

Remember first that we defined concordances, in conformity with the classical picture, in a contravariant way: two infinity morphisms $f,g:V\to W$ between homotopy Leibniz algebras, i.e. two {\small DGA} morphisms $f^*,g^*:\mathrm{Zin}(s^{-1}W^*)\to \mathrm{Zin}(s^{-1}V^*)$, are concordant if there is a morphism $$\eta\in\mathrm{Hom}_{\mathrm{DGA}}(\mathrm{Zin}(s^{-1}W^*),\mathrm{Zin}(s^{-1}V^*)\otimes\Omega^{\star}_1)\;,$$ whose values at $0$ and $1$ are equal to $f^*$ and $g^*$, respectively. Although we will use this definition in the sequel (observe that we slightly adapted it to future purposes), we temporarily prefer a dual, covariant definition (which has the advantage that the spaces $V,W$ need not be finite-dimensional).\medskip

The problem that the linear dual of the infinite-dimensional {\small DGCA} $\Omega^{\star}_1$ (let us recall that $\star$ stands for the (co)homological degree) is not a coalgebra, has already been addressed in \cite{BM12}. The authors consider a space $(\Omega_1^{\star})^{\vee}$ made up by the formal series, with coefficients in $\mathbb{K}$, of the elements $\alpha_i=(t^i)^\vee,$ $i\ge0$, and $\beta_i=(t^i\,dt)^\vee,$ $i\ge 0$. For instance, $\sum_{i\in\mathbb{N}}{\mathfrak K}^i\alpha_i$ represents the map $\{t^i\}_{i\in\mathbb{N}}\to \mathbb{K}$ and assigns to each $t^i$ the coefficient ${\mathfrak K}^i$. The differential $\partial$ of $(\Omega_1^{\star})^\vee$ is (dual to the de Rham differential $d$ and is) defined by $\partial(\alpha_i)=0$ and $\partial(\beta_i)=(i+1)\alpha_{i+1}$. As for the coalgebra structure $\delta$, we set
\begin{gather*}
\delta(\alpha_i)=\sum_{a+b=i}\alpha_a\otimes\alpha_b\;,\\
\delta(\beta_i)=\sum_{a+b=i}(\beta_a\otimes\alpha_b+\alpha_a\otimes\beta_b)\;.
\end{gather*}
When extending to all formal series, we obtain a map $\delta:(\Omega_1^{\star})^\vee\to (\Omega_1^\star)^\vee\widehat\otimes(\Omega_1^\star)^\vee$, whose target is the completed tensor product. To fix this difficulty, one considers the decreasing sequence of vector spaces $(\Omega_1^\star)^\vee=:\Lambda^0\supset \Lambda^1\supset\Lambda^2\supset\ldots\;$, where $\Lambda^i=\delta^{-1}(\Lambda^{i-1}\otimes\Lambda^{i-1})$, $i\ge 1$, and defines the universal {\small DGCC} $\Lambda:=\cap_{i\ge 1}\Lambda^i$.\medskip

A concordance can then be defined as a map $$\eta\in\mathrm{Hom}_{\mathrm{DGC}}(\mathrm{Zin}^c(sV)\otimes\Lambda,\mathrm{Zin}^c(sW))$$ (with the appropriate boundary values). It is then easily seen that any Quillen homotopy, i.e. any element in $\mathrm{MC}(L\otimes\Omega_1^\star)$, gives rise to a concordance. Indeed, when writing ${\cal V}$ instead of $sV$, we have maps $$L\otimes\Omega_1^\star=\mathrm{Hom}_{\mathbb{K}}(\mathrm{Zin}^c({\cal V}),W)\otimes\Omega_1^\star=\mathrm{Hom}_{\mathbb{K}}\left(\bigoplus_{i\ge 1}{\cal V}^{\otimes i},W\right)\otimes\Omega_1^\star=\left(\prod_{i\ge 1}\mathrm{Hom}_{\mathbb{K}}({\cal V}^{\otimes i},W)\right)\otimes\Omega_1^\star$$ $$\longrightarrow\; \prod_{i\ge 1}\left(\mathrm{Hom}_{\mathbb{K}}({\cal V}^{\otimes i},W)\otimes\Omega_1^\star\right)\;\longrightarrow \; \prod_{i\ge 1}\mathrm{Hom}_{\mathbb{K}}({\cal V}^{\otimes i}\otimes\Lambda,W)=\mathrm{Hom}_{\mathbb{K}}(\mathrm{Zin}^c({\cal V})\otimes\Lambda,W)\;.$$ Only the second arrow is not entirely obvious. Let $(f,\omega)\in \mathrm{Hom}_{\mathbb{K}}({\cal V}^{\otimes i},W)\times\Omega_1^\star$ and let $(v,\lambda)\in{\cal V}^{\otimes i}\times \Lambda$. Set $\omega=\sum_{a=0}^{N}k_at^a+\sum_{b=0}^N\kappa_b\,(t^b\,dt)$ and, for instance, $\lambda=\sum_{i\in\mathbb{N}}{\mathfrak K}^i\alpha_i$. Then $$g(f,\omega)(v,\lambda):=(\sum_{i\in\mathbb{N}}{\mathfrak K}^i\alpha_i)(\omega)f(v):=(\sum_{a=0}^N{\mathfrak K}^ak_a)f(v)\in W\;$$ defines a map between the mentioned spaces. Eventually, a degree $-1$ element of $L\otimes\Omega_1^\star$ (resp., an {\small MC} element of $L\otimes\Omega_1^\star$, a Quillen homotopy) is sent to an element of $\mathrm{Hom}_{\mathrm{GC}}(\mathrm{Zin}^c({\cal V})\otimes\Lambda,\mathrm{Zin}^c{\cal W})$ (resp., an element of $\mathrm{Hom}_{\mathrm{DGC}}(\mathrm{Zin}^c({\cal V})\otimes\Lambda,\mathrm{Zin}^c{\cal W})$, a concordance).\medskip

The relationship between Quillen and gauge homotopy is (at least on the chosen level of rigor) much clearer. Indeed, an element $\bar\gamma\in\mathrm{MC}_1(\bar L) = \mathrm{MC}(L\otimes\Omega_1^\star)$ can be decomposed as
$$ \bar\gamma=\gamma(t)\otimes 1+ r(t)\otimes dt,$$
where $t\in[0,1]$ is the coordinate of $\Delta^1$. When unraveling the {\small MC} equation of the $\bar{\cal L}_i$ according to the powers of $dt$, one gets
\begin{equation}\label{Getzler-Shoiket_homotopy}
\begin{array}{l}
\sum_{p=1}^{\infty}\frac{1}{p!}{\mathcal{L}}_p(\gamma(t),...,\gamma(t))=0\;,\\
\\
\frac{d\gamma}{dt}=-\sum^{\infty}_{p=0}\frac{1}{p!}\mathcal{L}_{p+1}(\gamma(t),...,\gamma(t),r(t))\;.
\end{array}
\end{equation}
A (nonobvious) direct computation allows to see that the latter {\small ODE}, see Definition \ref{DolSho} of gauge homotopies and Equations (\ref{FlowEq}) and (\ref{SpecVectfield}), is dual (up to dimensional issues) to the {\small ODE} (\ref{DECon}), see Proposition \ref{CharConcord} that characterizes concordances.

\subsection{Infinity category of Leibniz infinity algebras}\label{KPQSection4}

We already observed that vertical composition of concordances is not well-defined and that Leibniz infinity algebras should form an infinity category. It is instructive to first briefly look at infinity homotopies between infinity morphisms of {\small DG} algebras.

\subsubsection{{\small DG} case}

Remember that infinity homotopies can be viewed as integral curves of specific vector fields $V_r$ of the {\small MC} quadric (with obvious endpoints). In the {\small DG} case, we have, for any $r\in L_0$, $$V_r:L_{-1}\ni\alpha\mapsto V_r(\alpha) = -{\cal L}_1(r)-{\cal L}_2(\alpha,r)\in L_{-1}\;.$$ In view of the Campbell-Baker-Hausdorff formula,

$$ \exp(t V_r)\circ \exp(t V_s) = \exp (t V_r + t V_s + 1/2\; t^2 [V_r, V_s]+ ...)\;. $$
The point is that $$V:L_0\to \mathrm{Vect}(L_{-1})$$ is a Lie algebra morphism -- also after restriction to the {\small MC} quadric; we will not detail this nonobvious fact. It follows that

$$\exp(t V_r)\circ \exp(t V_s) = \exp (t V_{r+s + 1/2\; t [r,s]+ ...})\;.$$
If we accept, as mentioned previously, time-dependent $r$-s, the problem of the vertical composition of homotopies is solved in the considered {\small DG} situation: the integral curve of the composed homotopy of two homotopies $\exp(t V_s)$ (resp., $\exp(t V_r)$) between morphisms $f,g$ (resp., $g,h$) is given by

$$c(t)=(\exp(t V_r)\circ \exp(t V_s))(f) = \exp (t V_{r+s + 1/2\; t [r,s]+ ...})(f)\;.$$

Note that this vertical composition is not associative. Moreover, the preceding approach does not go through in the infinity situation (note e.g. that in this case $L_0$ is no longer a Lie algebra). This again points out that homotopy algebras form infinity categories.

\subsubsection{Shortcut to infinity categories}\label{InftyCatIntro}

This subsection is a short digression that should allow to grasp the spirit of infinity categories. For additional information, we refer the reader to \cite{Gro10},\cite{Fin11} and \cite{Nog12}.\medskip


{\it Strict $n$-categories} or {\it strict $\omega$-categories} (in the sense of strict infinity categories) are well understood, see e.g. \cite{KMP11}. Roughly, they are made up by 0-morphisms (objects), 1-morphisms (morphisms between objects), 2-morphisms (homotopies between morphisms)..., up to $n$-morphisms, except in the $\omega$-case, where this upper bound does not exist. All these morphisms can be composed in various ways, the compositions being associative, admitting identities, etc. However, in most cases of higher categories these defining relations do not hold strictly. A number of concepts of weak infinity category, e.g. infinity categories in which the structural relations hold up to coherent higher homotopy, are developed in literature. Moreover, an $(\infty,r)$-category is an infinity category, with the additional requirement that all $j$-morphisms, $j>r$, be invertible. In this subsection, we actually confine ourselves to {\it $(\infty,1)$-categories, which we simply call $\infty$-categories}.

\paragraph{First examples}

Of course,

\begin{ex} {\it $\infty$-categories should include ordinary categories}.\end{ex}

There is another natural example of infinity category. When considering all the paths in $T\in{\tt Top}$, up to homotopy (for fixed initial and final points), we obtain the {\it fundamental groupoid} $\Pi_1(T)$ of $T$. Remember that the usual `half-time' composition of paths is not associative, whereas the induced composition of homotopy classes is. Hence, $\Pi_1(T)$, with the points of $T$ as objects and the homotopy classes of paths as morphisms, is (really) a category in which all morphisms are invertible. To encode more information about $T$, we can use a 2-category, the {fundamental 2-groupoid} $\Pi_2(T)$, whose 0-morphisms (resp., 1-morphisms, 2-morphisms) are the points of $T$ (resp., the paths between points, the homotopy classes of homotopies between paths), the composition of 1-morphisms being associative only up to a 2-isomorphism. More generally, we define the {\it fundamental $k$-groupoid} $\Pi_k(T)$, in which associativity of order $j\le k-1$ holds up to a $(j+1)$-isomorphism. Of course, if we increase $k$, we grasp more and more information about $T$. The {\it homotopy principle} says that the weak fundamental infinity groupoid $\Pi_{\infty}(T)$ recognizes $T$, or, more precisely, that {\it $(\infty,0)$-categories are the same as topological spaces}. Hence,

\begin{ex} $\infty$-categories, i.e. $(\infty,1)$-categories, should contain topological spaces.\end{ex}

\paragraph{Kan complexes, quasi-categories, nerves of groupoids and of categories}

Let us recall that the {\it nerve functor} $N:{\tt Cat}\rightarrow {\tt SSet}$, provides a fully faithful embedding of the category ${\tt Cat}$ of all (small) categories into ${\tt SSet}$ and remembers not only the objects and morphisms, but also the compositions. It associates to any ${\tt C}\in{\tt Cat}$ the simplicial set $$(N{\tt C})_n=\{C_0\to C_1\to\ldots\to C_n\}\;,$$ where the sequence in the {\small RHS} is a sequence of composable {\tt C}-morphisms between objects $C_i\in{\tt C}$; the face (resp., the degeneracy) maps are the compositions and insertions of identities. Let us also recall that the {\it $r$-horn} $\Lambda^r[n]$, $0\le r\le n$, of $\Delta[n]$ is `the part of the boundary of $\Delta[n]$ that is obtained by suppressing the interior of the $(n-1)$-face opposite to $r$'. More precisely, the $r$-horn $\Lambda^r[n]$ is the simplicial set, whose nondegenerate $k$-simplices are the injective order-respecting maps $[k]\to [n]$, except the identity and the map $\delta^r:[n-1]\rightarrowtail [n]$ whose image does not contain $r$.\medskip

We now detail four different situations based on the properties `Any (inner) horn admits a (unique) filler'.

\begin{defi} A simplicial set $S\in{\tt SSet}$ is \emph{fibrant} and called a \emph{Kan complex}, if the map $S\to \star$, where $\star$ denotes the terminal object, is a Kan fibration, i.e. has the right lifting property with respect to all canonical inclusions $\Lambda^r[n]\subset \Delta[n]$, $0\le r\le n$, $n>0$. In other words, $S$ is a Kan complex, if {any horn} $\Lambda^r[n]\to S$ can be extended to an $n$-simplex $\Delta[n]\to S$, i.e. if {\sf any horn in $S$ admits a filler}.\end{defi}

The following result is well-known and explains that a simplicial set is a nerve under a quite similar extension condition.

\begin{prop} A simplicial set $S$ is the nerve $S\simeq N{\tt C}$ of some category ${\tt C}$, if and only if {\sf any inner horn $\Lambda^r[n]\to S$, $0<r<n$, has a unique filler $\Delta[n]\to S$}. \end{prop}

Indeed, it is quite obvious that for $S=N{\tt C}\in{\tt SSet}$, an inner horn $\Lambda^1[2]\to N{\tt C}$, i.e. two {\tt C}-morphisms $f: C_0\to C_1$ and $g:C_1\to C_2$, has a unique filler $\Delta[2]\to N{\tt C}$, given by the edge $h=g\circ f:C_0\to C_2$ and the `homotopy' $\mathrm{id}:h\Rightarrow g\circ f$ (1).\smallskip

As for Kan complexes $S\in{\tt SSet}$, the filler property for an outer horn $\Lambda^0[2]\to S$ (resp., $\Lambda^2[2]\to S$) implies for instance that a horn $f:s_0\to s_1$, $\mathrm{id}:s_0\to s_2=s_0$ (resp., $\mathrm{id}: s'_0\to s'_2=s_0'$, $g: s'_1\to s'_2$) has a filler, so that any map has a `left (resp., right) inverse' (2).\medskip

It is clear that simplicial sets $S_0,S_1,S_2,\,\ldots$ are candidates for $\infty$-categories. In view of the last remark (2), Kan complexes model $\infty$-groupoids. Hence, fillers for outer horns should be omitted in the definition of $\infty$-categories. On the other hand, $\infty$-categories do contain homotopies $\eta: h\Rightarrow g\circ f$, so that, due to (1), uniqueness of fillers is to be omitted as well. Hence, the

\begin{defi} A simplicial set $S\in{\tt SSet}$ is an \emph{$\infty$-category} if and only if {\sf any inner horn $\Lambda^r[n]\to S$, $0<r<n$, admits a filler $\Delta[n]\to S$}. \end{defi}

We now also understand the

\begin{prop} A simplicial set $S$ is the nerve $S\simeq N{\tt G}$ of some groupoid ${\tt G}$, if and only if {\sf any horn $\Lambda^r[n]\to S$, $0\le r\le n$, $n>0$, has a unique filler $\Delta[n]\to S$}.\end{prop}

Of course, Kan complexes, i.e. $\infty$-groupoids, $(\infty,0)$-categories, or, still, topological spaces, are $\infty$-categories. Moreover, nerves of categories are $\infty$-categories. Hence, the requirement that topological spaces and ordinary categories should be $\infty$-categories are satisfied. Note further that what we just defined is a model for $\infty$-categories called {\it quasi-categories} or {\it weak Kan complexes}.

\paragraph{Link with the intuitive picture of an infinity category}\label{InftyCatComp}

In the following, we explain that the preceding model of an $\infty$-category actually corresponds to the intuitive picture of an $(\infty,1)$-category, i.e. that in an $\infty$-category all types of morphisms do exist, that all $j$-morphisms, $j>1$, are invertible, and that composition of morphisms is defined and is associative only up to homotopy. This will be illustrated by showing that any $\infty$-category has a homotopy category, which is an ordinary category.\medskip

We denote simplicial sets by $S, S',\ldots$, categories by ${\tt C},{\tt D},\dots$, and $\infty$-categories by ${\tt S},{\tt S'},\ldots$\medskip

Let {\tt S} be an $\infty$-category. Its {\it 0-morphisms} are the elements of ${\tt S}_0$ and its {\it 1-morphisms} are the elements of ${\tt S}_1$. The {\it source} and {\it target} maps $\sigma,\tau$ are defined, for any 1-morphism $f\in {\tt S}_1$, by $\sigma f=d_1f\in{\tt S}_0$, $\tau f=d_0f\in{\tt S}_0$, and the {\it identity map} is defined, for any 0-morphism $s\in {\tt S}_0$, by $\mathrm{id}_s=s_0s\in {\tt S}_1$, with self-explaining notation. In the following, we denote a 1-morphism $f$ with source $s$ and target $s'$ by $f:s\to s'$. In view of the simplicial relations, we have $\sigma \mathrm{id}_s=d_1s_0s=s$ and $\tau \mathrm{id}_s=d_0s_0s=s$, so that $\mathrm{id}_s:s\to s$.\medskip

Consider now two morphisms $f:s\to s'$ and $g:s'\to s''$. They define an inner horn $\Lambda^1[2]\to {\tt S}$, which, as {\tt S} is an $\infty$-category, admits a filler $\phi:\Delta[2]\to {\tt S}$, or $\phi\in {\tt S}_2$. The face $d_1\phi\in{\tt S}_1$ is of course a candidate for the composite $g\circ f.$

\begin{rem}\label{CompInftyCat} Since the face $h:=d_1\phi$ of any filler $\phi$ is a (candidate for the) composite $g\circ f$, composites of morphisms are in $\infty$-categories not uniquely defined. We will show that they are determined only up to `homotopy'.\end{rem}

\begin{defi}\label{2MorphInftyCat} Let {\tt S} be an $\infty$-category and let $f,g:s\to s'$ be two morphisms. A \emph{2-morphism} or \emph{homotopy} $\phi:f\Rightarrow g$ between $f$ and $g$ is an element $\phi\in{\tt S}_2$ such that $d_0\phi=g, d_1\phi=f, d_2\phi=\mathrm{id}_{s}$.\end{defi}

Indeed, if there exists such a 2-simplex $\phi$, there are two candidates for the composite $g\circ \mathrm{id}_{s}$, namely $f$ and, of course, $g$. If we wish now that all the candidates be homotopic, the existence of $\phi$ must entail that $f$ and $g$ are homotopic -- which is the case in view of Definition \ref{2MorphInftyCat}. If $f$ is homotopic to $g$, we write $f\simeq g$.

\begin{prop} The homotopy relation $\simeq$ is an equivalence in ${\tt S}_1$.\end{prop}

\begin{proof} Let $f:s\to s'$ be a morphism and consider $\mathrm{id}_f:=s_0f\in{\tt S}_2$. It follows from the simplicial relations that $d_0\mathrm{id}_f=f,d_1\mathrm{id}_f=f,d_2\mathrm{id}_f=s_0s=\mathrm{id}_s$, so that $\mathrm{id}_f$ is a homotopy between $f$ and $f$. To prove that $\simeq$ is symmetric, let $f,g:s\to s'$ and assume that $\phi$ is a homotopy from $f$ to $g$. We then have an inner horn ${\psi}:\Lambda^2[3]\to {\tt S}$ such that $d_0\psi=\phi$, $d_1\psi=\mathrm{id}_g$, and $d_3\psi=\mathrm{id}_{\mathrm{id}_s}=:\mathrm{id}^2_s$. The face $d_2\Psi$ of a filler $\Psi:\Delta[3]\to {\tt S}$ is a homotopy from $g$ to $f$. Transitivity can be obtained similarly. \end{proof}

\begin{defi} The \emph{homotopy category} ${\tt Ho}(\tt S)$ of an $\infty$-category $\,{\tt S}$ is the (ordinary) category with objects the objects $s\in {\tt S}_0$, with morphisms the homotopy classes $[f]$ of morphisms $f\in {\tt S}_1$, with composition $[g]\circ [f]=[g\circ f]$, where $g\circ f$ is any candidate for the composite in ${\tt S}$, and with identities $\mathrm{Id}_s=[\mathrm{id}_s]$.\end{defi}

To check that this definition makes sense, we must in particular show that all composites $g\circ f$, see Remark \ref{CompInftyCat}, are homotopic. Let thus $\phi_1,\phi_2\in{\tt S}_2$ be two 2-simplices such that $(d_0\phi_1,d_1\phi_1,d_2\phi_1)=(g,h_1,f)$ and $(d_0\phi_2,d_1\phi_2,d_2\phi_2)=(g,h_2,f)$, so that $h_1$ and $h_2$ are two candidates. Consider now for instance the inner horn $\psi:\Lambda^2[3]\to {\tt S}$ given by $\psi=(\phi_1,\phi_2,\bullet,\mathrm{id}_f)$. The face $d_2\Psi$ of a filler $\Psi:\Delta[3]\to {\tt S}$ is then a homotopy from $h_2$ to $h_1$. To prove that the composition of morphisms in ${\tt Ho}({\tt S})$ is associative, one shows that candidates for $h\circ (g\circ f)$ and for $(h\circ g)\circ f$ are homotopic (we will prove neither this fact, nor the additional requirements for ${\tt Ho}({\tt S})$ to be a category).

\begin{rem} It follows that in an $\infty$-category composition of morphisms is defined and associative only up to homotopy.\end{rem}

We now comment on higher morphisms in $\infty$-categories, on their composites, as well as on invertibility of $j$-morphisms, $j>1$.

\begin{defi} Let $\phi_1:f\Rightarrow g$ and $\phi_2:f\Rightarrow g$ be 2-morphisms between morphisms $f,g:s\to s'$. A \emph{3-morphism} $\Phi:\phi_1\Rrightarrow \phi_2$ is an element $\Phi\in{\tt S}_3$ such that $d_0\Phi= \mathrm{id}_g$, $d_1\Phi =\phi_2$, $d_2\Phi =\phi_1$, and $d_3\Phi= \mathrm{id}^2_s$.\end{defi} Roughly, a 3-morphism is a 3-simplex with faces given by sources and targets, as well as by identities. Higher morphisms are defined similarly \cite{Gro10}.\medskip

As concerns composition and invertibility, let us come back to transitivity of the homotopy relation. There we are given 2-morphisms $\phi_1:f\Rightarrow g$ and $\phi_2:g\Rightarrow h$, and must consider the inner horn $\psi=(\phi_2,\bullet,\phi_1,\mathrm{id}^2_s)$. The face $d_1\Psi$ of a filler $\Psi$ is a homotopy between $f$ and $h$ and is a candidate for the composite $\phi_2\circ\phi_1$ of the 2-morphisms $\phi_1,\phi_2$. If we now look again at the proof of symmetry of the homotopy relation and denote the homotopy from $g$ to $f$ by $\psi'$, we see that $\psi\circ\psi'\simeq\mathrm{id}_g$. We obtain similarly that $\psi'\circ\psi\simeq\mathrm{id}_f$, so that 2-morphisms are `invertible'.

\begin{rem} Eventually, all the requirements of the intuitive picture of an $\infty$-category are (really) encoded in the existence of fillers of inner horns.\end{rem}

\subsubsection{Infinity groupoid of infinity morphisms between Leibniz infinity algebras}\label{sec_VcoH}

\paragraph{Quasi-category of homotopy Leibniz algebras}\label{InftyCatInftyAlg}

Let $\Omega^\star_\bullet$ be the {\small SDGCA} introduced in Subsection \ref{SimpicialDR}. The `Yoneda embedding' of $\Omega^\star_\bullet$ viewed as object of ${\tt SSet}$ and ${\tt DGCA}$, respectively, gives rise to an adjunction that is well-known in Rational Homotopy Theory: $$\Omega^\star:{\tt SSet}\rightleftarrows {\tt DGCA}^{\mathrm{op}}:\mathrm{Spec}_{\bullet}\;.$$ The functor $\Omega^\star=\mathrm{Hom}_{\tt SSet}(-,\Omega^\star_\bullet)=:\mathrm{SSet}(-,\Omega^\star_\bullet)$ associates to any $S_\bullet\in{\tt SSet}$ its Sullivan {\small DGCA} $\Omega^\star(S_\bullet)$ of piecewise polynomial differential forms, whereas the functor $\mathrm{Spec}_{\bullet}=\mathrm{Hom}_{\tt DGCA}(-,\Omega^\star_\bullet)$ assigns to any $A\in {\tt DGCA}$ its simplicial spectrum $\mathrm{Spec}_{\bullet}(A)$.\medskip

Remember now that an $\infty$-homotopy between $\infty$-morphisms between two Leibniz infinity algebras $V,W$, is an element in $\mathrm{MC}_1({\bar L})=\mathrm{MC}(L\otimes\Omega^\star_1)$, where $L=L(V,W)$.\medskip

The latter set is well-known from integration of $L_{\infty}$-algebras. Indeed, when looking for an integrating topological space or simplicial set of a positively graded $L_{\infty}$-algebra $L$ of finite type (degree-wise of finite dimension), it is natural to consider the simplicial spectrum of the corresponding quasi-free {\small DGCA} $\mathrm{Com}(s^{-1}L^*)$. The dual of Equation (\ref{NatBijCoAlg}) yields $$\mathrm{Spec}_\bullet(\mathrm{Com}(s^{-1}L^*))=\mathrm{Hom}_{\tt DGCA}(\mathrm{Com}(s^{-1}L^*),\Omega^\star_\bullet)\simeq \mathrm{MC}(L\otimes \Omega^\star_\bullet)\;.$$ The integrating simplicial set of a nilpotent $L_{\infty}$-algebra $L$ is actually homotopy equivalent to $\mathrm{MC}_\bullet({\bar L}):=\mathrm{MC}(L\otimes\Omega^\star_\bullet)$ \cite{Get09}. It is clear that the structure maps of $\mathrm{MC}_\bullet({\bar L})\subset L\otimes\Omega^\star_\bullet$ are $\tilde{d}^n_i=\mathrm{id}\otimes d^n_i$ and $\tilde{s}^n_i=\mathrm{id}\otimes s^n_i$, where $d^n_i$ and $s^n_i$ were described in Subsection \ref{SimpicialDR}.\medskip

Higher homotopies ($n$-homotopies) are usually defined along the same lines as standard homotopies (1-homotopies), i.e., e.g., as arrows depending on parameters in $I^{\times n}$ (or $\Delta^n$) instead of $I$ (or $\Delta^1$) \cite{Lei04}. Hence,

\begin{defi} \emph{$\infty$-$n$-homotopies} $($\emph{$\infty$-$(n+1)$-morphisms}$)$ between given Leibniz infinity algebras $V,W$ are Maurer-Cartan elements in $\mathrm{MC}_n(\bar L)=\mathrm{MC}(L\otimes\Omega^\star_n)$, where $L=L(V,W)$ and $n\ge 0$.\label{InftyNHomot}\end{defi}

Indeed, $\infty$-1-morphisms are just elements of $\mathrm{MC}(L)$, i.e. standard $\infty$-morphisms between $V$ and $W$.\medskip

Note that if $\tt S$ is an $\infty$-category, the set of $n$-morphisms, with varying $n\ge 1$, between two fixed objects $s,s'\in{\tt S}_0$ can be shown to be a Kan complex \cite{Gro10}. The simplicial set $\mathrm{MC}_\bullet(\bar L)$, whose $(n-1)$-simplices are the $\infty$-$n$-morphisms between the considered Leibniz infinity algebras $V,W$, $n\ge 1$, is known to be a Kan complex ($(\infty,0)$-category) as well \cite{Get09}.

\begin{rem} We interpret this result by saying that Leibniz infinity algebras and their infinity higher morphisms form an $\infty$-category $($$(\infty,1)$-category$)$. Further, as mentioned above and detailed below, composition of homotopies is encrypted in the Kan property.\end{rem}

Note that $\mathrm{MC}_\bullet(\bar L)$ actually corresponds to the `d\'ecalage', the `down-shifting', of the simplicial set $\tt S$.\smallskip

Let us also \emph{emphasize} that Getzler's results are valid only for nilpotent $L_{\infty}$-algebras, hence in principle not for $L$, which is only complete (an $L_{\infty}$-algebra is \emph{pronilpotent}, if it is complete with respect to its lower central series, i.e. the intersection of all its compatible filtrations, and it is \emph{nilpotent}, if its lower central series eventually vanishes). However, for our remaining concern, namely the explanation of homotopies and their compositions in the 2-term Leibniz infinity algebra case, this difficulty is irrelevant. Indeed, when interpreting the involved series as formal ones and applying the thus obtained results to the 2-term case, where series become finite for degree reasons, we recover the results on homotopies and their compositions conjectured in \cite{BC04}. An entirely rigorous approach to these issues is being examined in a separate paper: it is rather technical and requires applying Henriques' method or working over an arbitrary local Artinian algebra.






\paragraph{Kan property}\label{Kan property}

Considering our next purposes, we now review and specify the proof of the Kan property of $\mathrm{MC}_\bullet(\bar L)$ \cite{Get09}. As announced above, to facilitate comparison, we adopt the conventions of the latter paper and apply the results mutatis mutandis and formally to our situation. In particular, we work in \ref{Kan property} with the cohomological version of Lie infinity algebras ($k$-ary bracket of degree $2-k$), together with Getzler's sign convention for the higher Jacobi conditions, see \ref{Gauge}, and assume that $\mathbb{K}=\mathbb{R}$.\medskip

Let us first recall that the lower central filtration of $(L,{\cal L}_i)$ is given by $F^1L=L$ and
$$F^iL=\sum_{i_1+\cdots+i_k=i}{\cal L}_k(F^{i_1}L,\cdots,F^{i_k}L),~~i>1\;.$$ In particular, $F^2L={\cal L}_2(L,L)$, $F^3L={\cal L}_2(L,{\cal L}_2(L,L))+{\cal L}_3(L,L,L)$, ..., so that $F^kL$ is spanned by all the nested brackets containing $k$ elements of $L$. Due to nilpotency, $F^iL=\{0\}$, for $i\gg$.\medskip

To simplify notation, let $\delta$ be the differential ${\cal L}_1$ of $L$, let $d$ be the de Rham differential of $\Omega^\star_n=\Omega^\star(\Delta^n)$, and let $\bar\delta+\bar d$ be the differential $\bar{\cal L}_1=\delta\otimes \mathrm{id}+\mathrm{id}\otimes d$ of $L\otimes\Omega^\star(\Delta^n)$. Set now, for any $n\ge 0$ and any $0\le i\le n$,
$$\mathrm{mc}_n(\bar L):=\{(\bar\delta+\bar d)\beta:\beta \in (L\otimes\Omega^{\star}(\Delta^n))^0\}\;\text{and}\; \mathrm{mc}_n^i(\bar L):=\{(\bar\delta+\bar d)\beta:\beta \in (L\otimes\Omega^{\star}(\Delta^n))^0, \;\bar\varepsilon^i_n\beta=0\}\;,$$ where $\bar\varepsilon_n^i:=\mathrm{id}\otimes \varepsilon_n^i$ is the canonical extension of the evaluation map $\varepsilon_n^i:\Omega^\star(\Delta^n)\to \mathbb{K}$, see \ref{SimpicialDR}. \begin{rem} In the following, we use the extension symbol `bar' only when needed for clarity.\end{rem}

$\bullet\quad$  There exist fundamental bijections \begin{eqnarray} B^i_n: \mathrm{MC}_n(\bar L)\stackrel{\sim}{\longrightarrow} \mathrm{MC}(L)\times \mathrm{mc}_n^i(\bar L)\subset \mathrm{MC}(L)\times\mathrm{mc}_n(\bar L)\label{important_bijection}\;.\end{eqnarray}

The proof uses the operators $$h_n^i:~\Omega^{\star}(\Delta^n)\to \Omega^{\star-1}(\Delta^n)$$ defined as follows. Let $\vec t=[t_0,\ldots, t_n]$ be the coordinates of $\Delta^n$ (with $\sum_i t_i=1$) and consider the maps $\phi_n^i:~I\times\Delta^n\ni (u,\vec t)\mapsto u\vec t+(1-u)\vec e_i\in \Delta^n$. They allow to pull back a polynomial differential form on $\Delta^n$ to a polynomial differential form on $I\times \Delta^n$. The operators $h_n^i$ are now given by \begin{eqnarray} h_n^i\omega=\int_I~(\phi_n^i)^*\omega\;.\nonumber\end{eqnarray} They satisfy the relations
\begin{eqnarray} \{d,h_n^i\}=\mathrm{id}_n-\varepsilon_n^i,~~~~\{h_n^i,h_n^j\}=0,~~~~~\varepsilon_n^ih_n^i=0\;,\label{ibp_proper}\end{eqnarray} where $\{-,-\}$ is the graded commutator (remember that $\varepsilon_n^i$ vanishes in nonzero cohomological degree). The first relation is a higher dimensional analogue of $$\{d,\int_0^t\}\omega=\{d,\int_0^t\}(f(u)+g(u)du)=d\int_0^tg(u)du+\int_0^td_uf\,du=g(t)dt+f(t)-f(0)=\omega-\varepsilon_1^0\omega\;,$$ where $\omega\in\Omega^\star(I).$\medskip

The natural extensions of $d,$ $h_n^i,$ and $\varepsilon_n^i$ to $L\otimes\Omega^\star(\Delta^n)$ verify the same relations, and, since we obviously have $\delta h_n^i=-\,{h}_n ^i\delta$, the first relation holds in the extended setting also for $d$ replaced by $\delta+d$.\medskip

Define now $B_n^i$ by
\begin{equation} B_n^i:\mathrm{MC}_n(\bar L)\ni\alpha\mapsto B_n^i\alpha:=(\varepsilon_n^i\alpha,(\delta+d)h_n^i\alpha)\in \mathrm{MC}(L)\times \mathrm{mc}_n^i(\bar L)\;.\label{BDir}\end{equation} Observe that $\alpha\in (L\otimes\Omega^\star(\Delta^n))^1$ reads $\alpha=\sum_{k=0}^n\alpha^k$, $\alpha^k\in L_{1-k}\otimes\Omega^k(\Delta^n)$, so that $\varepsilon_n^i\alpha=\varepsilon_n^i\alpha^0\in L_1$. Moreover, it follows from the definition of the extended $L_{\infty}$-maps $\bar{\cal L}_i$ that \begin{equation}\label{EpsilonMC}\sum_{i\ge 1}\frac{1}{i!}{\cal L}_i(\varepsilon_n^i\alpha,\ldots,\varepsilon_n^i\alpha)=\varepsilon_n^i\sum_{i\ge 1}\frac{1}{i!}\bar{\cal L}_i(\alpha,\ldots,\alpha)=0\;.\end{equation} In view of the last equation (\ref{ibp_proper}), the second component of $B_n^i\alpha$ is clearly an element of $\textrm{mc}_n^i(\bar L)$.\medskip

The construction of the inverse map is based upon a method similar to the iterative approximation procedure that allows to prove the fundamental theorem of {\small ODE}-s. More precisely, consider the Cauchy problem $y\,'(t)=F(t,y(t))$, $y(0)=Y$, i.e. the integral equation $$y(s)=Y+\int_0^sF(t,y(t))dt\;.$$ Choose now the `Ansatz' $y_0(s)=Y$ and define inductively $$y_k(s)=Y+\int_0^sF(t,y_{k-1}(t))dt\;,$$ $k\ge 1$. It is well-known that the $y_k$ converge to a function $y$, which is the unique solution and depends continuously on the initial value $Y$.\medskip

Note now that, if we are given $\mu\in\textrm{MC}(L)$ and $\nu=(\delta+d)\beta\in \textrm{mc}^i_n(\bar L)$, a solution $\alpha\in\mathrm{MC}_n(\bar L)$ -- i.e. an element $\alpha\in(L\otimes\Omega^\star(\Delta^n))^1$ that satisfies $$(\delta+d)\alpha+\sum_{i\ge 2}\frac{1}{i!}\bar{\cal L}_i(\alpha,\ldots,\alpha)=:(\delta+d)\alpha+\bar{R}(\alpha)=0\;\;\;\text{--}\;$$ such that $\varepsilon_n^i\alpha=\mu$ and $(\delta+d)h_n^i\alpha=\nu$, also verifies the integral equation \begin{equation}\alpha=\mathrm{id}_n\alpha=\{\delta+d,h_n^i\}\alpha+\varepsilon_n^i\alpha=\mu+\nu+h_n^i(\delta+d)\alpha=\mu+\nu-h_n^i\bar R(\alpha)\;.\label{IntEq}\end{equation} We thus choose the `Ansatz' $\alpha_0=\mu+\nu$ and set $\alpha_{k}=\alpha_0-h_n^i\bar R(\alpha_{k-1}),$ $k\ge 1$. It is easily seen that, in view of nilpotency, this iteration stabilizes, i.e. $\alpha_{k-1}=\alpha_k=\ldots =:\alpha$, for $k\gg$, or, still, \begin{eqnarray} \alpha=\alpha_0-h_n^i\bar R(\alpha)\;.\label{stable}\end{eqnarray} The limit $\alpha$ is actually a solution in $\mathrm{MC}_n(\bar L)$. Indeed, remember first that the generalized curvature
\begin{eqnarray} \bar{\cal F}(\alpha)=(\delta+d)\alpha+\bar{R}(\alpha)=(\delta+d)\alpha+\sum_{i\ge 2}\frac{1}{i!}\bar{\cal L}_i(\alpha,\ldots,\alpha)\;,\nonumber\end{eqnarray}
whose zeros are the {\small MC} elements, satisfies, just like the standard curvature, the Bianchi identity
\begin{eqnarray} (\delta+d)\bar{\cal F}(\alpha)+\sum_{k\ge 1}\frac{1}{k!}\bar{\cal L}_{k+1}(\alpha,\ldots,\alpha,\bar{\cal F}(\alpha))=0\;.\label{biancchi}\end{eqnarray}
It follows from (\ref{stable}) and (\ref{ibp_proper}) that
\begin{eqnarray} \bar{\cal F}(\alpha)=(\delta+d)(\alpha_0-h_n^i\bar R(\alpha))+\bar{R}(\alpha)=(\delta+d)\mu+h_n^i(\delta+d)\bar R(\alpha)+\varepsilon_n^i\bar R(\alpha)\;.\nonumber\end{eqnarray} From Equation (\ref{EpsilonMC}) we know that $\varepsilon_n^i\bar R(\alpha)=R(\varepsilon_n^i\alpha)=R(\mu)$, with self-explaining notation. As for $\varepsilon_n^i\alpha=\mu$, note that $\varepsilon_n^i\mu=\mu$ and that $$\varepsilon_n^i\nu=\varepsilon_n^i(\delta+d)\beta=\varepsilon_n^i(\delta+d)\sum_{k=0}^n\beta^k,\; \beta^k\in L_{-k}\otimes\Omega^k(\Delta^n),\;\text{ so that }\; \varepsilon_n^i\nu=\varepsilon_n^i\delta\beta^0=\delta\varepsilon_n^i\beta^0=\delta\varepsilon_n^i\beta=0\;.$$ Hence, $$\bar{\cal F}(\alpha)=(\delta+d)\mu+h_n^i(\delta+d)(\bar{\cal F}(\alpha)-(\delta+d)\alpha)+R(\mu)$$ $$={\cal F}(\mu)+h_n^i(\delta+d)\bar{\cal F}(\alpha)=-h_n^i\sum_{k\ge 1}\frac{1}{k!}\bar{\cal L}_{k+1}(\alpha,\ldots,\alpha,\bar{\cal F}(\alpha))\;,$$ in view of (\ref{biancchi}). Therefore, $\bar{\cal F}(\alpha)\in F^i\bar L$, for arbitrarily large $i$, and thus $\alpha\in\mathrm{MC}_n(\bar L)$. This completes the construction of maps \begin{equation}\label{BInv}{\cal B}_n^i:\mathrm{MC}(L)\times \mathrm{mc}_n^i(\bar L)\to \mathrm{MC}_n(\bar L)\;.\end{equation}

We already observed that $\varepsilon_n^i{\cal B}_n^i(\mu,\nu)=\varepsilon_n^i\alpha=\mu$. In fact, $B_n^i{\cal B}_n^i=\mathrm{id}$, so that $B_n^i$ is surjective. Indeed, Equations (\ref{stable}) and (\ref{ibp_proper}) imply that $$(\delta+d)h_n^i\alpha=-h_n^i(\delta+d)\alpha_0+\alpha_0-\varepsilon_n^i\alpha_0=-h_n^i\delta\mu+\nu=\nu\;.$$ As for injectivity, if $B_n^i\alpha=B_n^i\alpha'=:(\mu,\nu)$, then both, $\alpha$ and $\alpha'$, satisfy Equation (\ref{IntEq}). It is now quite easily seen that nilpotency entails that $\alpha=\alpha'$.\bigskip

$\bullet\quad$  The bijections $$B_n^i:\mathrm{MC}_n(\bar L)\to \mathrm{MC}(L)\times \mathrm{mc}_n^i(\bar L)$$ allow to prove the Kan property for $\mathrm{MC}_\bullet(\bar L)$. The extension of a horn in $\mathrm{SSet}(\Lambda^i[n],\mathrm{MC}_\bullet(\bar L))$ will be performed as sketched by the following diagram:
\begin{eqnarray}
\xymatrix{
\textrm{SSet}(\Lambda^i[n],\textrm{MC}_{\bullet}(\bar L)) \ar@{.>}^-{ }[r]\ar@{->}_-{ }[d]&\textrm{MC}_n(\bar L)\\
\textrm{SSet}(\Lambda^i[n],\textrm{MC}(L)\times\textrm{mc}_{\bullet}(\bar L))\ar@{->}^-{}[r]&\textrm{MC}(L)\times\textrm{mc}^i_n(\bar L)\ar@{->}_-{}[u]
}\label{commute_square}\end{eqnarray}

Of course, the right arrow is nothing but ${\cal B}_n^i$.\medskip

$\star\quad$  As for the left arrow, imagine, for simplicity, that $i=1$ and $n=2$, and let $$\alpha\in \mathrm{SSet}(\Lambda^1[2], \mathrm{MC}_\bullet(\bar L))\;.$$ The restrictions $\alpha|_{01}$ and $\alpha|_{12}$ to the 1-faces $01$, $12$ (compositions of the natural injections with $\alpha$) are elements in $\mathrm{MC}_1(\bar L)$, so that the map $B_1^1$ sends $\alpha|_{01}$ to $(\mu,\nu)$ in $\mathrm{MC}(L)\times\mathrm{mc}_1(\bar L)$ (and similarly $B_1^0(\alpha|_{12})=(\mu',\nu')\in \mathrm{MC}(L)\times\mathrm{mc}_1(\bar L)$). Of course, $\mu=\varepsilon_1^1(\alpha|_{01})=\varepsilon_1^0(\alpha|_{12})=\mu'$. Since $\nu=(\delta+d)\beta$ and $\beta(1)=\varepsilon_1^1\beta=0$, we find $\nu(1)=\varepsilon_1^1\nu=0$ (and similarly $\nu'=(\delta+d)\beta'$ and $\beta'(1)=\nu'(1)=0$). Thus, $$(\mu;\nu,\nu')\in\mathrm{SSet}(\Lambda^1[2],\mathrm{MC}(L)\times\mathrm{mc}_\bullet(\bar L))\;,$$ which explains the left arrow.\medskip

$\star\quad$ For the bottom arrow, let again $i=1,n=2$. Since $\mu$ is constant, it can be extended to the whole simplex. To extend $(\nu,\nu')$, it actually suffices to extend $(\beta,\beta')$. Indeed, restriction obviously commutes with $\delta$. As for commutation with $d$, remember that $\Omega^\star:\Delta^{\mathrm{op}}\to {\tt DGCA}$ and that the {\tt DGCA}-map $d^2_2=\Omega^\star(\delta^2_2)$ sets the component $t_2$ to 0. Hence, $d^2_2$ coincides with restriction to $01$ and commutes with $d$. Let now $\bar\beta$ be an extension of $(\beta,\beta')$. Since $$(d\bar\beta)|_{01}=d^2_2d\bar\beta=dd^2_2\bar\beta=d\beta$$ and similarly $(d\bar\beta)|_{12}=d\beta'$.\medskip

It now remains to explain that an extension $\bar\beta$ does always exist. Consider the slightly more general extension problem of three polynomial differential forms $\beta_0$, $\beta_1,$ and $\beta_2$ defined on the 1-faces $12$, $02$, and $01$ of the 2-simplex $\Delta^2$, respectively (it is assumed that they coincide at the vertices). Let $\pi_2:\Delta^2\to 01$ be the projection defined, for any $\vec t=[t_0,t_1,t_2]$, as the intersection of the line $u\vec t+(1-u)\vec e_2$ with $01$. This projection is of course ill-defined at $\vec t=\vec e_2$. In coordinates, we get $$\pi_2:[t_0,t_1,t_2]\mapsto [t_0/(1-t_2), t_1/(1-t_2)]\;.$$ It follows that the pullback $\pi_2^*\beta_2$ is a rational differential form with denominator $(1-t_2)^N$, for some integer $N.$ Hence, $$\gamma_2:=(1-t_2)^N\pi_2^*\beta_2$$ is a polynomial differential form on $\Delta^2$ that coincides with $\beta_2$ on $01$. It now suffices to solve the same extension problem as before, but for the forms $\beta_0-\gamma_2|_{12}$, $\beta_1-\gamma_2|_{02}$, and $0$. When iterating the procedure -- due to Renshaw \cite{Sul77} --, the problem reduces to the extension of $0,0,0$ (since the pullback preserves 0). This completes the description of the bottom arrow, as well as the proof of the Kan property of $\mathrm{MC}_\bullet(\bar L)$.

\subsection{2-Category of 2-term Leibniz infinity algebras}\label{KPQSection5}

Categorification replaces sets (resp., maps, equations) by categories (resp., functors, natural isomorphisms). In particular, rather than considering two maps as equal, one details a way of identifying them. Categorification is thus a sharpened viewpoint that turned out to provide deeper insight. This motivates the interest in e.g. categorified algebras (and in truncated infinity algebras -- see below).\medskip

Categorified Lie algebras were introduced under the name of Lie 2-algebras in [BC04] and further studied in \cite{Roy07}, \cite{SL10}, and \cite{KMP11}. The main result of \cite{BC04} states that Lie 2-algebras and 2-term Lie infinity algebras form equivalent 2-categories. However, {\bf infinity homotopies of 2-term Lie infinity algebras} (resp., {\bf compositions of such homotopies}) {are not explained, but appear as some God-given natural transformations read through this} {\textsc{equivalence}} (resp., compositions are addressed only in \cite{SS07Structure} and performed in the {\textsc{algebraic or coalgebraic settings}}).\medskip

This circumstance is not satisfactory, and {\bf the attempt to improve our understanding of infinity homotopies and their compositions is one of the main concerns of the present paper}. Indeed, in \cite{KMP11} (resp., \cite{BP12}), the authors show that the {\textsc{equivalence}} between $n$-term Lie infinity algebras and Lie $n$-algebras is, for $n> 2$, not as straightforward as expected -- which is essentially due to the largely ignored fact that the category {\tt Vect} $n${\tt -Cat} of linear $n$-categories is symmetric monoidal, but that the corresponding map $\boxtimes:L\times L'\to L\boxtimes L'$ is not an $n$-functor (resp., that the understanding of a concept in the {\textsc{algebraic framework}} is far from implying its comprehension in the infinity context -- a reality that is corroborated e.g. by the comparison of concordances and infinity homotopies).\medskip

In this section, we obtain {\bf explicit formulae for infinity homotopies and their compositions}, applying the \textsc{Kan property} of $\mathrm{MC}_\bullet(\bar L)$ to the 2-term case, thus staying inside the \textsc{infinity setting}.

\subsubsection{Category of 2-term Leibniz infinity algebras}

For the sake of completeness, we first describe 2-term Leibniz infinity algebras and their morphisms. Propositions \ref{2term_Loday_in_algebra} and \ref{2term_Loday_morphism} are specializations to the 2-term case of Definitions \ref{LeibInftyAlg} and \ref{LeibInftyAlgMorph}; see also \cite{SL10}. The informed reader may skip the present subsection.

\begin{prop}\label{2term_Loday_in_algebra} A \emph{2-term Leibniz infinity algebra} is a graded vector space $V=V_0\oplus V_1$ concentrated in degrees 0 and 1, together with a linear, a bilinear, and a trilinear map $l_1,$ $l_2,$  and $l_3$ on $V$, of degree $|l_1|=-1,$ $|l_2|=0,$ and $|l_3|=1,$ which verify, for any $w,x,y,z\in V_0$ and $h,k\in V_1$,
\begin{enumerate}[(a)]
\item \label{identity_a} $l_1l_2(x,h)=l_2(x,l_1h)\;,$\\
$l_1l_2(h, x) = l_2(l_1h, x)\;,$
\item \label{identity_b}$l_2(l_1h,k)=l_2(h,l_1k)\;,$
\item \label{identity_c}$l_1l_3(x,y,z)=l_2(x,l_2(y,z))-l_2(y,l_2(x,z))-l_2(l_2(x,y),z)\;,$
\item \label{identity_d}$l_3(x,y,l_1h)=l_2(x,l_2(y,h))-l_2(y,l_2(x,h))-l_2(l_2(x,y),h)\;,$\\
$l_3(x, l_1h, y) = l_2(x, l_2(h, y)) - l_2(h, l_2(x, y)) - l_2(l_2(x,h), y)\;,$\\
$l_3(l_1h, x, y) = l_2(h, l_2(x, y)) - l_2(x, l_2(h, y)) - l_2(l_2(h, x), y)\;,$
\item \label{identity_e}$l_2(l_3(w,x,y),z)+l_2(w,l_3(x,y,z))-l_2(x,l_3(w,y,z))+l_2(y,l_3(w,x,z))-l_3(l_2(w,x),y,z)$\\
$+l_3(w,l_2(x,y),z)-l_3(x,l_2(w,y),z)-l_3(w,x,l_2(y,z))+l_3(w,y,l_2(x,z))-l_3(x,y,l_2(w,z))=0\;.$
\end{enumerate}
\end{prop}

\begin{prop}\label{2term_Loday_morphism} An \emph{infinity morphism between 2-term Leibniz infinity algebras} $(V,l_1,l_2,l_3)$ and $(W,m_1,m_2,m_3)$ is made up by a linear and a bilinear map $f_1,f_2$ from $V$ to $W$, of degree $|f_1|=0, |f_2|=1$, which verify, for any $x,y,z\in V_0$ and $h\in V_1$,
\begin{enumerate}[(a)]\label{fdg1}
\item\label{morphism_a} $m_1f_1h=f_1l_1h\;,$
\item $m_2(f_1x,f_1y)+m_1f_2(x,y)=f_1l_2(x,y)\;,\label{morphism_b}$
\item \label{morphism_c}$m_2(f_1x,f_1h)=f_1l_2(x,h)-f_2(x,l_1h)\;,$\\
$m_2(f_1h, f_1x) = f_1l_2(h, x) - f_2(l_1h, x)\;,$
\item \label{morphism_d}$m_3(f_1x,f_1y,f_1z)-m_2(f_2(x,y),f_1z)+m_2(f_1x,f_2(y,z))-m_2(f_1y,f_2(x,z))=$\\
$f_1l_3(x,y,z)+f_2(l_2(x,y),z)-f_2(x,l_2(y,z))+f_2(y,l_2(x,z))\;.$
\end{enumerate}
\end{prop}

\begin{cor} The category ${\tt 2Lei_{\infty}}$ of 2-term Leibniz infinity algebras and infinity morphisms is a full subcategory of the category ${\tt Lei_{\infty}}$-${\tt Alg}$ of Leibniz infinity algebras and infinity morphisms.\end{cor}

\subsubsection{From the Kan property to 2-term infinity homotopies and their compositions}\label{KanHomComp}

\begin{defi} A \emph{2-term infinity homotopy} between infinity morphisms $f=(f_1,f_2)$ and $g=(g_1,g_2)$, which act themselves between 2-term Leibniz infinity algebras $(V,l_1,l_2,l_3)$ and $(W,m_1,m_2,m_3)$, is a linear map $\theta_1$ from $V$ to $W$, of degree $|\theta_1|=1$, which verifies, for any $x,y\in V_0$ and $h\in V_1$,
\begin{enumerate}[(a)]\label{fdg2}
\item $g_1x-f_1x=m_1\theta_1x\;,$ \label{homotopy_a}
\item $g_1h-f_1h=\theta_1l_1h\;,$ \label{homotopy_b}
\item \label{homotopy_c} $g_2(x,y)-f_2(x,y)=\theta_1l_2(x,y)-m_2(f_1x,\theta_1y)-m_2(\theta_1x,g_1y)\;.$
\end{enumerate}
\label{HomTheo}\end{defi}

The characterizing relations (a) - (c) of infinity Leibniz homotopies are the correct counterpart of the defining relations of infinity Lie homotopies \cite{BC04}. However, rather than choosing the preceding relations as a mere definition, we deduce them here from the Kan property of $\mathrm{MC}_\bullet(\bar L)$. More precisely,

\begin{thm} There exist surjective maps $S^{\,i}_1$, $i\in\{0,1\},$ from the class $\cal I$ of $\infty$-homotopies for 2-term Leibniz infinity algebras to the class $\cal T$ of $\,2$-term $\infty$-homotopies for 2-term Leibniz infinity algebras.\label{KanHom1}\end{thm}

\begin{rem} The maps $S^{\,i}_1$ preserve the source and the target, i.e. they are surjections from the class ${\cal I}(f,g)$ of $\;\infty$-homotopies from $f$ to $g$, to the class ${\cal T}(f,g)$ of $\;2$-term $\,\infty$-homotopies from $f$ to $g$. In the sequel, we refer to a preimage by $S_1^i$ of an element $\theta_1\in{\cal T}$ as a \emph{lift} of $\theta_1$ by $S_1^i$ .\end{rem}

\begin{proof} Henceforth, we use again the homological version of infinity algebras ($k$-ary bracket of degree $k-2$), as well as the Lada-Stasheff sign convention for the higher Jacobi conditions and the {\small MC} equation. \medskip

Due to the choice of the homological variant of homotopy algebras, $\delta={\cal L}_1$ has degree $-1$. For consistency, differential forms are then viewed as negatively graded; hence, $d:\Omega^{-k}(\Delta^n)\to \Omega^{-k-1}(\Delta^n)$, $k\in\{0,\ldots,n\}$, and $\bar{\cal L}_1=\delta\otimes\mathrm{id}+\mathrm{id}\otimes\, d$ has degree $-1$ as well. Similarly, the degree of the operator $h_n^i$ is now $|h_n^i|=1$. It is moreover easily checked that $L$ cannot contain multilinear maps of nonnegative degree, i.e. that $L=\oplus_{k\ge 0}L_{-k}$. It follows that an element $\bar\alpha\in(L\otimes\Omega^\star(\Delta^n))^{-k}$, $k\ge 0$, reads $$\bar\alpha=\sum\alpha_{-k}\otimes\omega^0+\sum\alpha_{-k+1}\otimes\omega^{-1}+\ldots\;,$$ where the {\small RHS} is a finite sum. For instance, if $n=2$, an element $\bar\alpha$ of degree $-1$ can be decomposed as $$\bar\alpha=\alpha(s,t)\otimes 1+\beta(s,t)\otimes ds+\beta'(s,t)\otimes dt\;,$$ where $(s,t)$ are coordinates of $\Delta^2$ and where $\alpha(s,t)\in L_{-1}[s,t]$ and $\beta(s,t),\beta'(s,t)\in L_0[s,t]$ are polynomial functions in $s,t$ with coefficients in $L_{-1}$ and $L_0$, respectively.\medskip

In the sequel, we evaluate the $L_{\infty}$-structure maps $\bar{\cal L}_i$ of $L\otimes \Omega^\star(\Delta^n)$ mainly on elements of degree $-1$ and $0$, hence we compute the structure maps ${\cal L}_i$ of $L=\mathrm{Hom}_\mathbb{K}(\mathrm{Zin}^c(sV),W)$ on elements $\alpha$ and $\beta$ of degree $-1$ and $0$, respectively. Let
\begin{eqnarray} \alpha=\sum_{p\geq1}\alpha^p\in L\;,~~~|\alpha|=-1\;,\nonumber\\
\beta=\sum_{p\geq1}\beta^p\in L\;,~~~|\beta|=0\;,\nonumber\end{eqnarray}
where $\alpha^p,\beta^p:~(sV)^{\otimes p}\to W$. The point is that the concentration of $V,W$ in degrees $0,1$ entails that almost all components $\alpha^p,\beta^p$ vanish and that all series converge (which explains why the formal application of Getzler's method to the present situation leads to the correct counterpart of the findings of \cite{BC04}). Indeed, the only nonzero components of $\alpha,\beta$ are
\begin{eqnarray} \alpha^1:&& sV_0\to W_0,~~sV_1\to W_1\;,\nonumber\\
\alpha^2:&&(sV_0)^{\otimes 2}\to W_1\;,\nonumber\\
\beta^1:&&sV_0\to W_1\;.\label{CompAlphaBeta}\end{eqnarray}
Similarly, the nonzero components of the nonzero evaluations of the maps ${\cal L}_i$ on $\alpha$-s and $\beta$-s are
\begin{eqnarray}
\mathscr{L}_1(\alpha):&&sV_1\to W_0\;,~~(sV_0)^{\otimes 2} \to W_0\;,~~sV_0\otimes sV_1\to W_1\;,~~(sV_0)^{\otimes 3}\to W_1\;,\nonumber\\
\mathscr{L}_1(\beta):&& sV_0\to W_0\;,~~sV_1\to W_1\;,~~(sV_0)^{\otimes 2}\to W_1\;,\nonumber\\
\mathscr{L}_2(\alpha_1,\alpha_2):&& (sV_0)^{\otimes 2}\to W_0\;,~~sV_0\otimes sV_1\to W_1\;,~~(sV_0)^{\otimes 3}\to W_1\;,\nonumber\\
\mathscr{L}_2(\alpha,\beta):&&(sV_0)^{\otimes 2}\to W_1\;,\nonumber\\
\mathscr{L}_3(\alpha_1,\alpha_2,\alpha_3):&& (sV_0)^{\otimes 3}\to W_1\;,\label{non_zero_maps}\end{eqnarray} see Proposition \ref{proposition_about_lie_inf_on_space_of_homomorphisms}.\medskip

We are now prepared to concretize the iterative construction of ${\cal B}_n^i(\mu,\nu)\in\mathrm{MC}_n(\bar L)$ from $\mu\in\mathrm{MC}(L)$ and $\nu=(\delta+d)\beta$, $\beta\in(L\otimes\Omega^\star(\Delta^n))^0$, $\varepsilon_n^i\beta=0$ (the explicit forms of ${\cal B}_n^i(\mu,\nu)$ for $n=1$ and $n=2$ will be the main ingredients of the proofs of Theorems \ref{KanHom1} and \ref{KomComp2}).\medskip

$\bullet\quad$ Let $\alpha\in{\cal I}(f,g)$, i.e. let $$\alpha\in\mathrm{MC}_1(\bar L)\stackrel{\sim}{\longrightarrow}(\mu,(\delta+d)\beta)\in\mathrm{MC}(L)\times\mathrm{mc}_1^0(\bar L)\;,$$ such that $\varepsilon_1^0\alpha=f$ and $\varepsilon_1^1\alpha=g$. To construct $$\alpha={\cal B}_1^0B_1^0\alpha={\cal B}_1^0(\varepsilon_1^0\alpha,(\delta+d)h_1^0\alpha)=:{\cal B}_1^0(f,(\delta+d)\beta)=:{\cal B}_1^0(\mu,\nu)\;,$$ we start from
\begin{eqnarray} \alpha_0=\mu+(\delta+d)\beta\;.\nonumber\end{eqnarray}
The iteration unfolds as
\begin{eqnarray} \alpha_{k}=\alpha_0-\sum_{j=2}^{\infty}\frac{1}{j!}h_1^0\bar{\mathscr{L}}_j(\alpha_{k-1},\cdots,\alpha_{k-1})\;,\quad k\ge 1\;.\nonumber\end{eqnarray}
Explicitly,
\begin{eqnarray} \alpha_1&=&\mu+(\delta+d)\beta-\frac12h_1^0\bar{\mathscr{L}}_2(\alpha_0,\alpha_0)-\frac1{3!}h_1^0\bar{\mathscr{L}}_3(\alpha_0,\alpha_0,\alpha_0)\nonumber\\
&=&\mu+(\delta+d)\beta-h_1^0\bar{\mathscr{L}}_2(\mu+\delta \beta,d\beta)-\frac1{2}h_1^0\bar{\mathscr{L}}_3(\mu+\delta \beta,\mu+\delta \beta,d\beta)\nonumber\\
&=&\mu+(\delta+d)\beta-h_1^0\bar{\mathscr{L}}_2(\mu+\delta \beta,d\beta)\;.\nonumber\end{eqnarray}
Observe that $\mu+\delta\beta\in L_{-1}[t]$ and $d\beta\in L_0[t]\otimes dt$, that differential forms are concentrated in degrees $0$ and $-1$, that $h_1^0$ annihilates 0-forms, and that the term in $\bar{\cal L}_3$ contains a factor of the type ${\cal L}_3(\alpha_1,\alpha_2,\beta)$ (notation of (\ref{non_zero_maps})), whose components vanish -- see above. 
Analogously,
\begin{eqnarray} \alpha_2&=&\mu+(\delta+d)\beta-h_1^0\bar{\mathscr{L}}_2(\mu+\delta \beta-h_1^0\bar{\mathscr{L}}_2(\mu+\delta \beta,d\beta),d\beta)\nonumber\\
&&-\frac12h_1^0\bar{\mathscr{L}}_3(\mu+\delta \beta-h_1^0\bar{\mathscr{L}}_2(\mu+\delta \beta,d\beta),\mu+\delta \beta-h_1^0\bar{\mathscr{L}}_2(\mu+\delta \beta,d\beta),d\beta)\nonumber\\
&=&\mu+(\delta+d)\beta-h_1^0\bar{\mathscr{L}}_2(\mu+\delta \beta,d\beta)\;.\nonumber\end{eqnarray}
Indeed, the term $h_1^0\bar{\cal L}_2(h_1^0\bar{\cal L}_2(\mu+\delta\beta,d\beta),d\beta)$ contains a factor of the type ${\cal L}_2({\cal L}_2(\alpha,\beta_1),\beta_2)$ (notation of (\ref{non_zero_maps})), and the only nonvanishing component of this factor, as well as of its first internal map ${\cal L}_2(\alpha,\beta_1)$, is the component $(sV_0)^{\otimes 2}\to W_1$ -- which entails, in view of Proposition \ref{proposition_about_lie_inf_on_space_of_homomorphisms}, that the considered term vanishes. Hence, the iteration stabilizes already at its second stage and \begin{equation}\label{HomotDef}\alpha={\cal B}_1^0(\mu,\nu)=\mu+(\delta+d)\beta-h_1^0\bar{\mathscr{L}}_2(\mu+\delta \beta,d\beta)\in\mathrm{MC}_1(\bar L)\;.\end{equation}



Remark first that the integral $h_1^0$ can be evaluated since $\bar{\mathscr{L}}_2(\mu+\delta \beta,d\beta)$ is a total derivative. Indeed, when setting $\beta=\beta_0\otimes P$ (sum understood), $\beta_0\in L_0$ and $P\in\Omega^0(\Delta^1)$, we see that $$\bar{\cal L}_2(\mu,d\beta)={\cal L}_2(\mu,\beta_0)\otimes dP=-d\bar{\cal L}_2(\mu,\beta)\;.$$ As for the term $\bar{\cal L}_2(\delta\beta,d\beta)$, we have \begin{eqnarray} 0 = (\delta+d)\bar{\cal L}_2(\beta,d\beta)=\bar{\mathscr{L}}_2(\delta \beta, d\beta)+\bar{\mathscr{L}}_2(\beta,\delta d\beta)\;,\nonumber\end{eqnarray} since $\bar{\cal L}_1=\delta+d$ is a graded derivation of $\bar{\cal L}_2$ and as $\bar{\cal L}_2(\beta,d\beta)=\bar{\cal L}_2(d\beta,d\beta)=0$. It is now easily checked that \begin{eqnarray} \bar{\mathscr{L}}_2(\delta\beta, d\beta)=-\frac12d\bar{\mathscr{L}}_2(\delta \beta, \beta)\;.\nonumber\end{eqnarray}
Eventually,
\begin{eqnarray} \alpha&=&\mu+(\delta+d)\beta+h_1^0d\bar{\mathscr{L}}_2(\mu,\beta)+\frac12h_1^0d\bar{\mathscr{L}}_2(\delta \beta,\beta)\nonumber\\
&=&\mu+(\delta+d)\beta+\bar{\mathscr{L}}_2(\mu,\beta)+\frac12\bar{\mathscr{L}}_2(\delta \beta,\beta)\;.\nonumber\end{eqnarray} Indeed, it suffices to observe that, for any $\ell_{-1}\otimes P\in L_{-1}\otimes\Omega^0(\Delta^1)$ which vanishes under the action of $\varepsilon_1^0$, we have \begin{eqnarray} h_1^0d(\ell_{-1}\otimes P)=-dh_1^0(\ell_{-1}\otimes P)+\ell_{-1}\otimes P-\varepsilon_1^0(\ell_{-1}\otimes P)=\ell_{-1}\otimes P\;.\nonumber\end{eqnarray}
We are now able to write the components of $g=\varepsilon_1^1\alpha\in L_{-1}$ (see (\ref{CompAlphaBeta})) in terms of $f=\mu$ and $\beta$:
\begin{eqnarray}
&&g^1=\varepsilon_1^1(\mu+(\delta+d)\beta-\bar{\mathscr{L}}_2(\mu,\beta)-\frac12\bar{\mathscr{L}}_2(\delta \beta,\beta))^1=\varepsilon_1^1(f+\delta \beta)^1=f^1+\varepsilon_1^1(\delta \beta)^1\;,\nonumber\\
&&g^2=\varepsilon_1^1(\mu+(\delta+d)\beta-\bar{\mathscr{L}}_2(\mu,\beta)-\frac12\bar{\mathscr{L}}_2(\delta \beta,\beta))^2=
\varepsilon_1^1(f+\delta \beta-\bar{\mathscr{L}}_2(f,\beta)-\frac12\bar{\mathscr{L}}_2(\delta \beta,\beta))^2\;,\nonumber\\
&&g^3=0\;,\label{CharRelHomotPreFin}\end{eqnarray}
where we changed signs according to our sign conventions and remembered that the first component of a morphism of the type ${\cal L}_2(\alpha,\beta)$ (see (\ref{non_zero_maps})) vanishes.\medskip

To obtain a 2-term $\infty$-homotopy $\theta_1\in{\cal T}(f,g)$, it now suffices to further develop the equations (\ref{CharRelHomotPreFin}).

As $$g_1:=g^1s,f_1:=f^1s\in \mathrm{Hom}_\mathbb{K}^{0}(V,W)\;,$$ we evaluate the first equation on $x\in V_0$ and $h\in V_1$. Therefore, we compute $\varepsilon_1^1(\delta\beta)^1s=\delta\beta(1)^1s$ on $x$ and $h$. Since $$\delta\beta(1)={\cal L}_1\beta(1)=m_1\beta(1)+\beta(1) D_V\;,$$ where $D_V\in \mathrm{CoDer}^{-1}(\mathrm{Zin}^c(sV))$, we have $D_V:sV_1\to sV_0\;, sV_0\otimes sV_0\to sV_0\;,\ldots$ Hence, \begin{equation}\label{HomotDefIntermed}\delta\beta(1)^1sx = m_1\,\beta(1)s\,x = m_1\theta_1x\;,\end{equation} where we defined the {\it homotopy parameter} $\theta_1$ by \begin{equation}\label{HomotPara}\theta_1:=\beta(1)s=\beta(1)s-\beta(0)s\;.\end{equation} Similarly, \begin{equation}\label{HomotDefIntermed1}\delta\beta(1)^1sh=\beta(1)D_Vsh=\beta(1)s\,s^{-1}D_Vs\,h=\theta_1l_1h\;.\end{equation} The characterizing equations (a) and (b) follow.

Since $$g_2:=g^2s^2,f_2:=f^2s^2\in\mathrm{Hom}_\mathbb{K}^1(V\otimes V,W)\;,$$ it suffices to evaluate the second equation on $x,y\in V_0$. When computing e.g. $\varepsilon_1^1\bar{\cal L}_2(\delta\beta,\beta)^2s^2(x,y)$, we get $${\cal L}_2(\delta\beta(1),\beta(1))(sx,sy)=m_2(\delta\beta(1)\,sx,\beta(1)\,sy)+m_2(\beta(1)\,sx,\delta\beta(1)\,sy)=$$ \begin{equation}\label{HomotDefIntermed2}m_2(m_1\theta_1x,\theta_1y)+m_2(\theta_1x,m_1\theta_1y)=2m_2(\theta_1x,m_1\theta_1y)\;,\end{equation} in view of Equation (\ref{HomotDefIntermed}) and Relation (b) of Proposition \ref{2term_Loday_in_algebra}. Similarly, \begin{equation}\label{HomotDefIntermed3}\varepsilon_1^1\bar{\mathscr{L}}_2(f,\beta)^2s^2(x,y)=m_2(f_1x,\theta_1y)+m_2(\theta_1x,f_1y)\;.\end{equation} Further, one easily finds \begin{equation}\label{HomotDefIntermed4} \varepsilon_1^1(\delta\beta)^2s^2(x,y)=\theta_1l_2(x,y)\;.\end{equation} When collecting the results (\ref{HomotDefIntermed2}), (\ref{HomotDefIntermed3}), and (\ref{HomotDefIntermed4}), and taking into account Relation (a), we finally obtain the characterizing equation (c).\bigskip

$\;\bullet\quad$ Recall that in the preceding step we started from $\alpha\in {\cal I}(f,g)$, set $\mu=f$, $$\beta=h_1^0\alpha\;,$$ $\nu=(\delta+d)\beta$, defined $$\theta_1=(\beta(1)-\beta(0))s\;,$$ and deduced the characterizing relations $g=f+{\cal E}(f,\beta(1)s)=f+{\cal E}(f,\theta_1)$ of $\theta_1\in{\cal T}(f,g)$ by computing $$\alpha={\cal B}^0_1(\mu,\nu)=\mu+(\delta+d)\beta+\bar{\mathscr{L}}_2(\mu,\beta)+\frac12\bar{\mathscr{L}}_2(\delta \beta,\beta)$$ at $1$. Let us mention that instead of defining the map $S_1^0:{\cal I}\ni\alpha\mapsto \theta_1\in {\cal T}$, we can consider the similarly defined map $S_1^1$.\medskip

To prove surjectivity of $S_1^i$, let $\theta_1\in{\cal T}(f,g)$ and set $\beta(i)=0$, $i\in\{0,1\}$, and $\beta(1-i)=(-1)^{i}\theta_1s^{-1}$. Note that by construction $\theta_1=(\beta(1)-\beta(0))s$. Use now Renshaw's method \cite{Sul77} to extend $\beta(0)$ and $\beta(1)$ to some $\beta\in L_0\otimes\Omega^0(\Delta^1)$, set $$\mu=(1-i)f+ig\quad\text{and}\quad\nu=(\delta+d)\beta\;,$$ and construct \begin{equation}\label{preimage}\alpha={\cal B}_1^i(\mu,\nu)\in\mathrm{MC}_1(\bar L)\;.\end{equation} If $i=0$, then $$\alpha(0)=\varepsilon_1^0\alpha=\mu=f\quad\text{and}\quad\alpha(1)=({\cal B}_1^0(\mu,\nu))(1)=f+{\cal E}(f,\theta_1)=g\;,$$ in view of the characterizing relations (a)-(c) of $\theta_1$. If $i=1$, one has also $\alpha(1)=g$ and $\alpha(0)=g+{\cal E}(g,-\theta_1)=f$, but to obtain the latter result, the characterizing equations (a)-(c), as well as Equation (b) of Proposition \ref{2term_Loday_in_algebra} are needed. To determine the image of $\alpha\in{\cal I}(f,g)$ by $S_1^i$, one first computes $h_1^i\alpha$, which, since $h_1^i$ sends 0-forms to 0, is equal to $$h_1^i(\delta+d)\beta=-(\delta+d)h_1^i\beta+\beta-\varepsilon_1^i\beta=\beta\;,$$ then one gets $$S_1^i\alpha=(\beta(1)-\beta(0))s=\theta_1\;,$$ which completes the proof.\end{proof}

\begin{thm}[Definition] If $\;\theta_1:f\Rightarrow g$, $\tau_1:g\Rightarrow h$ are 2-term $\infty$-homotopies between infinity morphisms $f,g,h:V\to W$, the vertical composite $\tau_1\circ_1 \theta_1$ is given by $\tau_1+\theta_1$.\label{KomComp2}\end{thm}

We will actually lift $\theta_1,\tau_1\in{\cal T}$ to $\alpha',\alpha''\in\mathrm{MC}_1(\bar L)$ (which involves choices), then compose these lifts in the infinity groupoid $\mathrm{MC}_\bullet(\bar L)$ (which is not a well-defined operation), and finally project the result back to ${\cal T}$ (despite all the intermediate choices, the final result will turn out to be well-defined).

\begin{proof} Let now $n=2$, take $\mu\in\textrm{MC}(L)$ and $\nu=(\delta+d)\beta\in\mathrm{mc}_2^1(\bar L)$, then construct $\alpha={\cal B}_2^1(\mu,\nu)$. The computation is similar to that in the 1-dimensional case and gives the same result:
\begin{eqnarray} \alpha=\mu+(\delta+d)\beta+\bar{\mathscr{L}}_2(\mu,\beta)+\frac12\bar{\mathscr{L}}_2(\delta \beta,\beta)\;.\label{HomotCompIntermed}\end{eqnarray}

To obtain $\tau_1\circ_1\theta_1$, proceed as in (\ref{preimage}) and lift $\theta_1$ (resp., $\tau_1$) to $$\alpha':={\cal B}_1^1(g,(\delta+d)\beta')\in{\cal I}(f,g)\subset\mathrm{MC}_1(\bar L)\quad (\text{resp.,}\;\; \alpha'':={\cal B}_1^0(g,(\delta+d)\beta'')\in{\cal I}(g,h)\subset\mathrm{MC}_1(\bar L))\;,$$ where $$\beta'(0)=-\theta_1s^{-1}\;\text{and}\;\beta'(1)=0\quad (\text{resp.,}\;\; \beta''(0)=0\;\text{and}\;\beta''(1)=\tau_1s^{-1})\;.$$ As mentioned above, we have by construction \begin{equation}\label{Thetas} \theta_1=(\beta'(1)-\beta'(0))s\quad(\text{resp.,}\;\;\tau_1=(\beta''(1)-\beta''(0))s)\;.\end{equation}

If we view $\alpha'$ (resp., $\alpha''$) as defined on the face $01$ (resp., $12$) of $\Delta^2$, the equation $\varepsilon_1^1\alpha'=\varepsilon_1^0\alpha''=g$ reads $\varepsilon_2^1\alpha'=\varepsilon_2^1\alpha''=g=:\mu$. This means that $$(\alpha',\alpha'')\in\mathrm{SSet}(\Lambda^1[2],\mathrm{MC}_\bullet(\bar L))\;.$$ We now follow the extension square (\ref{commute_square}). The left arrow leads to $$(\mu;(\delta+d)\beta',(\delta+d)\beta'')\in\mathrm{SSet}(\Lambda^1[2],\mathrm{MC}(L)\times\mathrm{mc}_\bullet(\bar L))\;,$$ the bottom arrow to $$(\mu,(\delta+d)\beta)\in\mathrm{MC}(L)\times\mathrm{mc}_2^1(\bar L)\;,$$ where $\beta$ is {\it any extension} of $(\beta',\beta'')$ to $\Delta^2$, and the right arrow provides $\alpha\in\mathrm{MC}_2(\bar L)$ given by Equation (\ref{HomotCompIntermed}). From Subsection \ref{InftyCatComp}, we know that all composites of $\alpha',\alpha''$ are $\infty$-2-homotopic and that a possible composite is obtained by restricting $\alpha$ to $02.$ This restriction $(-)|_{02}$ is given by the ${\tt DGCA}$-map $d_1^2$. Hence, we get $$\alpha|_{02}=\mu+(\delta+d)\beta|_{02}+\bar{\mathscr{L}}_2(\mu,\beta|_{02})+\frac12\bar{\mathscr{L}}_2(\delta \beta|_{02},\beta|_{02})\in{\cal I}(f,h)\subset\mathrm{MC}_1(\bar L)\;.$$

We now choose the projection $S_1^0\alpha|_{02}\in{\cal T}(f,h)$ of the composite-candidate of the chosen lifts of $\theta_1,\tau_1$, as composite $\tau_1\circ_1\theta_1$. Since $$h_1^0\alpha|_{02}=-(\delta+d)h_1^0\beta|_{02}+\beta|_{02}-\beta(0)=\beta|_{02}-\beta(0)\;,$$ we get
$$ S_1^0\alpha|_{02}=(\beta|_{02}(2)-\beta(0)-\beta|_{02}(0)+\beta(0))s=(\beta(2)-\beta(0))s=$$ $$(\beta''(2)-\beta''(1))s+(\beta'(1)-\beta'(0))s=\tau_1+\theta_1\;,$$ in view of (\ref{Thetas}). Hence, by definition, the vertical composite of $\theta_1\in{\cal T}(f,g)$ and $\tau_1\in{\cal T}(g,h)$ is given by \begin{equation}\label{VertComp}\tau_1\circ_1\theta_1=\tau_1+\theta_1\in{\cal T}(f,h)\;.\end{equation}
\end{proof}

\begin{rem} The composition of elements of ${\cal I}=\mathrm{MC}_1(\bar L)$ in the infinity groupoid $\mathrm{MC}_\bullet(\bar L)$, which is defined and associative only up to higher morphisms, projects to a well-defined and associative vertical composition in ${\cal T}$.\end{rem}

Just as for concordances, horizontal composition of $\infty$-homotopies is without problems. The horizontal composite of $\theta_1\in{\cal T}(f,g)$ and $\tau_1\in{\cal T}(f',g')$, where $f,g:V\to W$ and $f',g':W\to X$ act between 2-term Leibniz infinity algebras, is defined by \begin{equation}\label{HorzComp} \tau_1\circ_0\theta_1=g'_1\theta_1+\tau_1f_1 = f'_1\theta_1+\tau_1g_1\;.\end{equation} The two definitions coincide, since $\theta_1,\tau_1$ are chain homotopies between the chain maps $f,g$ and $f',g'$, respectively, see Definition \ref{HomTheo}, Relations (\ref{homotopy_a}) and (\ref{homotopy_b}). The identity associated to a 2-term $\infty$-morphism is just the zero-map. As announced in \cite{BC04} (in the Lie case and without information about composition), we have the

\begin{prop} There is a strict 2-category ${\tt 2Lei_{\infty}}$-${\tt Alg}$ of 2-term Leibniz infinity algebras.\end{prop}

\subsection{2-Category of categorified Leibniz algebras}\label{KPQSection6}

\subsubsection{Category of Leibniz 2-algebras}

Leibniz 2-algebras are categorified Leibniz structures on a categorified vector space. More precisely,

\begin{defi}
A \emph{Leibniz 2-algebra} $(L,[-,-],{\mathbf J})$ is a linear category $L$ equipped with
\begin{enumerate}
\item a \emph{bracket} $[-,-]$, i.e. a bilinear functor $[-,-]:L\times L\rightarrow L$, and
\item a \emph{Jacobiator} $\mathbf{J}$, i.e. a trilinear natural transformation
$$
\mathbf{J}_{x,y,z}:[x,[y,z]]\rightarrow[[x,y],z]+[y,[x,z]],\quad x,y,z\in L_0,
$$
\end{enumerate}
which verify, for any $w,x,y,z\in L_0,$ the \emph{Jacobiator identity}
\begin{equation}\label{Jacobiator_diagramm}
\hspace{-5mm}
\xymatrix@C=4pc@R=4pc{
&[w,[x,[y,z]]]\ar@{->}
[dr]_-{\mathbf{1}}\ar@{->}[dl]^-{[\mathbf{1}_w,\mathbf{J}_{x,y,z}]}\\
[w,[[x,y],z]]+[w,[y,[x,z]]]\ar@{->}
[d]^-{\mathbf{J}_{w,[x,y],z}+\mathbf{J}_{w,y,[x,z]}}&&
[w,[x,[y,z]]]\ar@{->}
[d]_-{\mathbf{J}_{w,x,[y,z]}}\\
{\begin{gathered}[t]
[[w,[x,y]],z]+[[x,y],[w,z]]\\
+[[w,y],[x,z]]+[y,[w,[x,z]]]
\end{gathered}}\ar@{->}
[d]^-{\mathbf{1}+[\mathbf{1}_y,\mathbf{J}_{w,x,z}]}
&&[[w,x],[y,z]]+[x,[w,[y,z]]]\ar@{->}
[d]_-{1+[\mathbf{1}_x,\mathbf{J}_{w,y,z}]}\\
{\begin{gathered}[t]
[[w,[x,y]],z]+[[x,y],[w,z]]\\
+[[w,y],[x,z]]+[y,[[w,x],z]]\\
+[y,[x,[w,z]]]
\end{gathered}}\ar@{->}
[dr]^-{[\mathbf{J}_{w,x,y},\mathbf{1}_z]}&&
{\begin{gathered}[t]
[[w,x],[y,z]]+[x,[[w,y],z]]\\
+[x,[y,[w,z]]]
\end{gathered}}\ar@{->}
[ld]^-{\mathbf{J}_{[w,x],y,z}+\mathbf{J}_{x,[w,y],z}+\mathbf{J}_{x,y,[w,z]}}\\
&{\begin{gathered}[t]
[[[w,x],y],z]+[[x,[w,y]],z]\\
+[[x,y],[w,z]]+[[w,y],[x,z]]\\
+[y,[[w,x],z]]+[y,[x,[w,z]]]
\end{gathered}}
}
\end{equation}
\end{defi}

The Jacobiator identity is a coherence law that should be thought of as a higher Jacobi identity for the Jacobiator.\medskip

The preceding hierarchy `category, functor, natural transformation' together with the coherence law is entirely similar to the known hierarchy `linear, bilinear, trilinear maps $l_1,l_2,l_3$' with the $L_{\infty}$-conditions (a)-(e). More precisely,

\begin{prop}
There is a 1-to-1 correspondence between Leibniz 2-algebras and 2-term Leibniz infinity algebras.
\end{prop}

This proposition was proved in the Lie case in \cite{BC04} and announced for the Leibniz case in \cite{SL10}. A generalization of the latter correspondence to Lie 3-algebras and 3-term Lie infinity algebras can be found in \cite{KMP11}. This paper allows to understand that the correspondence between higher categorified algebras and truncated infinity algebras is subject to cohomological conditions, and to see how the coherence law corresponds to the last nontrivial $L_{\infty}$-condition.\medskip

The definition of Leibniz 2-algebra morphisms is God-given: such a morphism must be a functor that respects the bracket up to a natural transformation, which in turn respects the Jacobiator. More precisely,

\begin{defi}
Let $(L,[-,-],\mathbf{J})$ and $(L',[-,-]',\mathbf{J}')$ be Leibniz 2-algebras $($in the following, we write $[-,-], \mathbf{J}$ instead of $\;[-,-]',\mathbf{J}'$$)$. A \emph{morphism $(F,\mathbf{F})$ of Leibniz 2-algebras} from $L$ to $L'$ consists of
\begin{enumerate}
\item a linear functor $F:L\rightarrow L'$, and
\item a bilinear natural transformation $$\mathbf{F}_{x,y}:[Fx,Fy]\rightarrow F[x,y],\quad x,y\in L_0\;,$$\end{enumerate}
which make the following diagram commute
\begin{equation}\label{diagram_for_morphism}
\xymatrix@C=10pc@R=3pc{
[Fx,[Fy,Fz]]\ar@{->}[d]^{[\mathbf{1}_x,\mathbf{F}_{y,z}]}\ar@{->}[r]^-{\mathbf{J}_{Fx,Fy,Fz}}&[[Fx,Fy],Fz]+[Fy,[Fx,Fz]]\ar@{->}[d]^{[\mathbf{F}_{x,y},\mathbf{1}_z]+[\mathbf{1}_y,\mathbf{F}_{x,z}]}\\
[Fx,F[y,z]]\ar@{->}[d]^{\mathbf{F}_{x,[y,z]}}&[F[x,y],Fz]+[Fy,F[x,z]]\ar@{->}[d]^{\mathbf{F}_{[x,y],z}+\mathbf{F}_{y,[x,z]}}\\
F[x,[y,z]]\ar@{->}[r]^{F\mathbf{J}_{x,y,z}}&F[[x,y],z]+F[y,[x,z]]
}
\end{equation}
\end{defi}
\begin{prop}
\vspace{2mm}
There is a 1-to-1 correspondence between Leibniz 2-algebra morphisms and 2-term Leibniz infinity algebra morphisms.
\end{prop}

For a proof, see \cite{BC04} and \cite {SL10}.\medskip

Composition of Leibniz 2-algebra morphisms $(F,\mathbf{F})$ is naturally given by composition of functors and whiskering of functors and natural transformations.

\begin{prop} There is a category ${\tt Lei2}$ of Leibniz 2-algebras and morphisms.\end{prop}

\subsubsection{2-morphisms and their compositions}

The definition of a 2-morphism is canonical:

\begin{defi}
Let $(F,\mathbf{F}),(G,\mathbf{G})$ be Leibniz 2-algebra morphisms from $L$ to $L'$. A \emph{Leibniz 2-algebra 2-morphism} $\boldsymbol\theta$ from $F$ to $G$ is a linear natural transformation $\boldsymbol{\theta}: F\Rightarrow G$, such that, for any $x,y\in L_0,$ the following diagram commutes
\begin{equation}\label{diagram_for_2_morphism}
\xymatrix@C=2pc@R=2pc{
[Fx,Fy]\ar@{->}[rr]^-{\mathbf{F}_{x,y}}\ar@{->}[dd]^-{[\boldsymbol{\theta}_x,\boldsymbol{\theta}_y]}&&F[x,y]\ar@{->}[dd]^-{\boldsymbol{\theta}_{[x,y]}}\\
\\
[Gx,Gy]\ar@{->}[rr]^-{\mathbf{G}_{x,y}}&&G[x,y]\\
}
\end{equation}
\end{defi}

\begin{thm}
There is a 1:1 correspondence between Leibniz 2-algebra 2-morphisms and 2-term Leibniz $\infty$-homotopies.
\end{thm}

Horizontal and vertical compositions of Leibniz 2-algebra 2-morphisms are those of natural transformations.

\begin{prop} There is a strict 2-category ${\tt Lei2Alg}$ of Leibniz 2-algebras.\end{prop}

\begin{cor} The 2-categories ${\tt 2Lei_{\infty}}$-${\tt Alg}$ and ${\tt Lei2Alg}$ are 2-equivalent.\end{cor}

\newpage
\section{Appendix}
\subsection{Leibniz infinity algebra}\label{LeibnizInftyAlgebra}
\begin{defi}
{\it $\mathrm{Lei}_\infty$ algebra on a graded vector space $V$} is given by the family of multilinear maps
$l_i: V^{\otimes i}\to V$ of degrees $(i-2)$, such that for any $n>0$ the {\it higher Jacobi identities} holds:
\begin{equation}\label{LodayInfinityAlgebraIdenities}
\begin{array}{l}
\displaystyle\sum\limits_{i+j=n+1}\sum\limits_{\substack{j\leqslant k\leqslant i+j-1}}\sum\limits_{\sigma\in Sh(k-j,j-1)}\varepsilon(\sigma)\cdot \mathrm{sign}(\sigma)\cdot(-1)^{(i-k+j)(j-1)}\cdot(-1)^{j(v_{\sigma(1)}+...+v_{\sigma(k-j)})}\times \\[4ex]
\times{l_i(v_{\sigma(1)},...,v_{\sigma(k-j)},l_j(v_{\sigma(k+1-j)},...,v_{\sigma(k-1)},v_k),v_{k+1},...,v_{i+j-1})}=0,\\
\end{array}
\end{equation}
\end{defi}\bigskip

$\mathrm{Lei}_\infty$ algebras are nonsymmetric versions of $L_\infty$ algebras, that is if we additionally impose the graded antisymmetry condition on the higher brackets $\{l_i\}$ we will get the $L_\infty$ algebra.
\begin{thm}
$\mathrm{Lei}_\infty$ algebra over a finite dimensional graded vector space $V$ is given by the differential
$d$ on the quasi-free DGZA  $\mathrm{Zin}(s^{-1}V^*)$ or dually by the codifferential $D=d^*$ on the  coalgebra  $\mathrm{Zin}^c(sV)$ . And the condition
$d^2=0$ (dually $D^2=0$) encodes the bunch of higher Jacobi identities (\ref{LodayInfinityAlgebraIdenities}).
\end{thm}
\begin{proof}
Consider an arbitrary finite dimensional graded vector space $W$ and the differential $d$ on the quasi-free DGZA $\mathrm{Zin}(W)$. The differential $d$ is given by its action on the generators, that is by the action on the vector space $W$. By applying the Leibniz rule one can get the action of the differential on the whole space $\mathrm{Zin}(W)$ knowing only its action on $W$. We define the components of the differential:
$$
d_k:W\to W^{\otimes k}\subset\mathrm{Zin}(W),\ d|_W=d_1+d_2+...
$$
Then for any element $w_1...w_p\in W^{\otimes p}\subset\mathrm{Zin}(W)$ we have
\begin{equation}\label{derivation_formula}
\begin{array}{l}
d(w_1...w_p)=d((((w_1\cdot w_2)\cdot w_3)...)\cdot w_p)=\sum\limits_{i=1}^p(-1)^{|w_1|+...+|w_{i-1}|}((((w_1\cdot w_2)...dw_i)\cdot w_{i+1})...)\cdot w_p)=\\
=\sum\limits_{i=1}^p(-1)^{|w_1|+...+|w_{i-1}|}(((((w_1w_2...w_{i-1})\cdot \sum\limits_{k>0}d_kw_i)\cdot w_{i+1})...)\cdot w_p)=\\
=\sum\limits_{i=1}^p\sum\limits_{k>0}\ \sum\limits_{\sigma\in Sh(i-1,k-1)}(\sigma^{-1}\otimes \mathrm{id}^{(p-i+1)})(\mathrm{id}^{\otimes (i-1)}\otimes d_k\otimes \mathrm{id}^{\otimes(p-i)})(w_1w_2...w_p).
\end{array}
\end{equation}

If for an arbitrary -1 degree derivation $d$ on $\mathrm{Zin}(W)$ the action of $d^2$ on generators is zero then it is zero on the whole space $\mathrm{Zin}(W)$, and therefore the derivation becomes a differential. It follows from the following identity:
$$
d^2(w\cdot v)=d^2w\cdot v+w\cdot d^2 v=0.
$$
The condition that $d^2=0$ reads as follows:
$$
\begin{array}{l}\label{}
d(dw)=\sum\limits_{p>0}dd_pw\stackrel{(\ref{derivation_formula})}{=}\sum\limits_{p>0}\hspace{1mm}\sum\limits_{\substack{{0<i\leqslant p}\\{k>0}}}\ \sum\limits_{\sigma\in Sh(i-1,k-1)}(\sigma^{-1}\otimes \mathrm{id}^{(p-i+1)})(\mathrm{id}^{\otimes (i-1)}\otimes d_k\otimes \mathrm{id}^{\otimes(p-i)})d_pw\\
=\sum\limits_{n>0}\hspace{1mm}\sum\limits_{\substack{0<i\leqslant p\\k+p-1=n}}\sum\limits_{\sigma\in Sh(i-1,k-1)}(\sigma^{-1}\otimes \mathrm{id}^{(p-i+1)})(\mathrm{id}^{\otimes (i-1)}\otimes d_k\otimes \mathrm{id}^{\otimes(p-i)})d_pw.
\end{array}
$$
In the dual language, the transposed map $D=d^*$ is a codifferential on the coalgebra ${\mathrm{Zin}}^c(W^*)$. And the transposition of the last identity will encode that $D^2=0$:
\begin{equation}\label{equationForCodiffereintalLoday}
\sum\limits_{n>0}\hspace{1mm}\sum\limits_{\substack{0<i\leqslant p\\k+p-1=n}}\sum\limits_{\sigma\in Sh(i-1,k-1)}D_p(\mathrm{id}^{\otimes (i-1)}\otimes D_k\otimes \mathrm{id}^{\otimes(p-i)})(\sigma\otimes \mathrm{id}^{(p-i+1)})=0,
\end{equation}
where the transposed maps $D_p=d^*_p: W^{*\otimes p}\to W^*$ called the corestrictions of the codifferential $D$.\medskip

In equation (\ref{equationForCodiffereintalLoday}) the weight of the operator inside the sum $\sum\limits_{n>0}$ is equal to $2-p-k=1-n$, so it depends only on $n$, i.e. the last equation splits into the series of equations:
\begin{equation*}\label{operatorForLodayInfinity}
(D^2)_{n}=\sum\limits_{\substack{0<i\leqslant p\\k+p-1=n}}\sum\limits_{\sigma\in Sh(i-1,k-1)}D_p(\mathrm{id}^{\otimes (i-1)}\otimes D_k\otimes \mathrm{id}^{\otimes(p-i)})(\sigma\otimes \mathrm{id}^{(p-i+1)})=0,\quad\mbox{where }n>0.
\end{equation*}
So the condition $D^2=0$ on the coalgebra $\mathrm{Zin}^c(sV)$ reads as follows:
$$
\sum\limits_{\substack{0<i\leqslant p\\k+p-1=n}}\sum\limits_{\sigma\in Sh(i-1,k-1)}D_p(\mathrm{id}^{\otimes (i-1)}\otimes D_k\otimes \mathrm{id}^{\otimes(p-i)})(\sigma\otimes \mathrm{id}^{(p-i+1)})(sv_1...sv_n)=0,\quad n>0.
$$
Now we insert identities of the type $(-1)^{\frac{i(i-1)}{2}}s^{\otimes i}(s^{-1})^{\otimes i}=\mathrm{id}^{\otimes i}$ in two places:
\begin{equation}\label{LodayInfinityBeforeEvaluation}
\begin{array}{l}
\sum\limits_{\substack{0<i\leqslant p\\k+p-1=n}}\overbrace{s^{-1}D_p s^{\otimes p}}^{l_p}(-1)^{\frac{p(p-1)}{2}}\overbrace{(s^{-1})^{\otimes{p}}\left(\mathrm{id}^{\otimes (i-1)}\otimes D_k\otimes \mathrm{id}^{\otimes(p-i)}\right)s^{\otimes n}}^{\pm \mathrm{id}^{\otimes(i-1)}\otimes l_k\otimes \mathrm{id}^{\otimes (p-i)}}(-1)^{\frac{n(n-1)}{2}}\circ\\
\circ\underbrace{(s^{-1})^{\otimes n}\sum\limits_{\sigma\in Sh(i-1,k-1)}\left(\sigma\otimes \mathrm{id}^{(p-i+1)}\right)s^{\otimes n}}_{\Big(\sum\limits_{\sigma\in Sh(i-1,k-1)}\pm\sigma\otimes \mathrm{id}^{(p-i+1)}\Big)}=0.
\end{array}
\end{equation}
The precise signs in the formula above are the following:
\begin{enumerate}
\item $s^{-1}D_p s^{\otimes p}=l_p,$
\item $(s^{-1})^{\otimes{p}}\left(\mathrm{id}^{\otimes (i-1)}\otimes D_k\otimes \mathrm{id}^{\otimes(p-i)}\right)s^{\otimes n}=(-1)^{[(i-1)+\frac{p(p-1)}{2}+(p-i)k]}\cdot \mathrm{id}^{\otimes(i-1)}\otimes l_k\otimes \mathrm{id}^{\otimes (p-i)},$
\item $(s^{-1})^{\otimes n}\sum\limits_{\sigma\in Sh(i-1,k-1)}\left(\sigma\otimes \mathrm{id}^{(p-i+1)}\right)s^{\otimes n}=\\[3mm]=\sum\limits_{\sigma\in Sh(i-1,k-1)}\left((-1)^{\frac{n(n-1)}{2}}\cdot sign(\sigma)\cdot\sigma\otimes \mathrm{id}^{(p-i+1)}\right)$.
\end{enumerate}

Now we change the indices of summation $p\rightarrow i$, $k\rightarrow j$, $i\rightarrow(k-j+1)$, to be better in keeping with the results in the literature,
we get for each $n>0$:
\begin{equation}
\begin{array}{l}
\sum\limits_{i+j=n+1}\sum\limits_{\substack{j\leqslant k\leqslant i+j-1}}\sum\limits_{\sigma\in Sh(k-j,j-1)}(-1)^{(i-k+j)(j-1)}\cdot \mathrm{sign}(\sigma)\\
 \hspace{4cm}l_i\left(\mathrm{id}^{\otimes(k-j)}\otimes l_j\otimes \mathrm{id}^{\otimes(i-k+j-1)}\right)\left(\sigma\otimes \mathrm{id}^{\otimes(i-k+j)}\right)=0
\end{array}
\end{equation}
One can expand out the condensed tensor notation and get (\ref{LodayInfinityAlgebraIdenities}).
\end{proof}
\subsection{Leibniz infinity algebra morphism}\label{LeinbizInftyAlgebraMorph}
\begin{defi}
The morphism between two $\mathrm{Lei}_\infty$ algebras over $V$ and $W$ is given by the family of multilinear maps $\varphi_i:V^{\otimes i}\to W$ of degree $(i-1)$ which satisfy the following identities:\\
\begin{equation}\label{LodayInfinityAlgebraMorphismIdenities}
\begin{array}{l}
\sum\limits_{i=1}^n\hspace{1mm}\sum\limits_{\substack{k_1+...+k_i=n}}\hspace{1mm}\sum\limits_{\sigma\in Hsh(k_1,...,k_i)}(-1)^{\sum\limits_{r=1}^{i-1}(i-r)k_r+\frac{i(i-1)}{2}}\cdot(-1)^{
\sum\limits_{r=2}^i(k_r-1)(v_{\sigma(1)}+...+v_{\sigma(k_1+...+k_{r-1})})}\times\\
\times \varepsilon(\sigma)\cdot \mathrm{sign}(\sigma)\cdot l_i\left(\varphi_{k_1}(v_{\sigma(1)},...,v_{\sigma(k_1)}),...,\varphi_{k_i}(v_{\sigma(k_1+...+k_{i-1}+1)},...,v_{\sigma(k_1+...+k_i)})\right)=\\[0.5cm]
=\sum\limits_{i+j=n+1}\hspace{1mm}\sum\limits_{j\leqslant k \leqslant i+j-1}\sum\limits_{\sigma\in Sh(k-j,j-1)}(-1)^{k+(i-k+j)j}\cdot(-1)^{j(v_{\sigma(1)}+...+v_{\sigma(k-j)})}\cdot\varepsilon(\sigma)\cdot \mathrm{sign}(\sigma)\times\\
\times{\varphi_i(v_{\sigma(1)},...,v_{\sigma(k-j)},l_j(v_{\sigma(k+1-j)},...,v_{\sigma(k-1)},v_k),v_{k+1},...,v_{i+j-1})}.
\end{array}
\end{equation}
\end{defi}
\begin{thm}
$\mathrm{Lei}_\infty$-algebra morphism between $\mathrm{Lei}_\infty$-algebras over $V$ and $W$ is given by the differential graded algebra homomorphism $f: \mathrm{Zin}(s^{-1}W^*)\to \mathrm{Zin}(s^{-1}V^*)$, or dually by the graded differential coalgebra homomorphism $F=f^*:\mathrm{Zin}^c(sV)\to \mathrm{Zin}^c(sW)$. The morphism condition $fd-df=0$ (or dually $FD-DF=0$) encodes the bunch of higher identities~(\ref{LodayInfinityAlgebraMorphismIdenities}).
\end{thm}
\begin{proof}
Consider an arbitrary graded vector spaces $U$ and $U'$. An arbitrary  algebra homomorphism $f:\mathrm{Zin}(U)\to \mathrm{Zin}(U')$ is given by its action on generators:
\begin{equation}
\begin{array}{l}\label{morphism_formula}
f(u_1...u_p)=f((((u_1\cdot u_2)\cdot u_3)...)\cdot u_p)=((((fu_1\cdot fu_2)\cdot fu_3)...)\cdot fu_p)=\\[2mm]

=\sum\limits_{k_1+...+k_p=p}^{\infty}((((f_{k_1}u_1\cdot f_{k_2}u_2)\cdot f_{k_3}u_3)...)\cdot f_{k_p}u_p)=\\
=\sum\limits_{k_1+...+k_p=p}^{\infty}\left(\sum\limits_{\sigma\in Sh(k_1,k_{2}-1)}\sigma^{-1}\otimes \mathrm{id}^{\otimes(1+k_3+...+k_p)}\right)...\left(\sum\limits_{\sigma\in Sh(k_1+...+k_{p-1},k_p-1)}\sigma^{-1}\otimes \mathrm{id}\right)\circ\\[7mm]
\circ\left(f_{k_1}\otimes...\otimes f_{k_p}\right)(u_1...u_p)=\\[2mm]
=\sum\limits_{k_1+...+k_p=p}^{\infty}\ \sum\limits_{\sigma\in Hsh(k_1,...,k_p)}\sigma^{-1}\left(f_{k_1}\otimes...\otimes f_{k_p}\right)(u_1...u_p),
\end{array}
\end{equation}
where the the maps $f_k:U\to U'^{\otimes k}$ are the components of the restriction of the homomorphism f on the $U$:
$$
f|_U=f_1+f_2+...
$$
If for the algebra morphism $f$ the condition $df-fd=0$ holds on the generators then it also holds on the whole space $\mathrm{Zin}(U)$, it follows from the fact that
$$
(fd-df)(u\cdot w)=(fd-df)u\cdot w +(-1)^{\overline{u}}u\cdot (fd-df)w.
$$
Then the condition that the algebra homomorphism $f$ is a {\it differential} algebra homomorphism reads as follows:
\begin{equation}\label{(Fd-dF)_formula}
\begin{array}{l}
 (fd-df)u=\sum\limits_{p>0}fd_pu-\sum\limits_{p>0}df_pu\stackrel{(\ref{derivation_formula})\&(\ref{morphism_formula})}{=}\\
 =\sum\limits_{n>0}\left[\sum\limits_{\substack{k_1+...+k_p=n\\0<p\leqslant n}}\ \sum\limits_{\sigma\in Hsh(k_1,...,k_p)}\sigma^{-1}\left(f_{k_1}\otimes...\otimes f_{k_p}\right)d_pu-\right.\\
-\left.\sum\limits_{\substack{0<i\leqslant p\\k+p-1=n}}\ \sum\limits_{\sigma\in Sh(i-1,k-1)}(\sigma^{-1}\otimes \mathrm{id}^{(p-i+1)})(\mathrm{id}^{\otimes (i-1)}\otimes d_k\otimes \mathrm{id}^{\otimes(p-i)})f_pu\right]=0.
  \end{array}
\end{equation}
In the dual language, the transposed map $F=f^*$ is a homomorphism of the coalgebras $F:\mathrm{Zin}^{c}(U'^*)\to\mathrm{Zin}^{c}(U^*)$. And the transposition of the last identity  encodes that $FD-DF=0$:
\begin{equation}\label{equationForZinbielCoalgebraMorphism}
\begin{array}{l}
\sum\limits_{n>0}\left[\sum\limits_{\substack{k_1+...+k_p=n\\0<p\leqslant n}}\ \sum\limits_{\sigma\in Hsh(k_1,...,k_p)}D_p\left(F_{k_1}\otimes...\otimes F_{k_p}\right)\sigma-\right.\\
-\left.\sum\limits_{\substack{0<i\leqslant p\\k+p-1=n}}\ \sum\limits_{\sigma\in Sh(i-1,k-1)}F_p(\mathrm{id}^{\otimes (i-1)}\otimes D_k\otimes \mathrm{id}^{\otimes(p-i)})(\sigma\otimes \mathrm{id}^{(p-i+1)})\right]=0,
\end{array}
\end{equation}
where the transposed maps $F_k=f^*_k: S^k(U'^*)\to U^*$ are called the corestrictions of the morphism  $F$.\medskip

In equation (\ref{equationForZinbielCoalgebraMorphism}) the weight of the operator inside the sum $\sum\limits_{n=1}^{\infty}$ is equal to $1-n$, so it depends only on $n$, that is the last equation splits into the series of equations:
$$
\begin{array}{l}
\sum\limits_{\substack{k_1+...+k_p=n\\0<p\leqslant n}}\ \sum\limits_{\sigma\in Hsh(k_1,...,k_p)}D_p\left(F_{k_1}\otimes...\otimes F_{k_p}\right)\sigma-\\
-\sum\limits_{\substack{0<i\leqslant p\\k+p-1=n}}\ \sum\limits_{\sigma\in Sh(i-1,k-1)}F_p(\mathrm{id}^{\otimes (i-1)}\otimes D_k\otimes \mathrm{id}^{\otimes(p-i)})(\sigma\otimes \mathrm{id}^{(p-i+1)})=0.
\end{array}
$$
Now we insert in certain places the identities of the type $(-1)^{\frac{i(i-1)}{2}}s^{\otimes i}(s^{-1})^{\otimes i}=\mathrm{id}^{\otimes i}$ and get:
\begin{equation}\label{LodayInfinityMorphismBeforeEvaluation}
\begin{array}{l}
\sum\limits_{p=1}^n\hspace{1mm}\sum\limits_{k_1+...+k_p=n}\overbrace{s^{-1}D_ps^{\otimes p}}^{l_p}(-1)^{\frac{p(p-1)}{2}}\overbrace{(s^{-1})^{\otimes p}\left(F_{k_1}\otimes...\otimes F_{k_p}\right)s^{\otimes n}}^{\pm \varphi_{k_1}\otimes...\otimes \varphi_{k_p}}\circ\\[8mm]
\circ(-1)^{\frac{n(n-1)}{2}}\overbrace{\sum\limits_{\sigma\in Hsh(k_1,...,k_p)}(s^{-1})^{\otimes n}\sigma s^{\otimes n}}^{\sum\limits_{\sigma\in Hsh(k_1,...,k_p)}\pm\sigma}=\\[8mm]
=\sum\limits_{k+p-1=n}\hspace{1mm}\sum\limits_{0<i\leqslant p}\overbrace{s^{-1}F_ps^{\otimes p}}^{\varphi_p}(-1)^{\frac{p(p-1)}{2}}\overbrace{(s^{-1})^{\otimes p}(\mathrm{id}^{\otimes (i-1)}\otimes D_k\otimes \mathrm{id}^{\otimes(p-i)})s^{\otimes n}}^{\pm(\mathrm{id}^{\otimes (i-1)}\otimes l_k\otimes \mathrm{id}^{\otimes(p-i)})}\circ\\[8mm]
\circ(-1)^{\frac{n(n-1)}{2}}\overbrace{\sum\limits_{\sigma\in Sh(i-1,k-1)}(s^{-1})^{\otimes n}(\sigma\otimes \mathrm{id}^{(p-i+1)})s^{\otimes n}}^{\sum\limits_{\sigma\in Sh(i-1,k-1)}\pm\sigma\otimes \mathrm{id}^{(p-i+1)}}.
\end{array}
\end{equation}
The precise signs in the formula above are the following:
\begin{enumerate}
\item $s^{-1}D_p\left(s\right)^{\otimes p}=l_p$,
\item $(s^{-1})^{\otimes p}\left(F_{k_1}\otimes...\otimes F_{k_p}\right)s^{\otimes n}=(-1)^{\left[\sum\limits_{r=1}^{p-1}(p-r)k_r\right]}\cdot \varphi_{k_1}\otimes...\otimes \varphi_{k_p}$,
\item $\sum\limits_{\sigma\in Hsh(k_1,...,k_p)}(s^{-1})^{\otimes n}\sigma s^{\otimes n}=\sum\limits_{\sigma\in Hsh(k_1,...,k_p)}(-1)^{\frac{n(n-1)}{2}}\cdot \mathrm{sign}(\sigma)\cdot\sigma$,
\item $(s^{-1})F_ps^{\otimes p}=\varphi_p$,
\item $(s^{-1})^{\otimes{p}}\left(\mathrm{id}^{\otimes (i-1)}\otimes D_k\otimes \mathrm{id}^{\otimes(p-i)}\right)s^{\otimes n}=(-1)^{[(i-1)+\frac{p(p-1)}{2}+(p-i)k]}\cdot \mathrm{id}^{\otimes(i-1)}\otimes l_k\otimes \mathrm{id}^{\otimes (p-i)}$,
\item $\sum\limits_{\sigma\in Sh(i-1,k-1)}(s^{-1})^{\otimes n}\left(\sigma\otimes \mathrm{id}^{(p-i+1)}\right)s^{\otimes n}=\sum\limits_{\sigma\in Sh(i-1,k-1)}\left((-1)^{\frac{n(n-1)}{2}}\cdot \mathrm{sign}(\sigma)\cdot\sigma\otimes \mathrm{id}^{(p-i+1)}\right)$.
\end{enumerate}
We evaluate the operators (\ref{LodayInfinityMorphismBeforeEvaluation}) on the elements $v_1\otimes...\otimes v_n$ and get identities (\ref{LodayInfinityAlgebraMorphismIdenities}).
\end{proof}

\end{document}